\documentclass[12pt]{book}

\usepackage{amssymb}
\usepackage{amsfonts}
\usepackage{amsmath}
\usepackage{amsthm}
\usepackage{graphicx}
\usepackage[boxruled]{algorithm2e}

\def\Rs{{\mathbb{R}^3}}
\def\Rp{{\mathbb{R}^2}}
\def\R{{\mathbb{R}}}
\def\N{{\mathbb{N}}}

\def\bs{{{\bigskip}}}
\def\o{{\omega}} 
\def\O{{\Omega}} 
\def\a{{\alpha}}
\def\b{{\beta}}
\def\g{{\gamma}}
\def\G{{\Gamma}}
\def\d{{\delta}}
\def\D{{\Delta}}
\def\e{{\varepsilon}} 
\def\q{{\theta}}
\def\r{{\rho}}
\def\s{{\sigma}}
\def\S{{\Sigma}} 
\def\t{{\tau}}
\def\L{{\Lambda}}
\def\conv{{\rm conv}}
\def\rconv{{\rm{rconv}}}
\def\ext{{\rm{ext}}}

\newtheorem{thm}{Theorem}[chapter]
\newtheorem{lm}[thm]{Lemma}
\newtheorem{prop}[thm]{Proposition}
\newtheorem{co}[thm]{Corollary}
\newtheorem{rmk}[thm]{Remark}
\newtheorem{ex}[thm]{Example}
\newtheorem{open}[thm]{Open Problem}

\usepackage{xcolor}

\newcommand{\squeezelist}{\setlength{\itemsep}{0pt}}
\newtheorem{innercustomthm}{Theorem}
\newenvironment{customthm}[1]
  {\renewcommand\theinnercustomthm{#1}\innercustomthm}
  {\endinnercustomthm}
\usepackage{enumitem}
\usepackage{url}
\newcommand{\figlab}[1]{\label{fig:#1}}
\newcommand{\seclab}[1]{\label{sec:#1}}
\newcommand{\chaplab}[1]{\label{chap:#1}}
\newcommand{\lemlab}[1]{\label{lem:#1}}
\newcommand{\thmlab}[1]{\label{thm:#1}}
\newcommand{\exlab}[1]{\label{ex:#1}}
\newcommand{\openlab}[1]{\label{open:#1}}
\newcommand{\rmklab}[1]{\label{rmk:#1}}
\newcommand{\eqnlab}[1]{\label{eqn:#1}}
\newcommand{\figref}[1]{\ref{fig:#1}}
\newcommand{\secref}[1]{\ref{sec:#1}}
\newcommand{\chapref}[1]{\ref{chap:#1}}
\newcommand{\lemref}[1]{\ref{lem:#1}}
\newcommand{\thmref}[1]{\ref{thm:#1}}
\newcommand{\exref}[1]{\ref{ex:#1}}
\newcommand{\openref}[1]{\ref{open:#1}}

\newcommand{\eqnref}[1]{\ref{eqn:#1}}
\newcommand{\hide}[1]{}
\usepackage{bm} 
\usepackage{centernot} 

\title{Reshaping Convex Polyhedra}

\author{Joseph O'Rourke and Costin V\^\i lcu}

\date{\today}

\begin{document}
\maketitle

\frontmatter
\tableofcontents


\chapter{Abstract}
\chaplab{Abstract}
%

Given a convex polyhedral surface $P$, we define a \emph{tailoring} as excising from $P$
a simple polygonal domain that contains one vertex $v$, and
whose boundary can be sutured closed to a new convex polyhedron
via Alexandrov's Gluing Theorem.
In particular, a \emph{digon-tailoring} cuts off from $P$  
a \emph{digon} containing $v$, a subset of $P$ bounded by two equal-length geodesic segments that share endpoints,
and can then zip closed.

In the first part of this monograph, we primarily study properties of the tailoring operation on convex polyhedra.
We show that $P$ can be reshaped to 
any polyhedral convex surface $Q \subset \conv (P)$ 
by a sequence of tailorings.
This investigation uncovered previously unexplored topics, including a notion
of \emph{unfolding of} $Q$ \emph{onto} $P$---cutting 
up $Q$ into pieces pasted non-overlapping onto $P$,
and to continuously folding $P$ onto $Q$.

In the second part of this monograph, we study 
\emph{vertex-merging} processes
 on convex polyhedra (each vertex-merge
being in a sense the reverse of a digon-tailoring),
creating embeddings of $P$ into enlarged surfaces.
We aim to produce 
non-overlapping polyhedral and planar unfoldings,
which led us to develop an apparently new theory of convex sets, and of minimal length enclosing polygons, on convex polyhedra.

All our theorem proofs are constructive, implying polynomial-time algorithms.

\bs
\hrule
\medskip
\subsection*{MSC Classifications}

\paragraph{Primary:}

52A15, 52B10, 52C45, 53C45, 68U05.

\paragraph{Secondary:}

52-02, 52-08, 52A37, 52C99.

\bs
\hrule
\medskip
\subsection*{Keywords and phrases}
\noindent
convex polyhedron\\
Alexandrov's Gluing Theorem\\
polyhedron truncation\\
cut locus\\
geodesics\\
quasigeodesics\\
unfolding polyhedra\\
net for polyhedron; star-unfolding \\
vertex-merging\\
convex sets and convex hulls on convex polyhedra\\
minimal length enclosing polygon on convex polyhedra

\chapter{Preface}
\chaplab{Preface}
The research reported in this monograph emerged from exploring a
simple question: 
Given two convex polyhedra $P$ and $Q$, with $Q$ inside $P$,
can one reshape $P$ to $Q$ by repeatedly ``snipping" off vertices?
We call this snipping-off operation \emph{tailoring}. 
A precise definition
is deferred to the introductory Chapter~\chapref{IntroductionPartI},
but here we contrast it with \emph{vertex truncation},
which slices off a vertex with a plane and replaces it with a new
facet lying in that plane. 
This is, for example, one way to construct the truncated cube:
see Fig.~\figref{TruncatedCube}.
\begin{figure}[htbp]
\centering
\includegraphics[width=0.4\linewidth]{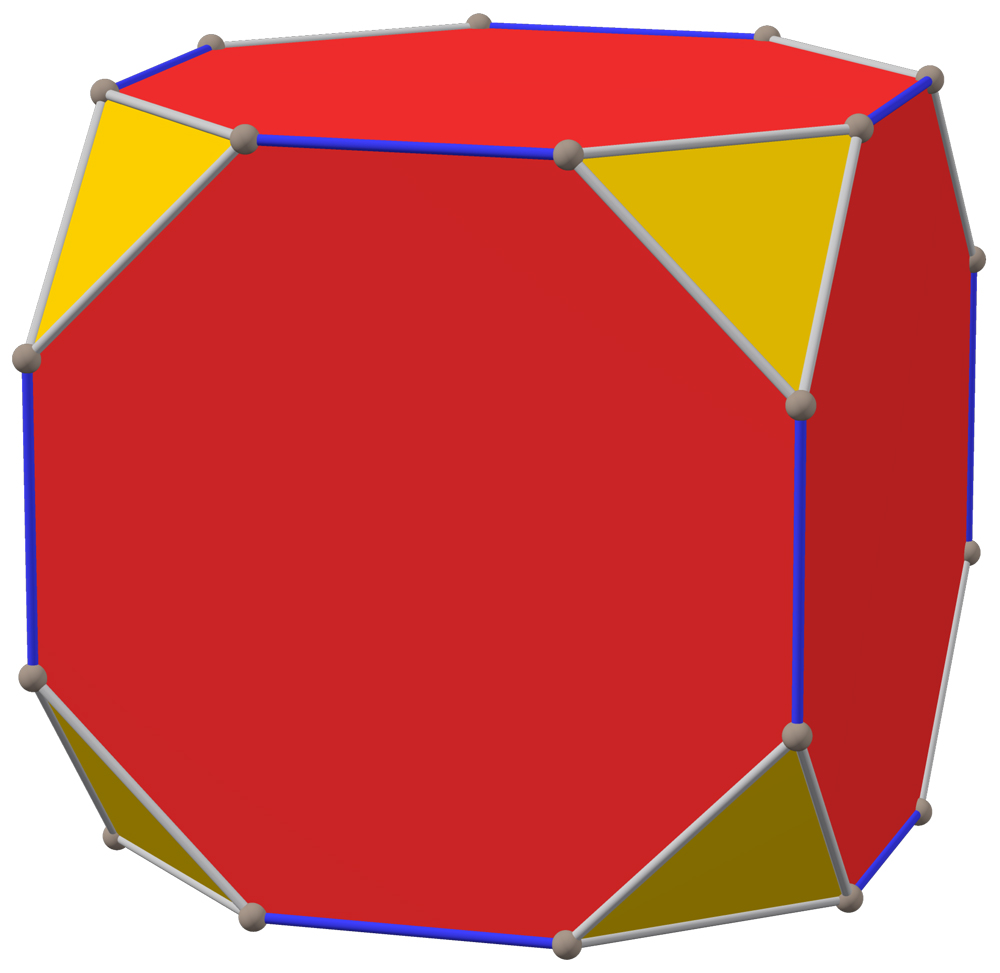}
\caption{Truncated cube. [Image by Tilman Piesk, Wikipedia].}
\figlab{TruncatedCube}
\end{figure}
Tailoring differs from vertex truncation in two ways:
first, it does not slice by a plane but instead uses a digon bound by a pair of
equal-length geodesics (again, definition deferred), and second, rather than filling the hole
with a new facet, the boundary of the hole is ``sutured" closed without the addition
of new surface.

Our first experiment started with a paper cube and tailored its $8$ vertices,
producing the shape shown in
Fig.~\figref{Cube8ExcisedPhoto}.
\begin{figure}[htbp]
\centering
\includegraphics[width=0.5\linewidth]{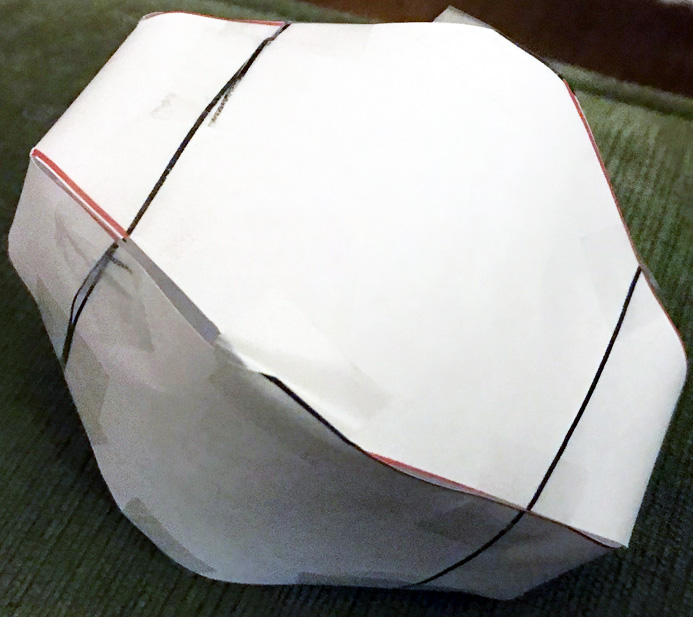}
\caption{Cube after tailoring $8$ vertices.}
\figlab{Cube8ExcisedPhoto}
\end{figure}
Although not evident from this crude model, this shape is actually a convex polyhedron
of $16$ vertices, an implication of Alexandrov's Gluing Theorem.
This led us imagine that continuing the process on the $16$ vertices might allow
reshaping the cube into a roughly spherical polyhedron---``whittling" a cube
to a sphere.
And indeed, one of our main theorems is that any $P$ can be reshaped to
any $Q \subset P$ by a finite sequence of tailorings
(Theorem~\thmref{MainTailoring}).
This holds whether $Q$ has fewer vertices than $P$
or more vertices than $P$:
see Fig.~\figref{CubeDodecaIcosa}.
\begin{figure}[htbp]
\centering
\includegraphics[width=0.8\linewidth]{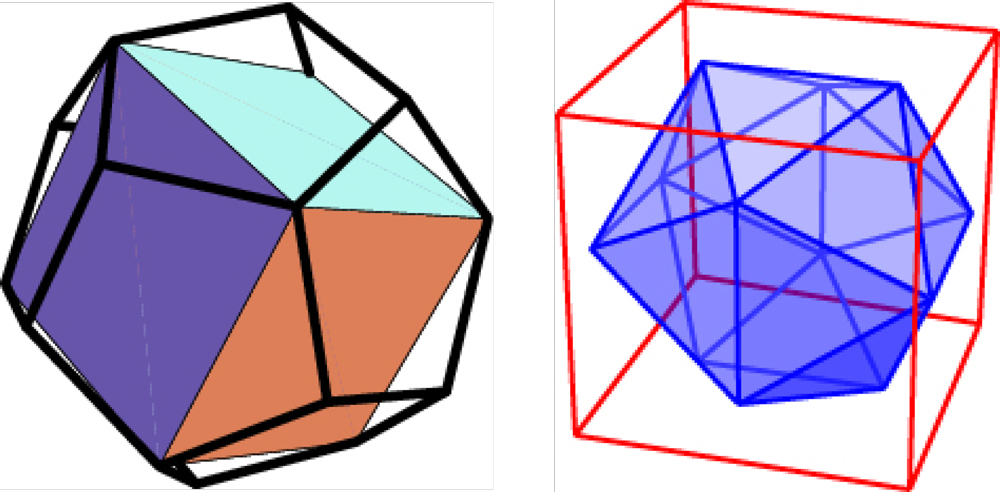}
\caption{Cube inside dodecahedron. Icosahedron inside cube.}
\figlab{CubeDodecaIcosa}
\end{figure}

\bs 

Investigating tailoring in Part~I uncovered previously unexplored topics, including a notion
of ``unfolding" $Q$ onto $P$---cutting up $Q$ into pieces pasted non-overlapping
onto $P$,
and continuously folding $P$ onto $Q$.
This led us to a systematic study of \emph{vertex-merging}, a technique introduced
by Alexandrov that can be viewed as a type of inverse operation to tailoring,
a study we carry out in Part~II.
Here we start with $P$ and gradually enlarge it via vertex-merges until 
$P$ is embedded onto a nearly flat shape $S$:
a doubly-covered triangle, 
an isosceles tetrahedron, a pair of cones, or
a cylinder. 
The first two ``vertex-merge irreducible shapes"
(doubly-covered triangle or 
isosceles tetrahedron) are the result of 
vertex-merging in arbitrary order, and can result in disconnecting 
an $n$-vertex $P$ into $O(n^2)$ pieces
(Corollary~\lemref{Pn2Pieces}).
In contrast we pursue the goal of minimizing the disconnection of $P$,
with the ultimate goal of achieving a planar \emph{net} for $P$---a nonoverlapping
simple polygon.

Toward this end, we partition $P$ into two half-surfaces by a quasigeodesic $Q$,
a simple closed curve on $P$, convex to both sides.\footnote{
In Part~I, $Q$ is the target polyhedron resulting from reshaping $P$.
In Part~II, there is no equivalent fixed ``target" polyhedron, and we use $Q$ 
to denote a quasigeodesic.}
Then we can prove that a particular spiral vertex-merge sequence avoids disconnecting
$P$ in each half.
(Theorems~\thmref{SpiralAlg} and~\thmref{Ztree}.)
This approach requires a notion of convexity 
and convex hull on the surface of
a convex polyhedron, 
apparently novel topics, explored in our longest chapter, Chapter~\chapref{Convexity}.


We can achieve a net for $P$ (Theorem~\thmref{RollingNet})
assuming the truth of a conjecture 
concerning quasigeodesics (Open Problem~\openref{QVert2}).
This is among the many 
avenues for future work uncovered by our investigations.

\bs
\subsection*{Acknowledgements}
We benefitted from discussions with Jin-ichi Itoh and with Anna Lubiw, on some topics of this work.

Costin V\^\i lcu's research was partially supported by 
UEFISCDI, project no. PN-III-P4-ID-PCE-2020-0794.

%
%
%

\mainmatter

\part{Tailoring for Every Body}

\chapter{Introduction to Part~I}
\chaplab{IntroductionPartI}
Let $P$ and $Q$ be convex polyhedra, each the convex hull of finitely many points in $\Rs$.
If $Q \subset P$, it is easy to see that $Q$ can be \emph{sculpted} from $P$ by ``slicing $P$ with planes." 
By this we mean intersecting $P$ with half-spaces each of whose plane boundary contains a face of $Q$.
If $Q \not\subset P$, we can shrink $Q$ until it fits inside.
So a \emph{homothet} of any given $Q$ can be sculpted from any given $P$, where a homothet is a copy possibly scaled, rotated, and translated.
Main results of Part~I 
are similar claims 
but via ``tailorings": a homothet of any given $Q$ can be tailored from any given $P$
(Theorems~\thmref{MainTailoring}, \thmref{CrestTailoring}, \thmref{TailoringFlattening}).\footnote{%
A preliminary version of Part~I 
is~\cite{ov-teb-2020}.}

With some abuse of notation,\footnote{
Justified by Alexandrov's Gluing Theorem, see ahead.}
we will use the same symbol $P$ for a polyhedral hull and its boundary.
We define two types of tailoring.
A \emph{digon-tailoring} cuts off a single vertex of $P$ along a digon, and then sutures the digon closed.
A \emph{digon} is a subset $D(x,y)$ of $P$ bounded by two equal-length geodesic segments that share endpoints $x$ and $y$.
A \emph{geodesic segment} is a shortest geodesic between its endpoints.
A \emph{crest-tailoring} cuts off a single vertex of $P$ but via a more complicated shape we call a ``crest."
Again the hole is sutured closed.
We defer discussion of crests to Chapter~\chapref{Crests}.
Meanwhile, we often shorten ``digon-tailoring" to simply \emph{tailoring}.
\begin{figure}
\centering
 \includegraphics[width=1.0\textwidth]{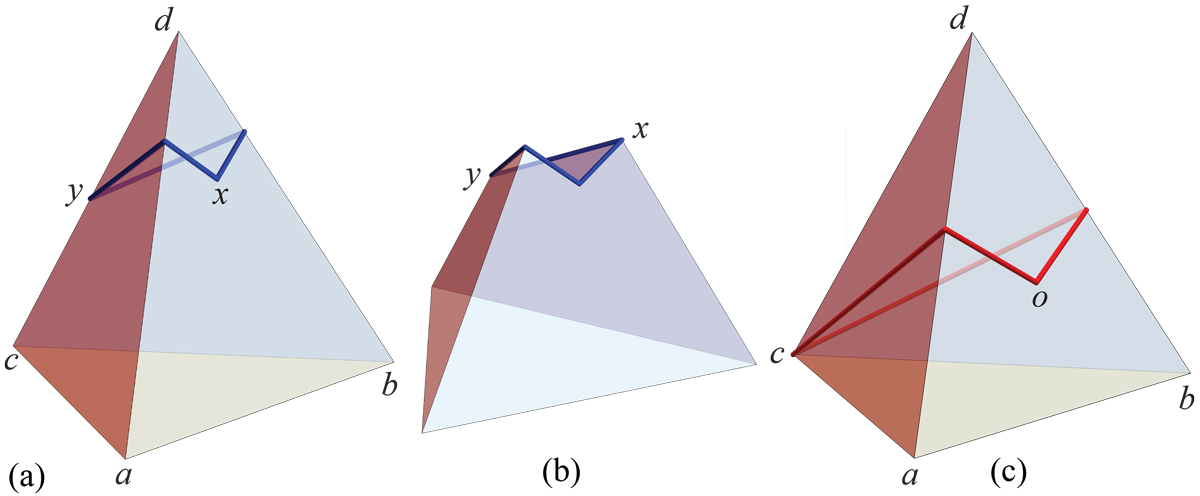}
\caption{
(a)~A digon $D(x,y)$ on the regular tetrahedron $P=abcd$, surrounding vertex $d$.
(b)~Digon excised.
(c)~A digon $D(o,c)$ with one endpoint at vertex $c$, the other the centroid $o$.}
\figlab{Tetra_3D_digons}
\end{figure}
Cutting out a digon means excising the portion of the surface between the geodesics, including the vertex they surround.\footnote{
An informal view (due to Anna Lubiw) is that one could pinch the surface flat in a neighborhood of the vertex, and then snip-off the flattened vertex with scissors.}
Once removed, the digon hole is closed by naturally identifying the two geodesics along their lengths.
This identification is often called ``gluing" in the literature, although we also call it 
``zipping" or ``suturing" or ``sealing." 
Fig.~\figref{Tetra_3D_digons} illustrates digons on a regular tetrahedron $P$.
In (a) of the figure, neither digon endpoint $x$ nor $y$ is a vertex of $P$;
(b) shows $P$ after excision but before sealing. After sealing, both $x$ and $y$ become
vertices.
(See ahead to Fig.~\figref{VertMerge} 
for a depiction of the polyhedron that results after sealing.)
In (c), one endpoint of the digon coincides with a vertex.
Here, after excision
and suturing, $o$ becomes a vertex and $c$ remains a vertex (of sharper curvature).


\section{Alexandrov's Gluing Theorem}
Throughout, we make extensive use of Alexandrov's Gluing Theorem~\cite[p.100]{a-cp-05},
which guarantees that the surface obtained after a tailoring of $P$
corresponds uniquely to a convex polyhedron $P'$.
A precise statement of this theorem, which we will abbreviate to AGT, is as follows.

\begin{customthm}{AGT}
\thmlab{AGT} 
Let $S$ be a topological sphere obtained by gluing planar polygons (i.e., naturally identifying pairs of sides of the same length) such that 
at most $2\pi$ surface angle is glued at any point.
Then $S$, endowed with the intrinsic metric induced by the distance in $\Rp$, is isometric to a convex polyhedron $P\subset\Rs$, 
possibly degenerated to a doubly-covered convex polygon. 
Moreover, $P$ is unique up to rigid motion and reflection in $\Rs$. 
\end{customthm}
\noindent
The case of doubly-covered convex polygon is important and we will encountered
it often.

Because the sides of the digon are geodesics, gluing them together to seal the hole leaves $2\pi$ angle at all but the digon endpoints.
The endpoints lose surface angle with the excision, and so have strictly less than $2\pi$ angle surrounding them.
So AGT applies and yields a new convex polyhedron.

This shows that tailoring is possible and alters the given $P$ to another convex polyhedron.
How to ``aim" the tailoring toward a given target $Q$ is a long story, told in subsequent sections.

Alexandrov's proof of his celebrated theorem is a difficult existence proof and gives little hint of the structure of the polyhedron guaranteed by the theorem. 
And as-yet there is no effective procedure to construct the three-dimensional shape of the polyhedron guaranteed by his theorem.
It has only been established that there is a theoretical pseudopolynomial-time
algorithm~\cite{kane2009pseudopolynomial},
achieved via an approximate numerical solution to the differential equation
established by Bobenko and Izmestiev~\cite{bi-atwdt-06}~\cite{o-cgc49-07}.
But this remains impractical in practice.
Only small or highly symmetric examples can be reconstructed,
for example~\cite{ado-cpfs-03} for the foldings of a square,
and more recently~\cite{arseneva2020complete} for polyhedra built from
regular pentagons.
Figs.~\figref{Pentahedron} and~\figref{VertMerge} (ahead)
were reconstructed through ad hoc methods.

AGT is a fundamental tool in the geometry of convex surfaces and, at a theoretical level, our 
work helps to elucidate its implications.
While AGT has proved useful in several investigations,
our application here to reshaping
has, to our knowledge, not been considered before as the central object of study.


\section{Tailoring Examples}

Before discussing background context, we present several examples.
Throughout we let $xy$ denote the line segment between points $x$ and $y$, $x,y \in \Rs$.
Also we make extensive use of vertex curvature.
The \emph{discrete} (or \emph{singular}) \emph{curvature} $\o(v)$ at a vertex $v \in P$ is the angle deficit: $2\pi$ minus the sum of the face angles incident to $v$.

\begin{ex}
\exlab{tetra-deg}
Let $T=abcd$ be a regular tetrahedron, and let $o$ be the center of the face $abd$;
see Fig.~\figref{Tetra_3D_digons}(c).
Cut out the digon on $T$ between $c$ and $o$ ``encircling" $d$, and zip it closed.

The unfolding $U(T)$ of $T$ with respect to $c$ is a planar regular triangle $c_a c_b c_d$ with center $o$.
Cutting out that digon from $T$ is equivalent to removing from $T$ the isosceles triangle $o c_a c_b$. See Fig.~\figref{TetraTailoring}(a).
We zip it closed by identifying the digon-segments $c_a o$ and $c_b o$, and refolding the remainder of $T$ by re-identifying $c_a b$ and $b c_d$,  and $c_b a$ and $a c_d$.
The result is the doubly-covered kite $K=aob c_d$, shown in Fig.~\figref{TetraTailoring}(b).
\end{ex}

\begin{figure}
\centering
\includegraphics[width=\textwidth]{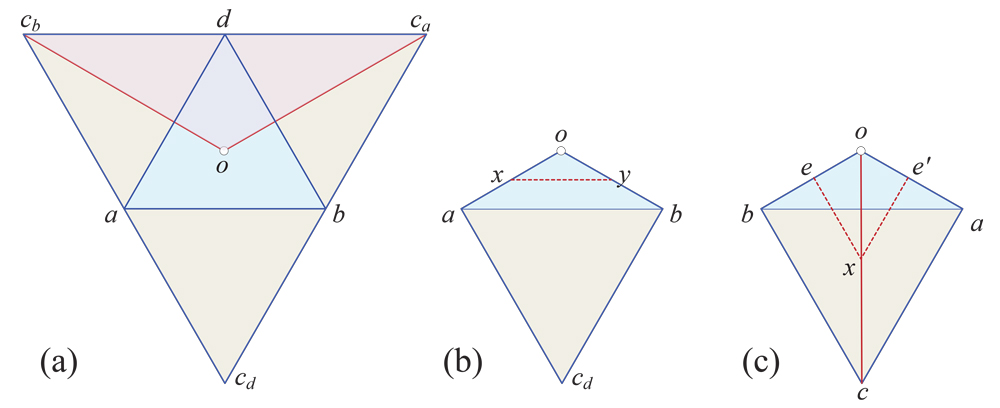}
\caption{Illustrations for Examples~\exref{tetra-deg}--\exref{deg-vol}.
In (c), $e$ and $e'$ are joined when the $D(x,y)$ digon is excised.}
\figlab{TetraTailoring}
\end{figure}

\begin{ex}
\exlab{deg-vol}
We further tailor the doubly-covered kite $K$ obtained in Example~\exref{tetra-deg}.
Excising a digon encircling $o$, between points $x \in oa$ and $y \in ob$, and zipping closed, 
yields a doubly-covered pentagon,  as in Fig.~\figref{TetraTailoring} (b).

If instead we excise
a digon encircling $o$, between corresponding points $x,y \in oc$ on different sides of $K$ (see Fig.~\figref{TetraTailoring} (c)), and seal closed, 
the result is a non-degenerate pentahedron, illustrated in Fig.~\figref{Pentahedron}.
This is among the rare cases where the polyhedron guaranteed by AGT can be
explicitly reconstructed.
\end{ex}

\begin{figure}
\centering
 \includegraphics[width=0.9\textwidth]{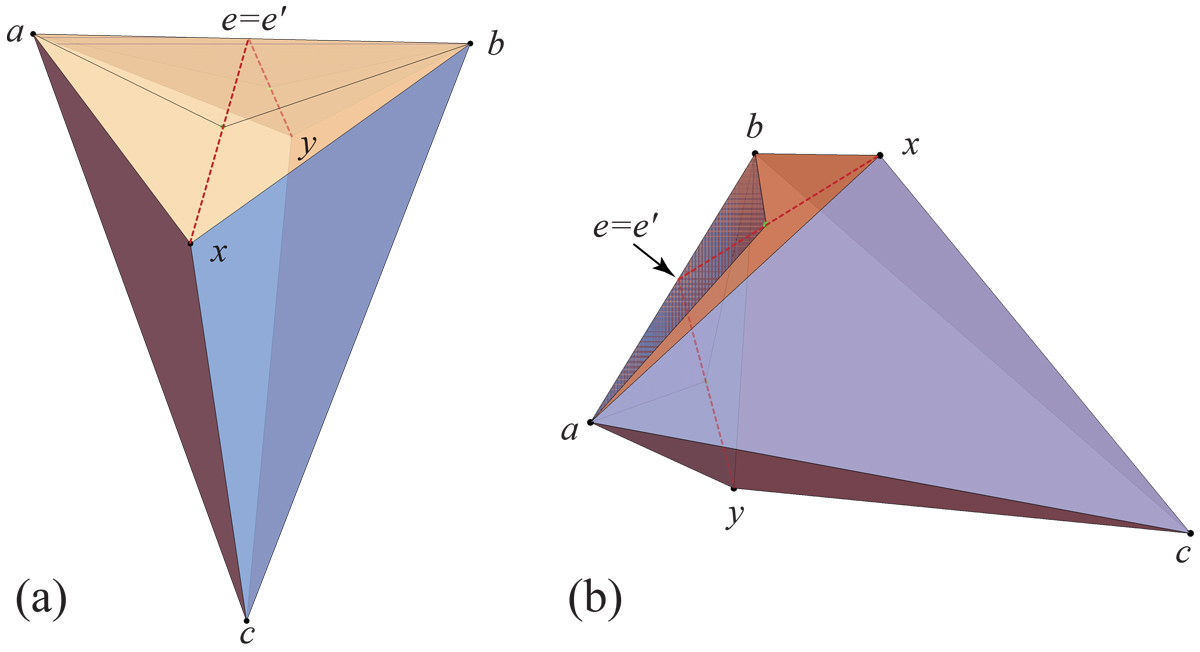}
\caption{The non-degenerate pentahedron obtained in Example~\exref{deg-vol}.
The point marked $e=e'$ is the join of the points $e$ and $e'$ in Fig.~\protect\figref{TetraTailoring}(c),
the midpoint of the edge $ab$.}
\figlab{Pentahedron}
\end{figure}

\begin{ex}
\exlab{iso_tetra}
A digon may well contain several vertices, but for our digon-tailoring we only consider those 
containing at most one vertex.
The limit case of a digon containing no vertex is an edge between two vertices.
(This will play a role in Chapter~\chapref{TailoringFlattening}.)
In this case, gluing back along the cut would produce the original polyhedron, but we can as well seal the cut 
closed from another starting point.
For example, cutting along an edge of an isosceles tetrahedron
(a  
polyhedron composed of four congruent faces)\footnote{
Also called a tetramonohedron, or an isotetrahedron.}
and carefully choosing the gluing 
leads to a doubly-covered rectangle.
See Fig.~\figref{IsoscTetra}.
\end{ex}

By AGT, this limit-case tailoring of an empty-interior digon
only applies between vertices $v_1, v_2$ of curvatures $\o_1,\o_2 \geq \pi$,
because both $v_1$ and $v_2$ will be identified with interior points of
the edge.\footnote{
In~\protect\cite[Sec.~25.3.1]{do-gfalop-07} this structure is called a ``rolling belt."}


\begin{figure}
\centering
 \includegraphics[width=1.0\textwidth]{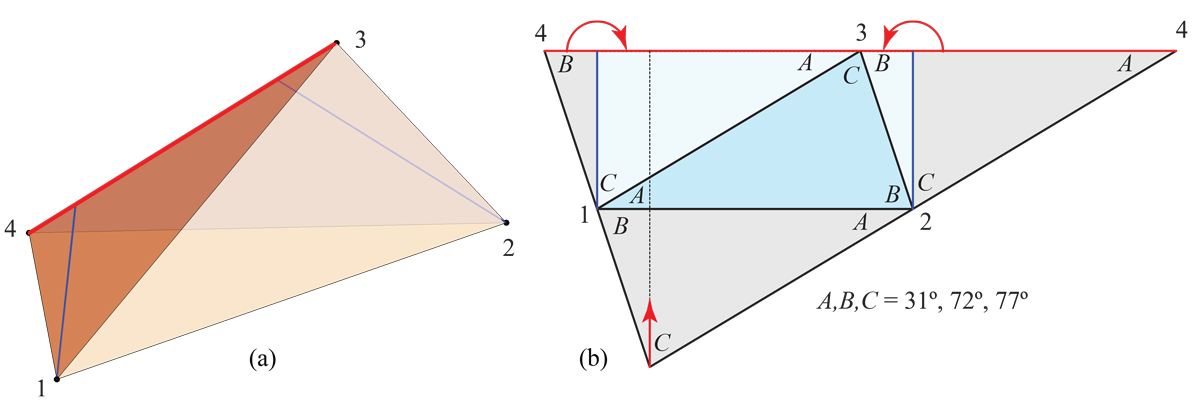}
\caption{(a)~An isosceles tetrahedron: All four faces are congruent.
Cutting along $v_3 v_4$ and regluing the two halves of that
slit differently, creasing at the blue segments,
yields a doubly-covered rectangle, as shown in~(b),
folding as shown.
}
\figlab{IsoscTetra}
\end{figure}

In view of Fig.~\figref{TetraTailoring}(c) and Fig.~\figref{Pentahedron},
it is clear that, even though tailoring is area decreasing, it is not necessarily volume decreasing.


Let a tailoring step remove vertex $v$ inside digon $D=(x,y)$.
In general, neither $x$ nor $y$ is a vertex before tailoring, but they become vertices after removing $D$, thereby increasing the number of vertices of $P$ by $1$.
This is illustrated in Fig.~\figref{Tetra_3D_digons}(ab).
If one of $x$ or $y$ is a vertex and the other not, then the total number of vertices
remains fixed, as in Fig.~\figref{Tetra_3D_digons}(c).
And if both $x$ and $y$ are vertices, then the number of vertices of $P$ is decreased by $1$. 
The challenge answered in our work is to direct tailoring to ``aim" from one polyhedon $P$ to the target $Q$, which may have a quite different number of vertices.


\section{Summary of Part-I Results}

Here we list our
main theorems in Part~I, each with a succinct (and at this stage, quite approximate) summary of their claims.
\begin{itemize}
\squeezelist
\item Theorem~\thmref{MainTailoring}: 
$Q$ may be digon-tailored from $P$, tracking a sculpting of $P$ to $Q$.
\item Theorem~\thmref{CrestTailoring}: 
$P$ may be crest-tailored to $Q$,
again tracking a sculpting.
%
\item Theorem~\thmref{TailoringFlattening}: 
A different proof of a similar
result, that $P$ may be digon-tailored to a homothet of $Q$, but this time without 
relying on sculpting.
\item Theorems~\thmref{SliceAlgorithm} and~\thmref{CrestTailoring} and \thmref{FlatteningAlg}: 
Tailoring algorithms have time-complexity $O(n^4)$.
\item Theorem~\thmref{Enlarge}: Reversing tailoring yields procedures for enlarging $Q$ inside $P$ to match $P$.
As a consequence, $Q$ may be cut up and ``unfolded" isometrically onto $P$.
\end{itemize}

\noindent
Along the way to our central theorems, we obtain results not directly related to AGT:
\begin{itemize}
\item Theorem~\thmref{Rigid}: If two convex polyhedra with the same number of vertices match on all but 
the neighborhoods of one vertex, then they are congruent.
\item Theorem~\thmref{DomePyr}: Every ``g-dome" can be partitioned into a finite sequence of pyramids by planes through its base edges.
\end{itemize}
The above results raise several open problems of various natures, either scattered along the text or presented in the last section of Part~I.

Finally, we sketch the logic behind the first result listed above,
whose statement in Chapter~\chapref{TailoringSculpting} we repeat here:
\begin{customthm}{\protect\thmref{MainTailoring}} 
Let $P$ be a convex polyhedron, and $Q \subset P$ 
a convex polyhedron resulting from repeated slicing of $P$ with planes. Then $Q$ can also be obtained from $P$ by tailoring.
Consequently, for any given convex polyhedra $P$ and $Q$, one can tailor $P$ ``via sculpting'' to obtain any homothetic copy of $Q$ inside $P$.
\end{customthm}
Start with $Q$ inside $P$, and imagine a sequence of slices by planes that sculpt $P$ to $Q$. 
Lemma~\lemref{VertexTruncation} shows how to digon-tailor one such slice, which then establishes the claim that we can tailor $P$ to $Q$.
Theorem~\thmref{MainTailoring} is achieved by first slicing off shapes we call ``g-domes," 
and then showing in Theorem~\thmref{DomePyr} that every g-dome can be reduced to its base by slicing off pyramids, i.e., by vertex truncations.
Lemma~\lemref{VertexTruncation} shows that such vertex truncations can be achieved by tailoring.
And the proof of Lemma~\lemref{VertexTruncation} relies on the rigidity established by Theorem~\thmref{Rigid}.
So the path  of logic is:
\begin{align*}
\textrm{plane slice} \;\to\; \textrm{g-domes} \;\to\; \textrm{pyramids}  \;\to\; & \textrm{digon removals} \;.\\
& \uparrow\\
& \textrm{rigidity}
\end{align*}


\chapter{Preliminaries}
\chaplab{Preliminaries}

In this chapter we 
present basic properties of cut loci on convex polyhedra,
the star-unfolding,
prove a rigidity result,
and describe the technique of vertex-merging.
All of these 
geometric tools will be needed subsequently.
The reader might 
skip this section and return to it as the tools are deployed.


\section{Cut locus properties}

The \emph{cut locus} $C(x)$ \emph{of the point $x$} on a convex polyhedron $P$ is the closure of the set of points to which there is more than one shortest path from $x$.
This concept goes back to Poincar\'e~\cite{p-lgsc-1905}, and
has been studied algorithmically since~\cite{ss-spps-86} (there, the cut locus is called the ``ridge tree'').
The cut locus is one of our main tools throughout this work.
The next lemma establishes notation and lists several known properties.

\begin{lm}[Cut Locus Basics]
\lemlab{CutLocusBasic}
The following hold for the cut locus $C(x)$:
\begin{enumerate}[label={(\roman*)}]
\item $C(x)$ has the structure of a finite $1$-dimensional simplicial complex which is a tree.
Its leaves (endpoints) are vertices of $P$, and all vertices of $P$, excepting $x$ (if it is a vertex) are included in $C(x)$. 
All points interior to $C(x)$ of tree-degree $3$ or more are known as \emph{ramification points} of $C(x)$.\footnote{%
In some literature, these points are called ``branch points" or ``junctions" of $C(x)$.}
All vertices of $P$ interior to $C(x)$ are also considered as ramification points.
\item Each point $y$ in $C(x)$ is joined to $x$ by as many geodesic segments as the number of connected
components of $C(x) \setminus {y}$.
For ramification points in $C(x)$, this is precisely their degree in the tree.
\item The edges of $C(x)$ are geodesic segments on $P$.
\item Assume the geodesic segments $\g$ and $\g'$ (possibly $\g = \g'$) from $x$ to $y \in C(x)$ are bounding a domain $D$ of $P$, 
which intersects no other geodesic segment from $x$ to $y$.
Then there is an arc of $C(x)$ at $y$ which intersects $D$ and it bisects the angle of $D$ at $y$.
\end{enumerate}
\end{lm}
\begin{proof} 
The statements (i)-(ii) and (iv) are well known.
The statement (iii) is Lemma 2.4 in~\cite{aaos-supa-97}.
\end{proof}

The following is Lemma~4 in~\cite{inv-cfcp-12}.

\begin{lm}[Path Cut Locus]
\lemlab{Path}
If $C(x)$ is a path, the polyhedron is a doubly-covered (flat) convex polygon, with $x$ on the rim.
\end{lm}

The following lemma will be invoked in Chapter~\chapref{SealGraph}.
\begin{lm}[Angle $< \pi$]
\lemlab{CutLocusAngles}
Let $x$ be a point on a convex polyhedron $P$ and let $y$ be a ramification point of $C(x)$ of degree $k \ge 2$.
Let $e_1,...,e_k$ be the edges of $C(x)$ incident to $y$, ordered counterclockwise, 
and let $\a_j$ be the angle of the sector between $e_j$ and $e_{j+1}$ at $y$ (with $k+1 \equiv 1$).
Then $\a_j < \pi$, for all $j=1,...,k$. 
\end{lm}
\begin{proof}
Let $\g_1,...\g_k$ be the geodesic segments from $x$ to $y$, say with $\g_j$ between $e_j$ and $e_{j+1}$ (again with $k+1 \equiv 1$).
See Fig.~\figref{RamifAngles}.
By Lemma~\lemref{CutLocusBasic} (iv), $e_j$ is the bisector of $\angle(\g_j, \g_{j+1})$. 
Put $2 \b_{j+1} = \angle(\g_j, \g_{j+1})$.
Then the total surface angle $\q_y$
incident to $y$ satisfies $\q_y=\sum_{j=1}^k 2 \b_j \leq 2 \pi$ 
if $k \geq 3$, and $\q_y < 2 \pi$ if $k=2$.
Therefore, when $k \geq 3$,  $\sum_{j=1}^k \b_j = \pi$ and so $\a_j = \b_j + \b_{j+1} < \pi$.
And when $k=2$, $\a_1 = \a_2 = \b_1 + \b_2 < \pi$.
\end{proof}

\begin{figure}[htbp]
\centering
\includegraphics[width=0.5\linewidth]{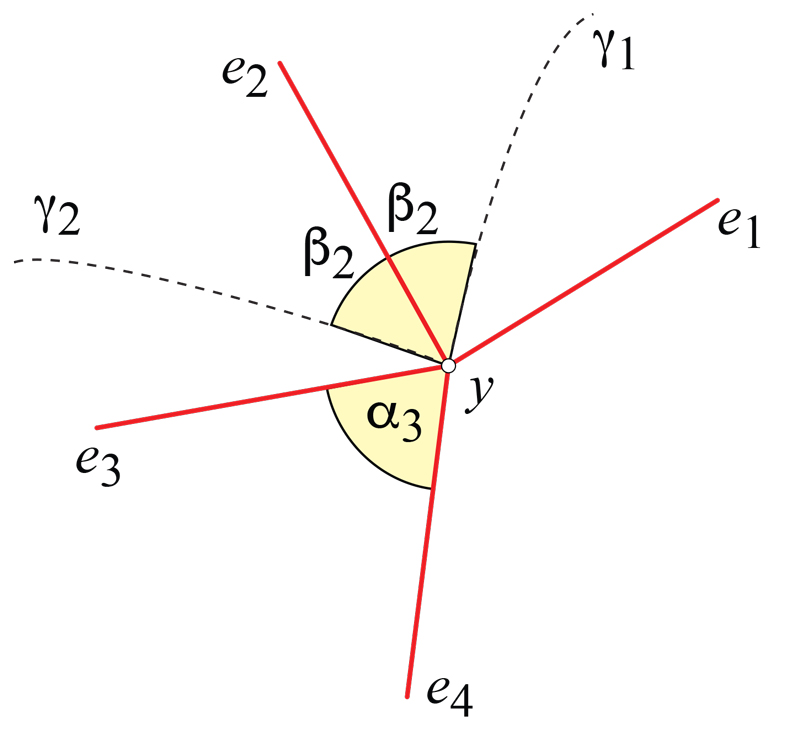}
\caption{Edge $e_2$ of $C(x)$ bisects  $\angle(\g_1, \g_2)=2\b_2$.}
\figlab{RamifAngles}
\end{figure}

\subsection{Star-Unfolding and Cut Locus}

Next we introduce a general method for unfolding any convex polyhedron $P$ to a simple (non-overlapping) polygon in the plane. 
We use this subsequently largely because of its connection to the cut locus.

To form the \emph{star-unfolding} of a $P$ with respect to $x$, one cuts $P$ along the geodesic segments 
(unique if $x$ is ``generic")
from $x$ to every vertex of $P$.
The idea goes back to Alexandrov~\cite{a-cp-05}; the  non-overlapping of the unfolding was established in~\cite{ao-nsu-92}, where the next result was also proved.
See Fig.~\figref{StarUnfCube}.

\begin{lm}[$S_P$ Voronoi Diagram]
\lemlab{*unfVD}
Let $S_P=S_P(x)$ denote the star-unfolding of $P$ with respect to $x \in P$.
Then the image of $C(x)$ in $S_P$ is the restriction to $S_P$ of the Voronoi diagram of the images of $x$.
\end{lm}

\begin{figure}
\centering
\includegraphics[width=1.0\textwidth]{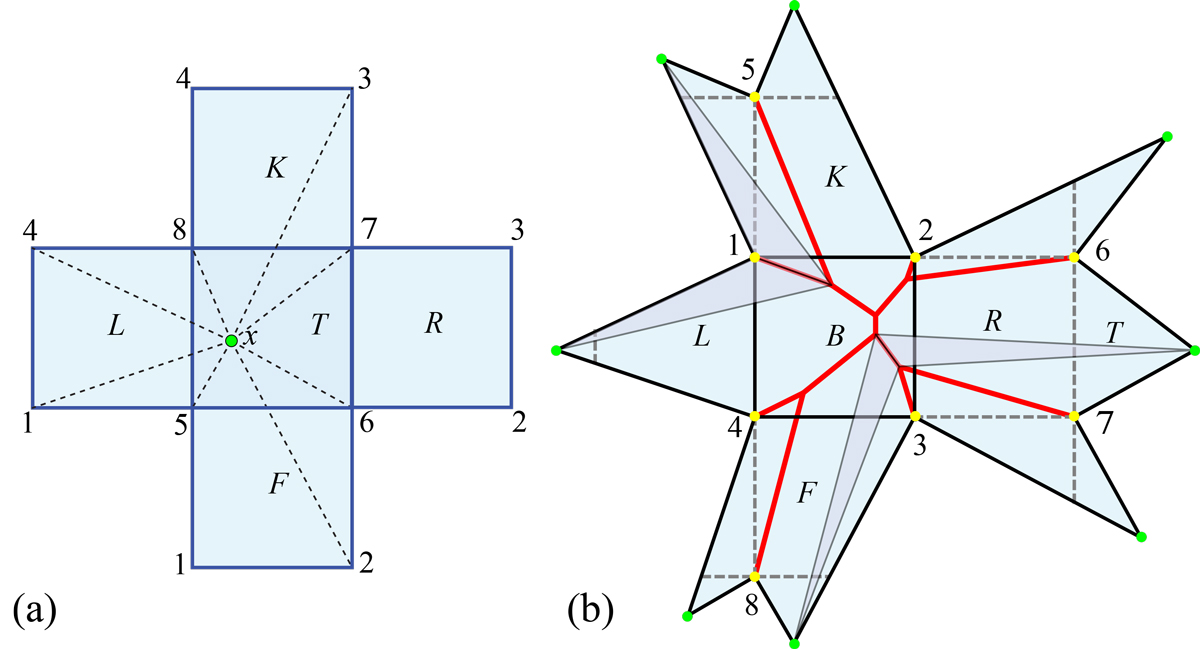}
\caption{(a)~Cut segments to the $8$ vertices of a cube from a point $x$
on the top face. (No cut is interior to the bottom face.)
T, F, R, K, L, B $=$ Top, Front, Right, Back, Left, Bottom.
(b)~The star-unfolding from $x$. The cut locus $C(x)$ (red) is the Voronoi
diagram of the $8$ images of $x$ (green).
Two pairs of fundamental triangles are shaded.}
\figlab{StarUnfCube}
\end{figure}

Notice that several properties of cut loci, presented in the previous section, could easily be derived from Lemma~\lemref{*unfVD}.


\subsection{Fundamental Triangles}

A geodesic triangle on $P$ (i.e., with geodesic segments as sides) is called 
\emph{flat} if its curvature vanishes.

\begin{lm}[Fundamental Triangles~\cite{inv-cfcp-12}]
\lemlab{FundTri} 
For any point $x \in P$,
$P$ can be partitioned into flat triangles whose bases are edges of $C(x)$,
and whose lateral edges are geodesic segments from $x$ to 
the ramification points or leaves of $C(x)$. 
Moreover, those triangles are isometric to
plane triangles, congruent by pairs. 
\end{lm}

\noindent
See Fig.~\figref{StarUnfCube}(b)
(and ahead to Fig.~\figref{CutLocusPath}).


\subsection{Cut Locus Partition}

Another tool we need (in Chapter~\chapref{TailoringSculpting})
is a cut locus partition lemma, a
generalization of lemmas in~\cite{inv-cfcp-12}.
On a polyhedron $P$,
connect a point $x$ to a point $y \in C(x)$ by two geodesic segments
$\g, \g'$.
This partitions $P$ into two ``half-surface" digons $H_1$ and $H_2$.
If we now zip each digon separately closed by joining $\g$ and $\g'$,
AGT leads to two convex polyhedra $P_1$ and $P_2$.
The lemma says that the cut locus on $P$ is the ``join" of the cut loci on $P_i$.
See Fig.~\figref{Partition}.

\begin{lm}[Cut Locus Partition]
\lemlab{Partition}
Under the above circumstances, the cut locus $C(x,P)$ of $x$ on $P$
is the \emph{join} of the cut loci on $P_i$:
$C(x,P) = C(x,P_1) \sqcup_y C(x,P_2)$, where $\sqcup_y$ joins the two cut loci at 
point $y$.
And starting instead from $P_1$ and $P_2$, the natural converse holds as well.
\end{lm}
\begin{proof}
Notice first that a straightforward induction and Lemma~\lemref{FundTri} on fundamental triangles
shows that
the cut locus of $x$ on $P_i$ is indeed the truncation of $C(x,P)$.
Therefore, $C(x,P) = C(x,P_1) \sqcup_y C(x,P_2)$.

Assume we start now from $P_1$ and $P_2$, 
having vertices $x_1,y_1 \in P_1$ and $x_2,y_2 \in P_2$ such that
\begin{itemize}
\item $\r_{P_1}(x_1,y_1)= \r_{P_2}(x_2,y_2)$,
where $\r_{P_i}(\;)$ is the geodesic 
distance between the indicated points on $P_i$.
\item $\q_{x_1} + \q_{x_2} \leq 2 \pi$, 
where $\q_x$ is the total surface angle incident to $x$,
and
\item $\q_{y_1} + \q_{y_2} \leq 2 \pi$.
\end{itemize}
Then we can cut open $P_i$ along a geodesic segment $\g_i$ from $x_i$ to $y_i$, $i=1,2$, and join the the two halves by AGT, such that 
$x_1,x_2$ have a common image $x$, and $y_1,y_2$ have a common image $y$.
 
Now, all geodesic segments starting at $x$ into $H_i$
remain in $H_i$, because geodesic segments do not branch.
Therefore, $H_1$ has no influence on $C(x,P_2)$ and
$H_2$ has no influence on $C(x,P_1)$.
\end{proof}
\begin{figure}[htbp]
\centering
\includegraphics[width=0.8\linewidth]{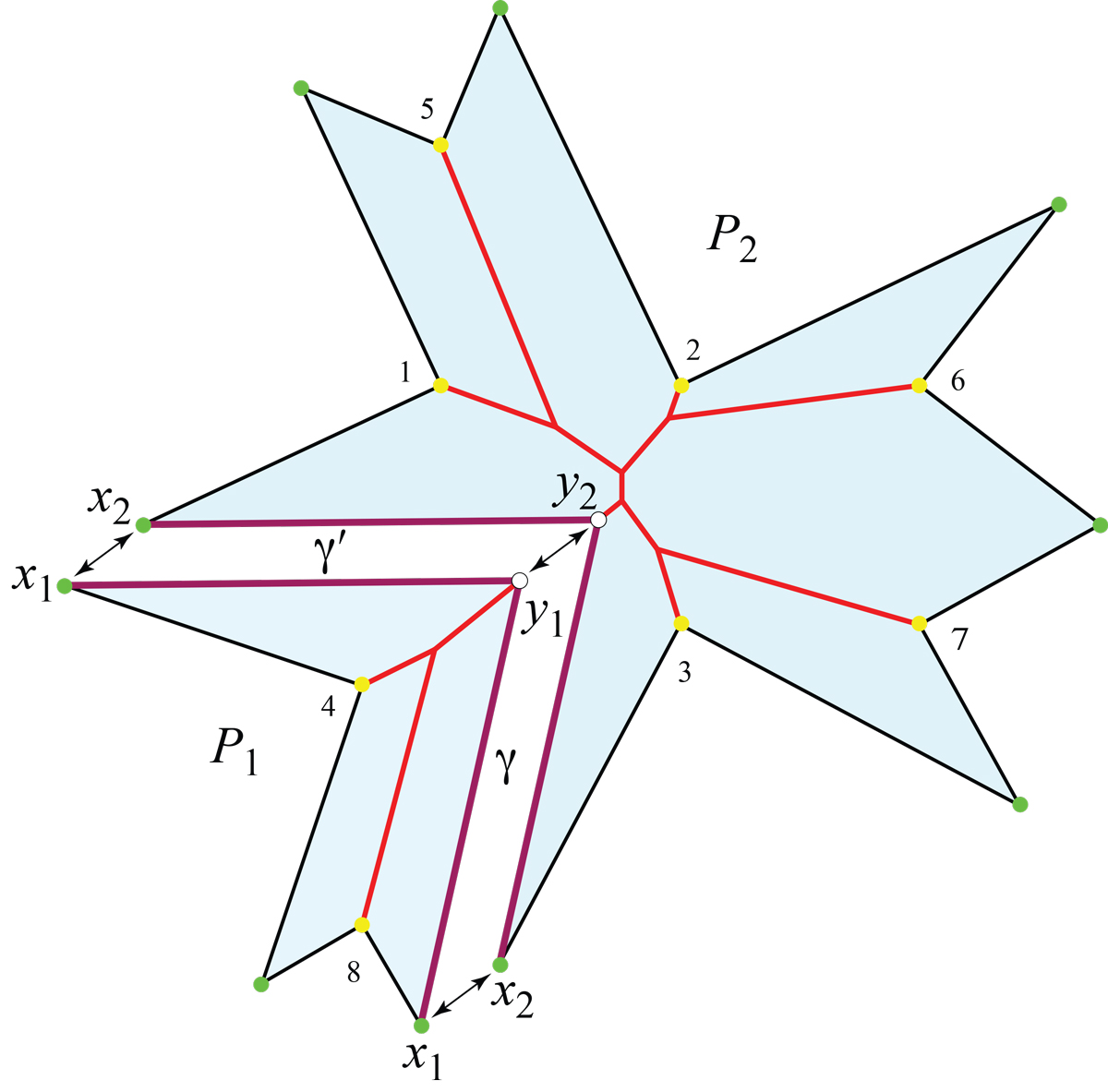}
\caption{Geodesic segments
$\g$ and $\g'$ (purple) connect $x{=}x_1{=}x_2$ to $y{=}y_1{=}y_2$.
$P_1$ folds to a tetrahedron, and $P_2$ to an $8$-vertex polyhedron, with $x$ and $y$ vertices in each.
$P_1$ and $P_2$ are cut open along geodesic segments from $x_i$ to $y_i$ and glued together to form $P$.
Based on the cube unfolding in Fig.~\protect\figref{StarUnfCube}(b).
}
\figlab{Partition}
\end{figure}


\section{Cauchy's Arm Lemma}
In several proofs, we will need an extension of Cauchy's Arm Lemma, which we now
describe.

Let $C = x_1,\ldots,x_n$ be a directed polygonal chain in the plane, with left angles
$\q_i = \angle x_{i-1} x_i x_{i+1}$, where $x_n=x_1$ if $C$ is closed,
and $\q_1$ and $\q_n$ undefined if $C$ is open.
Let $d=|x_1 x_n|$ be the distance between the endpoints, with $d=0$ if $C$ is closed.
If $\q_i \le \pi$ for all $i$, then $C$ is a \emph{convex chain}.
\emph{Cauchy's Arm Lemma} applies to reconfiguration of
$C$ while all the edge lengths remain fixed
(the edges are \emph{bars}).
We use primes to indicate the reconfigured chain.
The lemma
says that if the $\q_i$ angles are ``straightened"
while remaining convex,
in the sense that $\q_i \le \q'_i \le \pi$,
then the distance $d' = |x'_1 x'_n|$ only increases: $d' \ge d$.

For the extension needed later, 
we reformulate in terms of turn angles $\tau_i = \pi-\q_i \ge 0$,
as described in~\cite{o-ecala-01}.
Now the straightening condition says that $0 \le \tau'_i \le \tau_i$.
The extended version of Cauchy's arm lemma says that,
as long as, in the reconfiguration,
$-\tau_i \le \tau'_i \le \tau_i$, then the same conclusion holds: $d' \ge d$.
Thus $C'$ might no longer be a convex chain, but its reflexivities are bounded
by the original turn angles.
One way to interpret $d' \ge d$ is that there is a ``forbidden" disk of radius 
$d=|x_1 x_n|$ centered on $x_1$ that $x'_n$ cannot penetrate.
See Fig.~\figref{CauchyForbidden}.
\begin{figure}[htbp]
\centering
\includegraphics[width=0.5\linewidth]{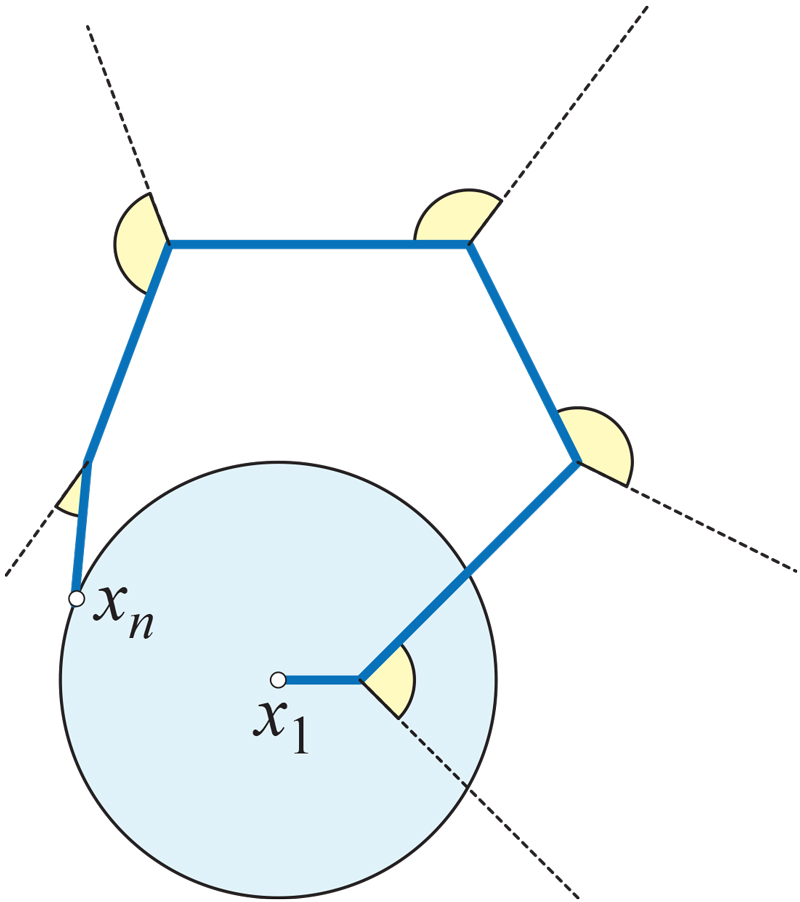}
\caption{Any turns within the indicated $\pm \tau_i$ angle ranges constitute straightening.
Then $x_n$ will not penetrate the $|x_1 x_n|$ disk.}
\figlab{CauchyForbidden}
\end{figure}
We summarize in a theorem:

\begin{thm}{(Cauchy's Arm Lemma.)}
\thmlab{CAL} 
A straightening reconfiguration of a planar convex chain $C$---retaining
edge-lengths fixed and confining new turn angles $\tau'_i$ to lie
in $[-\tau_i, \tau_i]$---only increases the distance between the 
endpoints: $|x'_1 x'_n| \ge |x_1 x_n|$.
\end{thm}

The next elementary result assures the angle increase, 
in the frameworks in which we will apply Cauchy's Arm Lemma.

\begin{lm}
\lemlab{Angles}
Consider three rays $r_1,r_2,r_3$ in $\R^3$, emanating from the point $w$, and put 
$\t_i=\angle(r_i, r_{i+1})$, with $3+1 \equiv 1$ mod $3$.
Then $\theta_1 \leq \theta_2 + \theta_3$.
\end{lm}
\begin{proof}
Imagine a unit sphere $S$ centered on $w$ and let $\{s_i \} = r_i \cap S$, and use $\r$ to indicate spherical distance.
Then the claim of the lemma is the triangle inequality for spherical distances: $\r(s_1,s_2) \le \r(s_1,s_3) + \r(s_2,s_3)$.
\end{proof}


\section{A Rigidity Result}
\seclab{Rigidity}

In this section we present a technical result for later use, which may be of independent interest.
The theorem says that two convex polyhedra that are isometric on all but the neighborhoods of one vertex, are in fact congruent.
We also show this result cannot be strengthened: two convex polyhedra can differ in the neighborhoods of just two vertices.

\begin{thm}
\thmlab{Rigid}
Assume $P,Q$ are convex polyhedra with the same number of vertices, such that 
there are vertices $p\in P$ and $q\in Q$, and respective small neighborhoods $N_p \subset P$, $N_q \subset Q$ not containing other vertices, 
and an isometry $\iota : P\setminus N_p \to Q \setminus N_q$.
Then $P$ is congruent to $Q$.
\end{thm}

\begin{proof}
The existence of $\iota$ on all but neighborhoods of $p$ and $q$ yields, in particular, that the curvatures 
$\o_P(p)$ of $P$ at $p$ and  $\o_Q(q)$ of $Q$ at $q$ are equal, to satisfy the curvature sum of $4\pi$ (by Gauss-Bonnet). 

Take a point $x \in P$ joined to each vertex of $P$ by precisely one geodesic segment,
a \emph{generic point} $x$.
Such an $x$ 
maybe be found in a ``ridge-free" region of $P$~\cite{aaos-supa-97};
it is equivalent to the fact that no vertex of $P$ is interior to $C(x)$.
Moreover, we may choose $x$ such that $\iota(x)$ has the same property on $Q$.

Denote by $u$ the ramification point of $C(x)$ neighboring $p$ in $C(x)$, i.e., the ramification point of degree $\ge 3$ closest to $p$.
Let $v$ be the similar ramification point of $C(\iota(x))$ neighboring $q$ in $C(\iota(x))$.
Since $N_p$ and $N_q$ are small, we may assume they are disjoint from 
$u$ and $v$ and all the segments described above.

Star unfold $P$ with respect to $x$, and $Q$ with respect to $\iota(x)$, and denote by $S_P$ and  $S_Q$ the resulting planar polygons.
We'll continue to use the symbols $p$ and $q$, $u$ and $v$ to refer to the corresponding points in $S_P$ and $S_Q$ respectively.
Let $x_i$, $i=1,2$ be the images of $x$ surrounding $p$ in $S_P$, and $\iota(x_i)$ the similar images in $S_Q$.
See Fig.~\figref{Tent_star_Cx}(a,b).

By hypothesis, we have respective neighborhoods $\bar{N}_p \subset S_P$ and $\bar{N}_q \subset S_Q$ 
and an isometry $\bar{\iota}$ induced by $\iota$, with $\bar{\iota} : S_P \setminus \bar{N_p} \to S_Q \setminus \bar{N_q}$. 
Thus in Fig.~\figref{Tent_star_Cx}(b), all of $S_P$ outside of the wedge $(x_1,u,x_2)$ is identical in $S_Q$.
Therefore the triangles  ${x}_1{u}{x}_2$ and ${x}'_1 {v} {x}'_2$ are congruent.
Moreover, ${p}$ lies on the bisector of the angle $\angle {x}_1{u}{x}_2$, and ${q}$ lies on the bisector of the angle $\angle {x}'_1 {v} {x}'_2$.
Since $\angle  {x}_1 {p} {u} = \angle  {x}'_1 {q} {v} = \pi - \frac{1}{2}\o(p)= \pi - \frac{1}{2}\o(q)$, $p$ and $q$ are uniquely determined.
Consequently,  $S_P$ and $S_Q$ coincide, and refolding according to the same gluing identifications leads to congruent $P$ and $Q$.
\end{proof}

\begin{figure}
\centering
\includegraphics[width=1.0\textwidth]{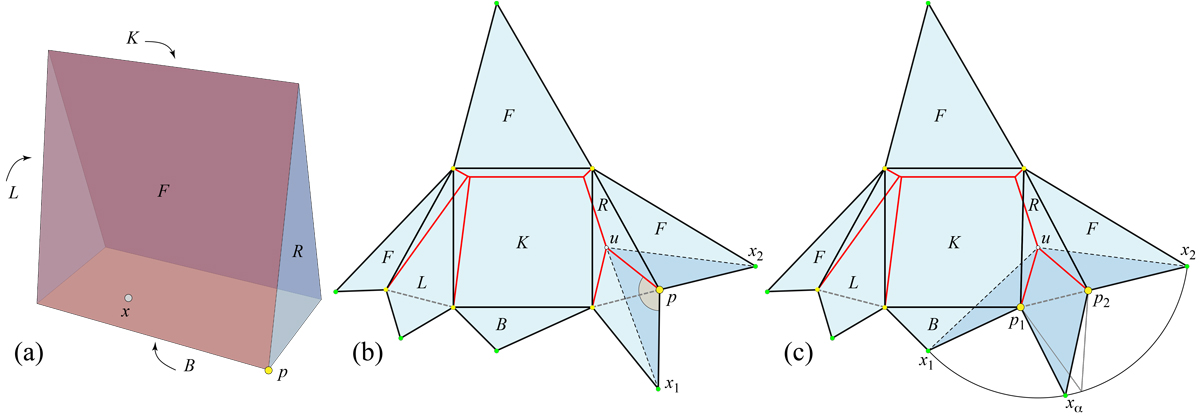}
\caption{
(a)~A $6$-vertex polyhedron $P$. The $F$ and $K$ faces are unit squares; 
$B$ is a $1 \times \frac{1}{2}$ rectangle, with $x \in B$.
(b)~Star-unfolding $S_P$ of $P$. $\angle  {x}_1 {p} {u}$ is marked.
(c)~Moving $x_\a$ on the circle arc moves the bisectors
incident to $u$, and so moves $p_1$ and $p_2$.
Refolding results in a polyhedron incongruent to~(a).}
\figlab{Tent_star_Cx}
\end{figure}

\begin{rmk}
\rmklab{RigidNo2}
It is perhaps surprising that the above result cannot be extended to 
claim that isometries excluding neighborhoods of two vertices
always imply congruence.
\end{rmk}

\begin{proof}
If the points $p_1, p_2\in P$ and $q_1, q_2 \in Q$ do not have a common neighbor in $C(x)$ and $C(\iota(x))$ respectively, the above proof establishes rigidity.

Next we focus on $P$, and try to find positions for $p_1, p_2\in P$ determined by the hypotheses.
Assume, in the following, that  $p_1, p_2\in P$  have a common degree-$3$ ramification neighbor $u$ in $C(x)$.

Star unfold $P$ with respect to some $x \in P$,  to $S_P$. 
The region of $S_P$ exterior to the wedge $(x_1, u, x_2)$ is uniquely determined and identical in $S_Q$.
See Fig.~\figref{Tent_star_Cx}(c).

Take a point $x_\a$ on the circle of center ${u}$ and radius $|x_1 u|=|x_2 u|$.
We now argue that positions of $x_\a$ on this circle allow $p_1$ and $p_2$ to vary while maintaining all outside of the $(x_1, u, x_2)$ wedge fixed.

Let $\angle {x}_\a{u}{x}_1 = 2 \a$.
On the bisector of that angle incident to $u$, one can uniquely determine a point ${p}_1$ such that $\angle  {x}_1 {{p_1}} {u} = \pi - \frac{1}{2}\o(p_1)$.
Similarly, one can uniquely determine a point ${p}_2$ on the bisector of that angle  $\angle {x}_\a{u}{x}_2$, such that $\angle  {x}_2 {{p_2}} {u} = \pi - \frac{1}{2}\o(p_2)$.

Thus we have identified a continuous $1$-parameter family of star-unfoldings, and consequently of convex polyhedra, verifying the hypotheses.
\end{proof}


\section{Vertex-Merging}
\seclab{VertexMerging}

Digon-tailoring is, in some sense, the opposite of \emph{vertex-merging}, a technique introduced by A.~D.~Alexandrov~\cite[p.~240]{a-cp-05},
and subsequently used later by others, 
see e.g.~\cite{z-ipt-07}, \cite{ov-ceccc-14}, \cite{o-vtcp-2020}.
We will employ vertex-merging
in Chapters~\chapref{TailoringFlattening} and~\chapref{Punfoldings}, 
and  focus on it in Part~II.

Consider two vertices $v_1, v_2$ of $P$ of curvatures $\o_1, \o_2$, with $\o_1 + \o_2 < 2 \pi$, and cut $P$ along a 
geodesic segment $\g$ joining $v_1$ to $v_2$. 
Construct a planar triangle $T = \bar{v}' \bar{v}_1 \bar{v}_2$ of base length $|\bar{v}_1 - \bar{v}_2|=|\g|$
and the base angles equal to $\o_1 /2$ and $\o_2 /2$ respectively. 
Glue two copies of $T$ along the corresponding lateral sides, and further glue the two bases of the copies to the two ``banks" of the cut of $P$ along $\g$. 
By Alexandrov's Gluing Theorem (AGT), the result is a convex polyhedral surface $P'$. 
On $P'$, the points (corresponding to) $v_1$ and $v_2$ are no longer vertices because exactly the angle deficit at each has been
sutured-in; they have been replaced by a new vertex $v'$ of curvature $\o' = \o_1 +\o_2$.
So vertex-merging always reduces the number of vertices of $P$ by one.
See Fig.~\figref{VertMerge}.

\begin{figure}[htbp]
\centering
\includegraphics[width=1.0\textwidth]{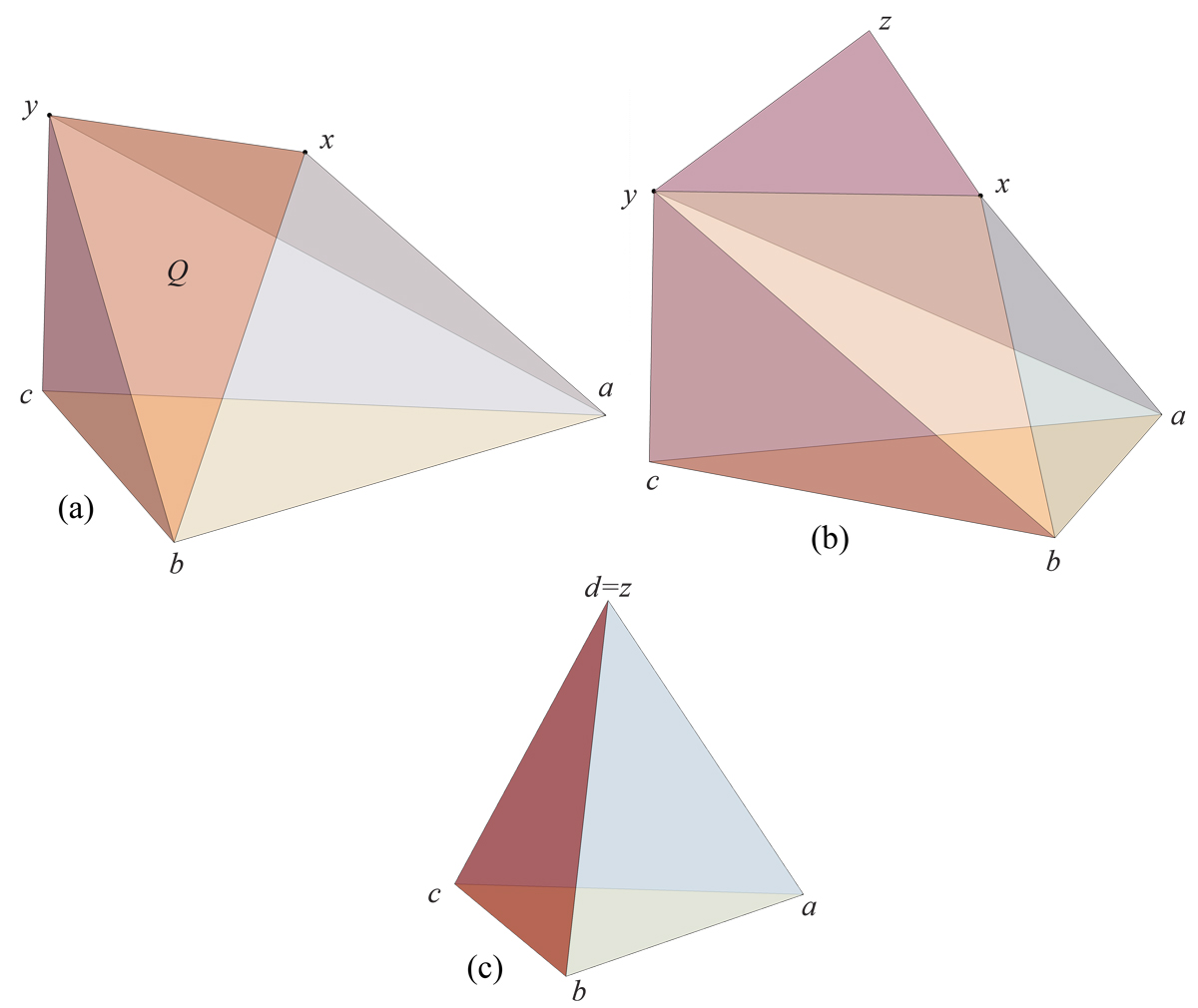}
\caption{(a)~$Q$ is a $5$-vertex, $6$-face polyhedron,
symmetric about a vertical plane through $xy$.
Its base $abc$ is an equilateral triangle.
(b)~Vertex merging $x$ and $y$ by gluing a pair of $xyz$ triangles.
(c)~The merging reduces $Q$ to a regular tetrahedron (not to same scale).
Cf.~Fig.~\protect\figref{Tetra_3D_digons}(a,b).
}
\figlab{VertMerge}
\end{figure}

In order to repeat vertex-merging, we need to know when there is 
a pair of vertices that can be merged.
This is answered in the following lemma.

\begin{lm}
\lemlab{IsoTetra}
Every convex polyhedron $Q$ has at least one pair of vertices admitting merging,
unless it is an isosceles tetrahedron or a doubly-covered triangle.
\end{lm}

\noindent
Recall from Chapter~\chapref{IntroductionPartI} that
an \emph{isosceles tetrahedron} $T$ is a tetrahedron with 
four congruent faces, and each vertex of curvature $\pi$.
See Fig.~\figref{IsoscTetra}.
\begin{proof}
If there is a pair of vertices whose sum of curvatures is strictly less than $2\pi$,
then vertex-merging is possible, as just described.
So assume that, for any two vertices of $Q$, their sum of curvatures is at least $2 \pi$.
In this case, it must be that $n \leq 4$. Indeed, since $\sum_{v \in Q} \o(v) = 4\pi$ 
(by the Gauss-Bonnet theorem),
if the sum of at least $5$ positive numbers is $4 \pi$ then the smallest two have sum $< 2\pi$.

If $n=3$, the cut locus of any vertex is a line-segment, by Lemma~\lemref{CutLocusBasic}(i), so $Q$ is a  doubly-covered triangle (see Lemma~\lemref{Path}).

If $n=4$ then necessarily all vertex curvatures of $Q$ are $\pi$. 
Indeed, if the sum of $4$ positive numbers is $4 \pi$ then the smallest two have the sum $\leq 2\pi$, with equality if and only if all are $\pi$. 
So $Q$ is an isosceles tetrahedron.
\end{proof}



\chapter{Domes and Pyramids}
\chaplab{Domes}
One of our goals in this work, achieved in Theorem~\thmref{MainTailoring}, is to show that if $Q$ can be obtained from $P$ by sculpting, then it can also be obtained from $P$ by tailoring.
A key step (Lemma~\lemref{Slice2gDomes}) repeatedly slices off shapes we call g-domes.
Each g-dome slice can itself be achieved by slicing off pyramids, i.e., by suitable vertex truncations.
Lemma~\lemref{VertexTruncation} will show that slicing off a pyramid can be achieved by tailoring, and thus leading to Theorem~\thmref{MainTailoring}.

Our main goal in this chapter is to prove that g-domes can be partitioned into ``stacked" pyramids, a crucial part of the above process.
We illustrate the slice~$\to$~g-domes and g-dome~$\to$~pyramids of the process
on a simple example, a tetrahedron $Q$ inside a cube $P$.
Later (in Chapter~\chapref{TailoringSculpting}) this example will be completed
by tailoring the pyramids.


\section{Domes}
\seclab{Domes}

As usual, a \emph{pyramid} $P$ is the convex hull of a convex polygon \emph{base} $X$,
and one vertex $v$, the \emph{apex} of $P$, that does not lie in 
the plane of $X$. The degree of $v$ is the number of vertices of $X$. 

A \emph{dome} is a convex polyhedron $G$ with a distinguished face $X$, the \emph{base}, and such that every other face of $G$ shares a (positive-length) edge with $X$.
Domes have been studied primarily for their combinatorial~\cite{eppstein2020treetopes},~\cite{eppstein2013bounds} or unfolding~\cite{do-gfalop-07} properties.
In~\cite{eppstein2020treetopes} they are called ``treetopes" because removing the base edges from the $1$-skeleton leaves a tree, which the author calls the \emph{canopy}.\footnote{
These polyhedra are not named in~\cite{eppstein2013bounds}.}
Here we need a slight generalization.

A \emph{generalized-dome}, or \emph{g-dome} $G$, has a base $X$, with every other face of $G$ sharing either an edge or a vertex with $X$.
Every dome is a g-dome, and it is easy to obtain every g-dome as the limit of domes.
An example is shown in Fig.~\figref{figg-dome}, which also shows that removing base edges from the $1$-skeleton does not necessarily leave a tree: $(v_1,x_2,v_2)$ forms a cycle.
Let us define the \emph{top-canopy} $T$ of a g-dome $G$ as the graph that results by deleting from the $1$-skeleton of $G$ all base vertices and their incident edges. 
In Fig.~\figref{figg-dome} the top-canopy is $v_1 v_2$.

\begin{figure}
\centering
 \includegraphics[width=0.5\textwidth]{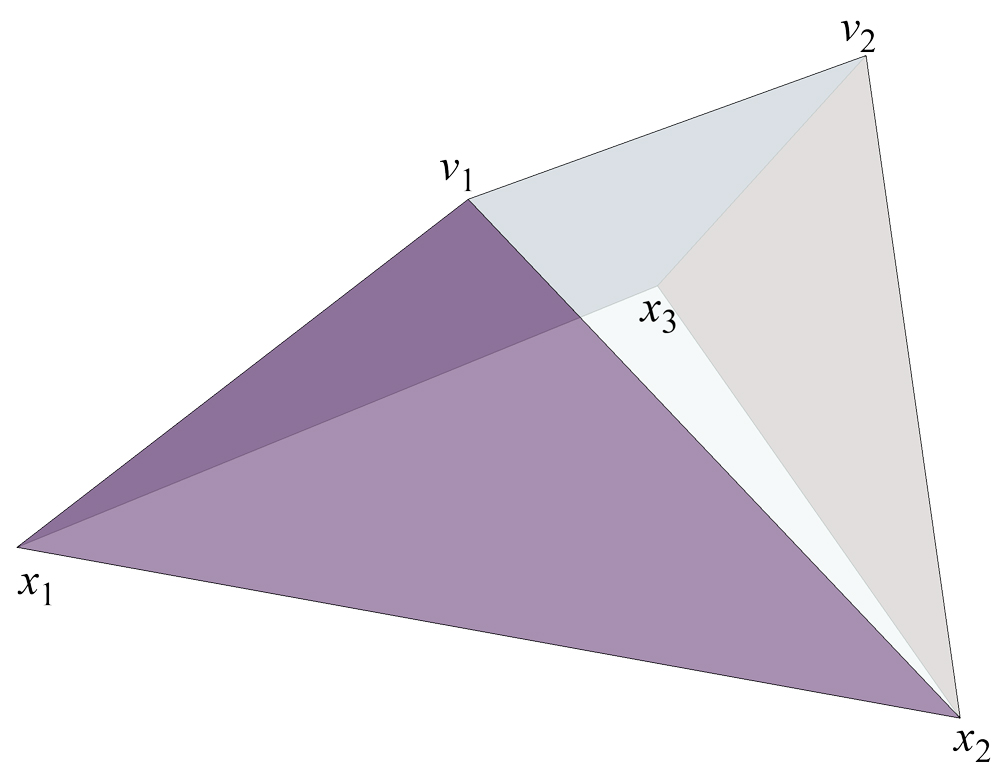}
\caption{A g-dome with base $\triangle x_1 x_2 x_3$ and top-canopy $v_1 v_2$.
}
\figlab{figg-dome}
\end{figure}

\begin{lm}
\lemlab{canopy}
The top-canopy $T$ of a g-dome $G$ is a tree.
\end{lm}
\begin{proof}
If $G$ is a dome, the claim follows, because even including the edges incident to the base $X$ results in a tree, and removing those edges leaves a smaller tree.

If $G$ is not a dome, then slice it with a plane parallel to, and at small distance above, the base. 
The result is a dome, and we can apply the previous reasoning.
\end{proof}


\section{Cube/Tetrahedron Example}
\seclab{CubeTetra12}

In this section we start to track an example, tailoring a cube
to a tetrahedron. 
$P$ is a triangulated unit cube.
$Q$ is a right tetrahedron in the corner of $P$;
see Fig.~\figref{CubeTetra_1}(a).
$Q$ can be obtained from $P$ by a single slice by plane 
$\Pi= p_1 p_3 p_8$.
The algorithm to be described in Chapter~\chapref{TailoringAlg1}
will first produce two g-domes $G_1$ and $G_2$,
and reducing those g-domes will lead to four pyramids $P_1,P_2,P_3,P_4$.
Later (in Chapter~\chapref{TailoringSculpting})
we will show how each of those pyramids is reduced by digon-tailoring,
finally yielding $Q$.


\subsection{Slice~$\to$~G-Domes for Cube/Tetrahedron}
Now we describe at an intuitive level what will be proved in
Lemma~\lemref{Slice2gDomes}:
that the slice by $\Pi$ can be effected
by partitioning
the sliced portion of $P$ into g-domes.

We start by 
rotating the plane $\Pi$ about the edge $p_1 p_3$ until it encounters a face of $P$;
call that plane $\Pi_0 = p_1 p_2 p_3$.
Now continue rotating until we reach $\Pi_1$ when
(a)~The portion of $P$ between $\Pi_0$ and $\Pi_1$ is a g-dome with
base on $\Pi_1$, and
(b)~Any further rotation would render the portion between the planes 
to a polyhedron no longer a g-dome.
$\Pi_1 = p_1 p_3 p_7 p_5$, because any further rotation would isolate
the face $p_5 p_6 p_7$ of $P$, no longer touching $\Pi_1$.
Note here the top face of the cube is triangulated with the diagonal $p_5 p_7$.
See Fig.~\figref{CubeTetra_1}(bc).
Continuing to rotate, we reach the original $\Pi$, partitioning the g-dome
illustrated in Fig.~\figref{CubeTetra_1}(d).
So $G_1$ has base $X_1= p_1 p_3 p_7 p_5$ and top-canopy $p_2 p_6$,
and $G_2$ has base $X_2= p_1 p_2 p_3$ and top-canopy $p_5 p_7$.
It is clear that removing $G_1 \cup G_2$ would reduce $P$ to $Q$.

\begin{figure}[htbp]
\centering
\includegraphics[width=1.1\linewidth]{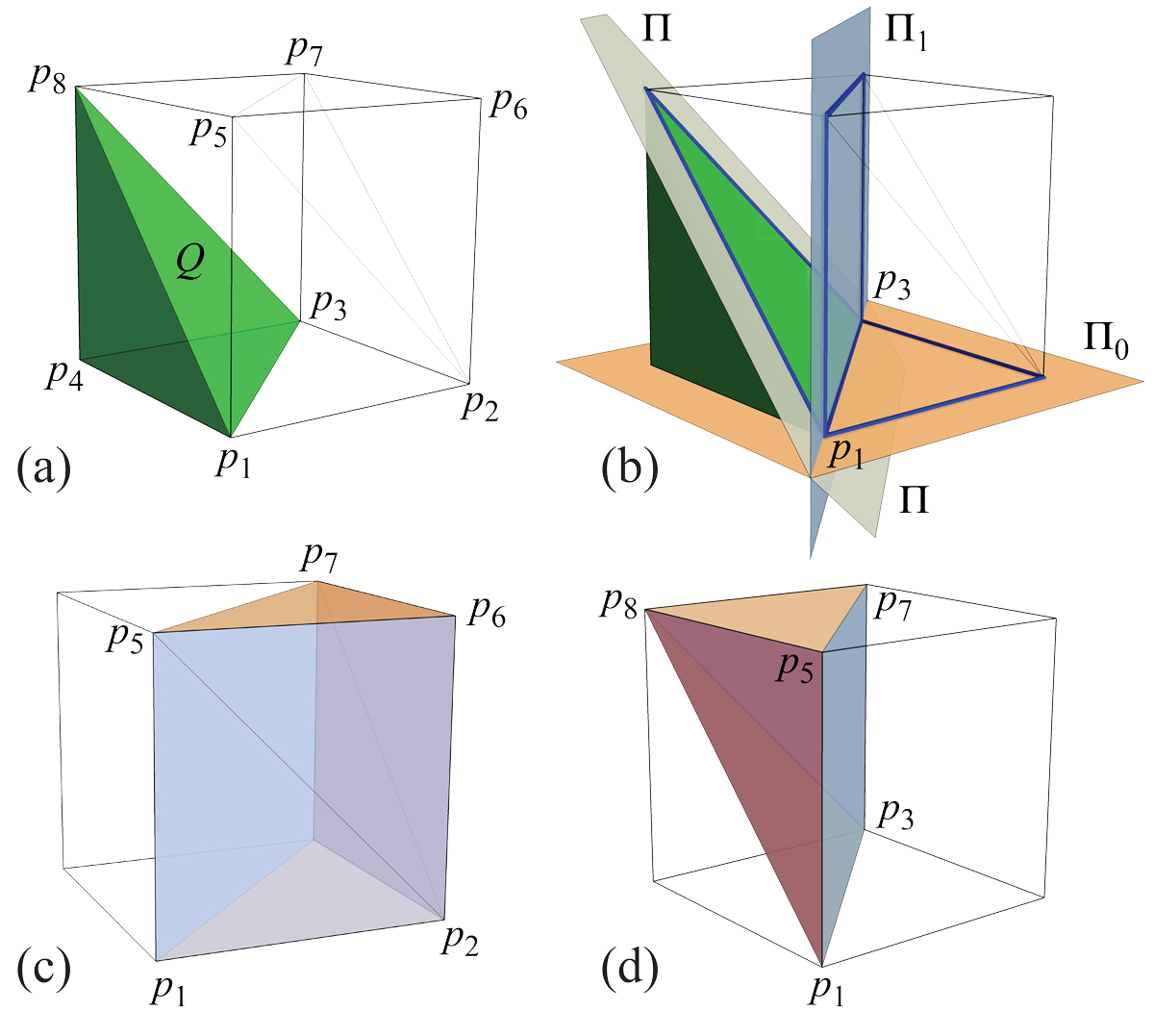}
\caption{(a)~$Q$ in corner of cube $P$.
(b)~Partitioning sliced portion of $P$ into two g-domes.
(c,d)~G-domes $G_1$ and $G_2$.}
\figlab{CubeTetra_1}
\end{figure}

\begin{figure}
\centering
 \includegraphics[width=1.0\textwidth]{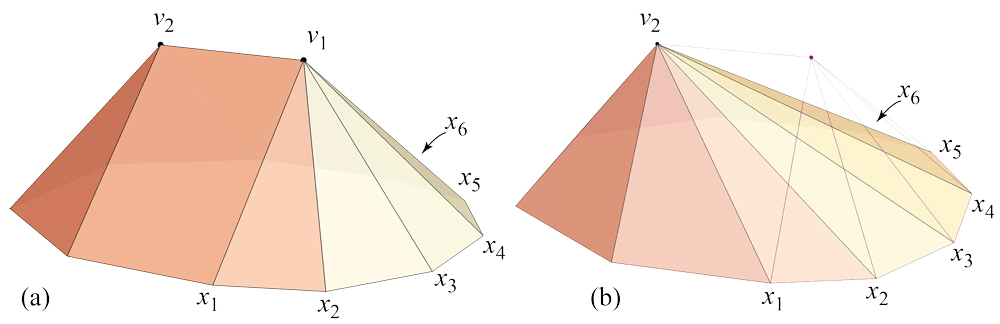}
\caption{
(a)~A dome $G$. Vertex $v_1$ of degree $k_1=6$ is adjacent to $X$, i.e.,
there are edges $v_1 x_i$. 
(b)~$G'$ after removal of $v_1$.}
\figlab{DomeSlice_begend}
\end{figure}


\section{Proof: G-dome~$\to$~Pyramids}

We now prove the reduction of a g-dome to pyramids,
using Fig.~\figref{DomeSlice_begend} to illustrate the proof details.
Afterward we will return to the cube/tetrahedron example
to apply the constructive proof to $G_1$ and $G_2$.

\begin{thm}
\thmlab{DomePyr}
Every g-dome $G$ of base $X$ can be partitioned into a finite sequence of $n$ pyramids $P_i$ with the following properties:
\begin{itemize}
\squeezelist
\item Each $P_i$ has a common edge with $X$.
\item Each $G_j = G \setminus \bigcup_{i=1}^j P_i$ is a g-dome, for all $j=1,...,n$.
\item The last pyramid $P_n$ in the sequence has the same base $X$ as $G$.
\end{itemize}
\end{thm}

The partition into pyramids specified in this theorem
has the special properties listed. Without these properties,
it would be easier to prove that every g-dome may be partitioned into pyramids.
But the properties are needed for the subsequent digon-tailoring
described in Chapter~\chapref{TailoringSculpting}.
In particular, the pyramids are ``stacked" in the sense that each can be
sliced-off in sequence, with each slice retaining $X$ unaltered.


The proof is a double induction, and a bit intricate.
One induction simply removes one vertex $v_1$ from the top-canopy.
The second induction, inside the first one, reduces the degree of $v_1$ to achieve removal of $v_1$, at the cost of increasing the degree of $v_2$.

\begin{figure}
\centering
 \includegraphics[width=1.0\textwidth]{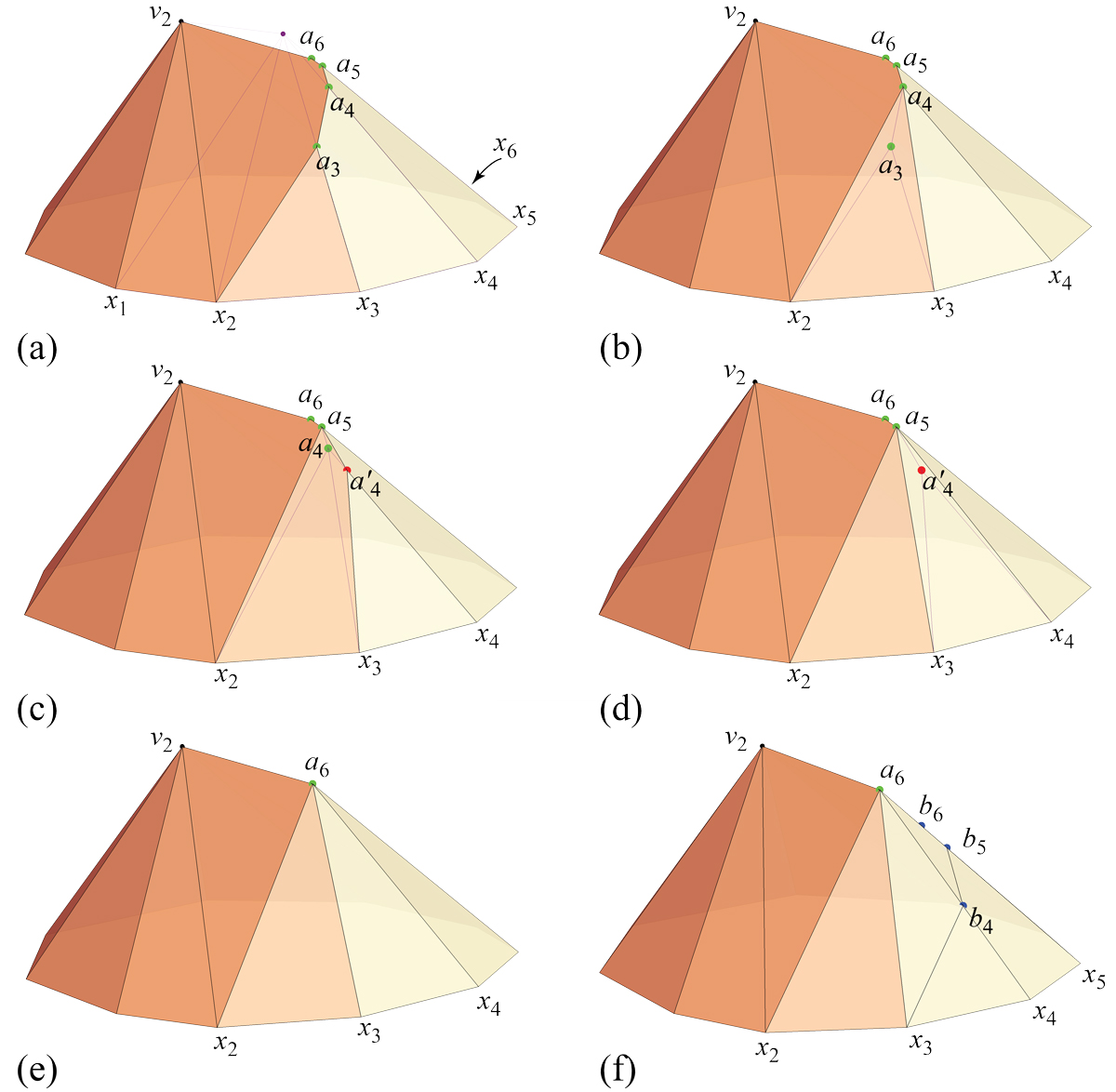}
\caption{
(a)~After slicing by the plane $\Pi_1=x_1x_2v_2$.
(b)~Removing $a_3$.
(c)~Removing $a_4$ creates $a'_4$.
(d)~Removing $a'_4$.
(e)~$a_6$ has replaced $v_1$ with one lower degree.
(f)~Intersection points with $\Pi_2=x_2 x_3v_2$.
}
\figlab{SliceAway_123456}
\end{figure}

\noindent
\begin{proof}
Let $m$ be the number of vertices in the top-canopy $T$ of $G$.
If $m=1$, $G$ is already a pyramid, and we are finished. So assume $G$'s top-canopy $T$ has at least two vertices.
Choose $v_1$ to be a leaf of $T$ given by Lemma~\lemref{canopy}, and $v_2$ its unique parent.
Let $v_1$ be adjacent to $k$ vertices of $X$.
If $G$ is a dome, $v_1$ has degree $k+1$; if $G$ is a g-dome, then possibly $v_1$ has degree $k+2$.
Since the later case changes nothing in the proof, we assume for the simplicity of exposition that $G$ is a dome.
Our goal is to remove $v_1$ through a series of pyramid subtractions.

Let the vertices of $X$ adjacent to $v_1$ be $x_1,x_2,\ldots,x_k$. 
Let $\Pi_1$ be the plane $x_1x_2v_2$.
This plane cuts into $G$ under $v_1$, and intersects the edges $v_1 x_i$, $i \ge 3$, in points $a_i$.
In Fig.~\figref{SliceAway_123456}(a), those points are $a_3,a_4,a_5,a_6$.
Remove the pyramid whose apex is $v_1$ and whose base (in our example) is $x_1,x_2,a_3,a_4,a_5,a_6,v_2$.

We now proceed to reduce the chain of new vertices $a_3,\ldots,a_k$ one-by-one until only $a_k$ remains.

First, with the plane $x_2x_3a_4$, we slice off the tetrahedron whose degree-$3$ apex is $a_3$; Fig.~\figref{SliceAway_123456}(b).
Next, with the plane $x_2x_3a_5$, we slice off the pyramid with apex $a_4$. 
Unfortunately, because $a_4$ has degree-$4$, this introduces a new vertex $a'_4$;  Fig.~\figref{SliceAway_123456}(c).
So next we slice with $x_3x_4a_5$ to remove the tetrahedron whose degree-$3$ apex is $a'_4$; Fig.~\figref{SliceAway_123456}(d).
Continuing in this manner, alternately removing a tetrahedron followed by a pyramid with a degree-$4$ apex, we reach Fig.~\figref{SliceAway_123456}(e).

Note that $v_1$ was connected to $k=6$ vertices of $X$, but $a_6$ is only connected to $5$:
the connection of $v_1$ to $x_1$ has in a sense been transferred to $v_2$. 
In general, $a_k$ has degree one less than $v_1$'s degree, 
and the degree of $v_2$ has increased by $1$.

Now we repeat the process, starting by slicing with $\Pi_2 = x_2x_3v_2$, which intersects the $a_k x_i$ edges at $b_4,\ldots,b_k$.
We remove the pyramid apexed at $a_6$ with base (in our example) of $x_2,x_3,b_4,b_5,b_6,v_2$; Fig.~\figref{SliceAway_123456}(f).
The same methodical technique will remove all but the last new vertex $b_k$, which replaces $a_k$ but has degree one smaller.

Continuing the process, slicing with $\Pi_i = x_ix_{i+1}v_2$, up to $i=k-1$, will lead to the complete removal of $v_1$, as previously illustrated 
in Fig.~\figref{DomeSlice_begend}(b), completing the inner induction.
Inverse induction on the number of vertices of the g-dome then completes the proof:
each step reduces the number of top vertices by $1$,
and the degree of $v_2$ increases by $k+1$, the degree lost at $v_1$.

With $G_0=G$, each $G_j$ is the intersection of $G_{k-1}$ with a closed half-space containing a base edge, so it is convex for all $j=1,...,n$. 
Indeed each $G_j$ is a g-dome, because all untouched faces continue to meet $X$ in either an edge or a vertex, and new faces always share an edge with $X$.
\end{proof}


The partition of a g-dome into pyramids given by Theorem~\thmref{DomePyr} is not unique. For our example in Fig.~\figref{DomeSlice_begend}(a), 
we finally get the pyramid apexed at $v_2$ in (b) of the figure,
but we could as well have ended with a pyramid apexed at $v_1$.



We will see in Chapter~\chapref{TailoringSculpting} that one g-dome of $O(n)$ vertices reduces to $O(n^2)$ pyramids of constant size,
and $O(n)$ pyramids each of size $O(n)$.


\subsection{G-domes~$\to$~Pyramids for Cube/Tetrahedron}

We now return to the cube/tetrahedron example, and show that following
the proof of Theorem~\thmref{DomePyr} partitions the two g-domes $G_1$ and $G_2$ into pyramids.
We will see that the two g-domes partition into two pyramids each.

The analysis for $G_1$ is illustrated in Fig.~\figref{CubeTetra_2}. 
Note $G_1$ has been reoriented so that its base $X_1= p_1 p_3 p_7 p_5$ 
is horizontal, and relabeled $x_1 x_2 x_3 x_4$ to match the proof.

We proceed to describe each step, making comparisons 
to Figs.~\figref{DomeSlice_begend} and~\figref{SliceAway_123456}.
\begin{enumerate}[label=(\alph*)]
\squeezelist
\item
In (a) of the figure,
the base of the g-dome has been oriented horizontally, and the correspondence between
the cube labels and the labels used in the proof of Theorem~\thmref{DomePyr} are shown.
The top-canopy of $G_1$ is $v_1 v_2$, with $v_1$ of degree-$5$
and $v_2$ of degree-$3$. Eventually $v_1$ will be removed and the degree
of $v_2$ increased by $1$, as in Fig.~\figref{DomeSlice_begend}.

\item
The first step is to slice off a pyramid apexed at $v_1$ by the plane $x_1 x_2 v_2$.
Due to the coplanar triangles, this has the effect of ``erasing" the diagonal $x_1 v_1$,
as would have occurred if that diagonal had a dihedral angle different from $\pi$.

\item
Next, the slice
$x_2 x_3 v_2$ cuts off a pyramid $P_1$ apexed at $v_1$, a pyramid we will analyze in
Section~\secref{PyramidRemovals}, leaving the shape in~(d).

\item
A new vertex $a_3$ is created by the slice in~(c), with degree-$4$.
This is the ``replacement" for $v_1$ but of smaller degree.
This corresponds to $a_3$ in Fig.~\figref{SliceAway_123456}(a),
although here there are no further vertices $a_4, \ldots$ along an $a$-chain.

\item
We again slice with $x_2 x_3 v_2$, which reduces the degree of $a_3$ to $3$.
This $a_3$ corresponds to $b_4$ in Fig.~\figref{SliceAway_123456}(f).

\item
A final slice by $x_3 x_4 v_2$ removes $a_3$, 
leaving the pyramid $P_2$ shown.
This completes 
the removal of $v_1$ in~(a),  increasing the degree of $v_2$ from $3$ to $4$.
Because what remains is a pyramid, the processing stops.

\end{enumerate}
So $G_1$ is partitioned into two pyramids, $P_1$ and $P_2$,
apexed at $p_2$ and $p_6$ respectively.

\begin{figure}[htbp]
\centering
\includegraphics[width=1.1\linewidth]{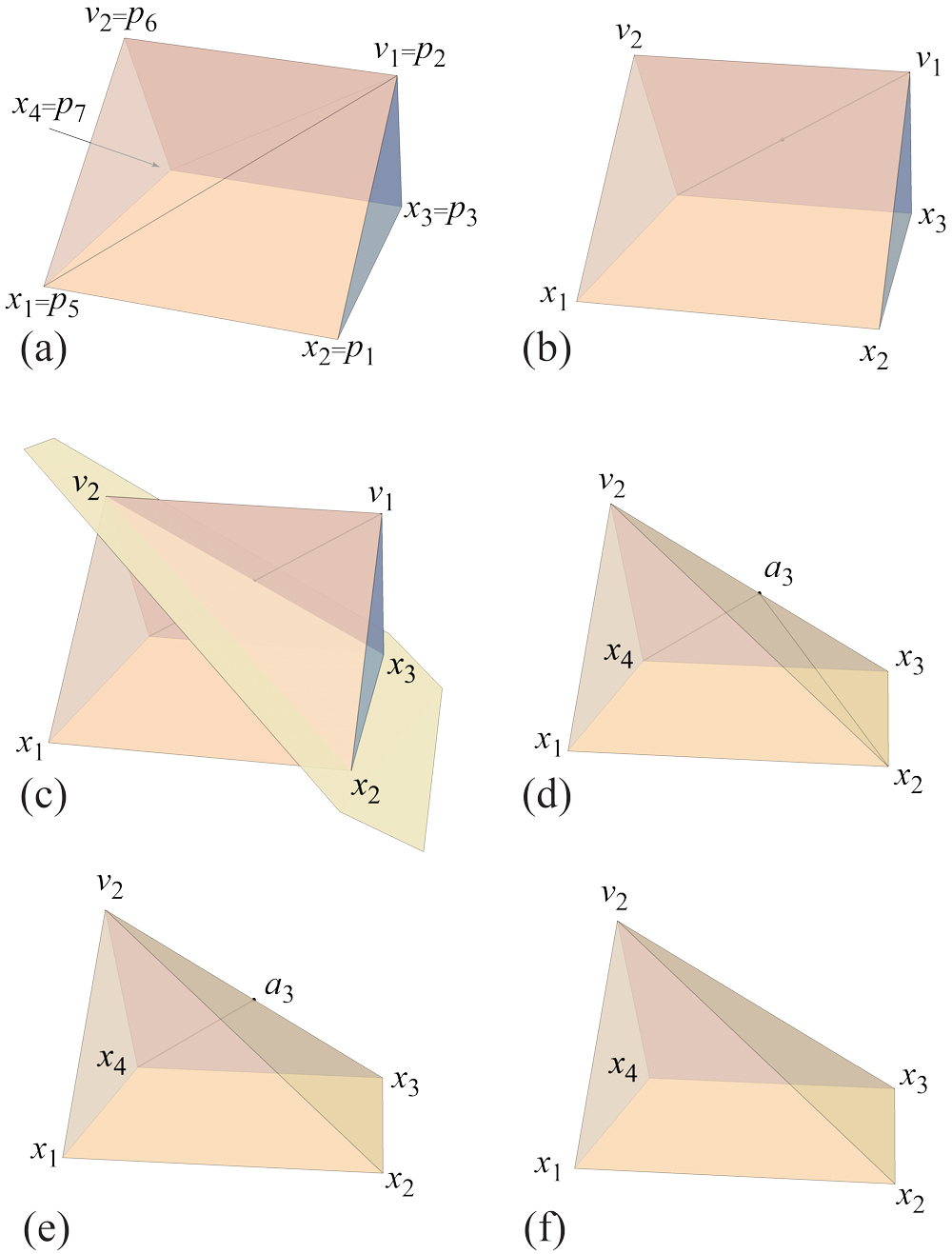}
\caption{(a)~g-dome $G_1$.
(b)~After slicing by $x_1 x_2 v_2$.
(c,d)~Slicing by $x_2 x_3 v_2$.
(e,f)~Two further slices remove $a_3$.}
\figlab{CubeTetra_2}
\end{figure}


We perform a similar analysis for the simpler $G_2$, shown in Fig.~\figref{CubeTetra_3},
again reoriented so that its base $p_8 p_1 p_3$ is horizontal and relabeled
$x_1 x_2 x_3$.
Here the top-canopy is $v_1 v_2$, and the first slice $x_1 x_2 v_2$
reduces $G_2$ to a pyramid $P_2$. So $G_2$ is also partitioned into two pyramids,
$P_3$ and $P_4$, apexed at $p_5$ and $p_7$ respectively.
Together $G_1$ and $G_2$ are partitioned into four pyramids.

\begin{figure}[htbp]
\centering
\includegraphics[width=0.9\linewidth]{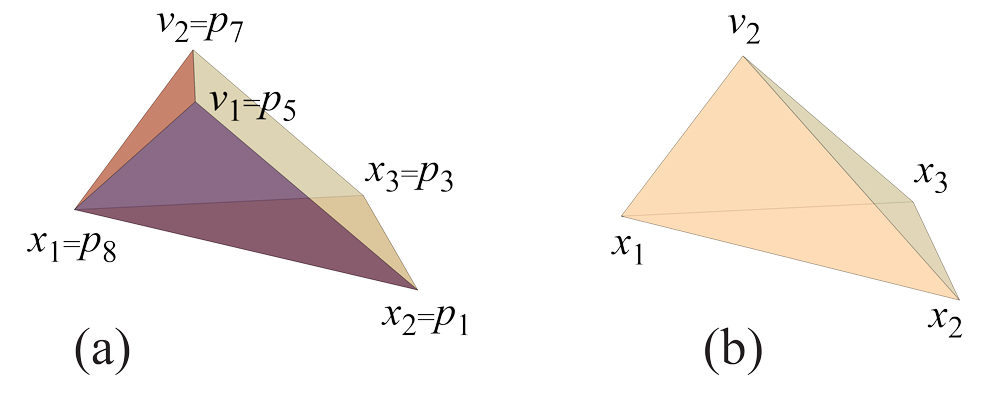}
\caption{(a)~$G_2$ with corresponding cube labels.
(b)~After slicing with $x_1 x_2 v_2$.}
\figlab{CubeTetra_3}
\end{figure}

We will complete this example by tailoring the four pyramids in the next chapter.


\chapter{Tailoring via Sculpting}
\chaplab{TailoringSculpting}

\begin{sloppypar}
In this chapter we complete the proof that one slice of $P$ by plane $\Pi$
can be tailored to the face of $Q$ lying in $\Pi$,
following the sequence
slice~$\to$~g-domes~$\to$~pyramids~$\to$~tailoring.
The previous chapter established the g-domes~$\to$~pyramids link.
Here we first prove the relatively straightforward slice~$\to$~g-domes process,
and then concentrate on the more complex pyramid~$\to$~tailoring step.
\end{sloppypar}


\section{Slice~$\to$~G-domes}

\begin{lm}
\lemlab{Slice2gDomes}
With $Q \subset P$ and $\Pi$ a plane slicing $P$ and containing a face $F$ of $Q$,
the sliced-off portion $P'$ can be partitioned into a \emph{fan} of $O(n)$ g-domes
$G_1, G_2, \ldots$, a fan in the sense that the bases of the g-domes all share
a common edge of $F$.
\end{lm}

\begin{figure}
\centering
 \includegraphics[width=0.5\textwidth]{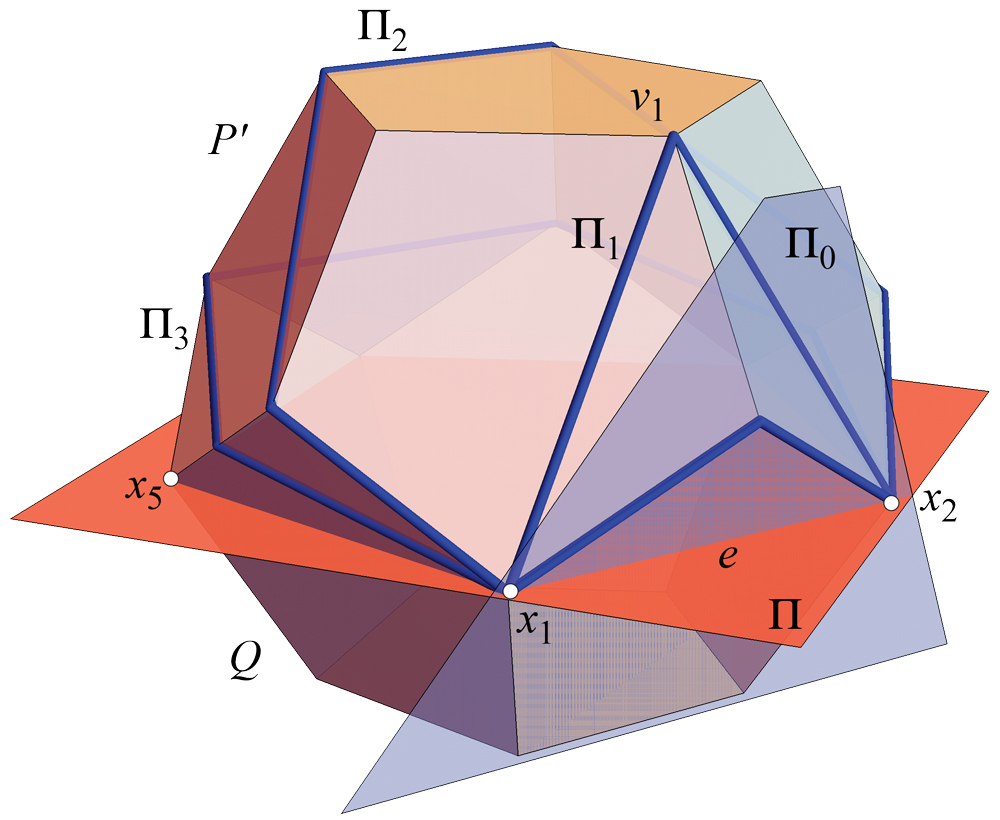}
\caption{Dodecahedron sliced by $\Pi$.
The polyhedron between each pair $\Pi_i, \Pi_{i+1}$ is a g-dome.
}
\figlab{DodecaGdome}
\end{figure}

\begin{proof}
Assume $\Pi$ is horizontal, with $Q$ the portion below $\Pi$ and $P'$ the portion above.
Denote by $x_i$ the vertices of $Q$ in $\Pi$, $i=1,\ldots,m$; call the top face of $Q$ with these vertices $F$.
Let $e$ be any edge of $F$, say $e=x_1 x_2$, and let $F'$ be the face of $Q$ sharing $e$ with $F$.
Call the plane lying on $F'$ $\Pi_0$.

Now imagine rotating $\Pi_0$ about $e$ toward $P'$, noting as it passes through each vertex $v_1,v_2,\ldots$ in that order.
For perhaps several consecutive vertices, the portion of $P'$ between the previous and the current plane is a g-dome, but rotating further
takes it beyond a g-dome.
More precisely, let $\Pi_1$ through $v_{j_1}$ be the plane such that, in the sequence
$$v_1,v_2,\ldots,v_{j_1},v_{j_1+1},\ldots,v_{j_2}, v_{j_2+1}\ldots \;,$$
the portion $G_1$ of $P'$ between $\Pi_0$ and $\Pi_1$, including the vertices $v_1,v_2,\ldots,v_{j_1}$, is a g-dome based on $\Pi_1$, but
the portion rotating further to include one more vertex, $v_{j_1+1}$, is not a g-dome.
$\Pi_2$ through $v_{j_2}$ is defined similarly: the portion $G_2$ between $\Pi_1$ and $\Pi_2$ including $v_{j_1+1},\ldots,v_{j_2}$
is a g-dome based on $\Pi_2$, but including $v_{j_2+1}$ it ceases to be a g-dome.
Here for each $\Pi_i, \Pi_{i+1}$ pair,
the base $X$ of the g-dome lies in $\Pi_{i+1}$.

Fig.~\figref{DodecaGdome} illustrates the process.
Here $\Pi_1$ is through $v_1=v_{j_1}$, and the last g-dome lies between $\Pi_3$ and $\Pi$.

So we have now partitioned $P'$ into g-domes $G_1, G_2, \ldots$.
\end{proof}

We will invoke Lemma~\lemref{VertexTruncation} (ahead) 
to reduce each g-dome to its base by tailoring, in the order $G_1, G_2, \ldots$. 
This will reduce $P'$ to just the top face $F$ of $Q$.

Having established the claim of the lemma for one slice, it immediately follows that it holds for an arbitrary number of slices. 
This was already illustrated in the cube/tetrahedron example in Chapter~\chapref{Domes},
and we will use it again in Theorem~\thmref{MainTailoring}.

We should note that it is at least conceivable that
only a single g-dome is needed above.
See Open Problem~\openref{OneXgdome}.


\section{Pyramid~$\to$~Tailoring}

The goal of this section is to prove that removal of a pyramid, i.e., a vertex truncation, can be achieved by (digon-)tailoring.
We reach this in Lemma~\lemref{VertexTruncation}: 
a degree-$k$ pyramid $P$ can be removed by $k-1$ tailoring steps, each step excising one vertex by removal of and then sealing a digon.
The first $k-2$ of these digons each have one endpoint a vertex, and so leave the
total number of vertices of $P$ at $k+1$.
The $(k-1)$-st digon has both endpoints vertices, and so its removal reduces the 
number of vertices to the $k$ base vertices.

We start with Lemma~\lemref{vertex_small} which claims the result but only under the assumption that the slice plane is close to the removed vertex.
Although this lemma is eventually superseded, it establishes the notation and the main idea.
Following that, Lemma~\lemref{pyramid} removes the ``sufficiently small" assumption of Lemma~\lemref{vertex_small}, but in the special case of $P$ a pyramid. 
Finally we reach the main claim in Lemma~\lemref{VertexTruncation},
which shows this special case encompasses the general case.

In the following, we use $\partial S$ to indicate the $1$-dimensional boundary of a $2$-dimensional surface patch $S$.

\subsection{Notation}
To help keep track of the notation throughout this critical section, we list the main
symbols below.
\begin{itemize}
\item Initially $P$ and $Q$ are polyhedra with $Q \subset P$, with $P$ sliced by
plane $\Pi$ to truncate vertex $v$ of degree $k$.
\item Later we specialize $P$ to be the pyramid sliced off.
\item $X = \Pi \cap P$ is the base of the pyramid, with vertices 
$\partial X = x_1 x_2 \ldots x_k$.
\item The lateral faces of $P$ are denoted by $L$, so $P = X \cup L$.
\item $D(x_i,y_i)$ is a digon with endpoints $x_i \in \partial X$ and $y_i$.
\item $P_j$ is the modified pyramid after $j$ digon removals.
\item $C(x_i,P_j)$ is the cut locus $C(x_i)$ on $P_j$.
\item The first ramification point of $C(x_i,P_j)$ beyond $y_i$ is $a_i$.
\end{itemize}
Notation will be repeated and supplemented within each proof.


\section{Small volume slices}

\begin{lm}
\lemlab{vertex_small}
Let $P$ be a convex polyhedron, and $Q$ the result obtained by slicing $P$ with a plane $\Pi$ at sufficiently small distance to a vertex $v$ of $P$, and removing precisely that vertex.
Then $Q$ can be obtained from $P$ by $k-1$ tailoring steps.
\end{lm}

\begin{proof}
Let the vertex $v$ to be removed have degree $k$ in the $1$-skeleton of $P$.
Let $e_i$, $i=1,\ldots,k$, be the edges incident to $v$, and $x_i$ the intersection of the slicing plane $\Pi$ with those edges:
$\{x_i\} = \Pi \cap e_i$.

We will illustrate the argument with the right triangular prism shown in Fig.~\figref{Sculpting}, 
where $k=3$ and $\Pi = x_1 x_2 x_3$.
Note that we do not exclude the case when some (or all) of the $x_i$ are vertices of $P$.

\begin{figure}
\centering
 \includegraphics[width=0.5\textwidth]{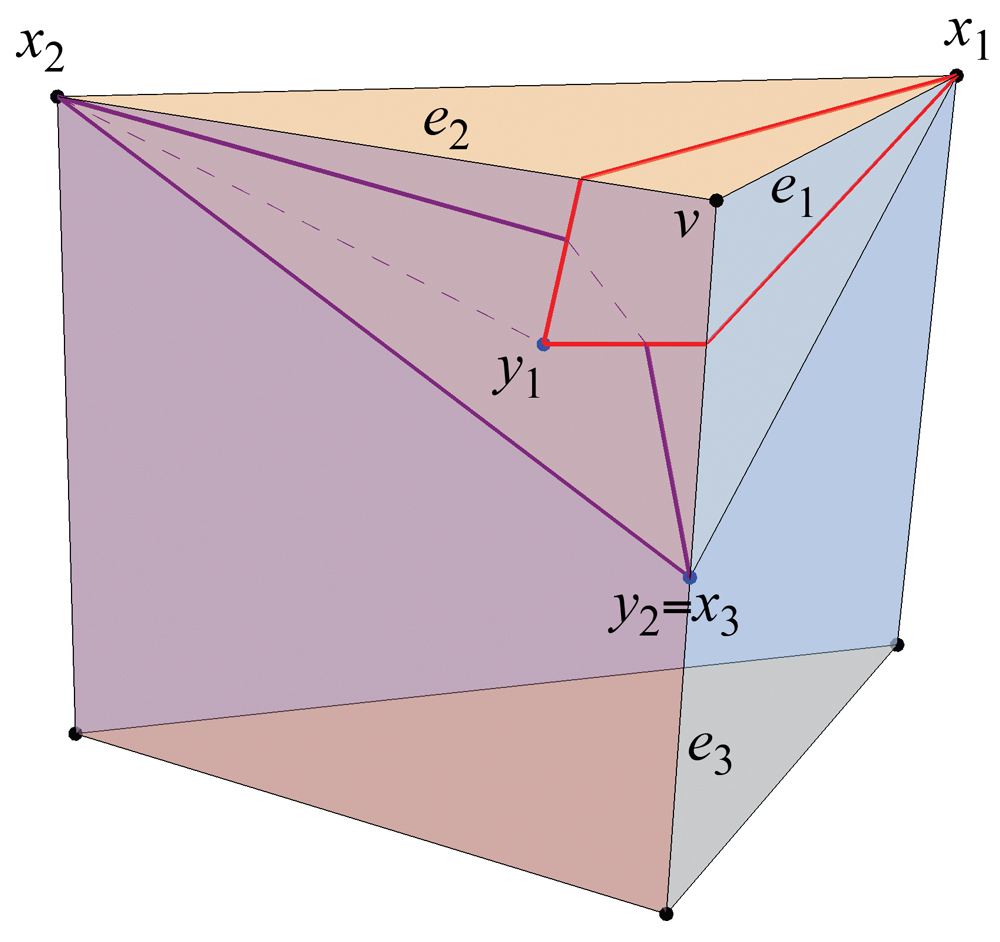}
\caption{$P$ is a prism whose top face is an isosceles right triangle.
The truncated vertex $v$ has degree $k{=}3$.
$x_3$ is the midpoint of $e_3$.
Digon $D_1 \supset \{v\}$ is shown red; $D_2 \supset \{y_1\}$ is purple.
$D_1$ is sutured closed before $D_2$ is excised.
The last replacement vertex $y_{k-1}$ must be identical to $x_k$.
}
\figlab{Sculpting}
\end{figure}

Denote by $\o_P(x_i)$ $i=1,\ldots,k$ the curvatures of $P$ at $x_i$, and by $\o_Q(x_i)$ the corresponding curvatures of $Q$.
The curvature $\o_P(v)$ will be distributed to the $x_i$.

In the figure, $\o_P(v)=90^\circ$, the curvatures of the three $x_i$ are $135^\circ,135^\circ,0^\circ$ in $P$,
and approximately $156^\circ,156^\circ,48^\circ$ in $Q$.
Indeed the increases sum to $90^\circ$: $21^\circ+21^\circ+48^\circ$.

The goal now is to excise $k-1$ digons with one end at $x_1,x_2,\ldots,x_{k-1}$,
removing precisely the surface angle needed to increase $\o_P(x_i)$ to $\o_Q(x_i)$.
After digon removals at $x_1,\ldots,x_i$, we call the resulting polyhedron $P_i$.

Let a digon with endpoints $x_i$ and $y_i$ be denoted $D_i=(x_i,y_i)$.
Cut out from $P$ the digon $D_1=(x_1,y_1)$ containing only the vertex $v$ in its interior, of angle at $x_1$ equal to $\o_Q(x_1) - \o_P(x_1)$.
By the assumption that the slice plane $\Pi$ is sufficiently close to $v$, the curvature difference is small enough so that $D_1$ includes only $v$.
Again by the sufficiently-close assumption, we may assume the digon endpoint $y_1$ lies on the edge of $C(x_1)$ incident to $v$, prior to the first ramification point of $C(x_1)$.
After suturing closed the digon geodesics, $y_1$ becomes a vertex of curvature $\o_P(v)-( \o_Q(x_1) - \o_P(x_1) )$.
In the figure, $y_1$ has curvature $90^\circ - 21^\circ \approx 69^\circ$. In a sense, $y_1$ ``replaces" $v$.

Next cut out a digon $D_2=(x_2,y_2)$ containing only the vertex $y_1$ in its interior, of angle at $x_2$ equal to $\o_Q(x_2) - \o_P(x_2)$.
The newly created vertex $y_2$  ``replaces" $y_1$.
Continue cutting out digons $D_i=(x_i,y_i)$ up to $i=k-1$, each $D_i$ surrounding $y_{i-1}$, and replacing $y_{i-1}$ with $y_i$.

Because these tailorings have sharpened the curvatures  $\o_P(x_i)$ to match the after-slice curvatures $\o_Q(x_i)$,
it must be that the curvature at the last replacement vertex $y_{i-1}$
is the same as the curvature at $x_k$: $\o_P(y_{k-1}) = \o_Q(x_k)$ (to satisfy Gauss-Bonnet).
So now the tailored $P_{k-1}$ matches $Q$ in both the positions of the vertices $x_i$, $i=1,\ldots,k-1$, and their curvatures; 
the only possible difference is the location of $y_{i-1}$ compared to $x_k$.
But the rigidity result, Theorem~\thmref{Rigid}, implies that $y_{k-1}=x_k$, and $P_{k-1}$ and $Q$ are now congruent.
\end{proof}

The ``sufficiently-small" assumption in the preceding proof allowed us to assume that the digon $D_i=(x_1,y_1)$ endpoint $y_1$ lay 
on the segment of $C(x_i)$ incident to $v$ prior to the first ramification point $a_1$ of $C(x_1)$. 
Recall that $\o(x_1)+\o(y_i)=\o(v)$, and the further along the segment $v a_1$
that $y_1$ lies, the larger the digon angle at $x_1$ is.
The procedure would be problematic if the digon angle at $x_1$ were not large enough even with $y_1$ at that ramification point $a_1$.
The next lemma removes the sufficiently-small assumption in the special case when $P$ is itself a pyramid, and the vertex truncation reduces $P$ its base, doubly-covered.
Following this, we will show that the case when $P$ is a pyramid is the ``worst case," and so the general case follows.


\subsection{Pyramid case}

\begin{lm}
\lemlab{pyramid}
Let $P$ be a pyramid over base $X$.
Then one can tailor $P$ to reduce it to $X$ doubly-covered, using $k-1$ digon removal steps.
\end{lm}

\begin{proof}
We continue to use the notation in the previous lemma, and introduce further notation needed here.
Let $L= P \setminus X$ be the lateral sides of the pyramid $P$; so $P= L \cup X$.
After each digon $D_i=(x_i,y_i)$ is removed and sutured closed, the convex polyhedron guaranteed by
Alexandrov's Gluing Theorem will be denoted by $P_i$.
We continue to view $P_i$ as $P_i= L_i \cup X_i$, even though already $P_1$, is in general no longer a pyramid. 
We will see that all the digon excisions occur on $L_i$, while $X_i$ remains isometric to the original base $X$, but no longer (in general) planar.

We will use $C(x_i, P_j)$ to mean the cut locus of $x_i$ on $P_j$. 
Regardless of which $P_j$ is under consideration, we will denote by $a_i$
the first ramification point of $C(x_i)$ immediately beyond the
vertex $y_{i-1}$ surrounded by the digon $D_i=(x_i,y_i)$.

We need to establish two claims:
\begin{description}
\item [Claim~(1):] The cut locus $C(x_{i+1},P_i)$ is wholly contained in $L_i$.
\item [Claim~(2):] The digon angle $\a_{i+1}$ at $x_{i+1}$ to $a_{i+1}$ is large
enough to reduce the $L$-angle at $x_{i+1}$ to its $X$-angle on the base.
\end{description}

Before addressing the general case of these claims, we illustrate the situation
for $x_1$,  referencing Fig.~\figref{Pent_unf}.
\begin{figure}
\centering
 \includegraphics[width=\textwidth]{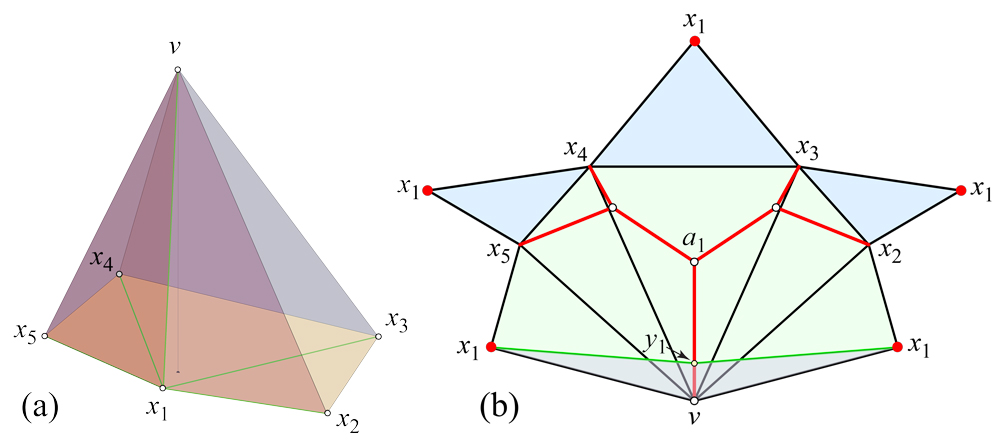}
\caption{(a)~A pyramid with pentagonal base $X$. Shortest paths from $x_1$ to all vertices
are marked green.
(b)~The star-unfolding with respect to $x_1$. 
The triangles from $X$ are blue; those from $L$ are green.
$C(x_1) \subset L$ is red.
The digon $D_1=(x_1,y_1)$ is shaded.
}
\figlab{Pent_unf}
\end{figure}
The digon $D_1(x_1,y_1)$ surrounding $v$ places $y_1$ on the $v a_1$ segment
of $C(x_1,P)=C(x_1)$. If one imagines $y_1$ sliding along $v a_1$ from $v$ to $a_1$, the digon angle at
$x_1$, call it $\d_1$, increases.
To show that $y_1$ can be placed so that $\d_1$ is large enough to reduce the angle
at $x_1$ to its angle in $X$ will require $a_1$ to lie in $L$ (rather than in $X$).

It turns out that $C(x_1) \subset L$ follows from a lemma in~\cite{aaos-supa-97}.\footnote{
Lem.~3.3: the cut locus is contained in the ``kernel" of the star-unfolding which in our case
is a subset of $L$.}
However, after removing $D_1$ and invoking Alexandrov's Gluing Theorem,
we can no longer apply this lemma.
With this background, we now proceed to the general case.

\paragraph{Claim~(1): $C(x_{i+1},P_i) \subset L$.}
Assume we have removed digons at $x_1,\ldots,x_i$, so that
$P_i = X_i \cup L_i$, and $L_i$ contains one vertex $y_i$,
the endpoint of the last digon $D_i$ removed, and $X_i$ contains no vertices.
Assume to the contrary of Claim~(1) that $C(x_{i+1},P_i)=C(x_{i+1})$ includes a point $z$
strictly interior to $X_i$. 
Because $z \in C(x_{i+1})$, there are two geodesic segments from $x_{i+1}$
to $z$, call them $\g^z_1$ and $\g^z_2$.
Because $X_i$ contains no vertices, it cannot be that both 
$\g^z_1$ and $\g^z_2$ are in $X_i$.
Say that $\g^z_1$ crosses $L_i$.
Let $p \in \partial X$ be the first point at which $\g^z_1$ enters $X_i$,
and let $\g_1 \subset \g^z_1$ be the portion from $x_{i+1}$ to $p$.
See Fig.~\figref{Claim1Abstract}.
\begin{figure}
\centering
 \includegraphics[width=0.75\textwidth]{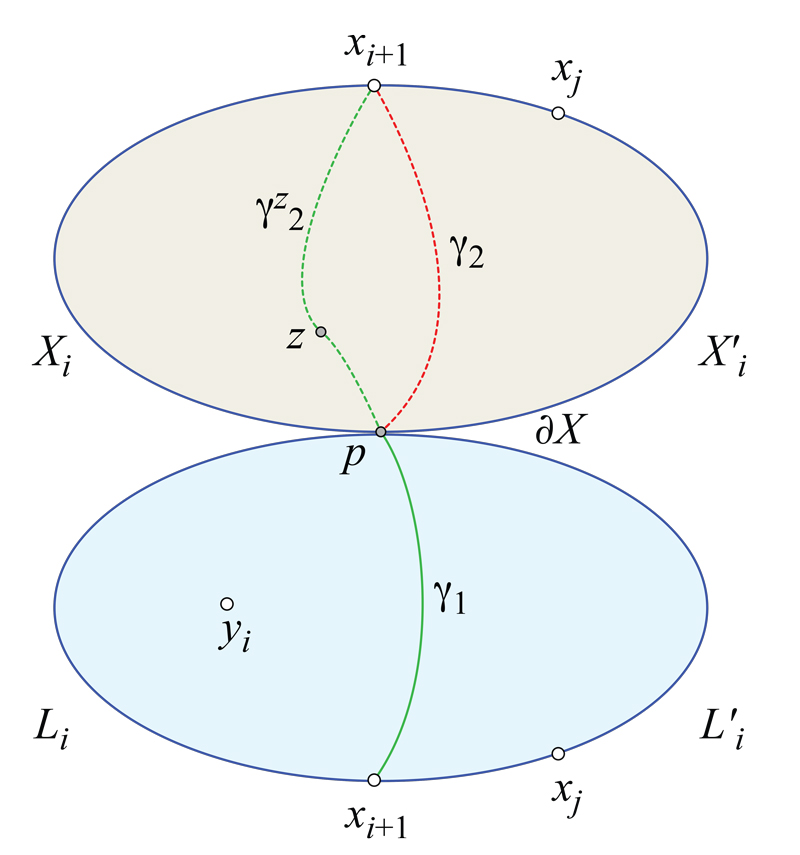}
\caption{$X_i$ is flipped above $L_i$ in this abstract illustration.
Cauchy's Arm Lemma ultimately shows that $\ell(\g_1) \ge \ell(\g_2)$.}
\figlab{Claim1Abstract}
\end{figure}

The geodesic segment $\g_1$ divides $L_i$ into two parts; let $L'_i$ be the part that does not contain the vertex $y_i$.
Join $x_{i+1}$ to $p$ with a geodesic $\g_2$ lying in $X_i$.
$\g_2$ was a shortest path to $p$ in $X$, but may no longer be shortest in $X_i$.
$\g_2$ also divides $X_i$ into two parts; let $X'_i$ be the part sharing a portion of $\partial X$ with $L'_i$.

Now we will argue that $\ell(\g_1) \ge \ell(\g_2)$, 
where $\ell(\g)$ is the length of $\g$.
This will yield a contradiction, for the following reasons.
$\g_1$ is a shortest geodesic to $p$, because it extends to $\g^z_1$, which
is a shortest geodesic to $z$. So $\g_2$ cannot be strictly shorter than $\g_1$.
Therefore we must have $\ell(\g_1) = \ell(\g_2)$, which implies that $p \in C(x_{i+1})$.
But then $\g_1$ cannot continue to $\g^z_1$ beyond $p$,
as $p$ is a cut point.

To reach $\ell(\g_1) \ge \ell(\g_2)$, we will use the extension of Cauchy's Arm Lemma
described in Chapter~\chapref{Preliminaries}.
Let $\q^X_j$ be the angle at $x_j$ in $X_i$, and $\q^L_j$ be the angle at $x_j$ in $L_i$.
For $j > i+1$, we know that $\q^L_j > \q^X_j$ because $P$ is a pyramid (see Lemma~\lemref{Angles}) and digon removal has not yet reached $x_j$.
We also know that $\q^X_j \le \pi$ because $X$ is convex.
However, $\q^L_j$ could be nearly as large as $2 \pi - \q^X_j$ if the pyramid $P$'s apex $v$ projects outside the base $X$.

Let $c_X$ be the planar convex chain in $\partial X$ that corresponds to $X'_i$; assume $c_X$ is
$x_{i+1}, x_{i+2}, \ldots, x_j, \ldots, p$ with angles $\q^X_j$. 
(The case where $c_X$ is includes the other part of $\partial X$, $x_{i+1}, x_i, \ldots, p$ can be treated analoguously.)
Then $\ell(\g_2)$ is the length of the chord between $c_X$'s endpoints.
In order to apply Cauchy's lemma, we rephrase the angles at $x_j$ as turn angles $\t_j = \pi - \q^X_j$.
The extension of Cauchy's lemma guarantees that, if the chain angles are modified so that the turn angles lie within $[-\t_j,\t_j]$,
then the endpoints chord length
cannot decrease. Roughly, opening (straightening) the angles stretches the chord.

Define $c_L$ as the planar (possibly nonconvex) chain composed of the same vertices $x_j$ that define $c_X$, but with angles $\q^L_j$.
Because $\q^L_j \le 2 \pi - \q^X_j$, the turn angles $\pi - \q^L_j$ in $c_L$ satisfy
$$
\pi - \q^L_j \ge \pi - (2 \pi - \q^X_j) = -(\pi - \q^X_j) = - \t_j \;.
$$
Also, because $\q^L_j > \q^X_j$, 
$$
\pi - \q^L_j \le \pi - \q^X_j = \t_j \;.
$$
So the $c_L$ turn angles are in $[-\t_j,\t_j]$, and we can conclude from Theorem~\thmref{CAL}
that the $c_L$ endpoints chord length $\ell(\g_1)$ is at least $\ell(\g_2)$, the $c_X$ endpoints chord length.

We have now reached $\ell(\g_1) \ge \ell(\g_2)$, whose contradiction described earlier shows that indeed $C(x_{i+1}) \subset L_i$.

\paragraph{Claim~(2): $\a_{i+1} \ge \q^L_{i+1} - \q^X_{i+1}$.}
Recall that $\a_{i+1}$ is the angle at $x_{i+1}$ of the digon from $x_{i+1}$
to $a_{i+1}$, the first ramification point of $C(x_{i+1})$ beyond the vertex $y_i$.
The claim is that $\a_{i+1}$ is large enough to reduce $ \q^L_{i+1}$ to $\q^X_{i+1}$.
We establish this by removing a path from $C(x_{i+1})$ and tracking angles,
as follows.

From Claim~(1), $C(x_{i+1}) \subset L_i$.
Let $\r_{i,i+2}$ be the path in the tree $C(x_{i+1})$ from $x_i$ to $x_{i+2}$.
See Fig.~\figref{Star_CutLocus_Digons}.
\begin{figure}
\centering
 \includegraphics[width=0.85\textwidth]{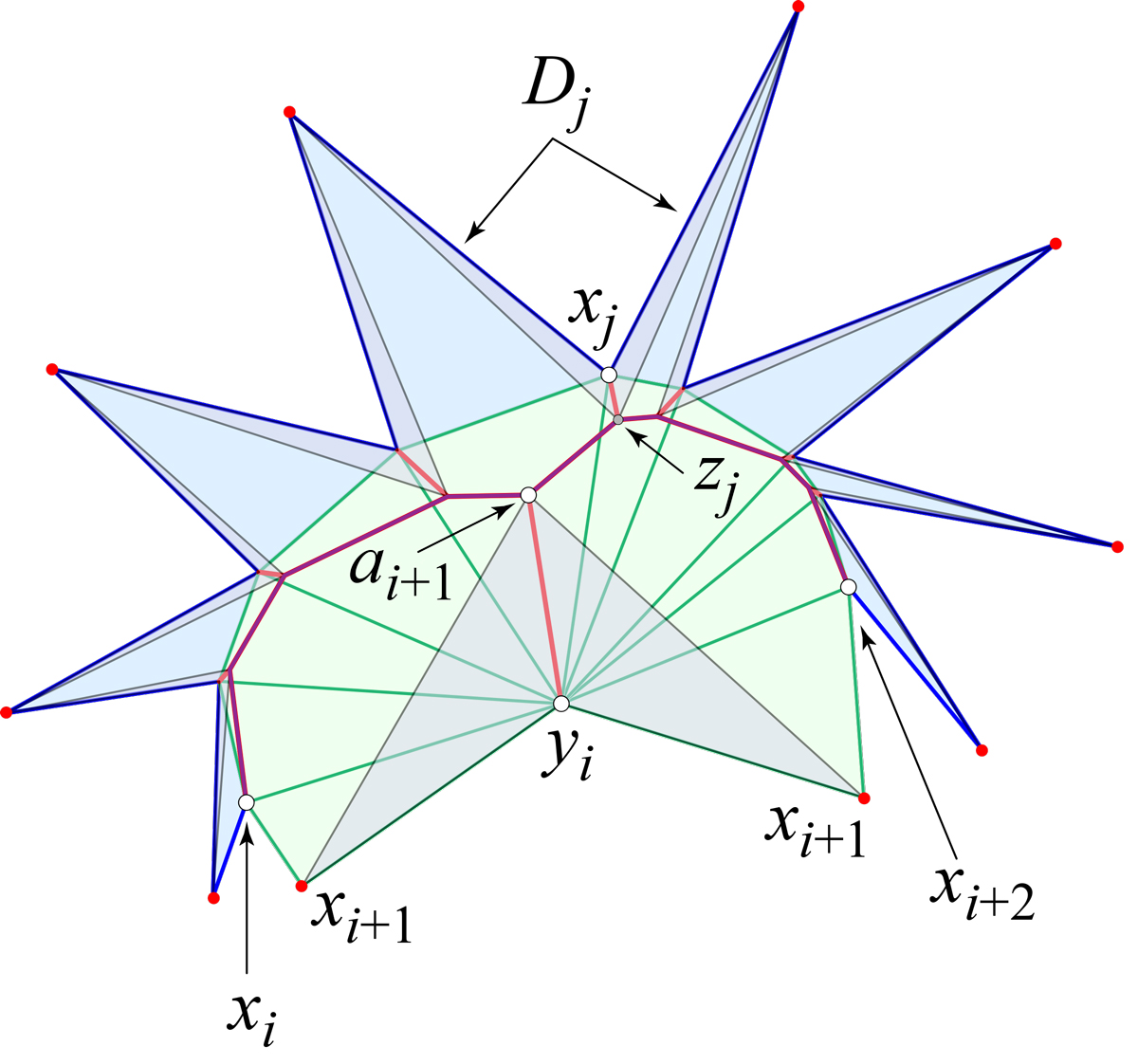}
\caption{
Star-unfolding of a pyramid with respect to $x_{i+1}$.
The triangles from $X_i$ are blue; those from $L_i$ are green.
Red points are images of $x_{i+1}$.
Digons are shaded. $\r_{i,i+2}$ is purple; remainder of $C(x_{i+1})$ is red.
}
\figlab{Star_CutLocus_Digons}
\end{figure}
Removal of $\r_{i,i+2}$ from $C(x_{i+1})$
disconnects $C(x_{i+1})$ into the edge $y_i a_{i+1}$, and a series
of subtrees $T_j$. Each $T_j$ shares a point $z_j$ with $C(x_{i+1})$.
Let $D_j$ be the digon from $x_{i+1}$ to $z_j$, and let $\d_j$ be the
angle of $D_j$ at $x_{i+1}$.
Finally, let $\D_j = \sum_j \d_j$.

Note that all the $\d_j$ angles are in $X_i$.
In contrast, the angle at $x_{i+1}$ in the digon $D_{i+1} = D(x_{i+1},a_{i+1})$ is in $L_i$, as illustrated in Fig.~\figref{Star_CutLocus_Digons}.
We defer justifying this claim to later.

Cut off all $D_j$, and also cut off $D_{i+1}$.
Suture the surface closed; call it $P^*=X^* \cup L^*$.
By Lemma~\lemref{Partition}, the cut locus $C(x_{i+1},P^*)$ is precisely the path $\r_{i,i+2}$.
Therefore, by Lemma~\lemref{Path}, $P^*$ is a doubly-covered convex polygon, so all angles at $x_j$ are equal above in $L^*$ and below in $X^*$.
In particular, $\q^{X^*}_{i+1} = \q^{L^*}_{i+1}$.
Now, because $\D$ angle was removed from $\q^{X}_{i+1}$, $\q^{X^*}_{i+1} = \q^{X}_{i+1} - \D$.
Because $\a_{i+1}$ was removed from $\q^L_{i+1}$, $\q^{L^*}_{i+1} = \q^{L}_{i+1} - \a_{i+1}$.
Therefore, 
\begin{eqnarray}
\q^{L}_{i+1} - \a_{i+1} &=& \q^{X}_{i+1} - \D \; \le \; \q^{X}_{i+1} \\
\a_{i+1} &\ge& \q^{L}_{i+1} - \q^{X}_{i+1}
\end{eqnarray}
which is Claim~(2).

It remains to show that $D_{i+1}$ is in $L_i$ rather than in $X_i$.
Suppose to the contrary that all the angle removal was in $X_i$. Then $\q^{L}_{i+1} = \q^{X}_{i+1} - \D - \a_{i+1}$.
So $\q^{L}_{i+1} < \q^{X}_{i+1}$, which is not possible for $P$ a pyramid.
This completes the proof of Claim~(2) and the lemma.
\end{proof}


\subsection{General case}

Lemma~\lemref{pyramid} is special in that $P$ sits over a base $X$.
In the general situation, $X$ is the intersection of $Q$ with the truncating
slice plane $\Pi$, but $X$ is not a face of $Q$. Rather in general, $X$ is inside
$Q$, the ``top" of the portion $Q'$ of $Q$ below $\Pi$.

\begin{lm}
\lemlab{VertexTruncation}
Let $Q$ be obtained from $P$ by truncating vertex $v$.
Then, if $v$ has degree-$k$, $Q$ may be obtained from $P$ by $k-1$ tailoring steps, each the excision of a digon surrounding one vertex.
\end{lm}

\begin{proof}
Here we argue that the general case is in some sense no different than
the special case of $P$ a pyramid just established in Lemma~\lemref{pyramid}.
In fact, the exact same digon excisions suffice to tailor $P$ to $Q$.

First we establish additional notation.
Let $\Pi$ be the plane slicing off $v$ above $\Pi$, and let $X = \Pi \cap P$.
Let the ``bottom" part of $P$ be $Q'$, with the final polyhedron $Q = Q' \cup X$.
We continue to use $L$ to denote the portion of $P$ above $\Pi$, so $P = Q' \cup L$.
After removal of digons at $x_1,x_2,\ldots,x_i$, we have $P_i = Q'_i \cup L_i$.

Below it will be important to distinguish between the three-dimensional extrinsic shape of $Q'_i$ and its intrinsic structure determined by the gluings that satisfy Alexandrov's Gluing Theorem.
We will use $\bar{Q}'_i$ for the embedding in $\R^3$ and $Q'_i$ for the intrinsic surface, and we will similarly distinguish between $\bar{X}_i$ and $X_i$.
Note that we can no longer assume that $C(x_{i+1}, P_i) \subset L_i$, for the cut locus could extend into $Q'$
(whereas it could not extend into $X$ in Lemma~\lemref{pyramid}).

\medskip

It suffices to show by induction that, on $P_i$, the following statements hold:
\\ (a) The shortest path $\g_{i+1}$ joining $x_{i+1}$ to $y_i$ is included in $L_i$.
\\ (b) The ramification point $a_{i+1}$ is still on $L_i$.
\\ (c) The $x_{i+1}$ angle $\a_{i+1}$ of the digon $D_{i+1} = D(x_{i+1},a_{i+1})$ is larger than or equal to $\o_Q(x_{i+1}) - \o_P(x_{i+1})$ (and so sufficient to reduce the curvature to $\o_Q(x_{i+1})$).

\medskip

To see (a), assume, on the contrary, that $\g_{i+1}$ intersects $Q'_i$.
Assume, for the simplicity of the exposition, that $\g_{i+1}$ enters $Q'_i$ only once, at $x_{i+1}$, and exits $Q'_i$ at $p \in \partial X$.
Let $\g'_{i+1}$ denote the part of $\g_{i+1}$ between $x_{i+1}$ and $p$.

We now check~(a) for $i=0$.
$Q=Q' \cup X$ and $X$ is planar, hence, because the orthogonal projection of any rectifiable curve onto a plane shortens or leaves its length the same,
$\g'_1$ is longer than or has the same length as its projection $\g''_1$ onto $X$.
So $p$ is a cut point of $x_1$ along $\g_1$, contradicting the extension of $\g_1$ as a geodesic segment beyond $p$.

By the induction assumption, all the digon excisions occur on $L_i$; $Q'_i$ is unchanged.
Nevertheless, as part of $P_i$, neither $\bar{Q}'_i$ nor $\bar{X}_i$ is (in general) congruent to the original $\bar{Q}$ and $\bar{X}$.
However, if we consider $Q'_i$ and $X_i$  separate from $P_i$, we can reshape them so that $\bar{Q}'_i = \bar{Q}$
and $\bar{X}_i = \bar{X}$, precisely because they have not changed.
Then $\bar{X}$ is planar and the projection argument used for $i=0$ works for all $i$.

\medskip

Next we check (b) and (c) for $i=0$.
Consider, as in Lemma~\lemref{pyramid}, the digon $D_1=(x_1,y_1)$ with $y_1 \in C(x_1)$,
with again $a_1$ the first ramification point of $C(x_1)$ beyond $v$.
The direction at $v$ of the edge $v a_1$ is only determined by the geodesic segment from $x_1$ to $v$, 
and hence is not influenced at all by $Q'$, because, by $i=0$ in~(a), that segment lies in $L$.

The ramification point $a_1$ is joined to $x_1$ by three geodesic segments, two of 
them---say $\g_1$ and $\g_2$---included in $L$.
The third geodesic $\g_3$ starts from $x_1$ towards $Q'$ and finally enters $L$ to connect to $a_1$. 
See Fig.~\figref{GeodesicsAbstract}.
Because these three geodesics have the same length, the longer $\g_3$ is, the longer are $\g_1$ and $\g_2$, and therefore more distant is $a_1$ to $v$.
So $a_1$ is closest to $v$, and the segment $v a_1$ shortest, when $Q' = X$ and $P$ is a pyramid.
It is when $v a_1$ is shortest that there is the least ``room" for $y_1$ on $v a_1$ to achieve the needed 
digon angle at $x_1$, for that angle is largest when $y_1$ approaches $a_1$.
Therefore, the case when $Q' = X$ and $P$ is a pyramid is the worst case, already settled in  Lemma~\lemref{pyramid}.

\begin{figure}
\centering
 \includegraphics[width=0.6\textwidth]{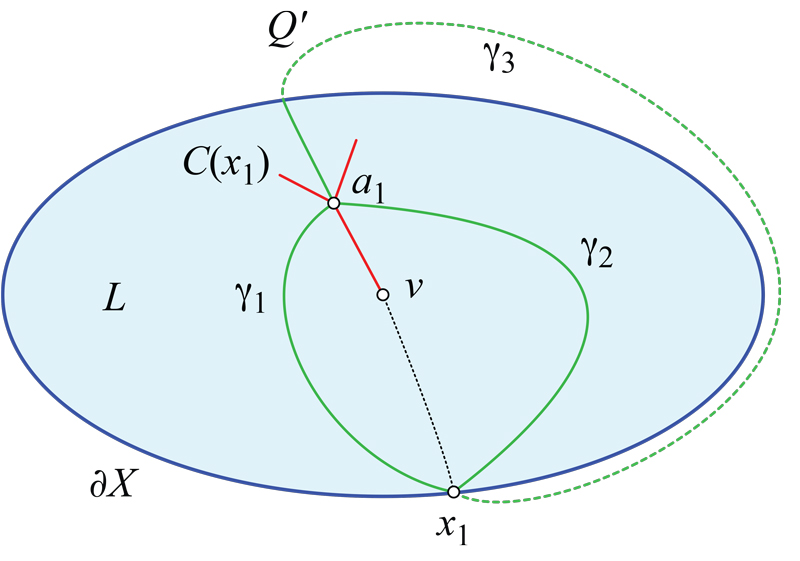}
\caption{Three geodesics to ramification point $a_1$ on $C(x_1)$. Dashed $\g_3$ partially in $Q'$.
The same situation holds under the changes:
$x_1 \to x_{i+1}$, $a_1 \to a_{i+1}$, $v \to y_i$, $Q' \to Q'_i$, $L \to L_i$.
}
\figlab{GeodesicsAbstract}
\end{figure}

Now we treat the general case for (b) and (c).
Again by the induction assumption, all changes to $P_i$ were made on its ``upper part'' $L_i$.

Because we ultimately need to reduce $L$ to $X$, the angle $\o_Q(x_{i+1}) - \o_P(x_{i+1})$ necessary to be excised at $x_{i+1}$, does not depend on $Q'_i$, only on $L_i$.
Thus the argument used for $i=0$ carries through.
The situation depicted in Fig.~\figref{GeodesicsAbstract} remains the same, with $x_1$ replaced by $x_{i+1}$, $v$ replaced by $y_i$, and $a_1$ replaced by $a_{i+1}$.
The ramification point $a_{i+1}$ is closest to $y_i$, and the segment $y_i a_{i+1}$ shortest, when $Q' = X$ and $P$ is a pyramid.
It is when $y_i a_{i+1}$ is shortest that there is the least ``room" for $y_{i+1}$ on $y_i a_{i+1}$ to achieve the needed angle excision at $x_{i+1}$, 
for that angle is largest when $y_{i+1}$ approaches $a_{i+1}$.
Therefore, the case when $Q' = X$ and $P$ is a pyramid is the worst case, already settled in  Lemma~\lemref{pyramid}.
\end{proof}

\noindent
Note that, in the end, the digon removals in Lemma~\lemref{vertex_small}, and then in  Lemma~\lemref{pyramid},
also work in the general case, Lemma~\lemref{VertexTruncation}.

\bigskip


\section{Cube/Tetrahedron: Completion}
\subsection{Pyramid Removals}
\seclab{PyramidRemovals}

We now return to the cube/tetrahedron example started in
Chapter~\chapref{Domes}.
We had reduced the sliced-off portion of
the original cube $P$ to four pyramids $P_1,P_2,P_3,P_4$.
Now each of these pyramids needs 
to be ``tailored away" to leave
the goal tetrahedron $Q$.
Fig.~\figref{CubeTetra_4} shows that when processed in order,
their removal reduces the portion of $P$
``above" $Q$, leaving the goal tetrahedron $Q$.
\begin{figure}[htbp]
\centering
\includegraphics[width=0.9\linewidth]{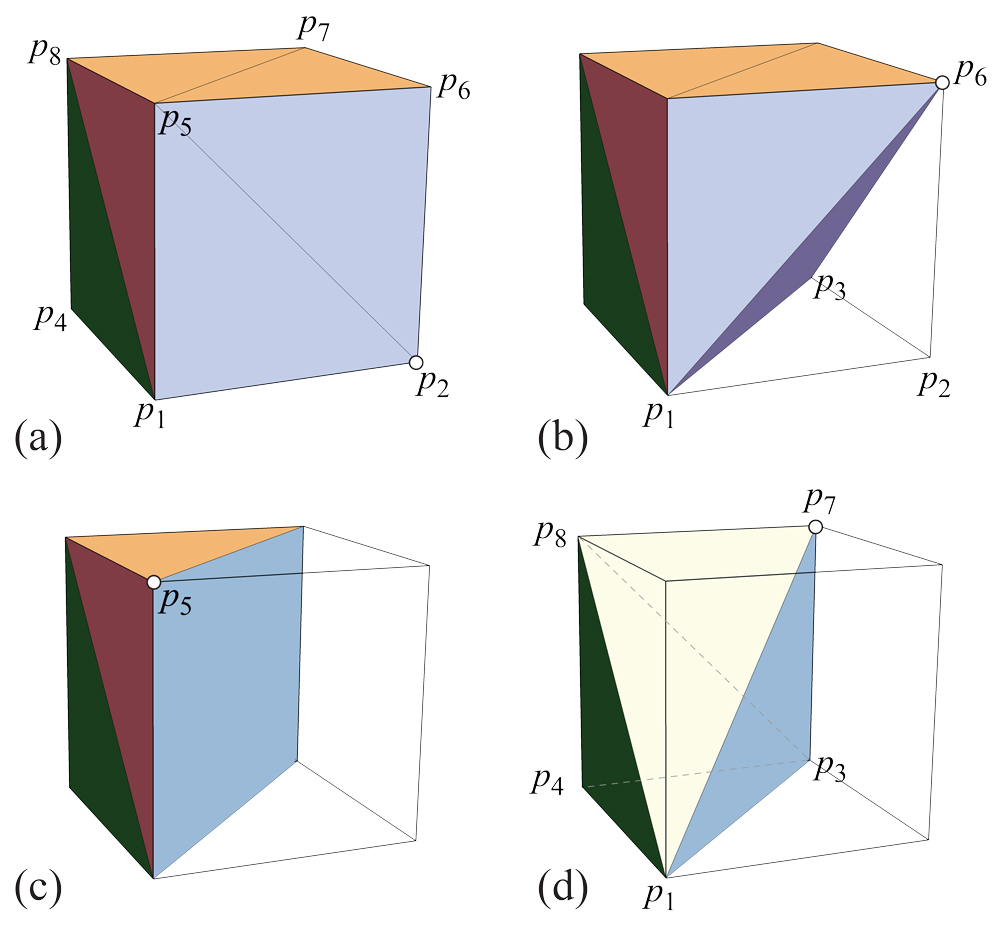}
\caption{Pyramid removals. The apex of the pyramid to be next removed is highlighted.
From (a) to (b), the pyramid $P_1$ with apex $p_2$ and base $p_1 p_3 p_6$ is removed.
Then $P_2$ apexed at $p_6$ is removed, producing~(c),
then removing $P_3$ apexed at $p_5$ leads to~(d).
The final removal of $P_4$ apexed at $p_7$, leaves the tetrahedron $Q=p_1 p_3 p_4 p_8$
previously illustrated in Fig.~\protect\figref{CubeTetra_1}(a).}
\figlab{CubeTetra_4}
\end{figure}

\subsection{Pyramid Reductions by Tailoring}
Now we follow Lemma~\lemref{VertexTruncation} 
to reduce each of the four pyramids
to their bases.
We only illustrate this for the first pyramid removed, $P_1$, with apex 
$p_2$ and base $X=p_1 p_3 p_6$. 
See Fig.~\figref{EqTriPyramid}(a).
The base is an equilateral triangle, with edge
lengths $\sqrt{2}$, with the apex is connected to the base vertices by unit-length
edges.
Because the apex $p_2$ is a cube corner, it has three incident $90^\circ$ angles,
so $\o(p_2)=90^\circ$.
The three faces incident to $p_2$ are $45^\circ{-}45^\circ{-}90^\circ$ triangles.
The base angles in $X$ are $60^\circ$. So each digon must reduce the incident $90^\circ$ 
angles by $30^\circ$ to match $60^\circ$.
Starting with $p_1$, the digon geodesics are $\pm 15^\circ$ around the $p_1 p_2$ edge.

The geometry is clearest if we unfold $P_1$'s lateral faces into the plane, as
shown Fig.~\figref{EqTriPyramid}(b).
Excising the first digon and sealing the cut results in a new polyhedron, with the
apex $p_2$ replaced by a new vertex, call it $y$, of curvature $\o(y)=60^\circ$
(because the two digon angles must sum to the $90^\circ$ curvature at $p_2$).
Unfortunately, we cannot display this new polyhedron
because of the difficulty of constructing what AGT guarantees exists.

Lemma~\lemref{VertexTruncation} says that just one more digon needs removal
(because $p_2$ has degree $k=3$),
again $\pm 15^\circ$ this time around the $p_3 y$ edge.
This removal reduces $P_1$ to its equilateral triangle base,
and, despite the nonconstructive nature of AGT, we know that the full polyhedron
is exactly what we illustrated earlier in Fig.~\figref{CubeTetra_4}(b).

\begin{figure}[htbp]
\centering
\includegraphics[width=1.0\linewidth]{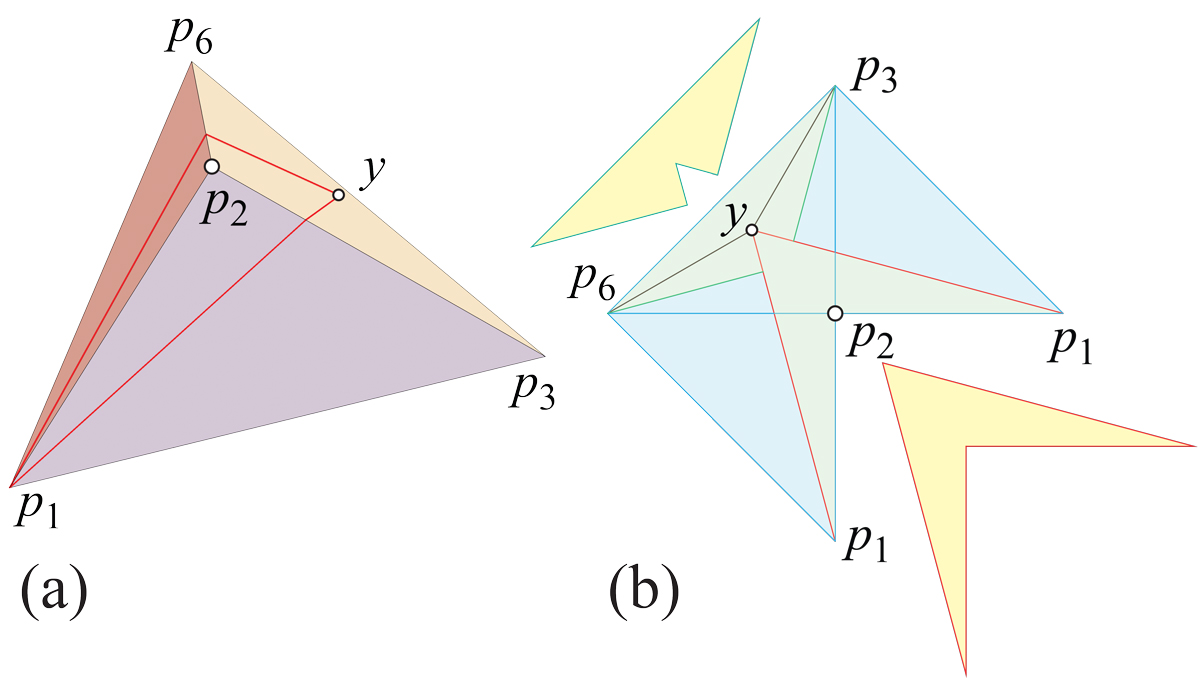}
\caption{(a)~Pyramid $P_1$: apex $p_2$, base $p_1 p_3 p_6$.
(b)~Excision of two digons. Each is bound by geodesics $\pm 15^\circ$ about the
edge to $p_2$. After removing $p_2$, a new vertex $y$ is created.
The removal of $y$ in the second excision flattens the pyramid to its base.}
\figlab{EqTriPyramid}
\end{figure}

One further remark on the shape of $P_1$ after excision of the first digon.
If we imagine $P_1$ standing alone on its base $X$ as illustrated in Fig.~\figref{EqTriPyramid}(a),
rather than as part of the full cube polyhedron, then it is not difficult to 
reconstruct the shape of $P'_1$. 
It is a flat doubly-covered quadrilateral, with the triangle $p_3 y p_6$ 
flipped over and joined to the $p_3 p_6$ edge
of the equilateral triangle base $X$.
However, with this $P'_1$ piece joined to the full cube, it seems much less straightforward to
determine the shape of the full polyhedron.

Each pyramid reduction proceeds in the same manner: $k-1$ digons are excised
if the apex has degree-$k$, and the lateral faces are reduced to the base.
So $P_2$ has apex $p_6$
and base $p_1 p_3 p_5 p_7$. After removal of three digons, the result is 
as illustrated in Fig.~\figref{CubeTetra_4}(c).
$P_3$'s apex $p_5$ has degree-$3$, so two digon removals lead to~(d).
The last pyramid removed, $P_4$, reduces to the face $p_1 p_3 p_8$ of $Q$, completing
the tailoring of the cube $P$ to tetrahedron $Q$.


\subsection{Seals}

After removal of the two digons illustrated in
Fig.~\figref{EqTriPyramid}(b), $P_1$ has been reduced to its equilateral
triangle base $X$. Sealing the first digon produces a seal $\s_1$,
which is then clipped to a segment $s_1$ by the second digon removal,
which produces $\s_2$ along the boundary of $X$.
The seal segments then are as shown in
Fig.~\figref{PyrBases}(a).
\begin{figure}[htbp]
\centering
\includegraphics[width=0.6\linewidth]{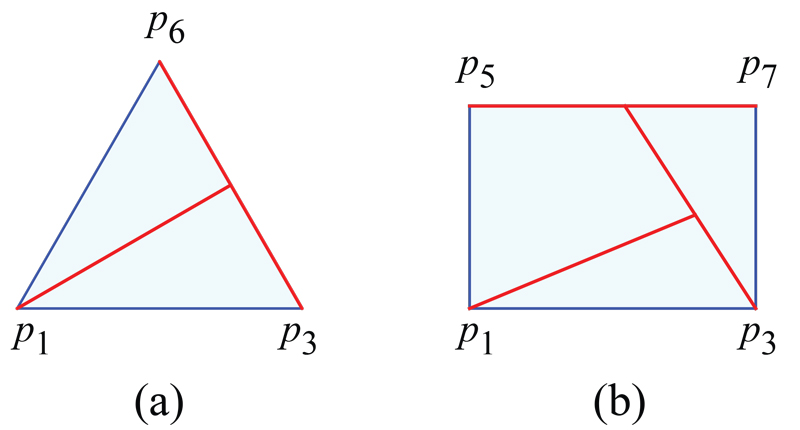}
\caption{(a)~Base of $P_1$. (b)~Base of $P_2$. Seal segments: red.}
\figlab{PyrBases}
\end{figure}
Fig.~\figref{PyrBases}(b) shows the three seal segments 
that result by reducing $P_2$ to its base.
These depictions of the seal graph $\S$ will be explored in some
detail in Chapter~\chapref{SealGraph}.

There is an aspect of the seals we are not tracking: 
As can be seen in Fig.~\figref{CubeTetra_4}(b,c),
one of the faces of $P_2$ is the equilateral triangle base of $P_1$.
Since that base is already crossed by seal segment $s_1$
when $P_2$ is undergoing digon removal,
that segment $s_1$ will be reflected as a cut in the reduced base 
of $P_2$, not depicted in Fig.~\figref{PyrBases}(b).
We have not attempted to track this complex overlaying of seal cuts
in the surface of $P$.
However, we explore seals for one pyramid tailoring in some
detail in Chapter~\chapref{SealGraph}.


\section{Hexagonal Pyramid Example}
\seclab{HexPyramid}

We now detail a more complex example 
following Lemma~\lemref{pyramid} to tailor a pyramid $P$ to its base $X$.
We continue to employ the notation used in the lemmas above.
The example is shown in Fig.~\figref{Hex3D}. $X$ is a regular hexagon, and $L$ consists of $k=6$ congruent, $70^\circ{-}70^\circ{-}40^\circ$ isosceles triangles. 
The curvature at the apex $v$ is $360^\circ - 6 \cdot 40 = 120^\circ$.
The angle at each $x_i$ in $X$ is $120^\circ$ whereas the angle in $L$ is $140^\circ$.
So each digon excision must remove $20^\circ$ from $x_i$.
As in the lemmas, we excise the digons in circular order around $\partial X$.

\begin{figure}
\centering
 \includegraphics[width=0.5\textwidth]{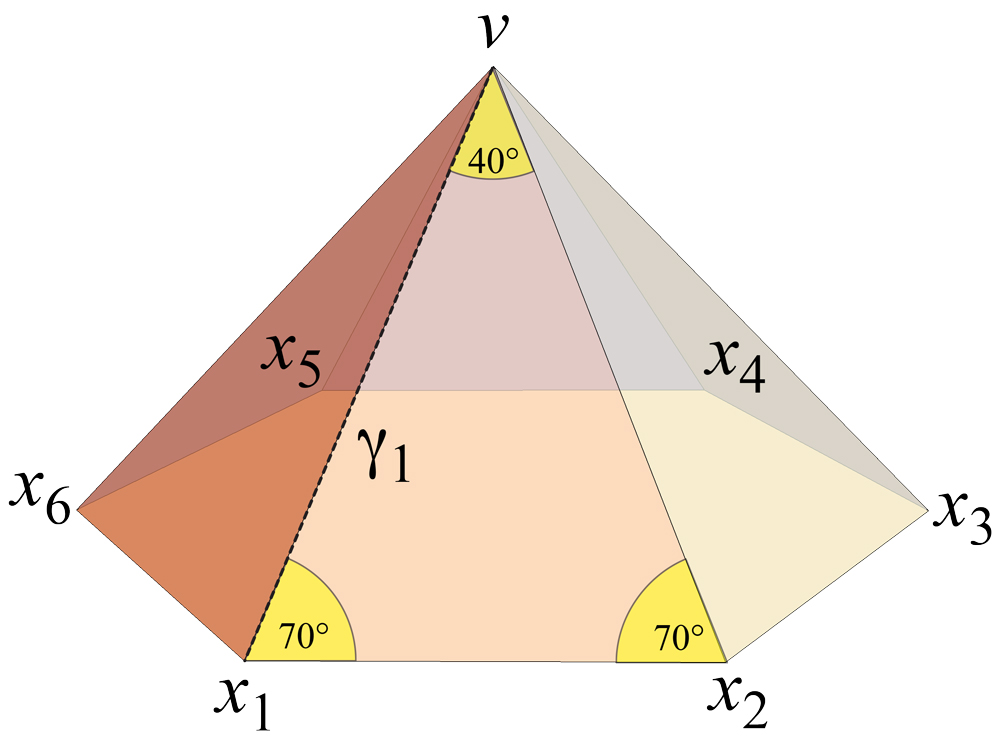}
\caption{Pyramid with regular hexagon base $X$;
all lateral faces congruent triangles.
}
\figlab{Hex3D}
\end{figure}

We display the progress of the excisions on the layout of $L$ in Fig.~\figref{DigonsHexFlatX}(a).
Let $v=y_0$ for ease of notation.
$D_1=(x_1,y_1)$ includes the geodesic $\g_1$ from $x_1$ to $y_0$,
and locates $y_1$ on the cut locus segment as described in the lemmas.
The digon boundary geodesics each remove $10^\circ$ from the
left and right neighborhood of $x_1$, and meet at $y_1$ at an angle of $100^\circ$,
which is then the curvature at the new vertex: $\o(y_1)=100^\circ$.
Notice that the digon angles $20^\circ + 100^\circ$ match the curvature
$\o(v)=120^\circ$ removed, as they must to satisfy Gauss-Bonnet.
\begin{figure}
\centering
 \includegraphics[width=1.0\textwidth]{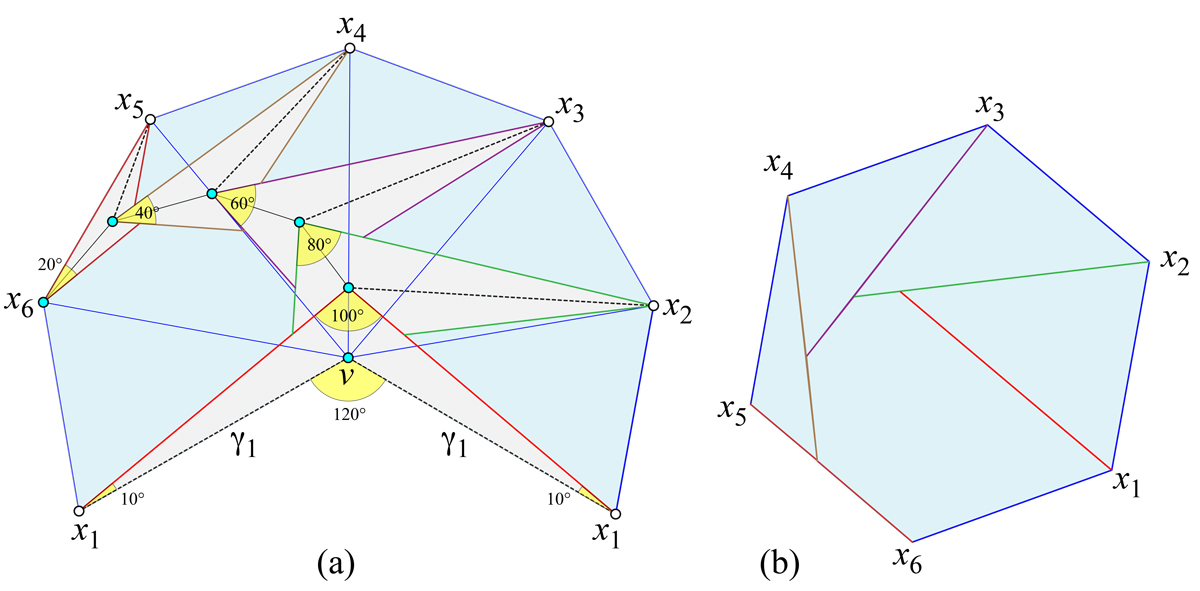}
\caption{(a)~Cone $L$ flattened; $\o(v)=120^\circ$.
Digons $D_i=(x_i,y_{i-1})$ shaded. Dashed lines are geodesics $\g_i$ from 
$x_i$ to the vertex $y_{i-1}$ (with $y_0=v$).
(b)~After excising all digons $D_1,\ldots,D_5$, $L_5$ is isometric to $X$.
Seals $\tilde{\s}_i$ are marked.
}
\figlab{DigonsHexFlatX}
\end{figure}

One should imagine that $D_1$ is sutured closed in Fig.~\figref{DigonsHexFlatX}(a), producing $L_1$, before constructing $D_2(x_2,y_2)$. 
Let $\s_i$ be the geodesic on $L_i$ that results from sealing $D_i$ closed; $\s_i$ is like a ``scar" from the excision.
Notice that one of the geodesics bounding $D_2$ crosses $\s_1$.

This pattern continues as all $k{-}1=5$ digons are removed, each time replacing vertex $y_{i-1}$ with $y_i$, flattening the curvature $\o(y_i)=\o(y_{i-1})-20^\circ$.
Finally, after $D_5(x_5,y_5)$ is removed, $y_5$ is coincident with $x_6$.
No further digon removal is needed, because $D_5$ removed $20^\circ$ from $x_6$.
So now each angle in $L_k=L_5$ at all $k{=}6$ vertices is $120^\circ$, and $L_5$ is isometric to a flat regular hexagon, i.e., to $X$.

This final hexagon is shown in Fig.~\figref{DigonsHexFlatX}(b).
The images $\tilde{\s}_i$ of the seals
are in general clipped versions of $\s_i$ on $L_i$, clipped by subsequent digon removals.
The particular circular order of digon removal followed in this example and the lemmas result
in a spiral pattern formed by $\tilde{\s}_i$.
Other excision orderings, which ultimately would result in the same flat $L_{k-1}$
(effectively proved in Lemmas~\lemref{vertex_small}--\lemref{VertexTruncation}) would create different seal patterns.
As mentioned earlier, we study seals in detail in
Chapter~\chapref{SealGraph}.

\section{Tailoring is finer than sculpting}
\seclab{TailoringFiner}

In this section we reach one of our main results, Theorem~\thmref{MainTailoring}, which says, roughly, 
that any polyhedron $Q$ that can be obtained by sculpting $P$ can be obtained by tailoring $P$.
Moreover, Lemma~\lemref{NotReachable} shows that polyhedra can be obtained by tailoring that cannot be obtained by sculpting. So, in a sense, tailoring is finer than sculpting.

\begin{lm}
\lemlab{NotReachable}
There are shapes $P$ and sequences of tailorings of $P$ that result in polyhedra not achievable by sculpting.
\end{lm}

\begin{proof}
We first tailor a regular tetrahedron $T$ as in Example~\exref{tetra-deg}, resulting in the kite $K$ in Fig.~\figref{TetraTailoring}(b).
We now show that $K$ cannot fit inside $T$, so it couldn't
have been sculpted from $T$.
Assume $T$ has edge-length $1$.
Then its extrinsic diameter is $1$ and its intrinsic diameter is $2/\sqrt{3}$ 
(see, e.g., Theorem~3.1 in~\cite{rouyer2003antipodes}).
Moreover, the extrinsic diameter of $K$ is precisely the intrinsic diameter of $T$, and so it cannot fit inside $T$.

Next we construct
a non-degenerate example, a modification of the previous one.
Consider a non-degenerate pentahedron $F$ close enough to $K=o a c_d b$ in Fig.~\figref{TetraTailoring}(b).
For example, it could have two vertices close to the vertex $a$ of $K$.
Insert into $F$ the removed digon from $T$; this is not affected by the new vertex, because it does not interfere with the geodesic segment from $c_d$ to $o$. 
We arrive at some surface $P$ close enough to the original tetrahedron $T$.
Therefore, the intrinsic and extrinsic diameters of $P$ and $F$ are close enough to those of $K$ and $T$, respectively,
and the above inequality between the extrinsic diameters of $P$ and $F$ still holds, because of the ``close enough" assumption.
\end{proof}

\begin{thm}
\thmlab{MainTailoring}
Let $P$ be a convex polyhedron, and $Q \subset P$ 
a convex polyhedron resulting from repeated slicing of $P$ with planes. Then $Q$ can also be obtained from $P$ by tailoring.
Consequently, for any given convex polyhedra $P$ and $Q$, one can tailor $P$ ``via sculpting'' to obtain any homothetic copy of $Q$ inside $P$.
\end{thm}

\begin{proof}
Lemma~\lemref{Slice2gDomes} established that one slice leads to domes,
Theorem~\thmref{DomePyr} showed that each dome leads to pyramids,
and Lemma~\lemref{VertexTruncation} showed that each pyramid can be
reduced to its base by tailoring.
Since this holds for one slice, it immediately follows that it holds for arbitrary slicing.

Concerning the domes~$\to$~pyramids step, we note that the
property that each pyramid $P_i$ has a common edge with $X$,
established in Theorem~\thmref{DomePyr}, allows reduction of
the pyramids in the order that they are obtained in that theorem.
After each reduction, the result is still a g-dome, allowing iteration until the
original g-dome is reduced to its base.

For the homothet-copy claim of the lemma,
shrink $Q$ by a dilation until it fits inside $P$, and then apply the 
reductions.
 \end{proof}

As we mentioned in the Preface, 
an informal consequence
of this theorem is that $P$ can be ``whittled" to e.g., a sphere $S$:

\begin{co}
\lemlab{approx_tailoring}
For any convex polyhedron $P$ and any convex surface $S$, one can tailor $P$ to approximate a homothetic copy of $S$.
\end{co}
\begin{proof}
Bring a homothetic copy of $S$ inside $P$.
Perform a series of slicings of $P$ with planes tangent to $S$.
Any degree of approximation desired can be achieved by
increasing the number of plane splicings. 
Call the result of these slicings $Q$.
Now apply Theorem~\thmref{MainTailoring}.
\end{proof}

Despite this corollary,
it does not seem possible to accomplish the reverse:
to start with a strictly convex surface and tailor it to a polyhedron.
However, one can of course sculpt a surface to a polyhedron.

\chapter{Pyramid Seal Graph}
\chaplab{SealGraph}



\section{Pyramid Digon Removal}

As we have seen in Theorem~\thmref{MainTailoring},
tailoring by tracking sculpting ultimately relies on digon removal reducing pyramids to their bases.
We have illustrated such reductions for a few low-degree pyramids 
in Figs.~\figref{Sculpting}, \figref{EqTriPyramid}, and~\figref{DigonsHexFlatX}.
In the latter two figures, we displayed the seals $\tilde{\s_i}$
on the base $X$ after reduction.
It however remains difficult to grasp in detail the digon-removal process for a pyramid $P$, for at least three reasons:
\begin{enumerate}
\squeezelist
\item After removing the first digon $D_1$, $P_1$ is (in general) no longer a pyramid. 
The difficulty of computationally realizing the subsequent intermediate shapes $P_i$, guaranteed by AGT, makes it  hard to envision the process.
\item The seals $\s_i$ that result from closing digon $D_i$
cross and clip one another.
\item The process depends on the order in which the digons are removed.
\end{enumerate}
We will continue to circumvent this last difficulty by only studying
the natural order of digon removal, anchored at $x_1, x_2,x_3,\ldots$ in 
counterclockwise order around $\partial X$.
In this section, we introduce a different way to view digon removal that
in some sense skirts the first two difficulties.

The process is complex enough to require somewhat
extensive notation, which we list in two parts before turning
to examples.


\subsection{Notation~I}

\begin{itemize}
\item $P$: a pyramid, $n$ vertices around base $X$.
\item Base vertices $x_1,x_2,\ldots,x_n$, in counterclockwise order
around $\partial{X}$.
\item Apex $y_0$ of degree-$n$. 
\item $y_i$: apex after removing digon $D_i$.
\item $D_i$: digon from $x_i$ to $y_i$, surrounding $y_{i-1}$.
\item $L_i$: (the remaining of the) lateral faces after removing digon $D_i$. $L_0$: initial faces before any removals.
The apex of $L_i$ is $y_i$.
\item $P_i$: The polyhedron $L_i \cup X$, guaranteed by AGT.
$P=P_0$ is the original, before any digon removal.
\end{itemize}

We should emphasize that although $P_i = L_i \cup X$, 
in general $X$ will not be planar in $P_i$ as it is in $P_0$,
and so $P_i$ is not a pyramid, as previously mentioned.


\section{Cone Viewpoint}
\seclab{Cone}

Although we do not know the structure of $P_i$, except at the beginning ($i=0$) and end ($i=n{-}1$),
when it is $P$ and doubly-covered $X$ respectively, we do know that the lateral faces $L_i$ 
contain only one vertex, $y_i$, hence they
form a subset
of a cone apexed at $y_i$.
Any cone can be cut open along a generator (a ray on the cone from the apex)
and laid flat in the plane.
Such a layout will have an angle gap of $\o(y_i)$ at the apex.
It is especially useful to cut along $x_i y_{i-1}$ before removing
digon $D_i$.
We will provide several examples, after presenting more notation.
We emphasize the indices $i{-1}$, $i$, and $i{+}1$ in the following,
in an attempt to avoid confusion.


\subsection{Notation~II}

\begin{itemize}
\item $\bar{L}_{i-1}$: Unfolding of $L_{i-1}$ cut open along $x_i y_{i-1}$. So
after removing and sealing digon $D_{i-1}$, but not yet $D_i$.
\item $\bar{L}_i$: Unfolding of $L_i$ cut open along $x_{i+1} y_i$. So
after removing and sealing digon $D_i$, but not yet $D_{i+1}$.
So $\bar{L}_0$ is $L_0$ cut open along $x_1 y_0$, and
$\bar{L}_1$ is $L_1$ cut open along $x_2 y_1$.
\item $\s_i = x_i y_i$ is the $i$-th seal after suturing closed the digon $D_i$.
We view the seals as directed from $x_i$ to $y_i$, so that they have
distinguished left and right sides. 
This direction is only used in the proofs; the seals are illustrated as undirected segments in several figures.
When the direction plays a role,
we use boldface: $\bm{\s}_i$.
\item $\S_i$: the seal graph after removing digon $D_i$. 
$\S_0 = \varnothing$, and $\S_1 = x_1 y_1$.
\item $s_j \subseteq \s_j$, $1 \le j \le i$, is the possibly truncated seal segment 
in $\S_i$, on the surface $P_j$.
%
So, after possibly other truncations, we reach $\tilde{\s_j}$, hence the informal inclusion $\tilde{\s_i} \subseteq s_j \subseteq \s_j$;
``informal'' because those geodesic segments live in different spaces.
\item $S_i$ is the subset of $L_i$ bounded by  $x_1 y_i$ and $x_i y_i$,
the \emph{sealed region} which we will later prove contains $\S_i$.
\end{itemize}


\section{Examples}

We start with $P_1$, previously displayed in Fig.~\figref{EqTriPyramid}.
$X$ is an equilateral triangle, with the apex centered above its centroid.
Fig.~\figref{P1_digons} shows the removal of $n{-}1=2$ digons $D_1, D_2$ that reduce $L_0$ to 
the equilateral triangle base $X$.
Images are repeated so that in one row the transition from $L_{i-1}$ to $L_i$ by removal of $D_i$ is evident.
\begin{figure}[htbp]
\centering
\includegraphics[width=0.70\textheight]{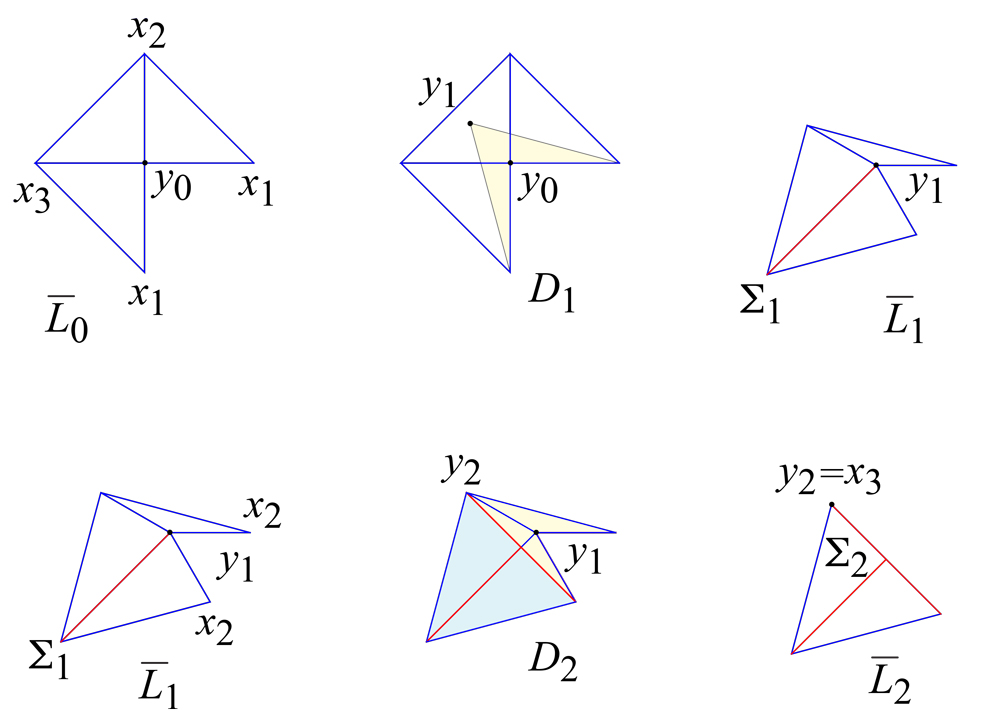}
\caption{$P_1$. $X$ is an equilateral triangle. 
Cf.~Fig.~\protect\figref{PyrBases}(a).}
\figlab{P1_digons}
\end{figure}

Next is the more complicated $P_2$ in Fig.~\figref{P2_digons}.
Here $X$ is a rectangle, and three digons are removed,
$D_1,D_2,D_3$, before reaching $X$.
\begin{figure}[htbp]
\centering
\includegraphics[width=0.75\textheight]{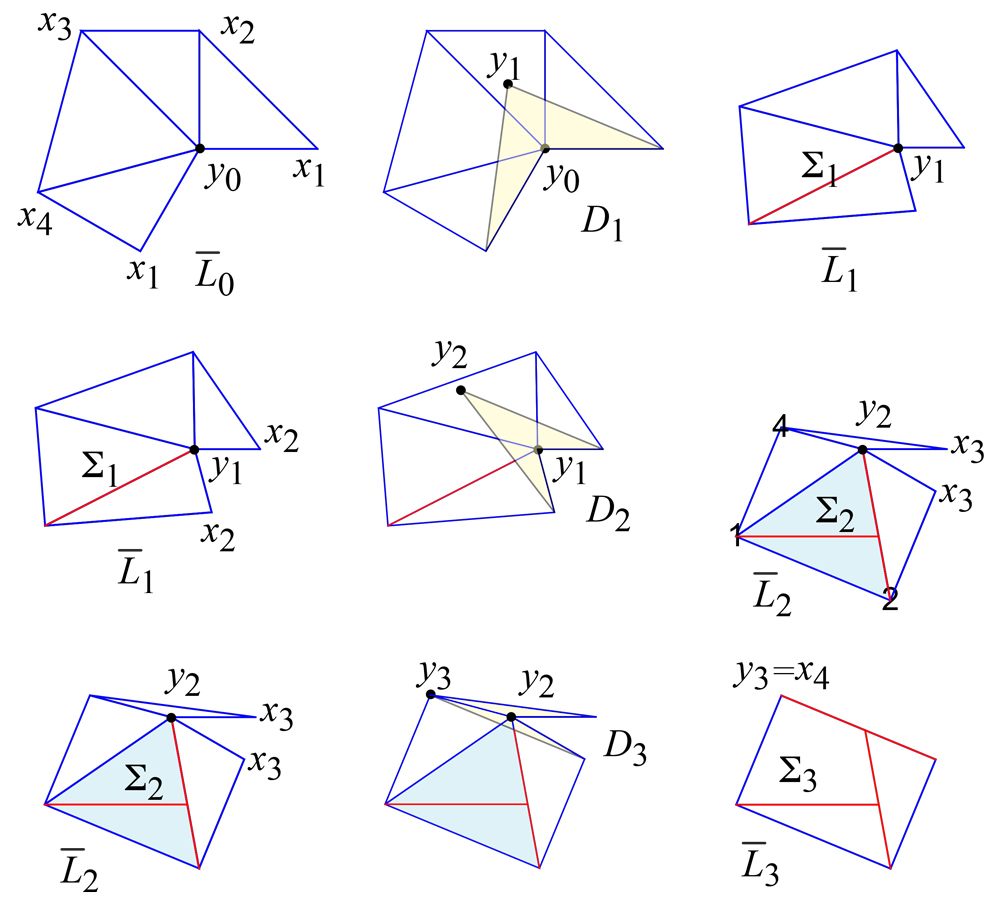}
\caption{$P_2$. $X$ is a rectangle.
Cf.~Fig.~\protect\figref{PyrBases}(b).}
\figlab{P2_digons}
\end{figure}
We should emphasize that any one of these figures could be cut out and closed
to a cone.
This cone would not be rigid, and its boundary $\partial X$ would not (in general)
be planar, as we mentioned earlier.

As a last example, we extract one row illustrating removal of $D_5$ from a
pyramid of degree-$12$ in
Fig.~\figref{Pn12_3figs}. We will refer to this figure subsequently.
\begin{figure}[htbp]
\centering
\includegraphics[width=1.1\linewidth]{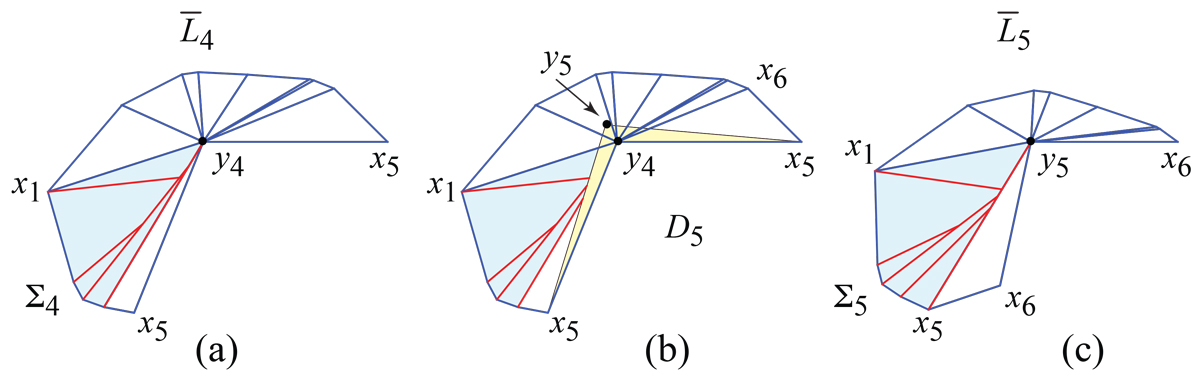}
\caption{Pyramid of $n{=}12$ base vertices $x_i$.
(a)~$\bar{L}_4$. $S_4$ (blue).
(b)~Digon $D_5$ (yellow).
(c)~$\bar{L}_5$.
$\S_i$ in red.}
\figlab{Pn12_3figs}
\end{figure}

\section{Preliminary Lemmas}
The viewpoint just described is simple enough to be implemented,
and to allow us to construct the seal graph $\S$ for any pyramid $P$.
See ahead to Fig.~\figref{Pyr_seals} for examples.
This leads to 
Theorem~\thmref{SealGraphTree}: $\S$ is a tree.
The proof of this claim is somewhat intricate, and presented in 
Section~\secref{SealGraphTree}.
That proof requires two lemmas, both involving
the structure of the cut locus, which we present first.
The reader might skip these proofs until later.


\begin{lm}
\lemlab{CutLocusPath}
After removing digons $D_1,\ldots,D_i$
and closing seals $\s_1,\ldots,\s_i$,
$C(x_{i+1}, P_i)$ includes the path $x_1,\ldots,x_i$,
with each node in that path of degree-$2$.
\end{lm}
\noindent
For example, in Fig.~\figref{Pn12_3figs}, $i=4$ and $C(x_5,P_4)$ includes
$x_1,x_2,x_3,x_4$.
\begin{proof}
We start with the leaf $x_i \in C(x_{i+1})$, and argue that
$(x_i,x_{i-1},\ldots,x_1)$ is a path $\r$ in $C(x_{i+1})$, i.e., that every point along $\r$ is of degree $\le 2$.
The proof uses techniques detailed in the proof of Lemma~\lemref{pyramid}.
In particular, Fig.~\figref{PathAbstract} below depicts the situation abstractly, similar to
Fig.~\figref{Claim1Abstract} in Lemma~\lemref{pyramid}.

First, $x_i$ is of degree-$1$ in $C(x_{i+1})$: The pyramid edge $x_{i+1} x_i$
is the shortest geodesic, unaffected by the digon removals up to $D_i$.
An edge $e_i$ of $C(x_{i+1})$ starts at $x_i$, and because of the
equal angles above on $L_i$ and below on $X$ at $x_i$, and because
$e_i$ is bisecting, initially it starts along the geodesic $x_i x_{i-1}$.
It then either continues to $x_{i-1}$, or reaches a ramification point.

Suppose the path continues $x_i, x_{i-1}, \ldots, x_j$, but then reaches a ramification point $r$ on $x_j x_{j-1}$.
Let $\r_j$ denote this path up to $r$.
We now analyze this situation and show it is contradictory.
Consult Fig.~\figref{PathAbstract} throughout.
\begin{figure}[htbp]
\centering
\includegraphics[width=0.75\linewidth]{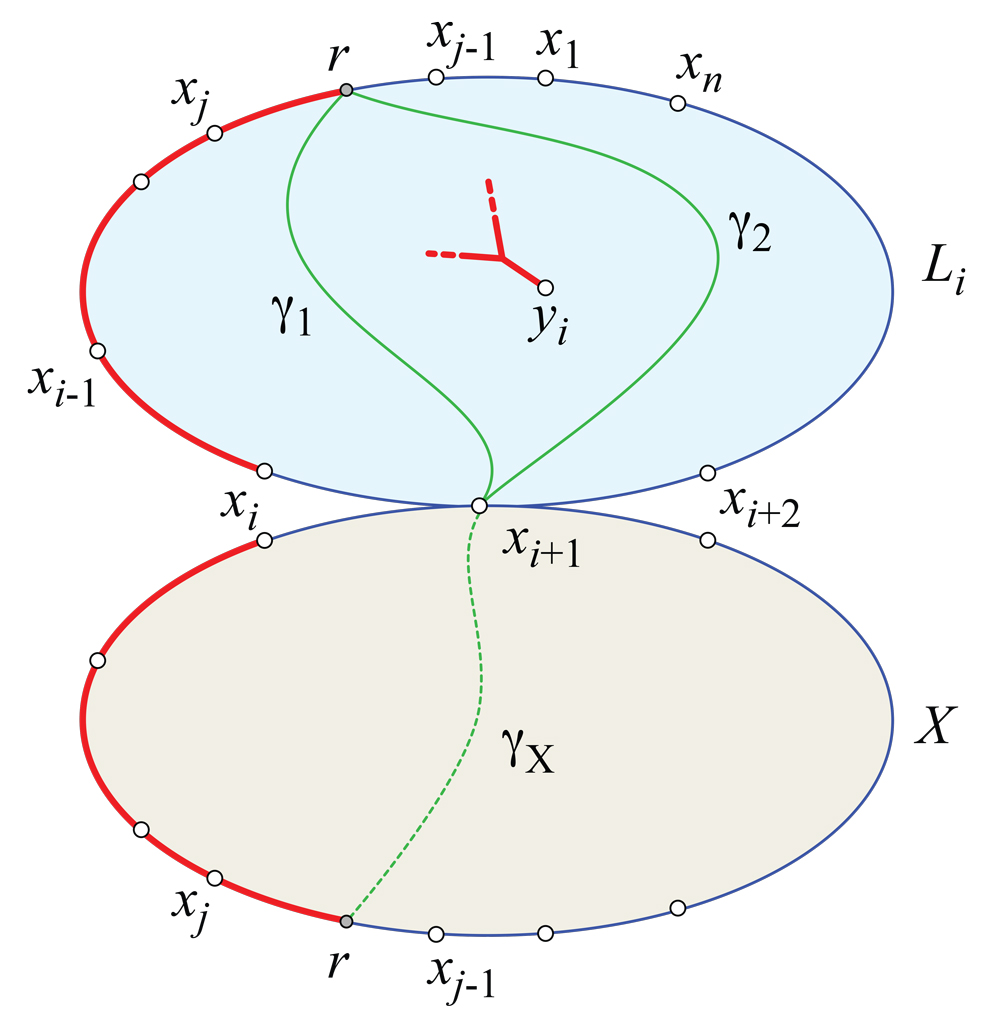}
\caption{Abstract depiction of $x_i,x_{i-1},\ldots,x_1$ path in $C(x_{i+1})$ (red).
$X$ is flipped below; $L_i$ above.}
\figlab{PathAbstract}
\end{figure}

The possible geodesics from $x_{i+1}$ to $r$ are:
$\g_X$ on $X$ below, $\g_1$ and $\g_2$ on $L_i$ above, and possibly $\g_3$ both above and below.
Note that there can only be the two $\g_1$ and $\g_2$ because there is just one vertex $y_i$ on $L_i$.

First consider $\g_1$, which, together with $\g_X$, encloses $\r_j$.
The planar convex chain along $\r_j$, $x_i,x_{i-1},\ldots,x_j, r$ is congruent above and below, because the angles above and below
are equal after digon removals. 
Thus the chords connecting the endpoints of the chains are equal, and so $|\g_1| = |\g_X|$.

Next consider $\g_2$, which, together with $\g_X$, encloses $x_{i+2},\ldots,x_n,x_1,\ldots,x_{j-1}$.
We again compare the planar convex chain with angles below on $X$ to the chain with angles above on $L_i$. 
Because some of the angles above are strictly larger than their counterparts below, we can apply Cauchy's Arm Lemma
just as we did in the proof of Lemma~\lemref{pyramid}(Claim~(1)) to conclude that $|\g_2| > |\g_X|$.
Therefore $\g_2$ cannot add to the degree of $r$ in $C(x_{i+1})$.

Finally, a geodesic $\g_3$ that lies on both $L_i$ and $X$ must have a portion completely above on $L_i$, to which we may apply the same 
arm-lemma argument to conclude that $|\g_3| > |\g_X|$.

Therefore, $r$ is in fact of degree $2$, it is not a ramification point, and $\r$ extends from $x_i$ to $x_1$ as claimed.
\end{proof}

\begin{lm}
\lemlab{RamifOutside}
Assume the digons $D_1,\ldots,D_i$ have been removed and $S_{i-1}$ is the sealed region containing
$\s_1,\ldots,\s_{i-1}$ (this inclusion will be proven later).
Let $a_i$ be the first ramification point of $C(x_i,P_{i-1})$, on the segment $y_{i-1} a_i$.
Then $a_i \not\in S_{i-1}$.
\end{lm}

\noindent
First we illustrate the claim of the lemma with the example from Fig.~\figref{Pn12_3figs}(b), repeated as Fig.~\figref{Pn12_CutLocus}(a).
In order for the lemma to be false, the situation instead must appear as in~(b) of the figure, with $a_5 \in S_4$.

\begin{proof}
For the purposes of contradiction, consider the situation depicted abstractly in Fig.~\figref{Pn12_CutLocus}(c), with $a_i \in S_{i-1}$.
By Lemma~\lemref{CutLocusPath}, $x_1,\ldots,x_{i-1}$ is a path of degree-$2$ nodes in $C(x_i,P_{i-1})=C(x_i)$.
The edge $y_{i-1} a_i$ of $C(x_i)$ bisects the 
angle formed by the two images of $x_i$.
From $a_i$, $C(x_i)$ must contain a path $\r$ that connects to $x_i$.
Because $a_i \in S_{i-1}$, $\r$ must start with an edge to the right of the line containing $y_{i-1} a_i$.

Note that $y_{i-1}$ is the only vertex on $L_{i-1}$, so $a_i$ is not a vertex, and therefore has a total surface angle $\q_{a_i}=2\pi$.
Lemma~\lemref{CutLocusAngles} requires a strictly leftward branch at $a_i$, left of the line containing $y_{i-1} a_i$.
Call the path continuing this left branch $\lambda$.
This path $\lambda$ is ``trapped": It cannot terminate in the interior of $S_{i-1}$ because there are no vertices in that region.
It cannot connect to any one of $x_1,\ldots,x_{i-1}$ because those nodes are degree-$2$. If $\lambda$ crossed $\r$ it would form a cycle.
Therefore we have reached a contradiction, and so $a_i \not\in S_{i-1}$.
\end{proof}

\begin{figure}[htbp]
\centering
\includegraphics[width=1.0\linewidth]{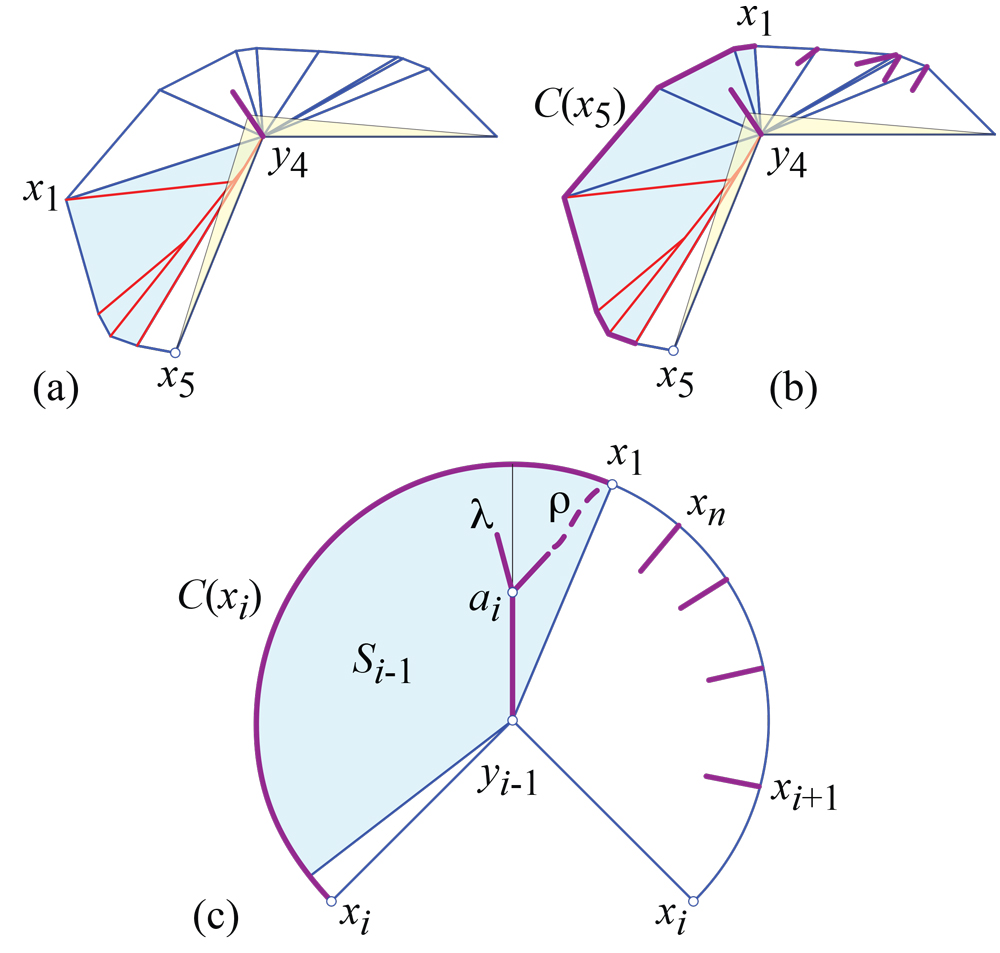}
\caption{(a)~Fig.~\protect\figref{Pn12_3figs}(b).
(b)~How a counterexample might appear.
(c)~Abstract counterexample. Portions of $C(x_i)$ purple.}
\figlab{Pn12_CutLocus}
\end{figure}


\section{Pyramid Seal Graph is a Tree}
\seclab{SealGraphTree}

Roughly, we prove in this section that the complete seal graph $\S=\S_n$, and intermediate seal graphs $\S_i$, have the structure of a spiraling tree.
Spiral trees---slit trees rather than seal trees---will play a significant role in Part~II.

We repeat some notation here for convenience.
The \emph{sealed region} $S_i$ is bounded by the segments $x_1 y_i$, $x_i y_i$, and the portion of $\partial X$ from $x_1$ to $x_i$.
Also recall that each seal $\bm{\s}_j = (x_j,y_j)$ is directed from $x_j$ to $y_j$.
The seal graph $\S_i$ is composed of segments $s_j$, each a subsegment of $\s_j$.

\begin{thm}
\thmlab{SealGraphTree}
The seal graph $\S_i$, after removal of digons $D_1,\ldots,D_i$ in counterclockwise order, has the following properties:
\begin{enumerate}[label=(\arabic*)]
\squeezelist
\item $\S_i \subset S_i$.
\item $\S_i$ is a directed tree with root $y_i$. 
%
%
\item Each segment $s_j$ of $\S_i$ is a (possibly truncated) seal
$\s_j= x_j y_j$ that remains anchored on its $x_j$ endpoint; i.e., the truncation
is on the $y_i$-end.
\item Each leaf $x_j$ is the start of a directed, convex path $\bm{\pi}_j$ to the root $y_i$.
\item The edges of $\bm{\pi}_j$ are portions of seal segments of increasing indices.
\item Along $\bm{\pi}_j$, lower-indexed edges terminate from the left on higher-indexed edges.
\item The last seal, $\bm{\s}_i= x_i y_i$, the \emph{root segment} of $\S_i$,
has no segments of $\S_i$ incident to its right side.
\end{enumerate}
The last segment of the complete seal graph $\S = \S_n$ coincides with the edge $x_{n-1} x_n$ of $\partial X$.
\end{thm}

\noindent
We again refer to Fig.~\figref{Pn12_3figs}(ab).
When $i=4$, $\S_4$ satisfies the properties,
and we seek to re-establish the properties for $\S_5$ in~(c) of the figure.
We enlarge~(b) of that figure in Fig.~\figref{Pn12_figb} to track in this proof.
It may also help to consult the complete seal graphs in Fig.~\figref{Pyr_seals}.
\begin{figure}[htbp]
\centering
\includegraphics[width=0.6\linewidth]{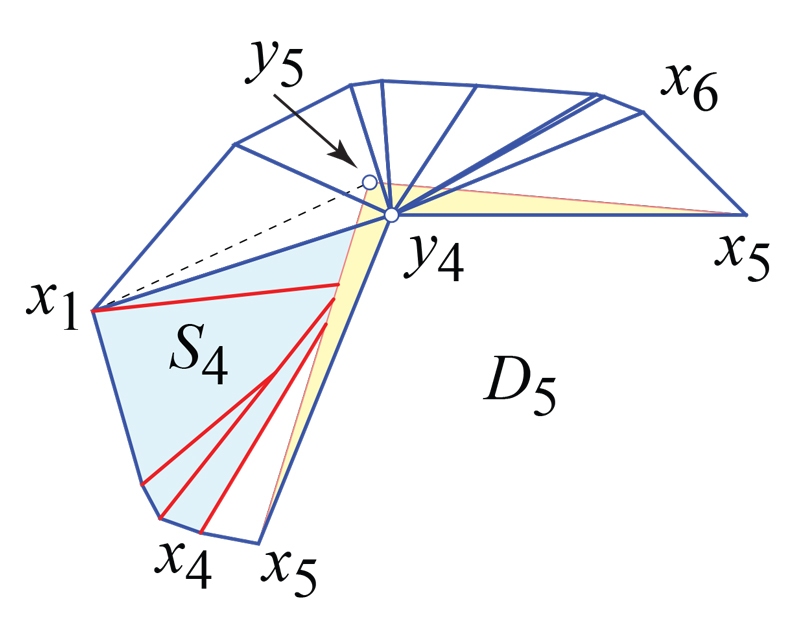}
\caption{$i=4$, before removal of digon $D_{i+1}=D_5$.
Detail from Fig.~\protect\figref{Pn12_3figs}(b).}
\figlab{Pn12_figb}
\end{figure}

\begin{proof}
\textbf{Induction Basis.}
These claims are trivially true for $i=0$ because $\S_0 = \varnothing$.
So assume $i=1$.
$L_1$ has just had the digon $D_1 = x_1 y_1$ excised around $y_0$.
$\bar{L}_1$ is then cut open along $x_2 y_1$.
$\S_1$ is the single segment $\s_1=x_1 y_1$, and all properties are easily verified.

\textbf{Induction Hypothesis.}
Assume that all properties hold for $\S_i$. The removal of
digon $D_{i+1} = x_{i+1} y_{i+1}$ leads to $\S_{i+1}$.
Now we establish the properties for $\S_{i+1}$.
\begin{enumerate}[label=(\arabic*)]
\squeezelist
\item $\S_{i+1} \subset S_{i+1}$.
$S_{i+1}$ grows on both sides:
by the triangle $x_1 y_i y_{i+1}$
(clockwise in Fig.~\figref{Pn12_figb}),
and to $x_{i+1} y_{i+1}$ (counterclockwise in the figure).
By Lemma~\lemref{RamifOutside}, we know that
$a_{i+1} \not\in S_i$, and because $y_{i+1}$ is on the $y_i a_{i+1}$ segment
of $C(x_{i+1})$, we know that $y_{i+1} \not\in S_i$.
Therefore the triangle $x_1 y_i y_{i+1}$ does in fact grow $S_i$ on the $x_1$-end.
On the $x_{i+1}$-end, the new seal
$\bm{\s}_{i+1}=x_{i+1} y_{i+1}$ is incorporated, and all seal segments right of
$\bm{\s}_{i+1}$ are clipped by the removal
of digon $D_{i+1}$. Therefore indeed $S_{i+1}$ 
expands to include all of $\S_{i+1}$.
\item $\S_{i+1}$ is a directed tree with root $y_{i+1}$.
We know that $\S_i$ is a directed tree with root $y_i$.
We first argue that the seal $\bm{\s}_{i+1}$ intersects the segments
of $\S_i$ from right-to-left, re-establishing property~(7).
Let $\g_1$ and $\g_2$ be the left and right geodesics of digon $D_{i+1} = x_{i+1} y_{i+1}$.
$\g_1$ starts within the triangle $\triangle=x_i x_{i+1} y_i$,
at an angle $(\a_{i+1}+\b_{i+1})/2$ left of the $x_{i+1} y_i$ edge of $\triangle$.
Therefore $\g_1$ cuts into $S_i$ through the $x_i y_i$ edge of $\triangle$.
The segments of $\S_i$ crossed and clipped by $\g_1$ are crossed from right-to-left.
And since $i+1$ is the highest indexed seal, the segments of $\S_i$
that meet $\bm{\s}_{i+1}$ satisfy property~(7):
lower-indexed segments terminate on the left of $\bm{\s}_{i+1}$.%
\footnote{%
We should mention that this property, that $\bm{\s}_{i+1}$ crosses segments
right-to-left, is dependent on the counterclockwise ordering of digon removal.
If after removing digon $D_i$, we next removed some $D_j$ with $j > i+1$,
it could be that it is $\g_2$ rather than $\g_1$ that clips $\S_i$.
This would result in seal graphs with a different structure.%
}
Henceforth, we use $\s_{i+1}$ as determined by $\g_1$.

By property~(8), $\bm{\s}_i = x_i y_i$ is the root segment of $\S_i$,
which we now know is crossed by $\bm{\s}_{i+1}$ right-to-left,
say crossing at point $p$.
Because $y_{i+1} \not\in S_i$ by Lemma~\lemref{RamifOutside},
$\bm{\s}_{i+1}$ has the root $y_i$ of $\S_i$ to its right, and all
the leaves $x_j$ of $\S_i$ to its left.

Now suppose $\S_{i+1}$ has an undirected cycle $\Phi$; see Fig.~\figref{TreeCycleNot}.
Because $\S_i$ has no cycle, any cycle in $\S_{i+1}$ must have one edge a subsegment of $\s_{i+1}$. 
Because all of $\S_i$ right of $\bm{\s}_{i+1}$ is removed, the cycle $\Phi$ must ``rest on" the left side of $\bm{\s}_{i+1}$.
It is clear $\Phi$ cannot rest on the $x_{i+1} p$ portion of $\bm{\s}_{i+1}$. 
So $\Phi$ must rest on the portion of $\bm{\s}_{i+1}$ left of $\bm{\s}_i$, as illustrated. 
But then, imagining removing $\bm{\s}_{i+1}$, $\Phi$ must have formed a cycle $\Phi'  \supset \Phi$ in $\S_i$, a contradiction. 
Therefore $\S_{i+1}$ is indeed a tree.

The remaining properties are now easily established.
%
%
\item Because $\bm{\s}_{i+1}$ clips segments of $\S_i$ to its right, each
segment $s_j \in \S_i$ remains anchored on $x_j$. And
the new segment $\bm{\s}_{i+1}$ is anchored on $x_{i+1}$.
\item The directed, convex path $\bm{\pi}_j$ remains, but now may by shortened
where it joins with $\bm{\s}_{i+1}$.
\item The segments along $\bm{\pi}_j$ have increasing indices, possibly now including
$i+1$.
\item We earlier established that lower-indexed segments of $\S_{i+1}$ terminate
from the left on higher-indexed segments, possibly now including $\bm{\s}_{i+1}$.
\item  The last and new seal $\bm{\s}_{i+1}=x_{i+1} y_{i+1}$ becomes the root segment
incident to the root $y_{i+1}$, and has no segments incident to its right side.
\end{enumerate}
Finally, it is a consequence of Lemma~\lemref{vertex_small} and the rigidity Theorem~\thmref{Rigid}
that $y_{n-1}=x_n$ so that the $(n{-}1)$-st seal coincides with the edge $x_{n-1} x_n$.
\end{proof}


\begin{figure}[htbp]
\centering
\includegraphics[width=0.6\linewidth]{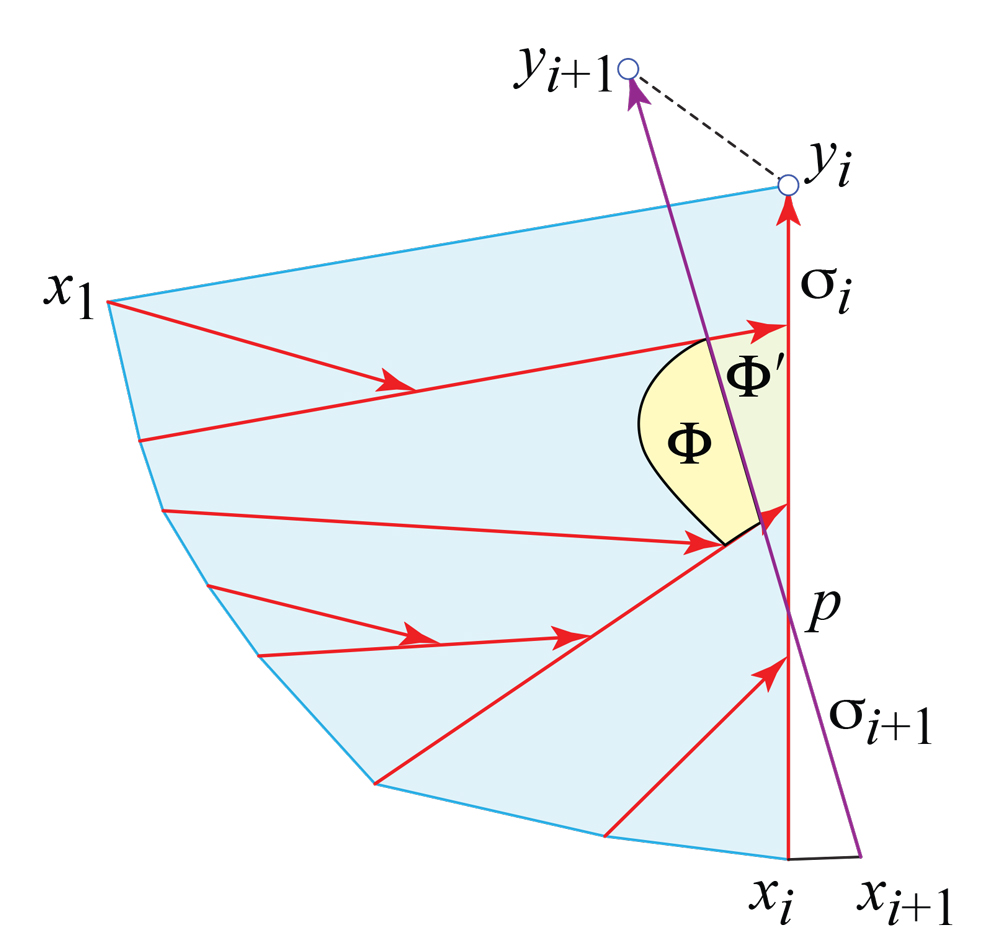}
\caption{Cycle $\Phi$ is not possible: 
removing $\s_{i+1}$ extends 
$\Phi$ to $\Phi'$.}
\figlab{TreeCycleNot}
\end{figure}

\begin{figure}[htbp]
\centering
\includegraphics[width=1.0\linewidth]{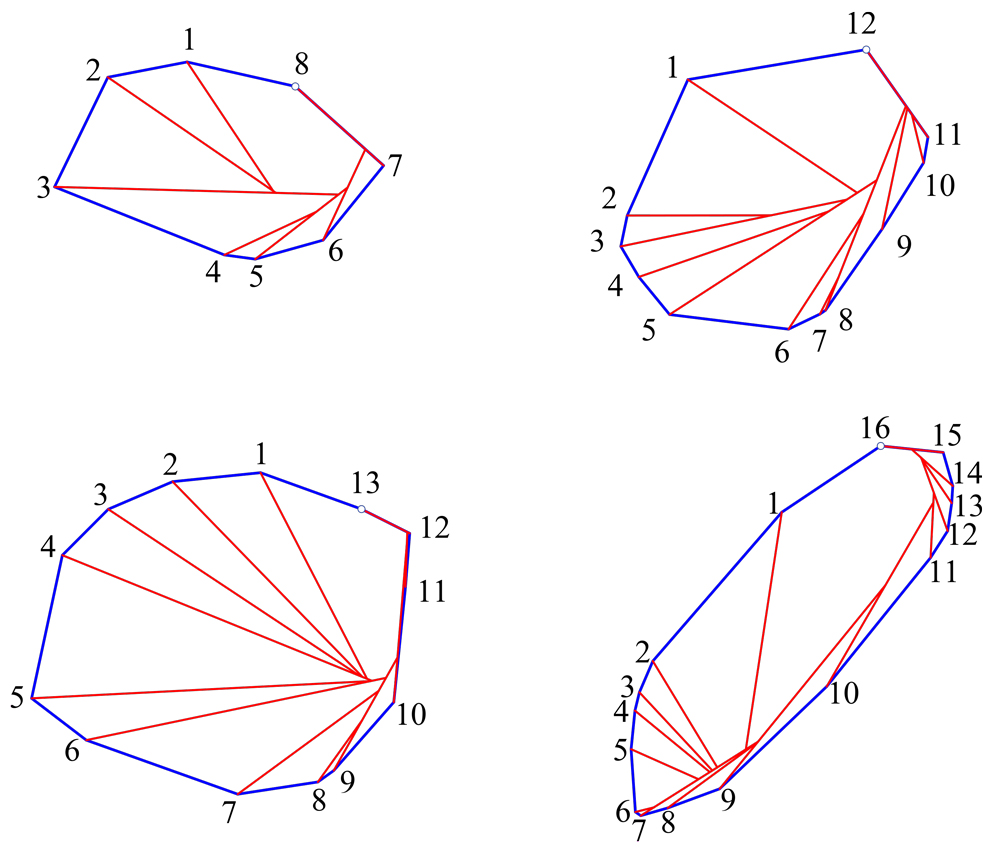}
\caption{Four complete seal graphs.}
\figlab{Pyr_seals}
\end{figure}

\subsection{Other Digon Orderings}
\seclab{OtherDigonOrderings}

The proof of Theorem~\thmref{SealGraphTree} depends
on removing the digons in the order $x_1,x_2,\ldots,x_{n-1}$ around $\partial X$.
This affects Lemma~\lemref{CutLocusPath}'s conclusion that $x_1,\ldots,x_i$
is a path in the cut locus, which then affects
Lemma~\lemref{RamifOutside}'s conclusion that the ramification point $a_{i+1}$
is outside the sealed region $S_i$.
In addition, which side of the removal of digon $D_{i+1}$ clips $\S_i$
is affected by knowing that $x_{i+1}$ is adjacent to $x_i$ on $\partial X$.
All of these consideration affect the structure of the seal graph. 
We leave the question of whether the seal graph is a tree for 
other orderings of digon removals to Open Problem~\openref{OtherDigonOrderings}.

Toward this open problem,
we only show here, with the next result, what a  degree-$4$ vertex in $\S$ 
must look like.

In the proof below, we simplify the notation of $\tilde{\s_i}$ on the base
to just $\s_i$, keeping in mind
that $\s_i$ and $\s_j$ formally live in different spaces,.

\begin{lm}
If a seal graph $\S$ for a pyramid $L$ has a degree-$4$ vertex $z$, then there exist $i<j<k$ such that $\s_i$ and $\s_j$ end at $z$, and $\s_k$ passes beyond $z$.
Moreover, the digon excision order is 
$D_k$ immediately after $D_j$ immediately after $D_i$.
\end{lm}

\begin{proof}
Consider a common point $z$ of $\s_i$ and $\s_j$, with $i < j$.
We may assume $j > i+1$ , since otherwise $\deg z = 3$ in $\S$.

Assume first that no other $\s_k$ passes through $z$.
Assume that $\deg z = 4$ in $\S$.
This implies that the digon $D_j$ crosses $\s_i$ and, since $\s_i$ remains a geodesic after the excision of $D_j$, 
$\s_i$ must be orthogonal to both geodesics bounding $D_j$.
Therefore, $\s_i$ creates with those two geodesics two geodesic triangles, both of positive curvature.
So each such triangle contains a vertex inside, contradicting that $D_j$ itself contains only one vertex.

Assume now that $\s_i$ ends at $z$, as does $\s_k$ for some $k \not\in \{i,j\}$.
Notice that we cannot have four $\s$s ending at $z$, because for the last one 
arriving---say $\s_k$---$\s_k$ would create a vertex at $z$, 
which will be excised by the digon $D_{k+1}$, breaking that degree-$4$ configuration at $z$.

So we may assume that $z$ belongs only to  $\s_i$, $\s_j$, and $\s_k$.
We may further assume, without loss of generality, that $i<j<k$.

Notice that both $\s_i$ and $\s_j$ end at $z$ and $\s_k$ does not, 
because otherwise $\s_k$ would create a vertex on $\s_i$ (or on $\s_j$) which would be excised by the digon $D_{k+1}$, 
breaking that degree-$4$ configuration at $z$.

The digon excision order follows: only $D_j$ could surround $y_i$, so its excision was just after $D_i$, and similarly for $D_j$.
\end{proof}

\medskip

For a particular digon-removal ordering, consider  the inverse image of $\S$ on $L$, and denote it by $\chi$. 
$\chi$ is a simple geodesic polygon surrounding $v$.
In Fig.~\figref{DigonsHexFlatX}(a), $\chi$ is the boundary of the gray region, 
effectively the union of the digons
(but recall that the digons $D_i$ live on different surfaces $P_{i-1}$; hence ``effectively").
Clearly, excising the surface bounded by $\chi$ from $L$ all at once achieves the same effect as excising the digons $D_i$ one-by-one.

The region of $L$ bound by a geodesic polygon $\chi$ is a particular instances of what we call a \emph{crest}:
a subset of $L$ enclosing $v$ whose removal and suitable suturing via AGT will reduce $L$ to $X$. 
Note that we allow the boundary of a crest to include portions of $\partial L$, e.g., $\chi$ in Fig.~\figref{DigonsHexFlatX}(a) includes the $x_i$ as well as the edge $x_5 x_6$.
In Chapter~\chapref{Crests} we will show that it is possible to construct crests directly on $L$ without deriving them from digon removals.

\chapter{Algorithm for Tailoring via Sculpting}
\chaplab{TailoringAlg1}

In this research we have also concentrated on achieving constructive proofs
of the theorems, constructive in the sense of leading to finite algorithms.
In this chapter we follow Theorem~\thmref{MainTailoring} to yield an algorithm for achieving the tailoring of $P$ to $Q$.
Throughout we measure computational complexity in terms of $n$,
where
$n=\max\{ |P|, |Q| \}$ is the number of vertices of the larger of $P$ or $Q$; so $|P|,|Q| = O(n)$.\footnote{
A finer analysis would treat the number of vertices of $Q$ and $P$ independently, say, $Q$ with $m$ vertices.
The algorithm steps would be the same, but the complexities would be apportioned differently.
}
Our goal for all the algorithms is to achieve polynomial-time complexity, $O(n^k)$,
but we have not worked hard to lower $k$, through, e.g., exploitation of efficient
data structures. 
Instead we are content to leave improvements for future work.
We will see that $k=4$ seems to suffice.

Euler's $V{-}E{+}F$ theorem implies that the number of edges and faces of a polyhedron are
linearly related to the number of vertices, so all components of $P$ and $Q$
are $O(n)$.
Thus for tailoring via slicing, the number of slices of $P$ is $O(n)$:
one slice per face of $Q$.
Each slice leads to g-domes, 
each g-dome to pyramids,
and each pyramid is reduced to its base by a series of tailoring steps: 
digon excisions and suturings.

We analyze the complexity in four parts, the $0$-th just the 
conceptual slicing of $P$ by planes on each face of $Q$:
\begin{itemize}
\squeezelist
\item Algorithm~0: Slice $P$ to $Q$: $O(n)$.
\item Algorithm~1: slice $\to$ g-domes, following Lemma~\lemref{Slice2gDomes}.
\item Algorithm~2: g-dome $\to$ pyramids, following Theorem~\thmref{DomePyr}.
\item Algorithm~3: pyramid $\to$ digons, following Lemma~\lemref{VertexTruncation}.\end{itemize}
Algorithms~$0,1,2$ operate on the extrinsic $3$-dimensional structure of the
polyhedra.
Algorithm~$3$ instead processes its calculations on the intrinsic structure of the surface.


\section{Algorithm~1: slice$~\to~$g-domes}

\begin{algorithm}[htbp]
\caption{From one slice, $O(n)$ g-domes.}
\DontPrintSemicolon

    \SetKwInOut{Input}{Input}
    \SetKwInOut{Output}{Output}

    \Input{One slice plane $\Pi$}
    \Output{$O(n)$ g-domes, a total of at most $O(n)$ vertices.}
       
     \BlankLine
     
     \tcp{Consult Fig.~\figref{DodecaGdome}.}
     
     Let $F$ be the face of $Q$ lying in $\Pi$, and $e$ be an edge of $F$.
     
     Sort vertices angularly about $e$. \tcp{$O(n \log n)$} 
     
     \For{$i=0,1,2,\ldots,k$}{
     
     Rotate $\Pi_i$ about $e$ until the portion swept is not a g-dome.
     
     \tcp{Following Lemma~\lemref{Slice2gDomes}.}
     
     Add to g-domes list.
     
     }
     
     \KwResult{List of $O(n)$ g-domes. }
     
\end{algorithm}




Algorithm~1 follows the proof of Lemma~\lemref{Slice2gDomes},
which clearly results in at most $O(n)$ g-domes, 
with the sum of the complexities of the g-domes for any one slice $O(n)$.
It can be implemented to run in $O(n \log n)$ time.
Therefore, over $O(n)$ slices, each resulting in $O(n)$ g-dome vertices,
we have at most $O(n^2 \log n)$ total g-dome complexity:
$O(n \log n)$ sorting repeated $O(n)$ times.

However, it could be that a clever choice of ordering of the slices always
results in a smaller total complexity, perhaps $O(n \log n)$.
Using $\O$ to indicate a lower bound,\footnote{
For example, a constant fraction of $n$ is $\O(n)$.}
the question is: Might each of 
the $\O(n)$ slices from Algorithm~$0$ lead to g-domes with
a total complexity of $\O(n^2)$, independent of the order of slicing?
Although resolving this question has only a small effect on the
overall time-complexity of the tailoring algorithm,
we take a detour to explore this issue via an example 
in the next subsection.

%


\subsection{Complexity of sculpting}

\begin{ex}
Consider Fig.~\figref{Qslice_PyPy_48}, two nested pyramids $Q \subset P$ sharing a common regular polygon base.
The combinatorial complexity of slicing $P$ with face planes of $Q$ is order-sensitive:
from $\O(n^2)$ to $\O(n \log n)$.
\end{ex}

\begin{figure}
\centering
 \includegraphics[width=1.0\textwidth]{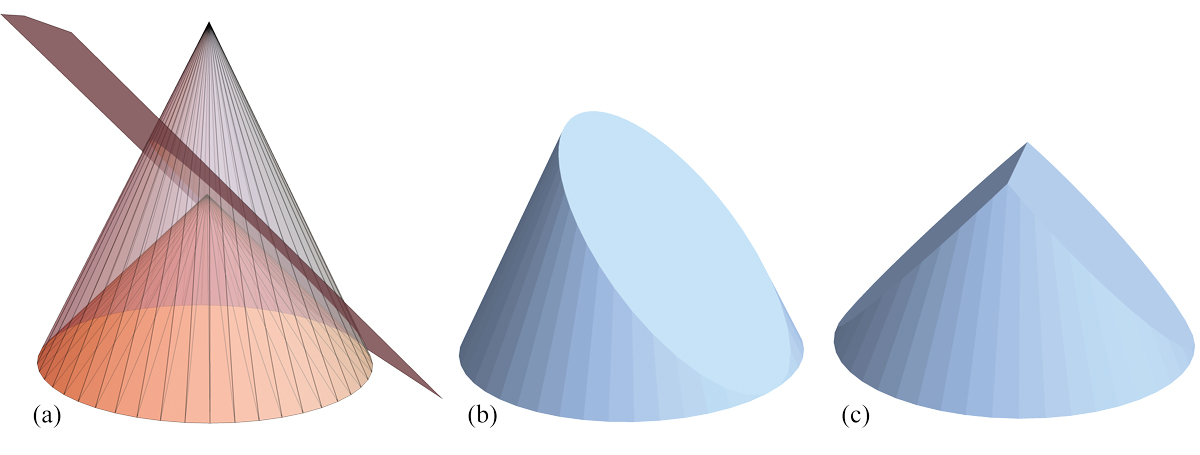}
\caption{$Q \subset P$. The shared base is a regular polygon of $n{=}48$ sides. 
(a)~One slice-plane $\Pi_0$ lying on a face of $Q$.
(b)~After slicing by $\Pi_0$.
(c)~A second, opposite slice $\Pi_{n/2}$.
}
\figlab{Qslice_PyPy_48}
\end{figure}

\begin{proof}
Let the faces of $Q$ ordered in sequence around the base be $F_0,F_1,\ldots,F_{n-1}$, each $F_i$ determining a plane $\Pi_i$.
Let $P_i$ be the polyhedron after slicing $P$ with planes $\Pi_0,\Pi_1,\ldots,\Pi_i$.
$\Pi_0$ cuts $n-2$ edges of $P$, as shown in Fig.~\figref{Qslice_PyPy_48}(a,b), and effectively removes half the edges of $P$ from later slices.
If $P$ is sliced in the order $i=1,2,\ldots,n-2$, following adjacent faces around the base, each plane cuts a diminishing number of the remaining edges.
An explicit calculation shows that plane $\Pi_i$ cuts $\lfloor \frac{n+1-i}{2} \rfloor $ edges of $P_{i-1}$.
And because
$$
\sum_{i=1}^{n-2} \left\lfloor \frac{n+1-i}{2} \right\rfloor  = \Omega( n^2 ) \;,
$$
with this plane-slice ordering, $\O(n)$ slices each have $\O(n)$ vertices, for a total quadratic complexity, $\O(n^2)$.

\medskip

However, if one instead orders the slices in a binary-search pattern, then the total complexity is $\O(n \log n)$, as we now show.
Let $n=2^m$ be a power of $2$ without loss of generality. 
The pattern is slicing with planes lying on faces $F_i$ with indices in the order
$$
s = \left( 0, \frac{n}{2}, \frac{n}{4}, \frac{3n}{4}, \frac{n}{8}, \frac{3n}{8}, \frac{5n}{8}, \frac{7n}{8}, \ldots \right) \;.
$$
We partition this sequence $s$ of indices
into subsequences, $s= (0, s_1, s_2, \ldots, s_m)$, as follows:
\begin{eqnarray*}
s_1 &=& \left( \frac{n}{2} \right)\\
s_2 &=& \left( \frac{n}{2^2}, \frac{3n}{2^2} \right)\\
s_3 &=& \left( \frac{n}{2^3}, \frac{3n}{2^3}, \frac{5n}{2^3}, \frac{7n}{2^3} \right)\\
\mbox{} &\cdots& \mbox{} \\
s_k &=& \left( \frac{j \,n}{2^k} \right), \; \textrm{where} \; j = 1,3,5,\ldots, 2^k{-}1  \; .
\end{eqnarray*}
Notice that the number of indices in sequence $s_k$, $|s_k|=2^{k-1}$.
As a check, the total number of indices in $s$ is
$$
1+ \sum_{k=1}^m |s_k| = 1+ \sum_{k=1}^m 2^{k-1} = 2^m = n \;.
$$

One can calculate that the slice at $\displaystyle{i= \frac{j \, n}{2^k}}$ only cuts $\displaystyle{\frac{n}{2^k}}$ edges of $P_{i-1}$. 
Only $\displaystyle{\frac{n}{2^{k-1}}}$ edges are ``exposed" to $\Pi_i$: for example, for $k=2$ and $i=1/4$, $n/2$ edges
are possibly available for cutting by $\Pi_i$, as can be seen in Fig.~\figref{Qslice_PyPy_48}(c).
However, because of the slant of $\Pi_i$, only half of those, $n/4$, are in fact cut by $\Pi_i$.

Now we compute the total number of edges sliced by the planes following the sequence $s$.
Because $|s_k|=2^{k-1}$, and each slice in $s_k$ cuts $\displaystyle{\frac{n}{2^k}} $edges, the total number of cuts over all $k$ is
$$
\sum_{k=1}^m 2^{k-1} \frac{n}{2^k} = n\sum_{k=1}^m \frac{1}{2} = n \, \frac{m}{2} \;.
$$
And since $m=\log n$, the total complexity is $\O(n \log n)$,
or an average of $\O(\log n)$ for each of $\O(n)$ slices.
\end{proof}

\noindent
In the absence of a resolution to this complexity question, we
will assume that Algorithm~0 and Algoirthm~1 together result in 
$O(n)$ g-domes each of $O(n)$ vertices, produced in time $O(n^2)$.


\section{Algorithm~2: g-dome$~\to~$pyramids}

\begin{algorithm}[htbp]
\caption{Partition one g-dome to $O(n)$ pyramids.}
\DontPrintSemicolon

    \SetKwInOut{Input}{Input}
    \SetKwInOut{Output}{Output}

    \Input{One g-dome $G$ of $O(n)$ vertices}
    \Output{$O(n)$ pyramids, each of size $O(n)$; and $O(n^2)$ pyramids, size $O(1)$.}

    \BlankLine

\tcp{Following Theorem~\thmref{DomePyr}.}



\For{each of $k=O(n)$ vertex-degree reductions}{

{As in Theorem~\thmref{DomePyr}, slice with plane: $O(k)$.}

{Remove pyramid $P$ of $O(k)$ vertices.}

{``Clean-up" by removing $k$ pyramids each of size $O(1)$.}

}

    \KwResult{List of pyramids. }
     
\end{algorithm}


We follow Theorem~\thmref{DomePyr} for partitioning each g-dome $G$ into pyramids.
Each vertex $v_i$ of the top-canopy of $G$ is removed,
as in Fig.~\figref{DomeSlice_begend}, until only one remains.
Removal of each $v_i$ follows the degree-removal steps
illustrated in 
Fig.~\figref{SliceAway_123456}.

Because the sum of the vertex degrees of a g-dome is $2E = O(n)$,
the asymptotic complexity of processing a g-dome with many vertices in its top-canopy
is no different than it is for just two vertices as in Fig.~\figref{DomeSlice_begend}.
Moreover, we can assume that $v_2$ has degree-$3$ while
$v_1$ has degree-$k$, with $k=O(n)$.

First a plane slice results in a pyramid of $k$ vertices with apex $v_1$, which is removed
(and reduced by Algorithm~3).
Next follows a ``clean-up" phase that removes $O(k)$ pyramids
each of $4$ or $5$ vertices, so of constant size, $O(1)$.

This is then repeated for the new apex of degree $k-1$:
removal of a pyramid of $k-1$ vertices, and cleanup of $O(k-1)$ pyramids of constant size.
After iterating through $k,k-1,k-2,\ldots$, the algorithm has sliced off
$O(k^2)$ pyramids of constant size,
and $O(k)$ pyramids of size $O(k)$.
In the worst case $k=O(n)$,
for a total complexity of $O(n^2)$.

\section{Algorithm~3: pyramid$~\to~$digons}

\begin{algorithm}[htbp]
\caption{Tailor one pyramid $P$ to its base $X$.}
\DontPrintSemicolon

    \SetKwInOut{Input}{Input}
    \SetKwInOut{Output}{Output}

    \Input{A pyramid $P$ of $O(n)$ vertices}
    \Output{$O(n)$ digons whose removal flattens $P$ to $X$. }

    \BlankLine

\tcp{Following Lemma~\lemref{VertexTruncation}.}

\tcp{Assume apex degree-$k$, with $k=O(n)$.}

\For{each of $x_i$, $i=1,2,\ldots,k$}{

{Construct digon $D_i(x_i,y_i)$: Locate $y_i$.}

{Locate $y_i$ by tracing geodesics: $O(k)$.}

}
  
  \KwResult{List of $O(n)$ digons}
     
\end{algorithm}

Lastly we concentrate on the cost of removing one pyramid $P$ of $O(n)$ vertices.
Following Lemma~\lemref{VertexTruncation}, this requires $O(n)$ digon removals.
For each digon $D_i(x_i,y_i)$, we need to calculate the location of $y_i$
on $C(x_i)$; then $y_i$ becomes a vertex for the removal of the next
digon $D_{i+1}$.
Fortunately, there is no need to compute the cut locus $C(x_i)$.

Let us focus on locating $y_i$, after the removal of $D_{i-1}(x_{i-1},y_{i-1})$
the previous iteration.
Recall that $y_i$ is the only vertex on $L_i$,
so it is immediate to find the shortest path $\g$ from $x_i$ to $y_i$.
We know the angle $\q=\o_Q(x_i) - \o_P(x_i)$ needed to be removed
by $D_i$,
so we know that geodesics $\g_1$ and $\g_2$ at angles $\q/2$ left and right
of 
$\g = x_i y_i$ will meet at $y_i$.
Tracing $\g_1$ and $\g_2$ over the surface might
cross sealed digons $D_1,\ldots,D_{i-1}$, the 
blue{seals $\s_i$}
in Fig.~\figref{DigonsHexFlatX}(b).
So the cost of computing $\g_1 \cap \g_2 = y_i$ is $O(n)$.

Thus the complexity of tailoring one pyramid of $O(n)$ vertices is $O(n^2)$.

Note that we do not need the extrinsic $3$-dimensional structure of the intermediate
polyhedra guaranteed by AGT to perform the calculations,
as is evident in the example described in
Fig.~\figref{DigonsHexFlatX}.


\section{Overall Tailoring Algorithm}

Putting the three algorithm complexities together, and 
for succinctness abbreviating 
$O(n)$ with $n$ and ``time-complexity" with ``complexity,"
we have:
\noindent
\begin{enumerate}[label={(\arabic*)}]
\item $n$ slices $\to$ $n$ g-domes, each of size $n$: complexity $n^2 \log n$.
\item $1$ g-dome of size $n$ $\to$
\begin{enumerate} 
\item $n$ pyramids, each of size $n$: complexity $n^2$.
\item $+$ $n^2$ constant-size pyramids: complexity $n^2$.
\end{enumerate}
\item $1$ pyramid of size $n$ $\to$ $n^2$  to reduce to base: complexity $n^2$.
\end{enumerate}
(1) and (2) together have complexity $n^3$, iterating over each of the $n$ domes.
(The $n \log n$ sorting need not be repeated after the g-domes are identified.)
(3) repeats $n^2$ over (potentially) $n^2$ pyramids: 
$n$ g-domes resulting in $n$ pyramids each.
Thus the total complexity is $n^4$.
We summarize in a theorem:

\begin{thm}
\thmlab{SliceAlgorithm}
Given convex polyhedra $P$ and $Q$ of at most $n$ vertices each, and $Q \subset P$,
$P$ can be tailored to $Q$, tracking a sculpting of $Q$ from $P$ as in Theorem~\thmref{MainTailoring}, in time $O(n^4)$.
\end{thm}


\chapter{Crests}
\chaplab{Crests}

In this chapter we revisit the suggestion made at the end of Chapter~\chapref{SealGraph} that
the digons to reduce one pyramid to its base could be cut out all at once.
As before, let $P = L \cup X$ be a pyramid with base $X$ and lateral sides $L$.
Recall that a \emph{crest} is a subset of $L$ enclosing $v$ whose removal and suitable suturing via AGT will reduce $L$ to $X$.

We derive a method for identifying a crest that does not rely on digon removals, but rather works directly on a pyramid.
This allows us to achieve in Theorem~\thmref{CrestTailoring} reshaping of $P$ to $Q$ by the removal of crests to flatten pyramids.
We call this process \emph{crest-tailoring}, in contrast to the \emph{digon-tailoring} explored in Chapter~\chapref{TailoringSculpting}.
It represents a tradeoff between the simplicity of digons and the number of excisions: 
for digon-tailoring, $O(n)$ per pyramid, for crests, one per pyramid.
We first illustrate the process of identifying a crest on two example pyramids before proving that it always works.


\section{Examples}
Let $P = L \cup X$ 
with $\partial X = \partial L = (x_1, x_2, \ldots, x_k)$
having vertices $x_i$.
The apex is $v$, which projects orthogonally to $\bar{v}$.
We will describe the procedure for identifying a crest first for $\bar{v} \in X$ and then for $\bar{v} \not\in X$. Although the cases initially feel different,
the proofs will show that they are nearly the same.

Fig.~\figref{PyHexV_1} illustrates the case $\bar{v} \in X$.
Let $T_i = x_i x_{i+1} v$ and $\bar{T}_i=x_i x_{i+1} \bar{v}$.
We proved in Lemma~\lemref{Angles} that the angle $\bar{\q}_i$ at $x_i$ in $X$ is strictly smaller than than the angle $\q_i$ on $L$, 
the sum of the angles in $T_{i-1}$ and $T_i$ incident to $x_i$ (as long as $|v \bar{v}| > 0$).
\begin{figure}
\centering
 \includegraphics[width=1.0\textwidth]{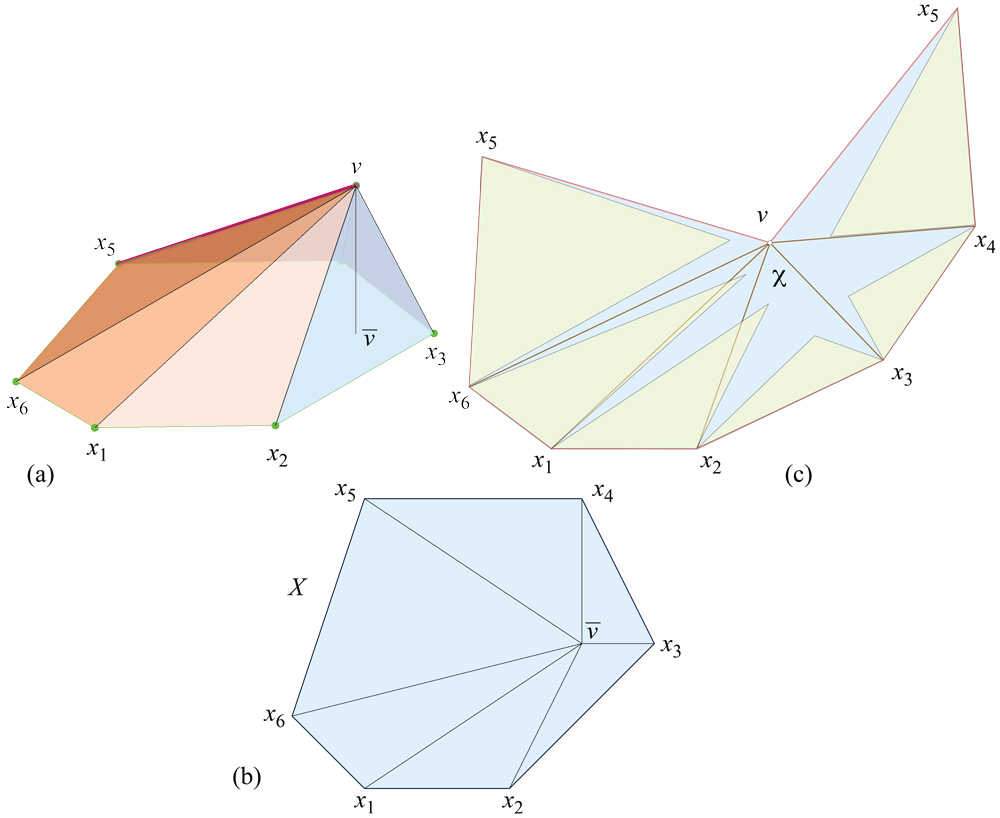}
\caption{(a)~Pyramid.
(b)~$X$ with projected triangles $\bar{T}_i=x_i x_{i+1} \bar{v}$.
(c)~Flattening of $L$ to $\bar{L}$, in this case by cutting the edge $x_5 v$.
The lifted triangles $T^L_i$ are shown yellow. The crest $\chi$ is blue.
}
\figlab{PyHexV_1}
\end{figure}

A key definition is the \emph{lift} of $\bar{T}_i$ onto $L$.
Let $\bar{\a}_i$ and $\bar{\b}_i$ be the base angles of $\bar{T}_i$, at $x_i$ and $x_{i+1}$ respectively.
On $L$, extend geodesic $\g_i$ from $x_i$ at angle $\bar{\a}_i$, and extend geodesic $\g_{i+1}$ from $x_{i+1}$ at angle $\bar{\b}_i$.
Let $v_i$ be the point on $L$ at which these geodesics meet; $v_i$ is the image of $\bar{v}$.
(We will not establish that indeed these geodesics meet on $L$ until Lemma~\lemref{Lifting}.)
Then $\operatorname{lift}(\bar{T}_i) = T^L_i$ is a geodesic triangle on $L$ isometric to $\bar{T}_i$.
Another way to view the lift of $\bar{T}_i$ is to imagine $\bar{T}_i$ rotating about $x_i x_{i+1}$ by the dihedral angle there and pasting it on the inside of $L$.

Yet another way to view the lift is as follows.
$L$ is isometric to a cone and can be flattened by cutting along a generator, i.e., a segment from $v$ to $\partial L$.
Let $\bar{L}$ be a particular flattening, with the cut generator not ``near" $x_i$ just for simplicity.
Then place a copy of $\bar{T}_i$ on $\bar{L}$ matching $x_i x_{i+1}$.
Then refold $\bar{L}$ to $L$.
We will continue to reason with a flattened $\bar{L}$ but remembering
that $\bar{L}$ is a representation of $L$, and so the cut edge is not relevant.

This last layout-viewpoint yields a method to construct a full crest, call it $\chi$.
The base $X$, partitioned into $\bar{T}_i$, can be modified by opening the angle at $x_i$ from $\bar{\q}_i$ to $\q_i$. 
After opening at all $x_i$, this figure can be superimposed on the flattened $\bar{L}$, matching the boundaries $x_1, \ldots, x_k$.
This is illustrated in Fig.~\figref{PyHexV_1}(c).
The crest is then the portions of $L$ not covered by the lifted $\bar{T}_i$.
It should be clear that cutting out $\chi$ and suturing closed the matching edges will reduce $L$ to $X$,
for the $\bar{T}_i$ remaining after removing $\chi$ exactly partition $X$.

We next illustrate the case when $\bar{v} \not\in X$;
see Fig.~\figref{PyHexV_2}.
We perform the exact same process of lifting triangles $\bar{T}_i$
to $L$, but we clip those triangles to $X$---i.e.,
form the polygon $\bar{T}_i \cap X$---as indicated in~(c) of the figure.
Notice that two triangles, $\bar{T}_3$ and $\bar{T}_4$, are removed by
the clipping intersection.
Again it should be clear that cutting out $\chi$ and suturing closed 
will reduce $L$ to $X$.
\begin{figure}
\centering
 \includegraphics[width=1.0\textwidth]{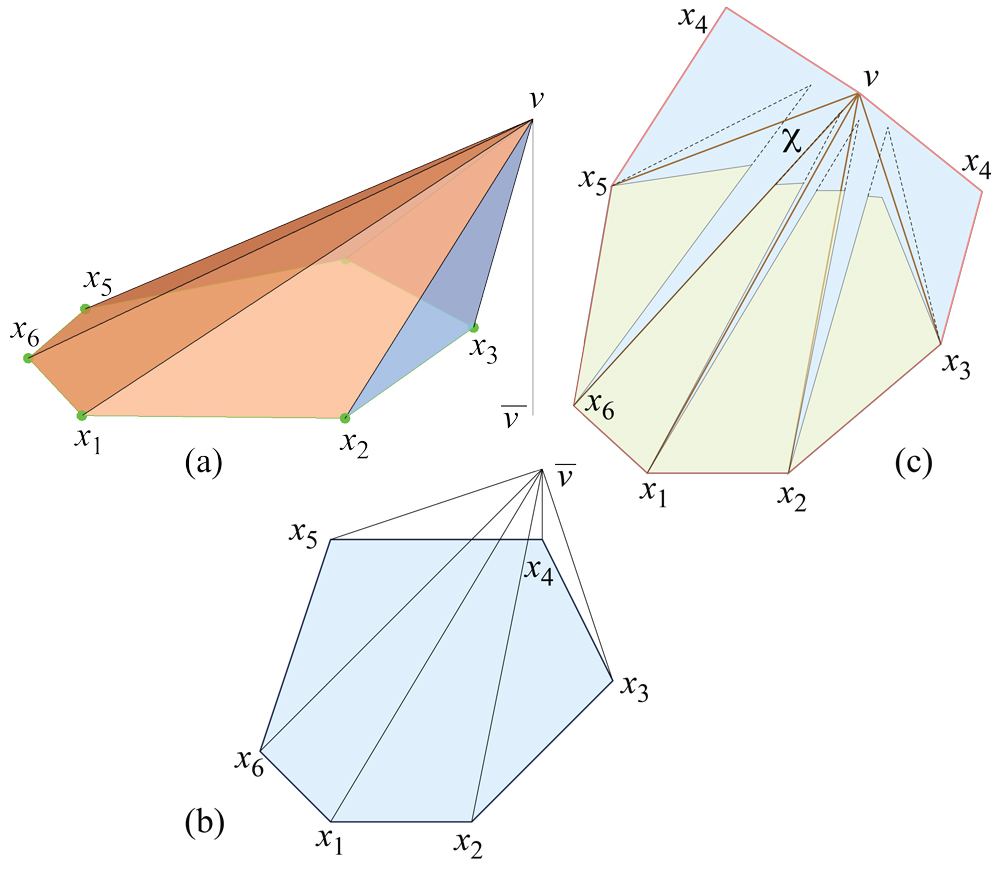}
\caption{Following the same conventions as in  Fig.~\protect\figref{PyHexV_1}:
(a)~Pyramid.
(b)~$X$; $\bar{v} \not\in X$.
(c)~$\bar{L}$, yellow clipped lifted triangles $T^L_i$, and blue crest $\chi$.}
\figlab{PyHexV_2}
\end{figure}


\section{Proofs}

We will need several geometric properties.

Let $\bar{\q}_i$ be the angle at $x_i$ in $X$, and  $\q_i$ the angle at $x_i$ on $L$, the sum of two triangle angles incident to $x_i$.
Then, by Lemma \lemref{Angles}, $\q_i > \bar{\q}_i$.

\begin{lm}
\lemlab{Altitude}
Let $T=a b c$ be a triangle in $\mathbb{R}^3$, with $a b$ on plane $\Pi$ and $c$ above that plane.
Let $\bar{c}$ be the orthogonal projection of $c$ onto $\Pi$, and $\bar{T}= a b \bar{c}$ the projected triangle.
Finally, let $T^r = a b c^r$ be the triangle $T$ flattened to $\Pi$ by rotating about $ab$.
Then $c^r$ and $\bar{c}$ lie on the altitude line perpendicular to the line containing $ab$.
\end{lm}

\begin{proof}
See Fig.~\figref{AltitudeLemma}.
The claim follows from the Theorem of the Three Perpendiculars.
Note that $T^r$ is congruent to $T$.
\end{proof}
\begin{figure}
\centering
 \includegraphics[width=1.0\textwidth]{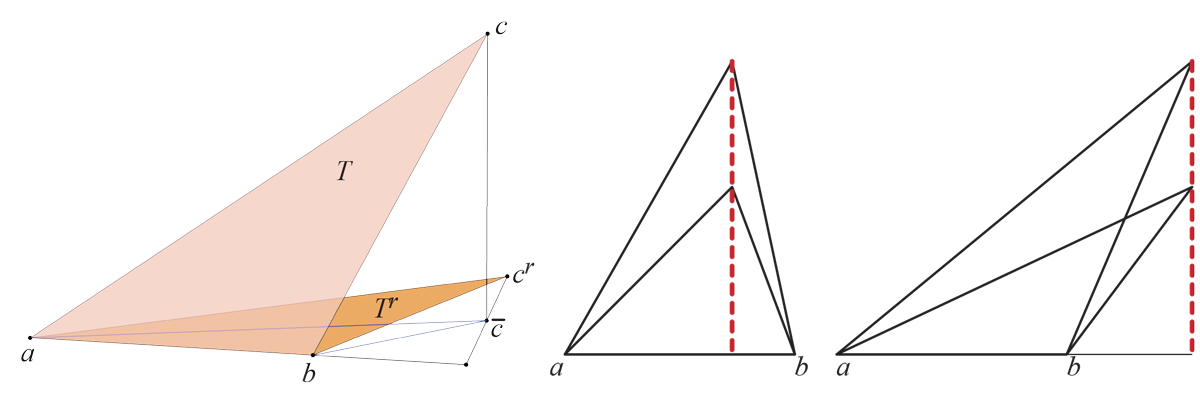}
\caption{
Rotated apex $c^r$ and projected apex $\bar{c}$ lie on same altitude.
}
\figlab{AltitudeLemma}
\end{figure}
The consequence of Lemma~\lemref{Altitude} is that, 
superimposing a lifted triangle $\bar{T}^L$ on a planar layout $\bar{L}$ of $L$,
the images of $v$ in $\bar{L}$ and $\bar{v}$ of $\bar{T}$, lie along the altitude of $T$.


The following lemma assumes that $\bar{v} \in X$.
The case when $\bar{v} \not\in X$ will be treated separately.
Let $\bar{L}$ be a planar layout of $L$, say, cut open at edge $x_1 v$.
Let $\t_i= \pi-(\b_{i-1}+\a_i)=\pi - \q_i$ be the turn angle at vertex $x_i$ in the layout.
Because $\bar{v} \in X$, $\t_i > 0$, i.e., the planar image
of $\partial L$ in $\bar{L}$ is a convex chain,
and also convex wrapping around the cut edge $x_1 v$.

Let $a_i$ be the segment altitude of triangle $T_i = x_i v x_{i+1}$ in the layout $\bar{L}$.
Let $\nu = 2\pi-\o(v)$ be the surface angle of $P$ incident to $v$.

\begin{lm}
\lemlab{AltTurns}
When  $\bar{v} \in X$ and consequentially $\bar{L}$ is a convex chain $x_1,\ldots,x_k$, the following hold (with $k+1 \equiv 1 \bmod{k}$):
\begin{enumerate}[label=(\alph*)]
\squeezelist
\item The sum of the turn angles, $\sum_i \t_i=\nu$, the surface angle incident to $v$.
\item The angle between $a_i$ and $a_{i+1}$ at $v$ on $L$ is exactly
equal to the turn angle $\t_i$ at vertex $x_i$.
\item The altitudes occur in order around $v$, in the sense that 
$a_{i+1}$ is counterclockwise of $a_i$ around $v$.
\end{enumerate}
\end{lm}

\begin{proof}
\begin{enumerate}[label=(\alph*)]
\item Viewing the entire layout $\bar{L}$ as a simple polygon, $\sum_i \t_i = 2 \pi$.
But we need to distinguish between $\t_1$ at the edge $x_1 v$ cut to flatten $L$, and $\t'_1$, the turns at the two images of $x_1$ in $\bar{L}$:
$$\t'_1 = 2\pi-(\a_1+\b_k) = \pi+ \t_1 \;.$$ 
The second anomalous turn in $\bar{L}$ is $\pi-\nu$ at $v$. So we have
\begin{eqnarray*}
\t'_1 + \sum_{i=2}^k \t_i + (\pi-\nu) &=& 2 \pi \;, \\
\sum_{i=1}^k \t_i &=& \nu \;.
\end{eqnarray*}

\item This can be seen by extending $T_i=x_i v x_{i+1}$ to a right triangle, with right angle at the foot of altitude $a_i$. 
See Fig.~\figref{AltitudeTurn}.

\item This follows directly from~(b). Note that here we rely on the turns $\t_i$ being positive, i.e., convex.
See Fig.~\figref{AltTurns}.
\end{enumerate}
\end{proof}
%
\begin{figure}
\centering
 \includegraphics[width=0.5\textwidth]{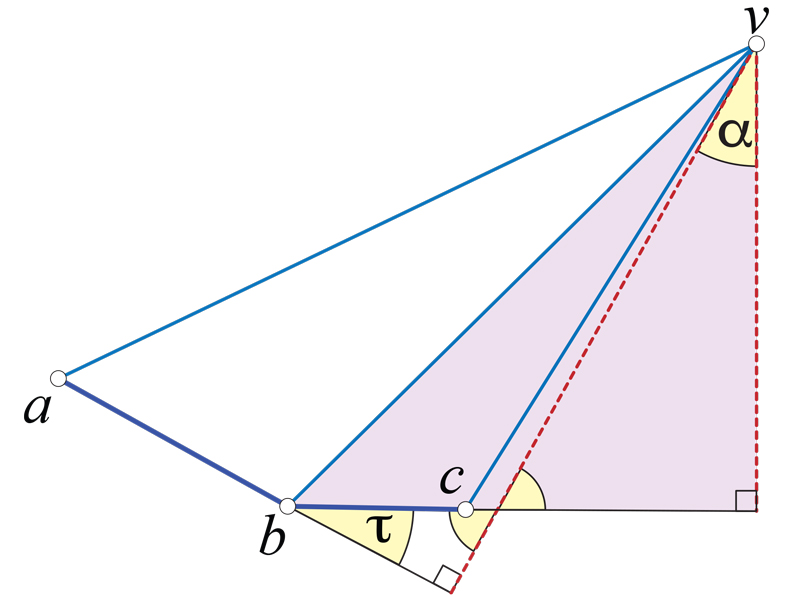}
\caption{
Triangles $avb$ and $bvc$ turn
$\t$ at $b$ of the convex chain $abc$. 
$\t$ is equal to the altitude turn $\a$;
angles of similar right triangles marked.
}
\figlab{AltitudeTurn}
\end{figure}

\begin{figure}
\centering
 \includegraphics[width=1.0\textwidth]{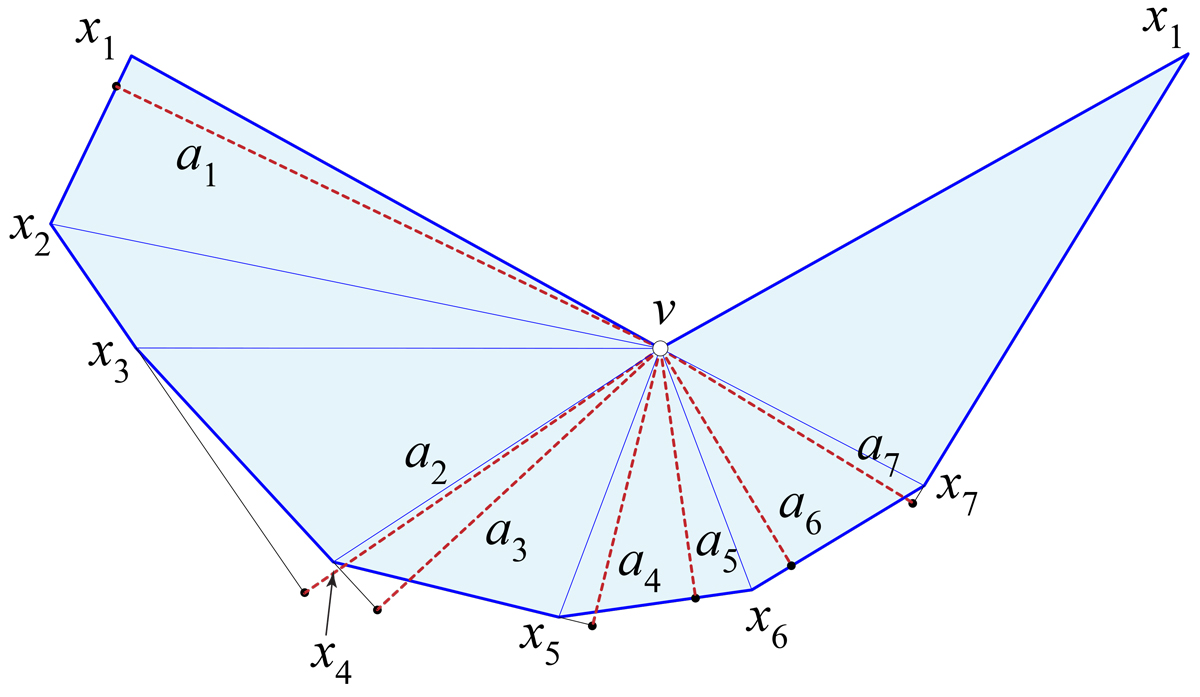}
\caption{
A convex chain $x_1,\ldots,x_7$ and the corresponding altitudes $a_1,\ldots,a_7$.
}
\figlab{AltTurns}
\end{figure}

The consequence of Lemma~\lemref{AltTurns} is that the surface angle $\nu$ around $v$
is partitioned by the altitudes $a_i$ in order, because $\sum_i \t_i = \nu$, and
the angle between $a_i$ and $a_{i+1}$ is $\t_i$.
Moreover, Lemma~\lemref{Altitude} shows that the apexes of each lifted triangle
$\bar{T}^L_i$ lie on those altitudes, at some positive distance from $v$.
Consequently, we can connect those apexes to form a simple geodesic polygon
enclosing $v$ on $L$.
Because every turn angle $\t_i$ is strictly less than $\pi$,
connecting two adjacent apexes along $a_i$ and $a_{i+1}$ 
will keep $v$ to the same (counterclockwise) side. 
Call this polygon the \emph{moat} $M$ of $P$.\footnote{
We do not know whether $M$ is always convex, but we only need it to be simple.}
Fig.~\figref{MoatV_1} illustrates the moat for the 
example in Fig.~\figref{PyHexV_1}(c).
\begin{figure}
\centering
 \includegraphics[width=0.8\textwidth]{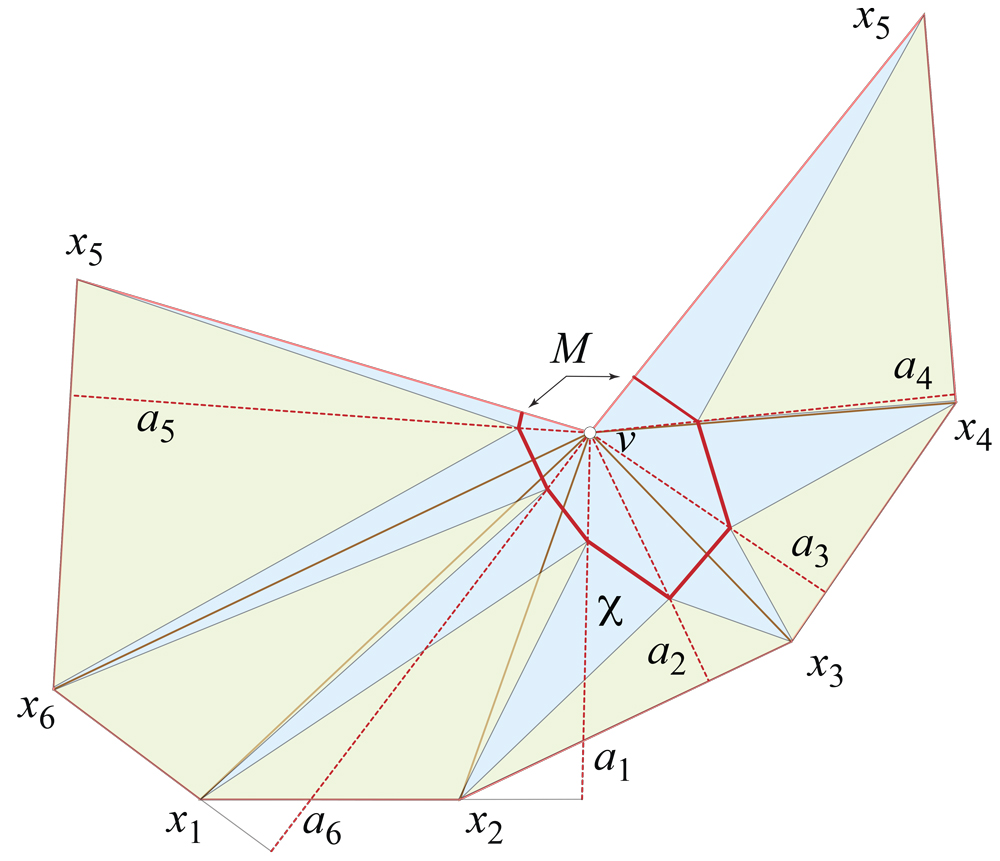}
\caption{The layout from Fig.~\protect\figref{PyHexV_1}(c)
shown with  moat $M$ and altitudes $a_i$ identified.
}
\figlab{MoatV_1}
\end{figure}

\begin{lm}
\lemlab{Lifting}
For the case $\bar{v} \in X$, the lifting of all triangles $\bar{T}_i$ to $\bar{T}^L_i$ onto $L$ has the following properties,
(where we shorten ``geodesic triangle" to ``triangle"):
\begin{enumerate}[label=(\alph*)]
\item Each lifted triangle $\bar{T}^L_i$ fits on $L$: $\bar{T}^L_i \subset L$.
\item $v$ does not lie in any triangle $\bar{T}^L_i$.
\item No lifted triangle self-overlaps, and no pair of triangles overlap.
\end{enumerate}
\end{lm}

\begin{proof}
\begin{enumerate}[label=(\alph*)]
\item
Because the apex of the lifted $\bar{T}^L_i$ is on the moat $M$ which surrounds $v$, $\bar{T}^L_i$ remains on the portion of $L$ outside the moat.
\item Therefore no $\bar{T}^L_i$ includes $v$.
\item If we view the overlay of $\bar{L}$ with the opening of $\partial X$ by the angle $\q_i-\bar{\q}_i$ at each $x_i$ image, as in Fig.~\figref{PyHexV_1}(c),
then Cauchy's Arm Lemma shows that two lifted triangles cannot overlap.
Suppose $\bar{T}^L_i$ and $\bar{T}^L_j$ overlap, $i < j$.
Then we can identify two points $p_i \in \bar{T}^L_i$ and $p_j \in \bar{T}^L_y$ that coincide in the layout.
But in $X$, $p_i$ and $p_j$ were separated by a positive distance $d=|p_i p_j|$.
In $X$, draw a convex chain from $p_i$ to $\partial X$, around that boundary, to $p_j$.
The layout opens this chain by the positive angles $\q_i-\bar{\q}_i$, and so in the layout, $p_i$ and $p_j$ must be separated further than $d$, a contradiction.
\end{enumerate}
\end{proof}

\noindent
Lemma~\lemref{Lifting} shows that $\chi$, the region of $L$ not covered by the lifted triangles,
is indeed a crest.

\medskip

We now turn to the case $\bar{v} \not\in X$.
The difficulty here is that $\partial L$ in a layout $\bar{L}$ of $L$ may not be a convex chain,
and Lemma~\lemref{AltTurns} relies on convexity for the altitudes to connect to $v$
in the same order as the vertices around $X$.
Indeed if $v$ were closer the plane of $X$ in the example in Fig.~\figref{PyHexV_2}(a),
then the angle at $x_3$ would be reflex. In general, a contiguous portion of $\partial L$
could be reflex. Lifting triangles incident to that reflex chain could lead to overlap,
violating~(c) of Lemma~\lemref{Lifting}.

However, as described earlier in Fig.~\figref{PyHexV_2}(c), the crest is formed by clipping the
triangles $\bar{T}_i$ to $X$.
Triangles $\bar{T}_3=x_3 \bar{v} x_4$ and $\bar{T}_4=x_4 \bar{v} x_5$ in Fig.~\figref{PyHexV_2}(b)
fall entirely outside $X$, and so play no role.
The convex portion of $\partial L$ still satisfies Lemma~\lemref{AltTurns},
so the corresponding altitudes are incident to $v$ in the same order as the vertices
along the convex chain.
This allows us to define a partial moat $M$, and then close it off to a simple polygon by a geodesic  path surrounding $v$.
This is illustrated in Fig.~\figref{MoatV_2}.

\begin{figure}
\centering
 \includegraphics[width=0.6\textwidth]{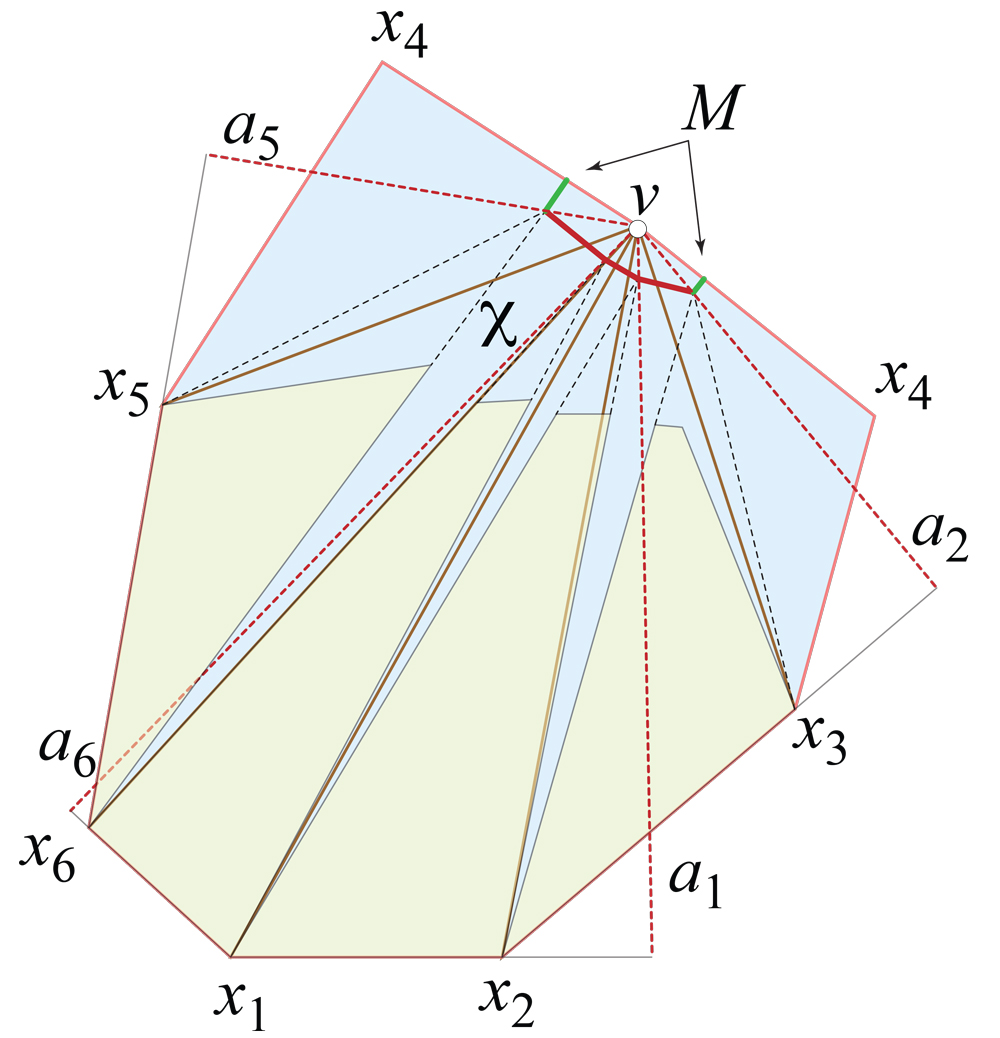}
\caption{The layout from Fig.~\protect\figref{PyHexV_2}(c)
shown with moat $M$ and altitudes $a_i$ marked.
Note $a_3$ and $a_4$ are missing because $T_3$ and $T_4$ are clipped
as outside $X$.
The green edges mark the closing of the partial moat around $v$.
}
\figlab{MoatV_2}
\end{figure}

This renders Lemma~\lemref{Lifting} true for the lifted triangles along the convex chain of $\partial L$, which are the only ones not clipped entirely away.
We summarize in a theorem:

\begin{thm}
\thmlab{thmCrest}
A crest $\chi$ can be constructed as the portion of $L$ not covered by lifted triangles in the case of $\bar{v} \in X$, 
and clipped lifted triangles in the case $\bar{v} \not\in X$, as described above.
\end{thm}

\noindent
We remark that the same procedure will work for other points $w \in X$ within some neighborhood of $\bar{v} \in X$, resulting in different crests.
However, a simple example shows that not every point $w \in X$ will produce a crest.
Consider $X=x_1 x_2 x_3$ an equilateral triangle of center $o$, and $v$ close to $o$, with $vo$ orthogonal to $X$.
Take $w \in o x_3$ close to $x_3$. Then the isosceles triangle $w x_1 x_2$ is  larger than the isosceles triangle
$v x_1 x_2$, so no congruent copy of the former can fit inside $L$ without encompassing $v$ and so self-overlapping.

We have this as a counterpart to Theorem~\thmref{MainTailoring}:

\begin{thm}
\thmlab{CrestTailoring}
For any convex polyhedra $P$ and $Q$, one can crest-tailor $P$ to any homothetic copy of $Q$ inside $P$, in time $O(n^4)$, where $n=\max\{ |P|, |Q| \}$.
\end{thm}

\begin{proof}
The lemmas leading to Theorem~\thmref{MainTailoring} 
established that ultimately we need to tailor single vertex truncations, i.e., tailor pyramids.
So the claim follows from Theorem~\thmref{thmCrest}.
\end{proof}


\section{Algorithm~4: pyramid $\to$ crest}

\begin{algorithm}[htbp]
\caption{Construct a crest $\chi$ that reduces a pyramid $P$.}
\DontPrintSemicolon

    \SetKwInOut{Input}{Input}
    \SetKwInOut{Output}{Output}

    \Input{A pyramid $P=L \cup X$ of $O(n)$ vertices}
    \Output{Crest $\chi$ whose removal flattens $P$ to $X$. }

    \BlankLine

\tcp{Assume apex $v$ degree-$k$, with $k=O(n)$.}

\tcp{See Figs.~\figref{PyHexV_1} and~\figref{PyHexV_2}.}

\For{each of $\bar{T}_i$, $i=1,2,\ldots,k$} {

	Compute where $\bar{T}_i$ edges (geodesics) cross $x_j v$. \tcp{$O(n)$.}

	Clip $\bar{T}_i$ to $X$.   \tcp{$O(n)$.}

}

\KwResult{Crest $\chi$.}
     
\end{algorithm}

The total cost of computing one crest $\chi$ is $O(n^2)$,
and we believe there are examples with total combinatorial complexity 
(number of geodesic/edge intersections) of $\Omega(n^2)$.
Because this
is the same complexity for reducing $P$ to $X$ via digon-tailoring described in Chapter~\chapref{TailoringAlg1}, 
the total time complexity is the same as in Theorem~\thmref{SliceAlgorithm}.

We are assuming that the combinatorial complexity of the crest $\chi$ on $L$
determines the time complexity of computing the crest.
However, a crest flattened to the plane, not overlaid on $\bar{L}$, has
combinatorial complexity $O(n)$---just $k=O(n)$ 
(possibly clipped) triangles---and can be constructed in $O(n)$ time.
Perhaps an 
implicit representation of shortest paths could suffice for subsequent calculations,
as they do in the optimal algorithm for shortest paths on a convex polyhedron~\cite{SchreiberSharir}.
Thus it may be that constructing 
an implicit representation of $\chi$ could lead to a lower time-complexity,
a question we leave for future work.
%

\chapter{Tailoring via Flattening}
\chaplab{TailoringFlattening}

In this chapter, we prove a completely different method for tailoring $P$ to $Q$.
The method mixes digon-tailoring steps with vertex-merge steps (Section~\secref{VertexMerging}).
The result is slightly weaker than either tailoring via sculpting
(Theorem~\thmref{MainTailoring})
or crest tailoring (Theorem~\thmref{CrestTailoring}),
weaker in the sense that the homothet of $Q$ obtained could be arbitrarily small.
Nevertheless, the proof and algorithm have the advantage of operating
entirely intrinsically: the 3D structure of $P$ and $Q$ is never invoked.

The proof depends on the observation that if both $P$ and $Q$ are 
doubly-covered polygons, it is easy to tailor one to the other:
Scale $Q$ to fit in $P$, and then cut the outline of $Q$ from $P$.
This can be accomplished by a series of digon-tailorings,
with each digon bounded by congruent segments on both sides,
or by truncations of several vertices at once.

At a high level, $Q$ is reduced to a flat polygon $Q_\textrm{flat}$, $P$ is reduced to a flat polygon $P_\textrm{flat}$ and tailored to match $Q_\textrm{flat}$. 
Finally, the steps used to reduce $Q$ are reversed and applied to the flat remnant of $P$.

In more detail, 
the proof (and algorithm that follows the proof) can be summarized
in these steps:
\begin{enumerate}[label={(\arabic*)}]
\item Reduce $P$ to $P_\textrm{flat}$ by a series of digon tailorings.
\item Reduce $Q$ to $Q_\textrm{flat}$ by a series of vertex-mergings.
$Q_\textrm{flat}$ is then composed of all of $Q$'s surface, plus surface-inserts
from the merges.
\item Scale $Q_\textrm{flat}$ to $Q^s_\textrm{flat}$ so that 
$Q^s_\textrm{flat}$ fits inside $P_\textrm{flat}$.
\item Trim $P_\textrm{flat}$ to match $Q^s_\textrm{flat}$ via digon-tailoring.
Call the result $P^s_\textrm{flat}$.
\item Reverse the  vertex-merging steps that reduced $Q$ to 
$Q_\textrm{flat}$, but now applied to $P^s_\textrm{flat}$.
Each reversal step removes a surface-insert via digon-tailoring.
\end{enumerate}
The end result, call it $Q^t$, is a polyhedron homothetic to $Q$, but composed
entirely of $P$-surface.

\medskip

We now prove these steps, tracking an example that
reduces a cube $P$ to the $5$-vertex polyhedron $Q$
in Fig.~\figref{VertMerge}(a), repeated in Fig.~\figref{HexaFlatten}(a).


\section{Proofs}

We follow the numbered outline above.

\subsection{Digon-tailor $P \to P_\textrm{flat}$}
Recall that Lemma~\lemref{Path} showed that,
if the cut locus $C(x)$ is a path, 
then the polyhedron is a doubly-covered convex polygon.
We use this lemma to reduce $P$ to $P_\textrm{flat}$.

We will use the cube example from Fig.~\figref{StarUnfCube} to illustrate the steps.
Assume $P$ is non-degenerate, i.e., not flat.
Let $x \in P$ be a point joined by unique geodesic segments to all vertices of $P$, 
and let $\rho$ be the unique path in $C(x)$ joining a pair of leaves of $C(x)$, i.e., joining the 
vertices $v_i$ and $v_j$ of $P$.

Then $C(x) \setminus \rho$ is a finite set of trees $T_k$.
Cut off from $P$ each $T_k$ by excising digons with one endpoint at $x$, and the other endpoint where $T_k$ joins $\rho$.
In Fig.~\figref{CutLocusPath}(a), $\rho$ connects $v_5$ and $v_7$, and separates four trees $T_i$.
After sealing each digon closed, we are left with a polyhedron $P_\textrm{flat}$ 
whose cut locus from the point corresponding to $x$ is precisely $\rho$ (by  Lemma~\lemref{Path}).

\begin{figure}[htbp]
\includegraphics[width=0.9\textwidth]{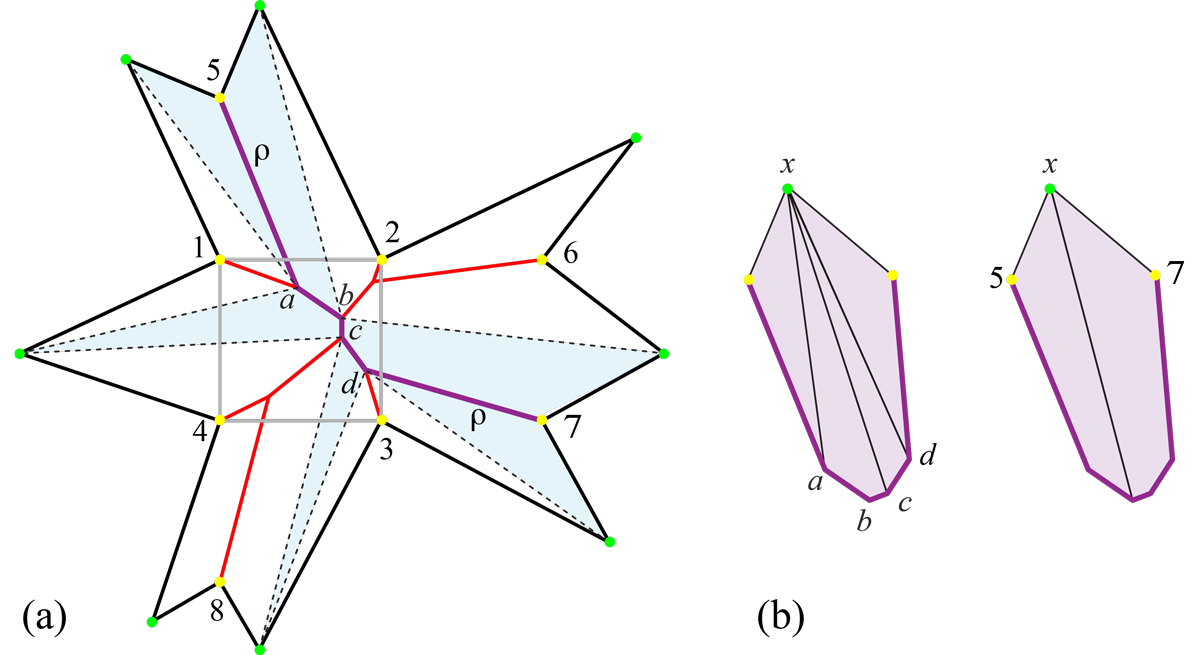}
\caption{Star unfolding of the cube in Fig.~\protect\figref{StarUnfCube}.
(a)~The path $\rho$ from $v_5$ to $v_7$ leaves four trees when removed from
$C(x)$. Excising four (white) digons leaves a surface (blue), which when
zipped closed folds to a doubly-covered $7$-gon, both sides of which
are shown in~(b).}
\figlab{CutLocusPath}
\end{figure}

If the path $\rho$ in the above proof is chosen to be as long as possible,
then $P_\textrm{flat}$ has larger surface area than if $\rho$ is short.

\subsection{Vertex-merge $Q \to Q_\textrm{flat}$}
Recall from Section~\secref{VertexMerging} that vertex-merging is
in a sense the inverse of digon-tailoring.
Two vertices are merged along a geodesic $\g$ and additional surface
in the form of two congruent triangles is sutured-in along $\g$.
Lemma~\lemref{IsoTetra} showed that 
every convex polyhedron $Q$ has at least one pair of vertices that
can be merged, unless $Q$ is an isosceles tetrahedron or a doubly-covered triangle.
So repeatedly vertex-merge $Q$ until it becomes flat, or is reduced 
to an isosceles tetrahedron.

In our example, 
Fig.~\figref{HexaFlatten}(a) shows $Q$, 
the same polyhedron in Fig.~\figref{VertMerge}(a).
The first vertex-merge leads to a regular tetrahedron, (c) in the figure.
As we discussed in 
Example~\exref{iso_tetra}
and illustrated in Fig.~\figref{IsoscTetra},
an isosceles tetrahedron can be reduced to a doubly-covered rectangle
by cutting an edge and regluing that edge differently.
Here we cut the edge $cd$ of the tetrahedron and reglue it by
creasing at the midpoints $y_1,y_2$ of the $cd$ 
slit, leading to the rectangle shown in~(d).

\begin{figure}[htbp]
\centering
\includegraphics[width=1.0\linewidth]{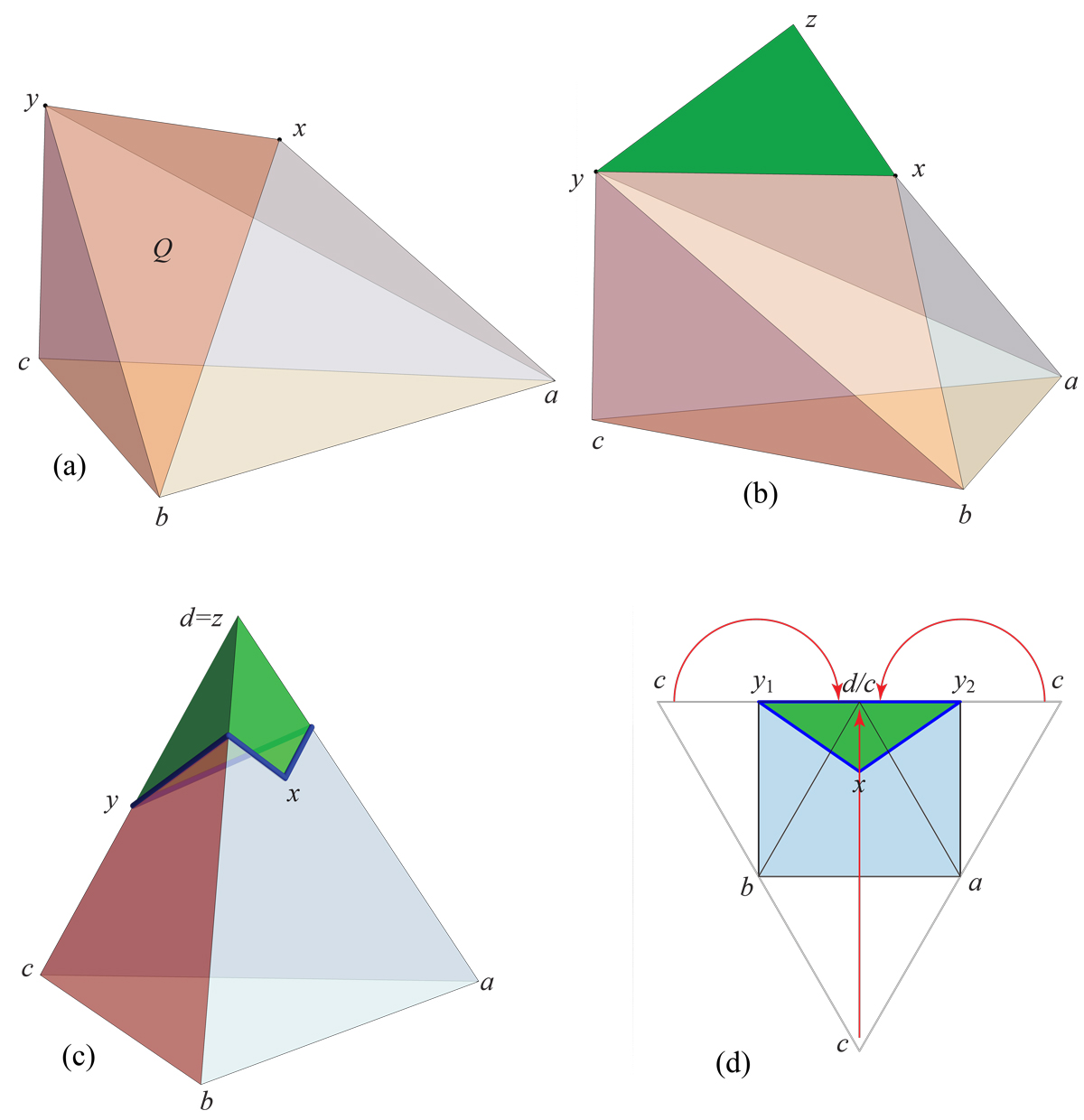}
\caption{(a,b)~Vertex-merge of $xy$.
(Repeat of Fig.~\protect\figref{VertMerge}.)
Green indicates surface-inserts.
(c)~After first vertex-merge.
(d)~After second special-case vertex-merge: doubly-covered square.
(Images not to same scale.)}
\figlab{HexaFlatten}
\end{figure}

\subsection{Scale $Q_\textrm{flat} \to Q^s_\textrm{flat}$}

The flattening of $P$ has reduced it in size in the sense that a portion of $P$'s
surface area has been excised,
while the flattening of $Q$ has augmented it by surface insertions and so has increased $Q$'s
surface area.
We next select a scale factor $s>0$ so that  $Q^s_\textrm{flat}$ can fit inside $P_\textrm{flat}$.
Clearly this is always possible.
In our example, the scaling might result in Fig.~\figref{ScaleQtoP}(a).
Note that we should view the top side of  $Q^s_\textrm{flat}$ on the top side
of $P_\textrm{flat}$, and similarly for the bottom sides.
\begin{figure}[htbp]
\centering
\includegraphics[width=1.0\linewidth]{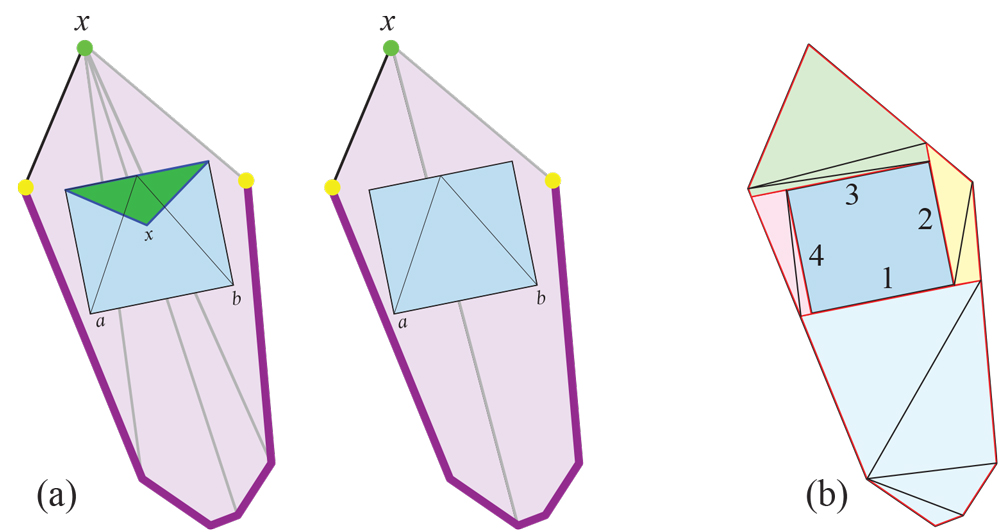}
\caption{(a)~$Q$ scaled to $Q^s_\textrm{flat}$ to fit in $P_\textrm{flat}$.
(b)~Trimming digon excisions by extending the edges of $P_\textrm{flat}$
in the order indicated.}
\figlab{ScaleQtoP}
\end{figure}

\subsection{Trim $P_\textrm{flat} \to P^s_\textrm{flat}$}
\seclab{TrimP}

As mentioned earlier, it is easy to tailor a doubly-covered convex polygon
$P_\textrm{flat}$ to match a doubly-covered polygon 
$Q^s_\textrm{flat}$ that fits inside $P_\textrm{flat}$.
This is in some sense similar to sculpting $P_\textrm{flat} \to P^s_\textrm{flat}$.
The differences are that (1)~here we only work in two dimensions, 
and (2)~this 2D sculpting is an intrinsic operation as well.
Fig.~\figref{ScaleQtoP}(b) shows one possible sequence of digon tailorings
for our example. 
We extend edge $e$ of 
$P^s_\textrm{flat}$ to a line $l_e$ which intersects $P_\textrm{flat}$, 
thus separating out a doubly-covered convex polygon, and repeat this for all edges $e$. 
Then each of those convex polygons is tailored one vertex
at a time, say tracking a triangulation.

\subsection{Reverse $P^s_\textrm{flat} \to Q^t$}

Finally we reverse the vertex-merge steps that produced $Q_\textrm{flat}$,
but applied to $P^s_\textrm{flat}$.
In our example, the last vertex-merge was the special-case step that
produced the rectangle in Fig.~\figref{HexaFlatten}(d).
The reverse step cuts the top edge $y_1 y_2$ between the two images of $y$,
and re-joins $y_1$ to $y_2$, leading to the
regular tetrahedron in (c)~of the figure.
Reversing the first vertex-merge applied to $Q$ excises the
inserted surface, green in (c), resulting in $Q^t$, 
a polyhedron homothetic to $Q$ 
(Fig.~\figref{HexaFlatten}(a) and~\figref{VertMerge})
but composed entirely of $P$-surface.

\subsection{Theorem: Tailoring via Flattening}
We have established this theorem:

\begin{thm}
\thmlab{TailoringFlattening}
For any given convex polyhedra $P$ and $Q$, one can tailor $P$ 
``via flattening'' so that it becomes homothetic to $Q$.
\end{thm}

\begin{rmk}
\rmklab{TailoringFlattening}
The result $Q^t$ obtained by Theorem~\thmref{TailoringFlattening} 
may be arbitrarily small compared to $Q$.


\begin{enumerate}
\item The example of a regular pyramid shows that the area of $P_\textrm{flat}$ may be as small as $2/n$ of the original area of $P$.

\item The ratio between the area of $Q_\textrm{flat}$ and the area of $Q$ can be arbitrarily large.

See Fig.~\figref{TrapezoidArea}.
\begin{figure}[htbp]
\centering
\includegraphics[width=0.4\linewidth]{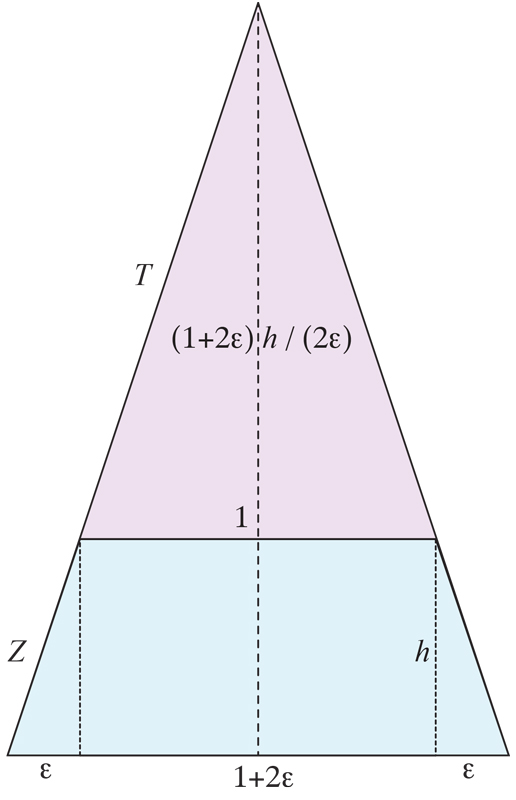}
\caption{The ratio of the areas of $T$ and $Z$ is arbitrarily large.}
\figlab{TrapezoidArea}
\end{figure}

\rm{To see this, consider an isosceles trapezoid $Z$ of base lengths $1$ and $1+ 2 \varepsilon$, and height $h$.
Its area is $(1+\varepsilon)h$.
Also consider the isosceles triangle $T$ obtained from $Z$ by extending its non-parallel sides until intersecting.
An elementary geometry argument provides the height of $T$, $(1+ 2 \varepsilon)h (2 \varepsilon)^{-1}$, and the area of $T$, $(1+ 2 \varepsilon)^2 h (4 \varepsilon)^{-1}$.

The ratio between the area of the doubles $Q_\textrm{flat}$ of $T$, and $Q$ of $Z$, is therefore 
$$
\frac{(1+ 2 \varepsilon)^2}{4  \varepsilon (1+ \varepsilon)} = 1+ \frac1{4  \varepsilon (1+ \varepsilon)}
$$ 
and can be arbitrarily large for $\varepsilon$ arbitrarily small.}
\end{enumerate}
The combination of $P_\textrm{flat}$ small and $Q_\textrm{flat}$ large leads to the arbitrarily-small claim for $Q^t$ with respect to $Q$.
\end{rmk}

\section{Algorithm for Tailoring via Flattening}
\seclab{TailoringAlg3}


\begin{algorithm}[htbp]
\caption{Tailor $P$ to $Q$ via Flattening.}
\DontPrintSemicolon

    \SetKwInOut{Input}{Input}
    \SetKwInOut{Output}{Output}

    \Input{Convex polyhedra $P$ and target $Q$}
    \Output{A tailored version of $P$ homothetic to $Q$}
    
    \BlankLine
    \tcp{(1)~Reduce $P$ to $P_\textrm{flat}$.}
    {Find generic point $x$.} \tcp{$O(n^4)$}
    {Compute cut locus $C(x)$.}
    {Select path $\rho$.}
    {Digon removals of trees until $P_\textrm{flat}$ attained.}
     \BlankLine

    \tcp{(2)~Vertex merge on $Q$ repeatedly.} 
     \While{$Q \neq T_\textrm{isos}$ and $|Q| > 3$}{
    
      {Identify two vertices $v_i$ and $v_j$ such that $\o_i+\o_j < 2 \pi$.} \tcp{$O(n^2)$}
      
      {Vertex merge $v_i$ and $v_j$, reducing $Q$ by one vertex to $Q'$.}
      
      {$Q \leftarrow Q'$.}
        
      }
      \BlankLine
      
      \If(\tcp*[h]{$Q$ isosceles tetrahedron}){$Q = T_\textrm{isos}$}{
      {Special tailor $Q$ to flat rectangle $Q_\textrm{flat}$.}
      }
     \BlankLine
     
     \tcp{(3)~Scale $Q_\textrm{flat}$ to $Q^s_\textrm{flat}$.}
     {Find largest inscribed and smallest circumscribed circles.} \tcp{$O(n)$.}
     {Scale $Q_\textrm{flat}$ by radii ratio.}
     \BlankLine
     
     \tcp{(4)~Trim $P_\textrm{flat}$ to $P^s_\textrm{flat}$.}
     {Extend edges of $P^s_\textrm{flat}$, triangulate each cut-off piece.} \tcp{ $O(n^2)$.}
     \BlankLine
     
     \tcp{(5)~Reverse steps to reduce $Q$, each applied to $P^s_\textrm{flat}$} 
          
     \ForEach{vertex-merging step applied to $Q$}
     {Reverse the step by cutting off the merge vertex.}

     \BlankLine
       
    \KwResult{A 3D polyhedron $Q^t$ homothetic to $Q$.}
     
\end{algorithm}

\noindent
In this section, we follow the proof of Theorem~\thmref{TailoringFlattening} and convert it to a polynomial-time algorithm.

As usual, let $n=\max \{ |P|, |Q| \}$ be the combinatorial size of the polyhedra.
We now establish an upper bound of $O(n^4)$ on the complexity of implementing
the algorithm.

Step~(1)
is to tailor $P$ to $P_\textrm{flat}$ using the cut locus $C(x)$ from a ``generic point" $x$, i.e., 
one with a unique shortest path to each vertex of $P$.
Although it is possible the need for uniqueness could be avoided, we
leave that future work.
We know of no way to find a generic $x$ short of computing all the ``ridge-free'' regions on $P$, which takes $O(n^4)$ time~\cite{aaos-supa-97}.
Independent of our work here, it is a interesting question if a generic $x$ can
be computed more quickly. We will see this $O(n^4)$ dominates the complexity of
the other calculations.

The star-unfolding $S_P(x)$ can be computed in $O(n \log n)$ time using the complex Schreiber-Sharir algorithm~\cite{SchreiberSharir}, 
or in $O(n^2)$ time with the Chen-Han algorithm~\cite{ChenHan1,ChenHan2}.
$S_P(x)$ only needs to be computed once.
With $S_P(x)$ computed, the cut locus $C(x)$ can be found from the Voronoi
diagram of the images of $x$.

Step~(2) is to repeatedly apply vertex-merging to $Q$ until it is
reduced to $Q_\textrm{flat}$, when $|Q_\textrm{flat}| \in \{3,4\}$.
Identifying two vertices $v_i$ and $v_j$ such that $\o_i+\o_j < 2 \pi$
can be achieved in $O(n \log n)$ time just by sorting the curvatures
$\o_i$ and selecting the two smallest. 
From the initial sorting onward, only $O(\log n)$ would be needed to update
the list, but we'll see this efficiency is not necessary.

With $v_i$ and $v_j$ selected, the shortest path $\g$ between
them needs to be computed. Although there is a complicated
optimal $O(n \log n)$ algorithm for computing shortest paths on a convex
polyhedron~\cite{SchreiberSharir}, that algorithm exploits the three-dimensional
structure of the polyhedron, which will not be available to us after the 
first vertex-merge. 
As mentioned earlier, there is no effective
procedure known to construct the polyhedron guaranteed by AGT.
However, we know the intrinsic structure of the polyhedron: its vertices, their
curvatures, a triangulation.
The algorithm of Chen and Han~\cite{ChenHan1,ChenHan2} can compute
shortest paths from this intrinsic data in $O(n^2)$ time.
Repeating this $n$ times to reach $Q_\textrm{flat}$ then
can be achieved in $O(n^3)$ time.

The scaling step~(3) can be accomplished in linear time, $O(n)$, as follows.
The largest circle inscribed in 
$P_\textrm{flat}$ is computed by the linear-time medial axis algorithm~\cite{chin1999finding};
say its radius is $r_P$.
The smallest circle circumscribing
$Q_\textrm{flat}$ is found in linear-time via Megiddo's algorithm~\cite{m-ltalp-83};
say its radius is $r_Q$.
Then scale $Q_\textrm{flat}$ by $s = r_P / r_Q$.

Trimming $P_\textrm{flat}$ to $P^s_\textrm{flat}$, step~(4), 
can be accomplished in many ways.
The method we described in Section~\secref{TrimP} can
easily be implemented in $O(n^2)$ time
by ray-shooting the edge extensions, and then triangulating
each convex polygon in linear time.
Likely the ray-shooting could be reduced to $O(n \log n)$ time.

Reversing the $Q$ vertex-merging steps, step~(5), amounts to digon tailorings cutting of the merged vertices on $P$.
This can easily be accomplished in $O(n \log n)$ time.
Keeping track of the considered vertices and employed digons gives in the end a correspondence between $Q^t$ and $Q$, and thus the 3D structure of $Q^t$.

So the whole algorithm time-complexity is dominated by the $O(n^4)$ cost of finding a guaranteed generic $x$.

\medskip

Because the ridge-free regions are determined by overlaying $n$ cut loci, the regions are delimited by a one-dimensional network of segments. 
Thus choosing a random point $x$ on $P$ is generic with probability $1$.
That still leaves the algorithm requiring $O(n^3)$ time.
We believe this time complexity could be improved, perhaps to $O(n^2)$. 
See Open Problem~\openref{CutLocusGeneric}.

We repeat Theorem~\thmref{TailoringFlattening} with the complexity bound included:
\begin{thm}
\thmlab{FlatteningAlg}
For any given convex polyhedra $P$ and $Q$, one can tailor $P$ until it becomes homothetic to $Q$ in time $O(n^4)$, where $n=\max\{ |P|, |Q| \}$.
\end{thm}


Finally, we note that
our model of reshaping in all cases excises just a single vertex via a tailoring step,
or inserts a single vertex via a vertex-merge step.
A rather different model, but related to the flattening algorithm above,
unfolds $P$ and $Q$ each to nets (non-overlapping planar polygons),
and places the unfolded $Q$ inside the unfolded $P$.
For example, Jin-ichi Itoh made the following
interesting suggestion:\footnote{
Personal communication, 2019.} 
first star-unfold $P$ to $S_P$ and $Q$ to $S_Q$, shrink $S_Q$ to fit inside $S_P$, 
then cut out $S_Q$ from $S_P$, and refold to obtain a homothet of $Q$ from $P$.

\chapter{Enlarging, $P$-Unfoldings, and Continuous Foldings}
\chaplab{Punfoldings}

In this chapter we show how to enlarge a convex polyhedron $Q$ to some $P \supset Q$, via tailoring. 
Based on this operation, we introduce the notion of a $P$-unfolding and propose a few methods to accomplish it.
Finally, we describe a method to continuously fold $P$ onto an enclosed $Q$.

\section{Enlarging and Reshaping}
Previous chapters have established three methods of tailoring $P$ to $Q$:
\begin{itemize}
\squeezelist
\item Theorem~\thmref{MainTailoring}: For $Q \subset P$, digon-tailor following a sculpting of $P$ to $Q$.
\item Theorem~\thmref{CrestTailoring}: For $Q \subset P$, crest-tailor following a sculpting of $P$ to $Q$.
\item Theorem~\thmref{TailoringFlattening}:  via flattening, digon-tailor $P$ to a (possibly small) homothet of $Q$.
\end{itemize}
If we do not have $Q \subset P$, then shrinking $Q$ until it can fit in $P$ leads
to a homothet of $Q$. 
All three approaches result in the homothet of $Q$ being composed entirely of $P$-surface.
In the tailoring-via-flattening algorithm, $Q$ is enlarged with vertex-merging inserts,
but the last steps remove all inserts.

\medskip

In this short chapter, we explore some results that can be achieved through a mix
of vertex-merge enlarging, and tailoring. Part~II will explore vertex-merging
more thoroughly.

Suppose $Q \subset P$. Then we can enlarge $Q$ to $P$
using any one of the three tailoring algorithms, as follows.
We first tailor $P$ to $Q$, tracking the cuts and digons 
removed.
For crest-tailoring, we cut the boundary of crests.
Then, starting with $Q$, we cut each sealed geodesic or crest-boundary
and insert the earlier removed corresponding digon 
or crest surface, in reverse order.
The result is that $Q$ is enlarged by the surface insertions 
(of $P$-surface)
in the reversing process
until it matches $P$.

If $Q \not\subset P$, so $P$ and $Q$ are of arbitrary relative sizes,
we can reshape $P$ to $Q$ by 
first enlarging $P$ to
some $P' \supset Q$ large enough to enclose $Q$, 
and then tailor $P'$ to $Q$.
Here $P'$ is an arbitrary polyhedron midway in the process
$P \to P' \to Q$; it only needs to be large enough.

These enlargings can be accomplished within the same $O(n^4)$ time complexity
of the tailoring algorithms.
And in analogy with Corollary~\lemref{approx_tailoring},
enlargings of $Q$ can approximate any target surface $S$.

We summarize the above discussion in a theorem:

\begin{thm}
\thmlab{Enlarge}
Let $P$ and $Q$ be convex polyhedra.
If $Q \subset P$, $Q$ may be enlarged with surface insertions to $P$.
If $Q \not\subset P$, $P$ can be reshaped to $Q$ with a combination of
surface insertions and excisions.
Either process can be accomplished in time $O(n^4)$,
and can approximate non-polyhedral convex surfaces.
\end{thm}


\section{$P$-unfoldings}
\seclab{Punfoldings}

For $Q \subset P$, we can view the enlarging of $Q$ to $P$ just described
as ``unfolding" $Q$ onto $P$. We call this a \emph{$P$-unfolding of $Q$}.
To our knowledge, this notion has 
not been considered previously.
We explore it  briefly here, and in more detail in Part~II.

\subsection{$P$-unfoldings and Reshaping}
\seclab{PunfReshaping}
Unfoldings of convex polyhedra to a plane have been studied extensively;
see, e.g., \cite{do-gfalop-07}.
Particular attention has been paid to unfolding to few pieces,
to connected unfoldings, and to non-overlapping unfoldings.
Often the goal is to achieve a \emph{net}, an unfolding to a single, simply connected, non-overlapping polygon in the plane.

Instead of unfolding a convex polyhedron $Q \subset P$ to a plane,
consider the question: 
Can one cut-up the surface $Q$ so that the pieces may be pasted onto $P$, non-overlapping,
and so form an isometric subset of $P$? 
This is a \emph{$P$-unfolding of $Q$}, or \emph{an unfolding of $Q$ onto $P$}.
It can be achieved via enlarging:
just enlarge $Q$ to $P$, and then remove all inserted digons or crests.
The result is a subset of $P$ isometric to the cut-up $Q$.
Note that all three enlarging methods retain the entire surface of $Q$---the enlarging
is accomplished by insertions of surface not part of $Q$.

This viewpoint incidentally yields a different proof of this known result:

\begin{co}
\thmlab{Areas}
Given convex polyhedra $P$ and $Q \subset P$, the area of $Q$ is smaller than the area of $P$.
\end{co}

%
%
%

To justify the use of the term ``unfolding," we describe an example of a simply connected $P$-unfolding embedding that does not achieve a net when unfolded to a plane.
Let $Q$ be the classical thin, nearly flat tetrahedron with an overlapping edge-unfolding.
See, e.g., \cite[Fig.~22.8, p.~314]{do-gfalop-07}.
Take $P$ to be a slightly larger homothet of $Q$.
Then the same edge-cuts that result in overlap in the plane embed $Q$ onto $P$ without overlap.
In a sense, it can be easier in some cases to unfold to $P$ than to unfold to a plane.

One focus of interest in the unfolding literature is the ``fewest nets" 
problem~\cite[Prob.~22.1, p.~308]{do-gfalop-07}: 
edge-unfold a polyhedron into the fewest number of nets.
Of course one net would be ideal, but it is a long-standing
open problem to determine if every convex polyhedron can be edge-unfolded
to a net---D{\"u}rer's problem~\cite{o-dp-13}.
This suggests it could be of interest 
to minimize the number of disconnected pieces of a $P$-unfolding
of $Q$, a goal we pursue in Part~II.
See also Open Problem~\openref{PunfConnected}.

\paragraph{Example: Hexagon inscribed in a triangle.}
Let $P$ be a doubly-covered triangle $abc$, and $Q$ a 
doubly-covered hexagon inscribed in $P$,
as in Fig.~\figref{HexInTri}.
Digon-tailoring $P$ to $Q$ is achieved by excising the vertices $a,b,c$ by
cutting digons bounded by the hexagon edges
$e_1,e_2,e_3$ respectively, and then sealing the cuts closed.
To enlarge $Q$ to $P$ by reversing the process,
each of those three edges is slit open and the pair of triangles
earlier removed is sutured back along the slit edges.
To obtain a $P$-unfolding, we simply skip inserting the pair of triangles.
The result is a ``pasting" of the two hexagons inside the two triangular faces of $P$.

Note that, to reverse a sealed digon, the resulting geodesic $\g$ is slit,
but the cut does not delete the endpoints $x_1,x_2$  of $\g$, as depicted in Fig.~\figref{DigonReverse}.
In other words, $\g \setminus \{x_1,x_2\}$ is doubled
to reproduce the digon boundary, but there remains one copy each of the $x_1,x_2$ endpoints.

Notice that each split $e_i$ is |blue{a} non-contractible closed curve inside the image of $Q$, so the embedding
of $Q$ onto $P$ in this case is not simply connected: it has three holes.

\begin{figure}[htbp]
\centering
\includegraphics[width=0.9\linewidth]{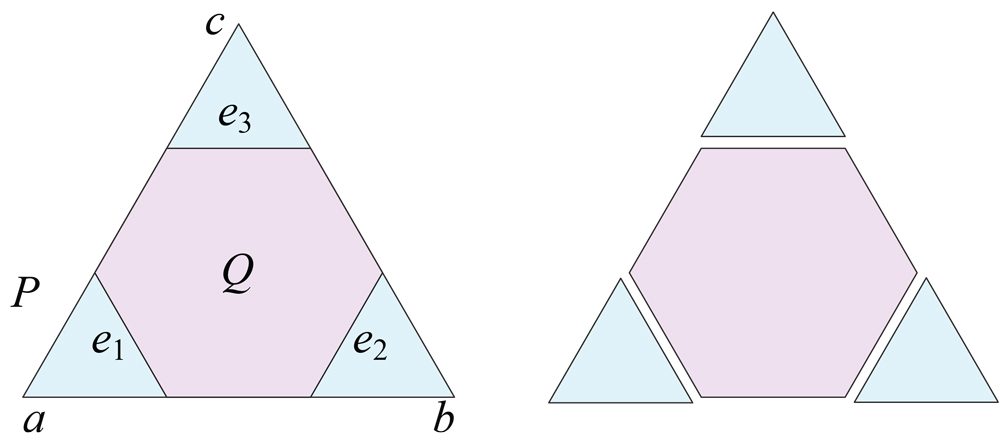}
\caption{Doubly-covered hexagon $Q$ inscribed inside / enlarged to a doubly-covered triangle $P$.}
\figlab{HexInTri}
\qquad 
\centering
\includegraphics[width=0.35\linewidth]{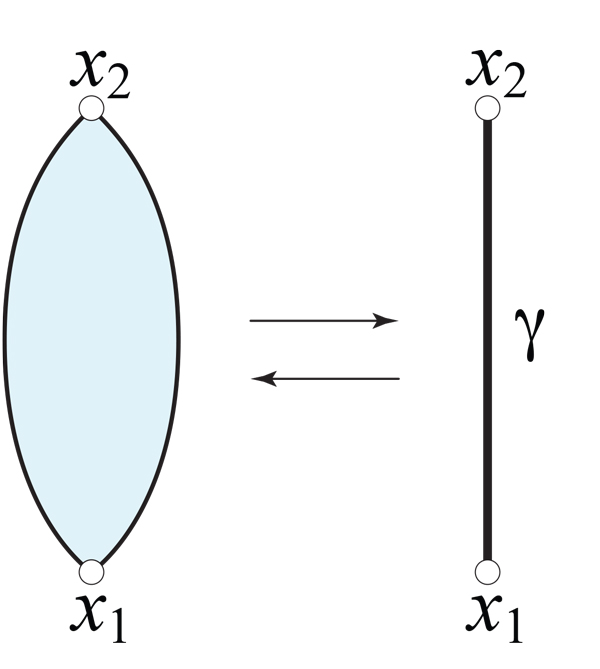}
\caption{Digon $x_1 x_2$ closes to geodesic $\g$.}
\figlab{DigonReverse}
\end{figure}

\medskip

The previous example shows
that the $P$-unfolding of $Q$ produced by enlarging is not necessarily simply connected.\footnote{
Notice that we do not require path-connectivity for the definition of simple connectivity.}
However, in general, that is indeed the case, as shown by the following result. 
We next explain the meaning of ``in general.''

\medskip

Consider the space ${\cal S}$ of all convex surfaces, endowed with the topology induced by the usual Pompeiu-Hausdorff metric.
Fix some $P \in {\cal S}$.
Consider in ${\cal S}$ the subset ${\cal P}={\cal P}_P^n$ of all polyhedra $Q\subset P$ with 
precisely $n$ vertices, with the induced topology.
Two polyhedra in ${\cal P}$ are then close to each other if and only if they have close 
corresponding vertices.
``General'' refers to polyhedra $Q$ in an open and dense subset of ${\cal P}$.

\begin{thm}
\thmlab{QPflat}
For any convex polyhedron $P$ and any $n \in \N$, there exists a subset ${\cal Q}={\cal Q}_P^n$ open and dense in ${\cal P}$, 
such that the $P$-unfolding $Q_P$ of each $Q \in {\cal Q}$ is flat (i.e., contains no internal vertices), and is simply connected.
\end{thm}

\begin{proof}
Assume we have some convex polyhedron $Q \subset P$ such that $Q_P$ contains an internal vertex $v$, and so the curvatures of $P$ and $Q$ at $v$ are equal: $\o_Q(v)=\o_P(v)$.
Slightly alter the position of the vertices of $Q$, to get $\o_Q(w) \neq \o_P(u)$, for any vertices $w \in Q$ and $u \in P$.
Of course, this remains valid in a small neighborhood of the new $Q$.

Assume now we have some convex polyhedron $Q \subset P$ such that $Q_P$ is not simply connected, 
i.e., $Q_P$ contains a noncontractible curve $\s \subset P \cap Q_P$.
The Gauss-Bonnet Theorem shows that the total curvature $\Omega_Q (\s)$ of $Q_P$ inside $\s$ equals the total curvature $\Omega_P (\s)$ of $P$ inside $\s$: $\Omega_Q (\s) = \Omega_P (\s)$. We next show that every such $Q$ that violates the theorem can be approximated with polyhedra that do satisfy the theorem.

Slightly alter the position of the vertices of $Q$, to get a new polyhedron $Q'$ on which the following property (V) is verified. 
(V): any partial sum of vertex curvatures is different from any partial sum of vertex curvatures in $P$. 
This implies that,  for any simple closed curve $\t$ on $Q'$, $\Omega_{Q'} (\t)$ cannot be written as the sum of vertex curvatures of $P$.
Therefore, $Q'_P$ has no curve in common with $P$, noncontractible in $Q'_P$. And so $Q'$ does satisfy the theorem.

Since the property (V) is valid on a neighborhood $N$ of $Q'$, it follows that all polyhedra in $N$ do satisfy the theorem.
\end{proof}


\subsection{$P$-unfoldings and the WBG Theorem}

In this section we show that the $P$-unfolding question can also be answered by applying the 
powerful hinged variant of Wallace-Bolyai-Gerwien dissection theorem.
However, this will result in a ``pseudopolynomial number of pieces" and 
pseudopolynomial running time~\cite{abbott2012hinged}.\footnote{
They define pseudopolynomial as follows:
``pseudopolynomial means polynomial in the combinatorial complexity ($n$) and the dimensions of an integer grid on which the input is drawn."}

The Wallace-Bolyai-Gerwien (WBG) Theorem states that any two simple polygons with equal area can be dissected into finitely many simple congruent polygons.
It was strengthened in 
the paper, ``Hinged Dissections Exist"~\cite{abbott2012hinged},
where it is shown that the pieces can be chosen in an arcwise-connected chain, 
i.e., the dissection can be hinged at vertices along that chain.

\begin{thm}
\thmlab{WBG_Punf}
Given convex polyhedra $P$ and $Q$ of at most $n$ vertices each, and $Q \subset P$, a connected $P$-unfolding of $Q$ can be determined in pseudopolynomial number of pieces  and pseudopolynomial running time, following the WBG theorem and the hinge-dissection.
\end{thm}

\begin{proof}
We start with a quote from~\cite{abbott2012hinged}:
\begin{quotation}
\noindent
``One interesting consequence of [our] theorem is that any finite set of polyhedral surfaces of equal surface area have a common hinged dissection: It is known that every polyhedral surface can be triangulated and then vertex-unfolded into a hinged chain of triangles~\cite{deeho-vusm-03}. Our results (specifically Theorem~6) show how to construct a single hinged chain that can fold into any finite set of such chains, which can then be folded (and glued) into the polyhedral surfaces.''
\end{quotation}

\noindent A \emph{vertex-unfolding} of a triangulated polyhedron
is a chain of triangle connected at vertices. A vertex-unfolding
of a triangulated cube is shown in Fig.~\figref{VertexUnf}.

Let $A_Q$ and $A_P$ be the areas of $Q$ and $P$.
Triangulate and vertex-unfold both $P$ and $Q$.
The plan is to show that $A_Q$ can ``fit" exactly within
a sequence of $P$'s triangles, with at most one of $P$'s triangles
partitioned into two triangles.

Denote by $T_1,...,T_m$ the triangles in the chain corresponding to $P$, and by  $a_j$ denote the area of the $T_j$, $j=1,...,m$.
So we have $A_Q<A_P= \sum_{j=1}^m a_j$.
Then there exists $1 \leq h \leq m$ such that $a_1+...+a_{h-1} \leq A_Q < a_1+...+ a_h$.
Triangle $T_h$ is then the transitional triangle with respect to area in the sequence
of triangles.

\begin{figure}[htbp]
\centering
\includegraphics[width=1.0\linewidth]{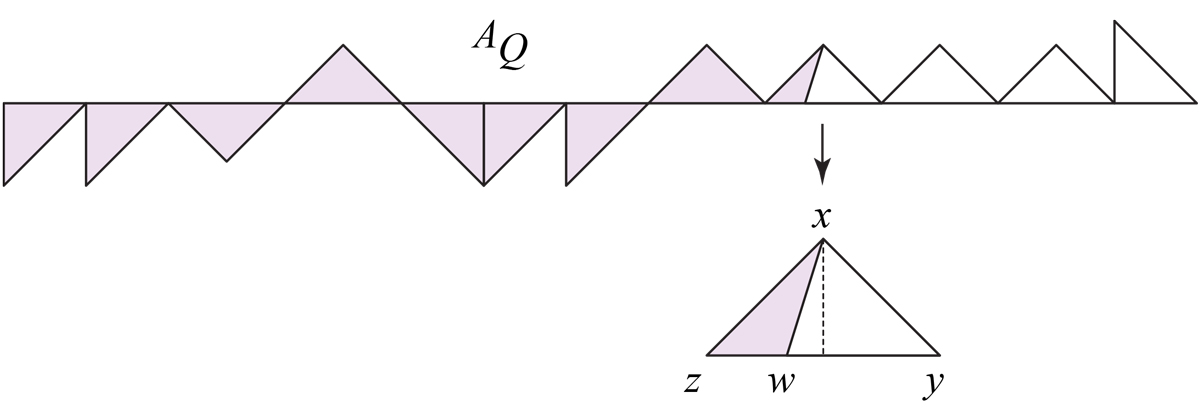}
\caption{Area $A_Q$ (shaded) in vertex-unfolding of a triangulated cube.
$T_h=xyz$.
Based on Fig.2(a) in~\protect\cite{deeho-vusm-03}.}
\figlab{VertexUnf}
\end{figure}

Now we further dissect the triangle $T_h$ into two parts, one of which has area $c=A_Q - (a_1+...+a_{h-1})$.
Specifically, let $T_h = xyz$ and choose a point $w \in yz$ such that $|w-y|/|z-y| = c/a_h$.
Then the area $A'$ of the triangle $xyw$ is equal to $c/a_h$ of the area of $T_h$ (because both triangles have the same altitude through $x$).

Therefore, in the chain of triangles obtained from $P$, we can find an arcwise-connected subchain of total area equal to 
$A_Q$, possibly after partitioning one triangle.

Then Theorem~6 in~\cite{abbott2012hinged} constructs a single
hinged dissection chain of the triangles forming $Q$
into the triangles up to $T_h$ of $P$.
Then the $Q$-chain triangles can be placed inside the $P$-chain triangles,
achieving a connected $P$-unfolding of $Q$.

The ``pseudopolynomial number of pieces" and pseudopolynomial running
time follows from~\cite{abbott2012hinged}.
\end{proof}
\noindent
Note that although the $P$-unfolding constructed in this theorem's proof
is connected, its interior is not connected.


\section{$P$ Continuously Folding onto $Q$}
\seclab{PfoldtoQ}
Cauchy's Rigidity Theorem shows that if one views the faces of a convex polyhedron
as rigid plates hinged along shared edges, then the polyhedron is rigid---it cannot flex.
For both convex and nonconvex polyhedra, a line of research has explored
\emph{continuous flattening} of a polyhedron to a plane~\cite[Chap.18]{do-gfalop-07}.
The idea is to allow ``moving" or ``rolling" creases, allowing the
faces of the polyhedron to bend, but not tear or pass through one another.
An early result was a method
to flatten each of the five Platonic solids~\cite{itoh2010continuous}.
This was followed by procedures for continuously flattening all convex 
polyhedra~\cite{inv-cfcp-12},~\cite{abel2014continuously}.
Most recently, continuous flattening of any polyhedron was achieved
by extending the model to permit infinite slicing~\cite{addklin-cf-2020}.

All of these results aim to fold the polyhedron to a ``flat folded state"---a non-self-crossing layering
of the surface on a plane.
Here we employ the tailoring/sculpting result (Theorem~\thmref{SliceAlgorithm}) to
address what appears to be a novel variation: continuously folding the surface of one 
convex polyhedron onto another.

In particular, we follow the slicing algorithm as summarized in Theorem~\thmref{SliceAlgorithm}:
Given convex polyhedra $P$ and $Q$ 
with $Q \subset P$,
$P$ can be tailored to $Q$, tracking a sculpting of $Q$ from $P$ as in Theorem~\thmref{MainTailoring}, 
via a polynomial number of tailoring steps.
This leads to the following theorem.
\begin{thm}
\thmlab{ContinuousFolding}
Given convex polyhedra $P$ and $Q$ 
with $Q \subset P$,
there exists a continuous folding of $P$ onto $Q$,
tracking a sculpting of $Q$ from $P$ as in Theorem~\thmref{MainTailoring}.
\end{thm}

Note we claim the existence of a continuous folding, but we do not know how to 
turn it into a finite algorithm to perform the folding;
see Open Problem~\openref{Continuous}.
Nevertheless, there are distinct steps.

Recall the tailoring-via-sculping algorithm (Theorem~\thmref{SliceAlgorithm})
ultimately identifies a polynomial number of tailoring steps (polynomial in the number of vertices $n$).
The folding in Theorem~\thmref{ContinuousFolding} follows the same sequence of digon
steps as the tailoring algorithm, but handling each digon differently.
From now on we concentrate on 
the tailoring of a single digon.

The high-level idea is as follows.
Instead of excising the digon and suturing closed its two sides,
we gradually bring together the two digon sides without removing the interior of the digon.
This process forms a growing doubly-covered triangle, which is folded onto the surface of $Q$.
Once the digon sides have been joined, we imagine the folded triangle as now melded into
$Q$'s surface. Subsequent digon processing may refold earlier foldings.

\begin{figure}[htbp]
\centering
\includegraphics[width=1.0\linewidth]{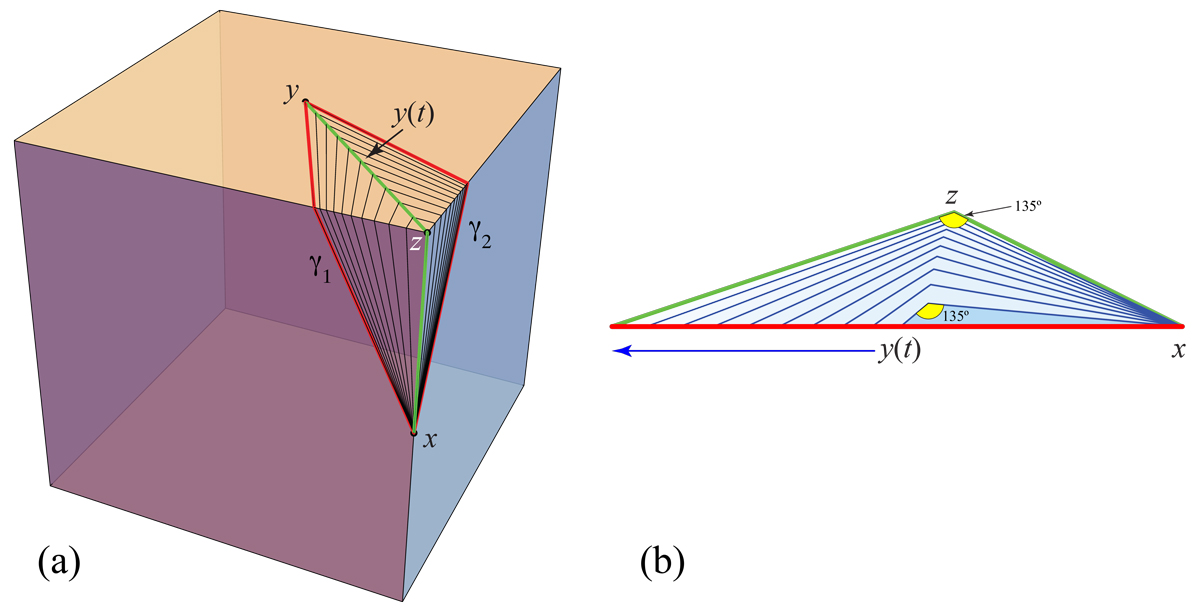}
\caption{(a)~Digon $xy$ surrounding vertex $z$. $y(t)$ moves along the geoseg
 $zy$.
(b)~Each $t$ leads to a doubly-covered triangle $T(t)$ with apex $z$ and base $x y(t)$.
The angle at $z$ is fixed to half of $3 \pi/2$ by the curvature of $\pi/2$ at vertex $z$.}
\figlab{CubeZipping}
\end{figure}

Recall from Chapter~\chapref{IntroductionPartI} that a digon $D(x,y)$ is a subset of $P$
bounded by two equal-length geodesic segments $\g_1$ and $\g_2$.
The tailoring algorithm only employs digons each surrounding one vertex.
If $z$ is the enclosed vertex, then $D(x,y)$ is intrinsically 
composed of two triangles,
one based on $\g_1$ and one of $\g_2$, both apexed at $z$.
When $\g_1$ is identified with $\g_2$, then the digon surface forms
a doubly-covered triangle $T = x z y$.

In the tailoring algorithm, $D(x,y)$ is excised, $\g_1$ is glued to $\g_2$,
and Alexandrov's Gluing Theorem~\thmref{AGT} 
is invoked to obtain a new convex polyhedron.
For continuous folding, we want to gradually close $\g_1$ and $\g_2$ until $T$ is obtained.
For this process, we fix $x$ and define a variable $y(t) \in yz$ with boundary conditions
$y(0)=z$ and $y(1)=y$.
See Fig.~\figref{CubeZipping}(a).
Define $\g_1(t)$ and $\g_2(t)$ to connect the fixed $x$ to the variable point $y(t)$.
As $t$ varies from $0$ to $1$, joining $\g_1(t)$ and $\g_2(t)$ lifts a triangle $T(t) = x z y(t)$
outside of the surface of $Q$.
We can imagine applying AGT at every infinitesimal advancement of $t \in (0,1]$.

The triangles $T(t)$ each have apex $z$ and a fixed angle at $z$ of $\frac{1}{2}(2 \pi - \o(z))$.
See Fig.~\figref{CubeZipping}(b).
The base $x y(t)$ of $T(t)$ lengthens with $t$, as does the altitude of $T(t)$.
When $t=1$, the digon is closed with $\g_1=\g_2$ and $T(1)=T$.

So far we have only observed that $T(t)$ sits outside of $Q$.
But at each $t$, we fold $t$ to one side or the other of the geodesic
segment $x y(t)$, i.e., $\g_1(t) = \g_2(t) = \g(t)$.
Because  $\g(t)$ is a geodesic segment, there is a neighborhood $N(t)$ of $\g(t)$ that is 
vertex-free,
in the sense that it contains no other vertex than $x$ and $y$.
See Fig.~\figref{AccordionFolding}(a).
If $T(t)$ fits in that neighborhood $N(t)$, then we fold $T(t)$ into $N(t)$ by creasing
along $\g(t)$. Because $N(t)$ is vertex-free,
$T(t)$ can flatten into $N(t)$, even if $\g(t)$ crosses edges of $Q$,
because $N(t)$ is isometric to a planar patch of surface, and $T(t)$ is itself planar.

If at some time $t$, $T(t)$ no longer fits into $N(t)$, then we ``accordion-fold"
the growing $T(t)$ so that it does fit.  As the altitude to $z$ grows with $t$,
the triangle might need to be folded back and forth,
as depicted in Fig.~\figref{AccordionFolding}(b).

\begin{figure}[htbp]
\centering
\includegraphics[width=0.8\linewidth]{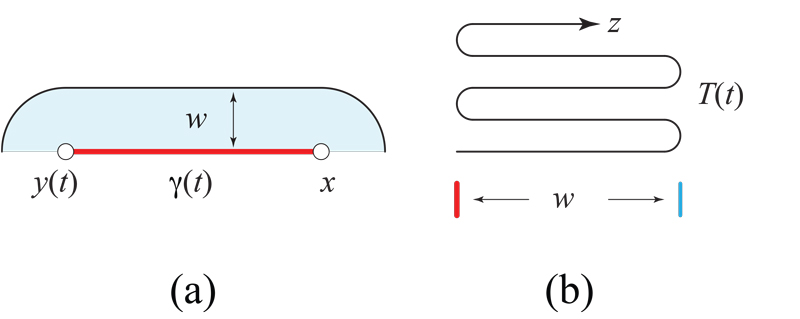}
\caption{(a)~Vertex-free region of width $w$ adjacent to geodesic $\g(t)$.
(b)~Side view of folding triangle $T(t)$ to fit inside the vertex-free region.}
\figlab{AccordionFolding}
\end{figure}

When $t$ reaches $1$, all of $T(1)=T$ is now folded onto $Q$, perhaps in several layers.
Now we imagine melding $T$ into $Q$'s surface, and we proceed with the next step of
the tailoring algorithm.
Note that a subsequent digon, say $x' y'$, might cross the now closed $xy$ seal.
Processing that digon will fold its doubly-covered triangle $T'$,
intruding underneath and into the earlier triangle $T$.
Although the subsequent foldings of $T$ caused by other digons
could be quite complicated, we continue to
view $T$ as merged with $Q$'s surface, and let the folding of each digon
proceed as described, possibly carrying along earlier doubly-covered triangle foldings.

After all digons have been processed in the order specified by the tailoring algorithm,
we have successfully continuously folded $P$ onto $Q$,
without ever any surface self-crossing.




\part{Vertex Merging and Convexity}

\chapter{Introduction to Part~II}
\chaplab{IntroII}
In the second part of our monograph we focus on
the operation of vertex-merging on a convex polyhedron $P$.
Used before as a proof technique by several authors, starting with A.D. Alexandrov, this operation 
has never before been considered, to our knowledge, as a main object of study.
The operation is defined in Chapter~\chapref{Preliminaries},
Section~\secref{VertexMerging}, 
and we employed it in Part~I, in Chapters~\chapref{TailoringFlattening} 
and~\chapref{Punfoldings}. 
We explained in Chapter~\chapref{TailoringFlattening} 
how vertex-merging is in some sense the inverse of digon-tailoring:
instead of cutting off a vertex $v$ by excising a digon, 
%
%
vertex-merging slits a geodesic
segment connecting two vertices $v_1$ and $v_2$,
and inserts a pair of triangles that flatten $v_1$ and $v_2$ and introduce a new
vertex $v$ with curvature $\o(v) = \o(v_1) + \o(v_2)$.
One might view this process as ``geometric surgery" in analogy with
topological surgery, which cuts out part of a manifold and replaces it with
part of another manifold, a technique pioneered by John Milnor
and further developed by Dennis Sullivan and others.

We called the geodesic segment that results from suturing-closed the digon boundary
a \emph{seal}. 
We will call the vertex-merge geodesic segment 
a \emph{slit}. 

A polyhedral surface $S$ 
that admits no vertex-merging is \emph{irreducible}.
By Lemma~\lemref{IsoTetra}, 
the irreducible polyhedra are isosceles tetrahedra and doubly-covered triangles.

Vertex-merging \emph{reductions}---the processes of repeatedly reducing the number of vertices via merging, 
from $P$ to an irreducible polyhedron---are
studied in Chapter~\chapref{VMSlitGraph} 
together with their induced \emph{slit graphs}. 
This can be seen as a theoretical exploration as well as
a refinement of a proof idea employed in 
Chapter~\chapref{TailoringFlattening}. 

Particularly interesting are the cases when slit graphs are forests of trees,
because then the unfolding (i.e., the image) $P_S$ of $P$ onto 
the irreducible surface $S$ has connected interior.
Even in this case, obtaining a net from $P_S$ is not always obvious, as shown in 
Section~\secref{Unf2xTri}. 
The overall goal, and the guiding thread, of this Part~II, is to identify slit graphs 
that are forests of trees.

Our basic algorithmic idea for finding such forests is a certain type of 
\emph{spiraling} slit tree.
We present this first in
the planar limit case in Chapter~\chapref{SpiralTree2D}.
The analysis is based on planar convex hulls of finitely many points,  
viewed as vertices of zero curvature.
The boundary of such a hull is equivalently obtained as the minimal length 
enclosing polygon of the given set of points.

Next, we aim to extend to convex polyhedra this idea of spiraling slit trees.
Toward this goal, we partition the surface into two \emph{half-surfaces}, 
sharing a simple closed quasigeodesic $Q$
as a common boundary.
Every convex polyhedron has at least three such 
quasigeodesics~\cite{p-qglcs-49}\cite{p-egcs-73}.

So we are led first to consider, on convex polyhedra, the notions of convexity and convex hull (in Chapter~\chapref{Convexity})
and then the notion of a 
minimal length enclosing geodesic polygon 
(in Chapter~\chapref{MinEnclPoly}).
Although identical in two dimensions,
these notions are not equivalent in our framework.
As far as we know, they have not been 
investigated in detail on convex polyhedra. 
Given their importance in the plane, we believe they
have a certain interest in their own right.

In Chapter~\chapref{SpiralTree3D} we show that the extension to half-surfaces (bounded by $Q$) of the planar spiraling algorithm works well for both
the convex hull and the minimal length 
enclosing geodesic polygon, 
of all vertices inside a simple closed quasigeodesic.

The next step, in Chapter~\chapref{Appl3DAlg}, is to join the two slit trees, previously obtained for half-surfaces.
Thus we obtain unfoldings of the starting polyhedron $P$ onto the union of two cones, or 
onto to a cylinder.
The last case appears if $Q$ contains at most two vertices, 
when rolling the cylinder on a plane leads to a net of $P$.

In the penultimate Chapter~\chapref{VertQuasigeo}, we prove the existence of a non-empty and open set ${\cal Q}_{\leq 2}$ of convex polyhedra $R$,
each polyhedron $R \in {\cal Q}_{\leq 2}$ having a simple closed quasigeodesic with at most two vertices.
We also conjecture that such a quasigeodesic exists on all $P$,
which would lead to a net of $P$ by rolling.

We end with a list of 
open problems from both Part~I and~II in Chapter~\chapref{OpenII}.

\chapter{Vertex-Merging Reductions and Slit Graphs}
\chaplab{VMSlitGraph}
In this chapter we initiate the systematic study of consecutive operations of vertex-merging,
already used in Chapter~\chapref{TailoringFlattening}. 
We introduce vertex-merging reductions and their associated slit graphs, and derive 
their basic properties for later use.
Several examples and open problems are discussed.
Our main goal, carried out over several chapters and topics, is to relate vertex-mergings to unfoldings and nets.
First steps toward this goal are reached in Theorem~\thmref{Connectivity}
and the discussion that follows in Section~\secref{VMUnf}.


\section{Slit Graphs for Vertex Mergings}
\seclab{VMSlitGraph}
Recall from Section~\secref{VertexMerging} 
and the Introduction to Part~II that merging of vertices $v_1$ and $v_2$ is only possible
(by Lemma~\lemref{IsoTetra})
when their curvatures $\o_1, \o_2$ satisfy $\o_1 + \o_2 < 2 \pi$.
The \emph{vertex-merging irreducible} surfaces (in short, vm-irreducible)
are doubly-covered triangles and isosceles tetrahedra.

A \emph{vertex-merging reduction}\footnote{The term ``reduction'' refers to the number of vertices. The surface area increases at each step.}
(in short, a vm-reduction) of a convex polyhedron $P$ is a maximal sequence of consecutive vertex-merging steps starting from $P$;
maximal, in the sense that it reduces $P$ to 
a vm-irreducible surface $S$.

Explicitly, during a reduction process we obtain a sequence of convex polyhedral surfaces $P=P_0 ,P_1,\ldots,P_k=S$, with $S$ vm-irreducible,
and piecewise isometries $\iota_j : P_j \to P_{j+1}$.
Of course, each $\iota_j$ is a polyhedral-unfolding of $P_j$ onto $P_{j+1}$, as is
$\iota=\iota_{k-1} \circ \iota_{k-2} \circ \ldots \iota_0 : P \to S$, where $\circ$ is the composition of functions. Put $P_S = \iota(P) \subset S$.

With some abuse, in the following we shall sometimes consider all isometries above as identities, thus identifying $P_j$ and $\iota_j ( P_j) \subset P_{j+1}$, for $j=0, \ldots, k-1$.

The \emph{slit graph for a vertex-merging reduction} $\iota=\iota : P \to S$ 
is the trace on $P$ of all geodesic segments 
used during the reduction, called \emph{slits}. 
Explicitly, at each reduction step we use a geodesic 
segment $\g_j$ on $P_j$ along which we merge two vertices of $P_j$.
Notice that $\bar{\iota}_j = \iota_{j-1} \circ \ldots \circ \iota_0 : P \to P_j$, 
hence the inverse $(\bar{\iota}_j)^{-1}: P_j \to P$ maps $P_j$ to $P$.
The slit graph $\L$ is the union for all $j=0,\ldots,k-1$ of the slits 
$\lambda_j = \bar{\iota}^{-1}_j  \left( \bar{\iota}_j  (P) \cap \g_j \right)$.
Note that only portions of a slit $\g_j$ on $P_j$ 
might lie on $P$ and so be part of $\L \subset P$.

The importance of studying slit graphs will become apparent 
later; see Section~\secref{VMUnf} and Chapter~\chapref{Appl3DAlg}.

Consider now the vm-reduction $\iota : P \to S$ and its inverse process, from the irreducible surface $S$ to $P$.
Explicitly, at each step of this 
reverse process, we tailor a digon (the one inserted for vertex-merging).
Then the slit graph $\L$ of $P$, for the vm-reduction $\iota$, is precisely the seal graph $\S$ for the corresponding tailoring 
of $S$ to $P$, studied in Chapter~\chapref{SealGraph}. 

\medskip

The next result shows that the slit graph is composed of as many geodesic segments as steps in the reduction process.

\begin{lm}
\lemlab{LmSlitSegment}
For any vertex-merging reduction of  any convex polyhedron, all slits are non-degenerate geodesic segments.
\end{lm}
\begin{figure}[htbp]
\centering
\includegraphics[width=1.0\linewidth]{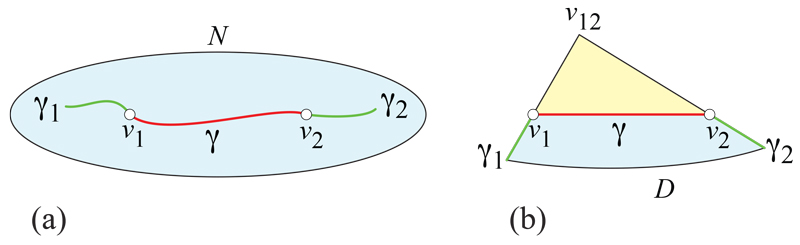}
\caption{(a)~Neighborhood $N$ of geodesic $\g$ connecting $v_1$ to $v_2$.
(b)~$N$ folded to a doubly covered domain $D$.}
\figlab{ND_geodesics}
\end{figure}
%
\begin{proof}
Vertex merging is an intrinsic process. The proof idea below is based on a convenient way of viewing it.

Assume we executed $j$ steps from a reduction process, starting from $P$ and reaching $P_j$.
Consider a geodesic segment $\g$ between vertices $v_1,v_2 \in P_j$ with $\o_1 + \o_2 < 2 \pi$.
Also consider geodesics $\g_1, \g_2$ starting at $v_1$ and $v_2$ such that, together with $\g$, they bisect the complete angles at $v_1$ and $v_2$, 
respectively.\footnote{
If $\g$ is the unique geodesic segment joining $v_1,v_2$ then $\g_1$ contains a geodesic segment included in $C(v_2)$, and similarly for $\g_2$.}
Also consider a neighborhood $N$ of the polygonal path composed by $\g_1$, $\g$ and $\g_2$, containing no vertex of $P$ excepting $v_1,v_2$.
See Fig.~\figref{ND_geodesics}(a).
Then, after merging $v_1$ and $v_2$, $N$ is isometric to a doubly covered planar domain $D$, as in (b)~of the figure. 

Retain the same notation for $D$.
Merging the vertices $v_1$ and $v_2$ can thus be represented in the plane, by the use of $D$.
There, it means extending the edges edges $\g_1, \g_2$ beyond $v_1,v_2$, until their intersection at $v_{12}$.
So each newly created edge (say $v_1v_{12}$) is the extension of
an edge of $D$, which derives from a geodesic on $P_j$ ($\g_1$ in this case). 
That is, $v_1 v_{12}$ is included in a geodesic which cannot be new from one end to the other.

Applying the previous reasoning to $\g$, which is already maximal with respect to inclusion, shows that it contains a geodesic subsegment already existing on $P_{j-1}$ and, inductively, we find a (possibly smaller) geodesic subsegment of $\g$ existing on $P$.
\end{proof}

A nearly immediate consequence of the preceding lemma is this:

\begin{co}
\lemlab{Pn2Pieces}
For a polyhedron $P$ of $n$ vertices,
the collection of slits (the slit graph) cuts $P$ into at most $O(n^2)$ pieces.
\end{co}

\begin{proof}
Because each slit is a geodesic segment, each pair of slits cross at most once.
So the slits
are pseudo-segments, subsegments of pseudo-lines.
(\emph{Pseudo-lines} are curves each pair of which intersects at most once,
where they properly cross.)
With $O(n)$ slits, the arrangement of pseudo-segments on the planar
surface of $P$ has combinatorial complexity $O(n^2)$,
and in particular, has at most $O(n^2)$ cells~\cite{agarwal2005pseudo}.
\end{proof}

Despite the possibly quadratic complexity of the complement of the slit graph, in the end 
the slit graph has only
$3$ or $4$ components, as established in the following lemma.

\begin{lm} 
\lemlab{34components}
The slit graph of any vm-reduction process $\iota:P \to S$ has at most $n_S \in \{3,4\}$ components, where $n_S$ is number of vertices of $S$.
\end{lm}

\begin{proof}
Every slit $\lambda \in \L$ is included in a geodesic $\g$ whose isometric image $\g_j \subset P_j$ creates a merge-vertex outside of $P$.
With some abuse, say $\lambda = \g_j \cap P$.
By Lemma~\lemref{LmSlitSegment}, $\lambda$ is a non-degenerate geodesic segment, hence it is included in $\partial P_S$.
But $S$ has at most $n_S$ vertices, and each component of $S \setminus P_S$ has at least one vertex, hence $S \setminus P_S$ has at most $n_S$ components.
The boundary of each component is a geodesic polygon $\G$ (see Figure~\figref{Hex_EqTri}(a)), 
so its inverse image $\iota^{-1}(\G)$ is a connected subgraph of $\L$ containing $\lambda$.
Therefore, $\L$ has at most $n_S$ components.
\end{proof}


\section{Example: Reductions of Flat Hexagon}
\seclab{Reductions of Flat Hexagon}

Consider a doubly-covered regular hexagon $H=x_1 x_2 \ldots x_6$.
The curvature at each vertex is $4 \pi /6=2 \pi/3$.
We describe two different merge sequences.

Assume first that we vertex-merge vertices two-by-two in increasing order of indices, so we 
flatten the six original vertices and create three new vertices
$x_{12}, x_{34}, x_{56}$, each of curvature $\o_{12}=\o_1 + \o_2 = 4 \pi /3$.
The resulting surface is a doubly covered equilateral triangle of area $1.5$ times larger than the area of $H$.
See Fig.~\figref{Hex_EqTri}(a,b).
The slit graph in this case is a forest composed of three single-edge trees, every other edge of $H$.
This shows that the slit graph for a vertex-merging reduction is not necessarily connected.

Note that for this example, digon-tailoring 
the equilateral triangle in~(b) to $H$ via
sculpting produces the same seals as the digon-tailoring via vertex-merging produces slits.

\begin{figure}[htbp]
\centering
\includegraphics[width=0.9\linewidth]{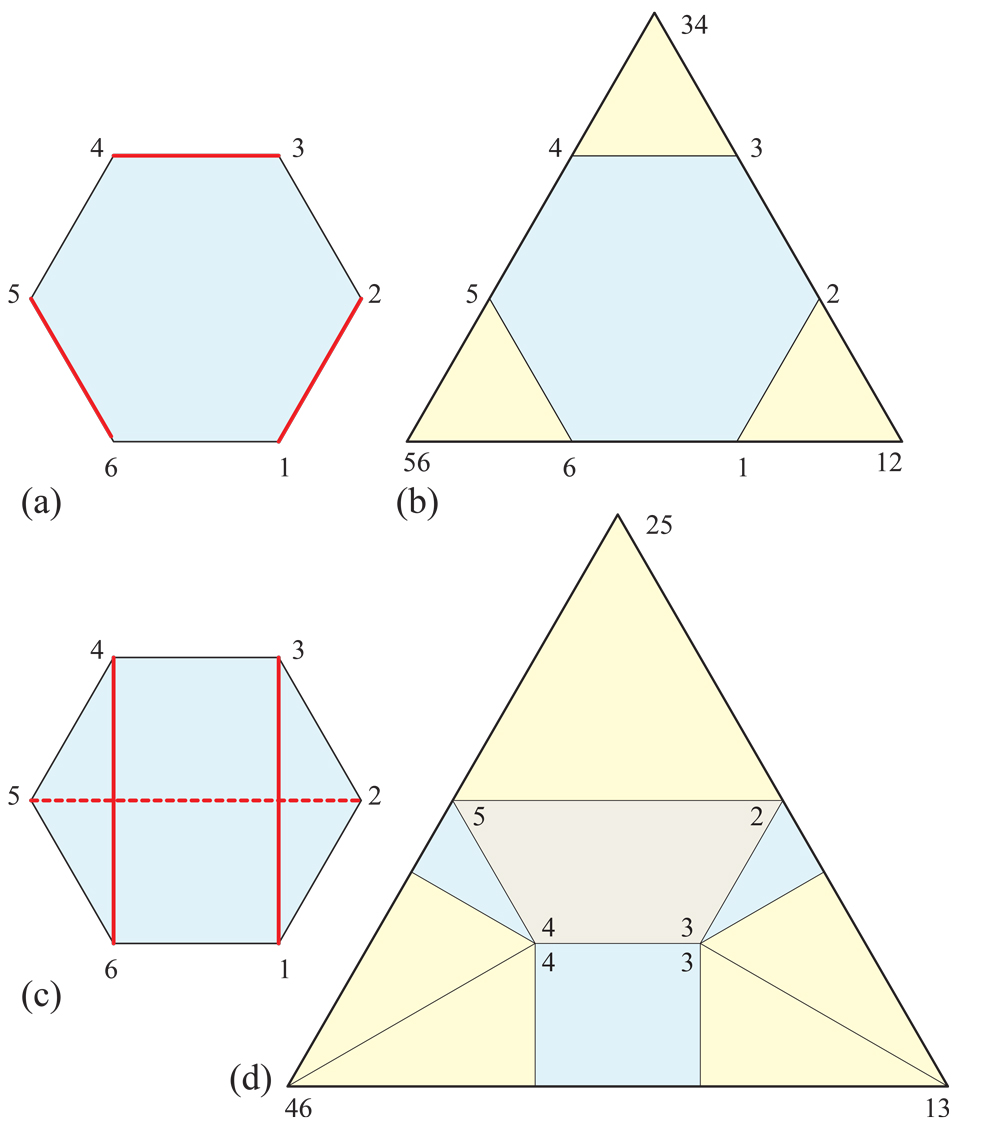}
\caption{
(a)~Doubly-covered hexagon $H$ with slits marked (red).
(b)~The result after three merges.
(c)~Same $H$ but different merge sequence.
(d)~Final equilateral triangle after vertex-merging (not to same scale as (c)).
Blue indicates front surface of the original $H$, tan back surface, and yellow insertions.
The back face is symmetric to the front face.
The images of the vertices $x_1$ and $x_6$ are on the back face, opposite to $v_3$ and $v_4$, respectively.}
\figlab{Hex_EqTri}
\end{figure}

Notice that we cannot vertex-merge $x_{12}$ with $x_j$, $3 \leq j \leq 6$, because the resulting curvature would be exactly $2 \pi$, 
violating the necessary merge condition.
Therefore, no matter in which order we vertex-merge the vertices $x_3,\ldots,x_6$ 
in pairs, the resulting slit graph will be a forest of three single-edge trees.
Similar reasoning shows that the slit graph for any vertex-merging reduction of $H$ is never connected.

On the other hand, for the example presented earlier in Fig.~\figref{VertMerge}, 
the slit graph is a single-edge tree, hence connected.
Moreover, each vm-reduction sequence
of a non-isosceles tetrahedron consists of a single vertex-merging operation, yielding a doubly covered triangle.
So for these surfaces, the slit graph for a vertex-merging reduction is always connected.

This leads us to consider 
the following question, a partial
answer of which will be given in the following sections.

\emph{For which convex polyhedra does there exist a vm-reduction whose slit graph is connected?}

\bigskip

Next we examine a different vertex-merge ordering:
merging $x_1$ and $x_3$ on the front of $H$,  $x_2$ and $x_5$ on the back, and $x_4$ and $x_6$ on the front.
This produces another doubly covered equilateral triangle, say $T$,%
\footnote{The intermediate shapes guaranteed by AGT are, however, not flat
as they were in the previous example,
but rather have positive volume. Only the final shape is flat.}
of area $2.5$ times larger than the area of $H$.
See Fig.~\figref{Hex_EqTri}(c,d).

In particular, the $P$-unfolding (in the sense of Chapter~\chapref{Punfoldings})
$H_T$ of $H$ onto $T$ is not simply connected, 
as for example, the boundary of the $v_1 v_3$ slit is
a non-contractible cycle;
see Fig.~\figref{Hex_EqTri}(d).
However, $H$ itself has not been disconnected by the vertex-merging slits.

Clearly for any given $P$, different vertex-merging reductions of $P$
lead to different surface areas of $S$.
We leave to
Open Problem~\openref{SurfaceArea}
exploring bounds on, or how to achieve, the min or max surface area.

%


\section{Example: Reductions of Cube}
\seclab{CubeExamples}

In this section we present two vm-reduction processes for the cube.

The first reduction, illustrated in Fig.~\figref{CubeVMFlattened}, reduces the cube to a doubly covered square, 
which is a degenerate isosceles tetrahedron and so vm-irreducible.
\begin{figure}[htbp]
\centering
\includegraphics[width=0.8\linewidth]{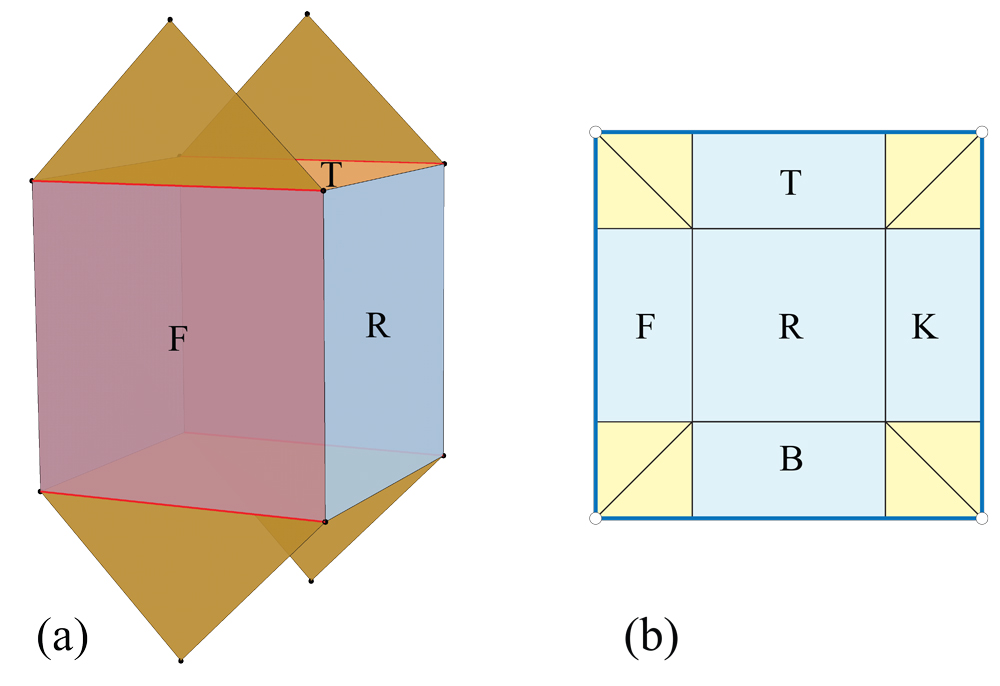}
\caption{(a) A cube and its vm-reduction. (b) The corresponding vm-irreducible surface. The pattern on the back side is similar, except with $R \to L$.
Inserted triangles in yellow.}
\figlab{CubeVMFlattened}
\end{figure}
In any order, slit the top front and back edges, and the bottom front and back edges,
and vertex-merge their endpoints.
The curvature at each cube corner is $\pi/2$, so the triangle inserts are
each isosceles right triangles, a pair of which flatten the vertices to which
they are incident.
The slit graph is a forest of four single-edge trees.

The second merge example is more intricate; it is
illustrated in Fig.~\figref{Cube_VM_triangle_3D}, 
and will be revisited in Chapter~\chapref{Appl3DAlg}.
It reduces the cube to a doubly covered triangle, (b) of the figure.
We now describe the reduction.
\begin{figure}[htbp]
\centering
\includegraphics[width=1.1\linewidth]{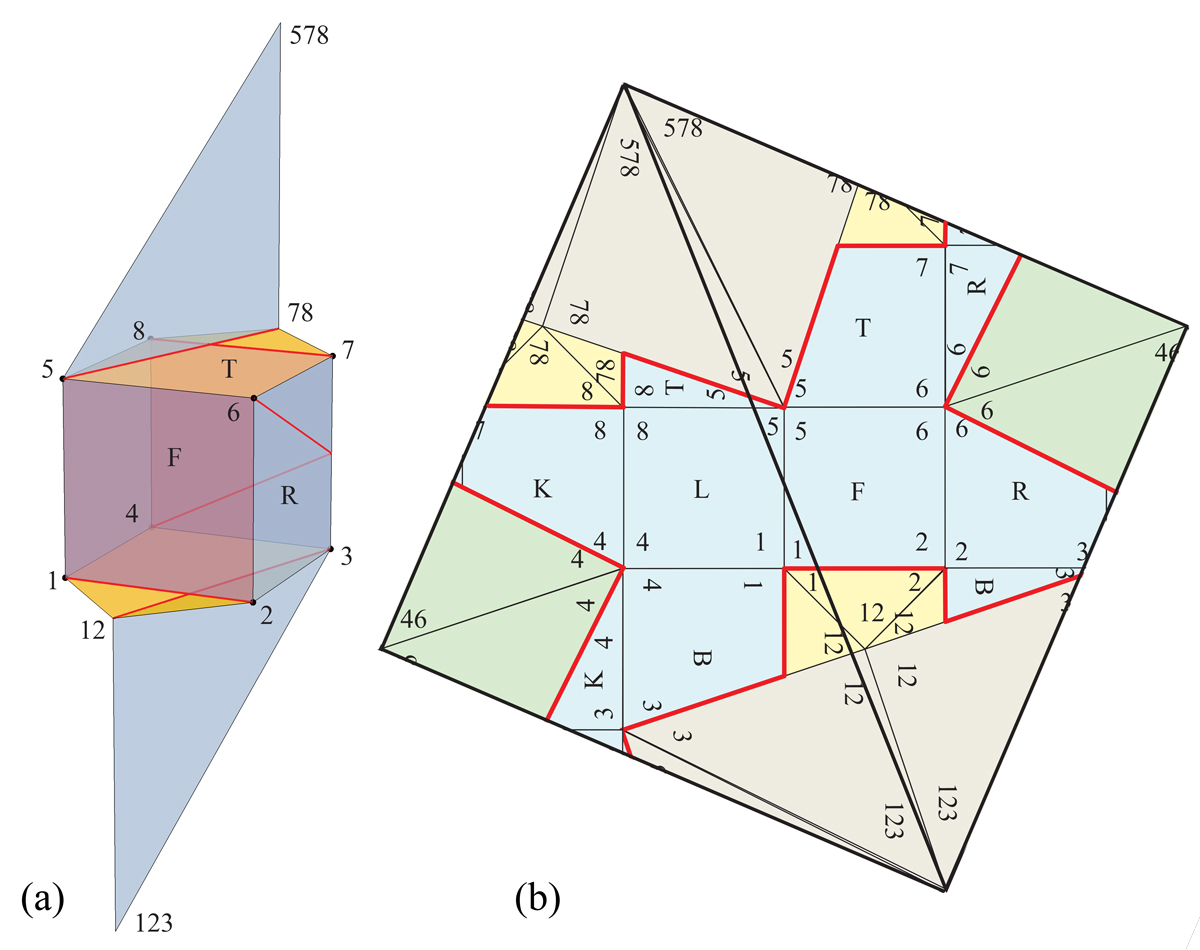}
\caption{(a)~Slits in red. Vertex merging: 
$1+2 \to 12$. $7+8 \to 78$. $5+78 \to 578$. $3+12 \to 123$.
$4+6 \to 46$.
(b)~The doubly covered triangle has vertices at
$578, 123, 46$. The triangle is a right isosceles triangle. Cube surface in blue.
Image of the slit graph in red.}
\figlab{Cube_VM_triangle_3D}
\end{figure}

First $v_7,v_8$ are merged to produce $v_{78}$, just as in the previous example
(although in (a) of the figure the triangle inserts are shown coplanar with the top $T$
of the cube).
Next $v_{78}$ is merged with $v_5$,
creating $v_{578}$. Note that the slit geodesic $v_5 v_{78}$
crosses the slit $v_7 v_8$, and so the inserted triangle pair separates 
the top $\triangle v_7 v_8 v_{78}$ (but not its mate underneath).
It is not straightforward to track the consequence of each insertion.

The same two merges are mirrored on the cube bottom face $B$:
$v_1$ and $v_2$, and $v_3$ with $v_{12}$.
Finally, $v_4$ and $v_6$ are merged.
So the five merges have reduced the eight cube vertices to just three:
a doubly-covered triangle. 
It is shown unfolded in Fig.~\figref{Cube_VM_triangle_3D}(b).
Note that the three vertices of this triangle,
$v_{578}, v_{123}, v_{46}$, are
each surrounded by a polygonal domain of triangle inserts.
This follows because all the cube vertices have been flattened.

The slit graph in this example has three components,
satisfying Lemma~\lemref{34components}.


\section{Example: Icosahedron}
\seclab{IcosahedronExample1}

We next detail using vertex-merges to reduce an icosahedron
to a doubly covered triangle, without disconnecting the icosahedron surface.
We will revisit and modify this example in Chapter~\chapref{Appl3DAlg},
where further details will be provided.

We label the $12$ vertices of $P$ as shown in Fig.~\figref{IcosahedronLabels_v1}.
We will merge in sequence five vertices, $v_1,v_2,v_3,v_4,v_6$, a merge affecting
three of the triangles incident to the top vertex $v_6$.
Symmetrically, we merge five bottom vertices, $v_{11},v_{10},v_9,v_8,v_{12}$.
Each of these sequential merges results in a merge vertex of curvature 
$\frac{5}{3} \pi$, the sum of five $\frac{1}{3} \pi$ curvatures.
Finally, we merge $v_5$ and $v_7$, cutting across the middle band of triangles,
creating a merge vertex of curvature $\frac{2}{3} \pi$.
(So the total curvature is $( \frac{5}{3} + \frac{5}{3} + \frac{2}{3} ) \pi = 4 \pi$,
satisfying the Gauss-Bonnet theorem.)
The resulting doubly covered triangle $abc$ has
angles $30^\circ, 30^\circ, 120^\circ$.

\begin{figure}[htbp]
\centering
\includegraphics[width=0.4\linewidth]{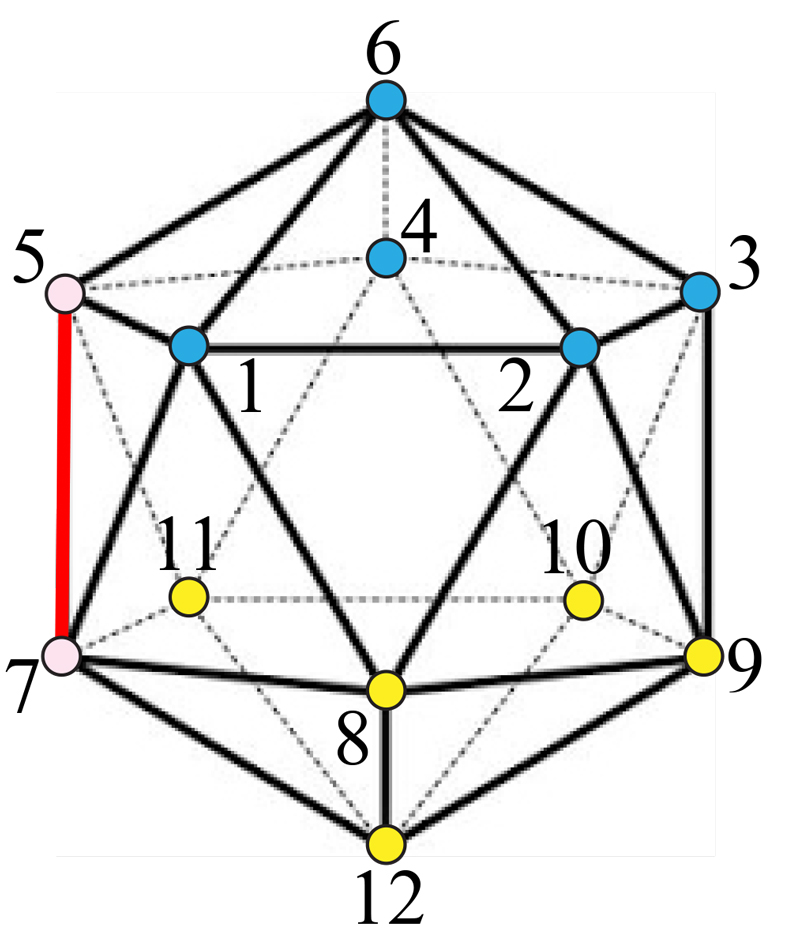}
\caption{Labels $i$ for vertices $v_i$.}
\figlab{IcosahedronLabels_v1}
\end{figure}

\begin{figure}[htbp]
\centering
\includegraphics[width=0.65\linewidth]{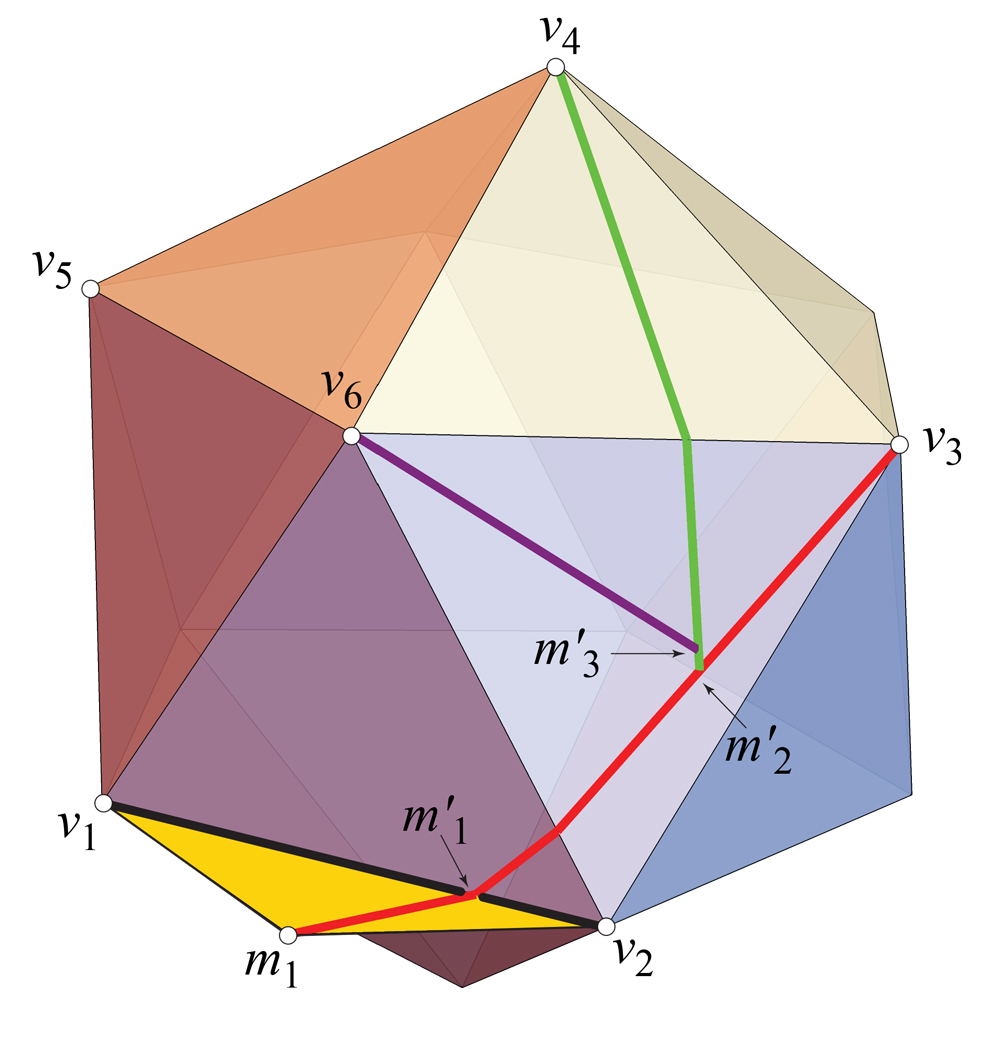}
\caption{Four geodesic slits on $P$, each entering $P$ 
at $m'_i$. $v_6$ is the top vertex of $P$.}
\figlab{Icosahedron_v1}
\end{figure}

We now detail the top sequential merge of five vertices.
Each of the four merges $i$ is accomplished by inserting two
copies of a triangle $T_i$, whose apex is the merge vertex $m_i$.
\begin{align*}
v_1 + v_2  & \to m_1 \;,\; T_1 \\
m_1 + v_3 & \to m_2  \;,\; T_2 \\
m_2 + v_4 & \to m_3  \;,\; T_3 \\
m_3 + v_6 & \to m_4  \;,\; T_4 \;.
\end{align*}
(Here we are using `$+$' and `$\to$' informally to mean `merge' and
`creating' respectively.)
The merge of $m_i$ to the next vertex $v_{i+2}$ slits a geodesic
$\g_{i+1}$ from $m_i$, down one copy of $T_i$, onto $P$ at the point we call $m'_i$.
This geodesic crosses
the previous geodesic cut $\g_i$ from $m_{i-1}$ to $v_{i+1}$ at  $m'_i$.
So $m'_i$ is a point on $P$, whereas $m_i$ is not on $P$, but rather on $P_i$,
the intermediate polyhedron after the $i$-th merge.
In Fig.~\figref{Icosahedron_v1} illustrates the four merge cuts on $P$,
with just $T_1$ shown to illustrate $\g_2$ crossing from $T_1$ onto $P$ at $m'_1$.

Returning to Fig.~\figref{IcosahedronLabels_v1}, we merge five vertices on the top
side of $P$, and symmetrically five on the bottom side, and two connecting
across the ``equitorial'' band of triangles. These merges reduce the
original $12$ vertices to $3$, so the result is a doubly covered triangle.
The slit graph is a forest of three trees.

The resulting doubly covered triangle $abc$ is shown in
Fig.~\figref{IcosaDoubleTriangle}, cut open so that both sides can be seen.
It may not be obvious, but the white regions in Fig.~\figref{IcosaDoubleTriangle}
form a non-overlapping net of the icosahedron,
when, for example, the doubly-covered triangle is cut along
edges $a c$ and $b c$ as illustrated.
\begin{figure}[htbp]
\centering
\includegraphics[width=1.0\linewidth]{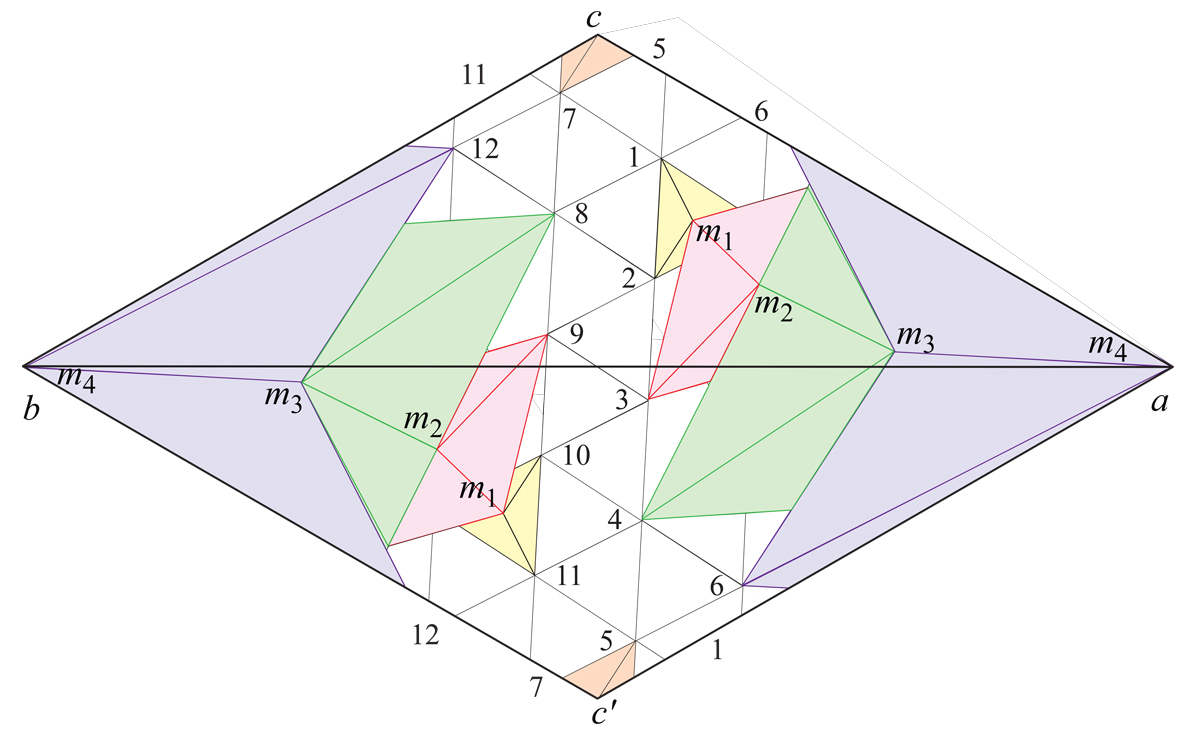}
\caption{Icosahedron surface is white, triangle inserts (merge domains) colored.
$\triangle a b c'$ is the back side of $\triangle a b c$, with $c=c'$ identified.}
\figlab{IcosaDoubleTriangle}
\end{figure}


\section{Example: Hexagonal Shape with Cycle}
\seclab{HexCycle}
The next example shows that the slit graph could have a cycle.
$P$ is the convex hull of two similar hexagons in parallel planes,
separated by a distance $h$.
Label the top hexagon with vertices $v_1,\ldots,v_6$,
and the larger base hexagon $x_1,\ldots,x_6$.
See Fig.~\figref{HexCycle}.
\begin{figure}[htbp]
\centering
\includegraphics[width=0.75\linewidth]{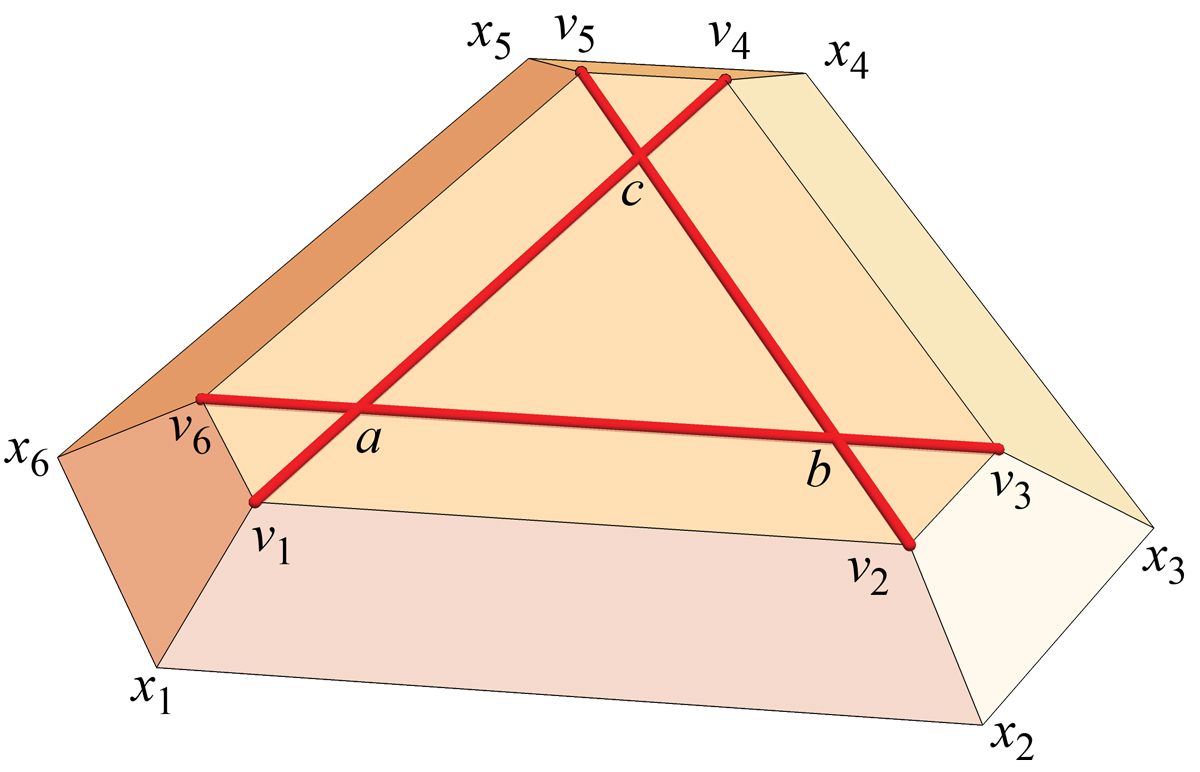}
\caption{Three vertex-merge slits (red) will form a cycle $abc$.}
\figlab{HexCycle}
\end{figure}
Note that as $h \to 0$, the curvature $\o(v_i) \to 0$.
Let $h=\e$ be a small positive height, so $\o=\o(v_i)$ is small.

Now we are going to merge $v_1 + v_4$, then $v_2 + v_5$, 
and finally $v_3 + v_6$, where again $+$ means ``merge."
The merge of $v_1$ and $v_4$ is accomplished by a pair of triangles
with angles $\o/2, \o/2, \pi-\o$. With $\o$ small, the inserted digon $D(v_1,v_4)$
is very narrow, akin to a ``fat" edge with endpoints $v_1$ and $v_4$.
Then the geodesic between $v_2$ and $v_5$ crosses near $c$ over
$D(v_1,v_4)$, but because that is narrow, the geodesic does not look too
different from the $v_2 v_5$ edge illustrated in the figure---the geodesic
``jags" slightly as it crosses $D(v_1,v_4)$.
Again the digon $D(v_2,v_5)$ is narrow, and so does not greatly deviate
the geodesic from $v_3$ to $v_4$, which now crosses both previously inserted
digons, near $b$ and $a$.

After these three vertex-merges, $P_3$ is a 9-vertex polyhedron.
It should be clear that the roughly triangular region $a b c$ in the figure
is disconnected by the three slits, which form a cycle in $\L_3$.
Further slits cannot ``repair" the disconnecting cycle,
so the final $\L$ will have at least one cycle.

\medskip

This leads us to consider the following question:
\emph{For which convex polyhedra do there exist vm-reductions whose slit graphs have no cycle?}
Although we pursue this question in subsequent chapters,
it remains unresolved. See Open Problem~\openref{NoSlitCycle}.


\section{Vertex Merging  and Unfoldings}
\seclab{VMUnf}

The goal of this section if to relate vm-reductions and their slit graphs to unfoldings.

An unfolding of a convex polyhedron $P$ 
cuts the surface along a spanning tree of the vertices, producing a 
polygon $U_P$ when developed in the plane. 
See, e.g., \cite{o-dp-13}.
When that polygon is \emph{simple}---non-self-intersecting,
i.e., non-overlapping---then $U_P$
is called a \emph{net} for $P$,
i.e., an injective embedding of $P$ into the plane.

Consider a reduction process $\iota : P \to S$ of $P$ onto the vm-irreducible surface $S$.
Then $P_S = \iota(P) \subset S$ is an unfolding of $P$ onto $S$
in the sense of Chapter~\chapref{Punfoldings}. 
Further unfolding $S$ to the plane (there are in general many ways to do this) provides an unfolding $U_P$ of $P$ in the plane.

\medskip

We recall here some definitions for subsets 
related to convex polyhedra.
\begin{itemize}
\item
A \emph{simple closed curve} is a closed curve without self-intersections (i.e., homeomorphic to a circle).
\item
A \emph{(geodesic)} polygon is a simple closed curve composed of (geodesic) segments.
\item
A \emph{domain} is a connected open set.
\item
A \emph{polygonal domain} is a the closure of a domain whose boundary is a finite union of polygons.
\item
A \emph{simple polygonal domain} is a polygonal domain with one boundary component.
\item
By a \emph{simply connected set} we understand a set each path-connected component of which is contractible; so the set itself is not required to be connected.
\end{itemize}

\begin{lm}
\lemlab{OneVertex}
Consider a partial vm-reduction $\iota_j : P \to P_j$, for some $1\leq j \leq k$, and identify, for simplicity, $P=\iota_j (P) \subset P_j$.
Every component of $P_j \setminus P$ is a simple polygonal domain containing precisely one vertex.
In particular, the conclusion holds for $S \setminus P_S$,
where $S$ is the vm-irreducible end result of the reduction process.
\end{lm}

We have seen this lemma verified for $S \setminus P_S$
in the series of examples above:
Figs.~\figref{Hex_EqTri}(b,d),
\figref{CubeVMFlattened}(b),
\figref{Cube_VM_triangle_3D}(b),
and~\figref{IcosaDoubleTriangle}.

\begin{proof}
The proof is a simple induction on the step $m$ of the vm-reduction sequence.

At step $m=1$, we merge the first two vertices of $P$, and $P_1 \setminus P$ consists of precisely the pair of inserted triangles, with one vertex, 
the shared apex of those triangles, outside $P$.

Notice that each (partial) vm-reduction $\iota_j$ yields a forest $F$ of binary rooted trees, every node of 
which is a vertex in some $P_l$, with $1\leq l \leq j$.
(These vertex trees are not to be confused with slit trees.)
After a complete reduction, the $3$ or $4$ vertices of $S$ are the roots of
the $3$ or $4$ vertex-trees that were merged to produce those root vertices.
Precisely, the leaves of $F$ are the vertices of $P$, and two nodes in $F$ have a common parent if and only if they are merged at some step $l$.
If a vertex $v \in P$ has not been merged, it is an isolated node of $F$.
If $v$ has been merged, then it has been flattened: $\o(v)=0$.

Assume now that the conclusion holds for the first $m-1$ steps.
At step $m$ we merge, say, the nodes $n_1$ and $n_2$ of $F$, along the geodesic segment $\g$ in $P_{m-1}$. 
Note $n_i$ could be a vertex of $P$, or a vertex in a domain of $P_{m-1} \setminus P$
resulting from a merge.
If $n_i$ is a vertex of $P$, then it is not surrounded by a polygonal domain; for simplicity, we assume in this case that the polygonal domain is $\{n_i\}$ itself.
Cutting along $\g$ and inserting the curvature triangles merges the simple polygonal domains around $n_1$ and $n_2$ into a larger simple polygonal domain,
containing precisely the resulting node $n_{12}$.
\end{proof}

\noindent
For example, in Fig.~\figref{Cube_VM_triangle_3D}, when $n_1=v_5$ and $n_2=v_{78}$,
the domain around $v_{78}$ (yellow in (b) of the figure) is merged with the new
curvature triangles (tan) to form one simple polygonal domain containing $n_{12}=v_{578}$.
In Fig.~\figref{IcosaDoubleTriangle}, the domains around triangle vertices
$a$ and $b$ are the result of four merges, colored in the figure.

\begin{thm}
\thmlab{Connectivity}
Consider a reduction process $\iota : P \to P_S \subset S$ of $P$ onto the vm-irreducible surface $S$, 
resulting in slit graph $\L$.
\begin{itemize}
\item If $\L$ is a forest of trees then $P_S$ is a polygonal domain in $S$.
\item If $\L$ is connected then $P_S$ is simply connected.
\item If $\L$ is a tree then $P_S$ is a simple polygonal domain in $S$.
\end{itemize}
\end{thm}

\begin{proof}
Each slit in $\L$ appears twice in the boundary of $P_S$, once for each bank.
Since $P$ has no boundary, the whole boundary of $P_S$ is produced in this way, hence it is a finite union of segments.

It also follows that each tree component of $\L$ yields a boundary component of $P_S$, which is a polygon.
Therefore, if $\L$ is a tree then $P_S$ is a simple polygonal domain in $S$.
And if $\L$ is a forest of trees then $P_S$ is a polygonal domain.

If a connected component of $P_S \subset S$ is not simply connected then it has several boundary components, impossible if $\L$ is connected.
\end{proof}

Next we offer a topological viewpoint.
View $P$ as a topological sphere with conical-point vertices.
A vertex-merge of $v_1$ and $v_2$ is a topological-circle hole cut in $P$, passing through $v_1, v_2$, which is then filled with a topological disk---the
two back-to-back triangles with a new vertex $v$ inside the disk.
Two connected vm-slits on $P$ correspond topologically to merging two holes to one hole.
A tree of vm-slits is then topologically a single hole, filled with doubly-covered triangles.
Thus a forest of trees is topologically a collection of disjoint holes on $P$.
As all holes are bounded by geometric segments, the result is a polygonal domain with polygonal holes.

\bs

Assume now that $\L$ is a tree, so $P_S$ is a simple polygonal domain in $S$
by Theorem~\thmref{Connectivity}, where $S$ a doubly-covered triangle or an isosceles tetrahedron.
With an appropriate unfolding of $S$, $P_S$ remains connected and thus becomes an unfolding of $P$ in the plane.
However, the unfolding may overlap, and so not constitute a net.
This is explained in the next section.



\section{Unfolding Irreducible Surfaces}
\seclab{Unf2xTri}
Assume we have a vm-reduction $\iota : P \to P_S \subset S$ 
resulting in a
tree slit graph $\L$  so that $P_S$ is
a simple polygonal domain of $S$, by Theorem~\thmref{Connectivity}.
We explore next the question of unfolding $S$ to the plane, preserving 
connectedness of $P_S$.
We start with $S$ a doubly-covered triangle.

\subsection{$S$: Doubly-covered Triangle}
Let $u_1,u_2,u_3$ be the three vertices of $S$.
Each is surrounded by polygonal domains $M_1, M_2, M_3$
created by repeated triangle inserts.
Two examples of these domains
are the cube vm-reduction in
Fig.~\figref{Cube_VM_triangle_3D}(b)
and the icosahedron vm-reduction in
Fig.~\figref{IcosaDoubleTriangle}.
Both are displayed with their triangles cut open.

Now we argue that $S$ can be cut open via a spanning tree 
without disconnecting $P$. So this provides an unfolding of $S$
that preserves $P$ as a single piece, as in those two figures.
Although we can guarantee a single piece, we have not established this piece
is a net, avoiding overlap. See Open Problem~\openref{NotDisconnect}.

Topologically,
$S$ is a sphere with three 
conical points / vertices $u_i$, each surrounded
by a \emph{merge-domain} $M_i$, with the remainder of $S$ the original $P$:
in loose notation,
$P = S \setminus \left( \cup_i M_i \right)$.
Each $M_i$ is topologically a disk,
as depicted in Fig.~\figref{MiTopology}(a).

There are only two combinatorially distinct spanning slit trees for unfolding $S$:
a path $u_1, u_2, u_3$ (and its index permutations),
or a {\sc Y}-tree: some point $x$ in $P$ with slits from $x$ to each $u_i$.
(In both 
Figs.~\figref{Cube_VM_triangle_3D}(b)
and~\figref{IcosaDoubleTriangle},
the slit tree is a path.)
For either spanning tree, a slit path $\rho$ must enter $M_i$ to reach $u_i$, and in the
path case, the slit path must also exit $M_i$.
\begin{figure}[htbp]
\centering
\includegraphics[width=1.0\linewidth]{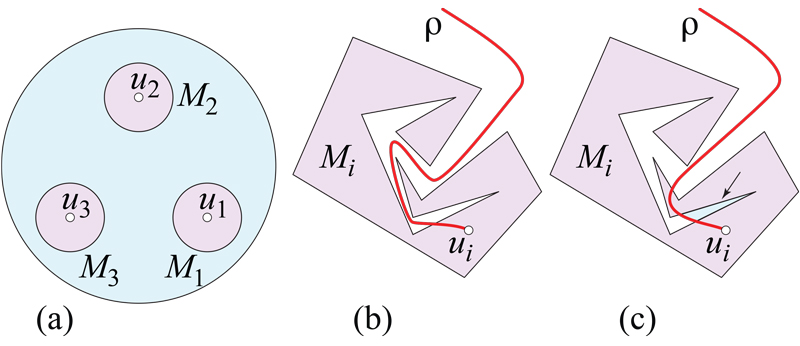}
\caption{(a)~Topology of $S$. (b,c)~Slit path $\rho$ to $u_i$.
In~(c), the marked region is disconnected from $P$.}
\figlab{MiTopology}
\end{figure}

Because each $M_i$ is a disk, there is a $\rho$ that crosses $\partial M_i$
just once on entrance, and in the path case, once again on exit.
Fig.~\figref{MiTopology}(b) illustrates such a one-cross $\rho$.
Fig.~\figref{MiTopology}(c) shows that, if $\rho$ crossed $\partial M_i$ at
more than one point, then a piece of $P$ is disconnected.
But because the spanning slit tree can be chosen to avoid the (c)~situation,
we are guaranteed that we can unfold $S$ to the plane so that $P \subset S$
remains a simply connected polygonal domain in the plane.


\subsection{Net and Overlap}

Let us temporarily ignore the structure of 
$S = P \cup \left( \cup_i M_i \right)$,
and just view $S$ as a doubly covered triangle $\triangle$.
If, for every spanning slit tree $T$, the unfolding of
$\triangle$ avoids overlap, then we immediately have that
the unfolding of $P$ avoids overlap, and so we have found a net for $P$.

Alas, this is not true for every $T$, as
the example in Fig.~\figref{Spiral_x2tri_cex} shows.
\begin{figure}[htbp]
\centering
\includegraphics[width=1.0\linewidth]{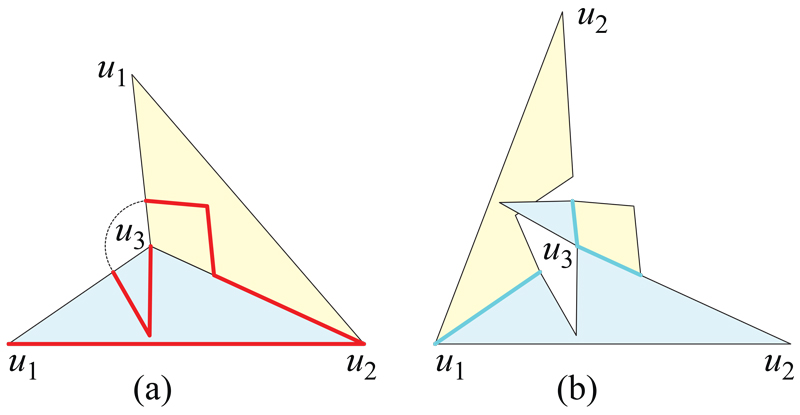}
\caption{(a)~Front: blue; back: yellow.
Slit path $T$: $(u_1, u_2, u_3)$. 
(b)~Unfolding after cutting $T$. Blue segments remain uncut.}
\figlab{Spiral_x2tri_cex}
\end{figure}
In (a) the $u_2,u_3$ slit path spirals around $u_3$ from the back to the front.
The unfolding is shown in~(b), with the uncut segments highlighted,
showing clear overlap.

Although this may seem a contrived example,
the merge regions $M_i$ might be created from spiraling slit trees,
as we will discuss in the the following chapters, 
and could be quite
complex, as the icosahedron example (Fig.~\figref{IcosaDoubleTriangle})
suggests.
A natural slit path follows the merge vertices into each $M_i$ region:
$\mu= (m_0,m_1,m_2,\ldots,m_i)$, where $m_0=u_1$.
However, we leave it as an open question 
(Open Problem~\openref{NotDisconnect})
of whether forming a 
slit tree by following $\mu$ for each $M_i$---or any other slit tree---unfolds $P$ to a net.


\subsection{$S$: Isosceles Tetrahedron}
Vertex-merging could result in an isosceles tetrahedron $S$
(rather than a doubly covered triangle).
This occurred in the first cube reduction, illustrated in
Fig.~\figref{CubeVMFlattened}(b).
This does not change the topological picture Fig.~\figref{MiTopology}(a)
significantly: just a fourth merge-domain disk $M_4$ is added, 
and there are two more possible spanning
tree structures.
Therefore, by similar reasoning, we can ensure a single-piece unfolding
of $P$ to the plane.
But again, we leave open the question of whether there is always a slit
tree that leads to a nonoverlapping single piece, a net for $P$.

\chapter{Planar Spiral Slit Tree}
\chaplab{SpiralTree2D}
The previous chapter showed that 
if the slit graph $\L$ of a vm-reduction 
is a tree, then we can unfold $P$ to the plane, and possibly to a non-overlapping net.
So we have made a goal of finding a vm-reduction ordering that results in $\L$
a tree.
Our overall plan is to partition $P$ into two ``halves" via a simple closed
quasigeodesic $Q$, and then to vm-reduce the vertices in each half separately.
(As mentioned in Chapter~\chapref{IntroII}, such quasigeodesics always exist.)
Let $V$ be the set of vertices inside $Q$.
We will eventually show (in Chapter~\chapref{SpiralTree3D}) 
how to vertex-merge all of $V$ so that the slit graph of that half is a tree.
In this chapter we illustrate the main idea by assuming $V$ lies in a plane.

\section{Sequential Spiral Merge}

Let $V =  \{v_1,v_2,v_3, \ldots, v_n \}$ be the vertices
in the above-half $P^+$ of polyhedron $P$.
Define a \emph{sequential merge} as a series of vertex-merges
where at each step, the previous merge vertex is connected to a
vertex of $V$. Thus the merge region---the union of all the triangle inserts---grows
a single connected domain enclosing, in the end, one vertex
by Lemma~\lemref{OneVertex}.

There are many possible sequential merge orders. We choose 
one that we call a \emph{spiral merge}.
See ahead to Fig.~\figref{SpiralTree_n50_s1} for why ``spiral" is 
an appropriate term.
In the planar situation, we view
vertices as having zero curvature, so that $P^+$ becomes planar.
In this context, a slit is a line segment $s$, wholly in $P$.
Each triangle insert degenerates to just $s$, and creates a merge vertex
somewhere along $s$.
We will show in Theorem~\thmref{SlitTree2D} that a spiral merge 
algorithm in two dimensions results in $\L$ a slit tree.


\section{Notation}

\begin{figure}[htbp]
\centering
\includegraphics[width=1.0\linewidth]{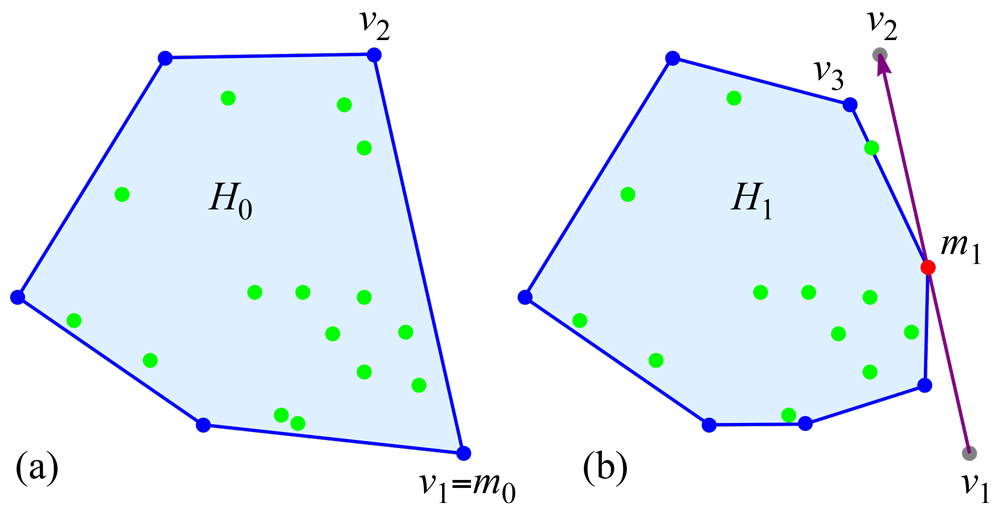}
\caption{$v_1$ is merged with $v_2$ to produce $m_1$.
$H_1 = \conv( V_1)$ where
$V_1 = V_0 \setminus \{v_1,v_2\} \cup m_1$.}
\figlab{SpSnap_12}
\end{figure}

The detailed argument needs considerable notation beyond that introduced %
earlier, which we gather below for reference.

\begin{itemize}
%
\item $V = V_0 = \{v_1,v_2,v_3, \ldots, v_n \}$: 
Vertices in the plane. 
(Later (in Chapter~\chapref{SpiralTree3D}) these will be vertices 
on the surface of the polyhedron $P$.)
\item $m_i$: merge vertex, the vertex created by merging $m_{i-1}$ with $v_{i+1}$
with $m_0=v_1$. We 
sometimes abbreviate this as $m_{i-1} + v_{i+1} \to m_i$.
\item $s_i = m_{i-1} v_{i+1}$. The $i$-th slit/merge segment. 
The new merge vertex $m_i$ lies on $s_i$.
\item So each slit segment has three labeled points: $m_{i-1}, m_i, v_{i+1}$.
\item $v_1$ is also given the label $m_0$, so $s_1 = v_1 v_2 = m_0 v_2$.
\item $m_i$ can lie anywhere along $s_i$. 
In Figs.~\figref{SpSnap_12}-\figref{SpiralTree_n50_s1},
$m_i$ was chosen at a random point on $s_i$.
\item $\L _i= \cup_i s_i$ is the slit graph after the $i$-th merge. $\L$ is the full slit graph.
\item $v_i$ is called \emph{flattened} if it has already been merged. 
(In Chapter~\chapref{SpiralTree3D}, when $v_i$ is a vertex of positive curvature,
the merge will reduce $v_i$'s curvature to zero.)
\item $V_i$ is the set of not-flattened vertices remaining after the $i$-th merge 
$s_i = m_{i-1} v_{i+1}$. $m_i \in V_i$.
$|V_i|$ is the number of vertices in $V_i$.
\item $H_i$ is the convex hull of $V_i$. 
We view $H_i$ as a closed region of the plane and $\partial H_i$ its boundary, a convex polygon.
$H_0$ is the convex hull of $V=V_0$.
\end{itemize}

\begin{figure}[htbp]
\centering
\includegraphics[width=1.0\linewidth]{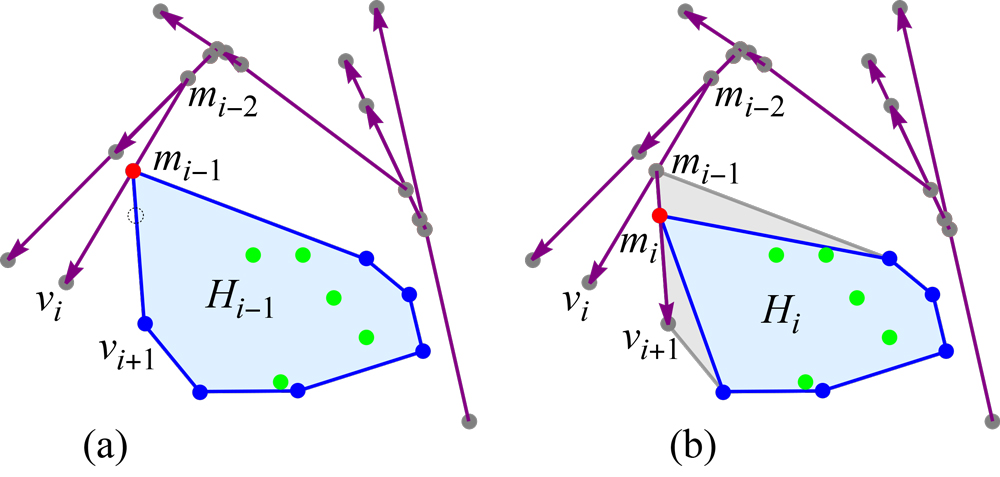}
\caption{General step: $i$-th merge.
(a)~$s_{i-1} = m_{i-2} v_i$. $s_{i-1} \cap H_{i-1} = m_{i-1}$.
(b)~$s_{i} = m_{i-1} v_{i+1}$. $s_{i} \cap H_{i} = m_{i}$.
$s_{i-1} \cap s_i = m_{i-1}$.
}
\figlab{SpSnap_mi}
\end{figure}


\section{Algorithm Description}
\seclab{SpiralAlgorithm}

\begin{description}
\item[First merge.] 
The algorithm starts by selecting any edge of the convex hull $H_0$ of $V=V_0$,
say $v_1$ and $v_2$, and merging them via the merge segment $s_1 = v_1 v_2$.
In the plane, this amounts to replacing $v_1=m_0$ and $v_2$ with 
with a new vertex 
$m_0 + v_2 \to m_1$
on the segment $s_1$.
The merge flattens $v_1$ and $v_2$,
which is why they are removed.
$V_1 = V \setminus \{v_1,v_2\} \cup m_1$,
and $H_1 = \conv(V_1)$.
See Fig.~\figref{SpSnap_12}.
\item[Second merge.] 
Next, $m_1$ is merged with the first vertex beyond $m_1$ on $H_1$.
Call this vertex $v_3$. Note that $v_3$ might not be the next vertex
after $v_2$ on $H_0$, because $V_1 \neq V_0$
and so $H_1 \neq H_0$.
The merge of $m_1$ and $v_3$ introduces a new merge vertex $m_2$ on
$s_2=m_1 v_3$, and both $m_1$ and $v_3$ are flattened and removed
from $V_1$, while $m_2$ is added, to produce $V_2$.
%
\item[General step.] 
The $i$-th merge connects $m_{i-1}$ to $v_{i+1}$,
where $v_{i+1}$ is the next vertex on $H_{i-1}$ beyond $m_{i-1}$.
See Fig.~\figref{SpSnap_mi}.
Then $m_i$ lies on $s_i = m_{i-1} v_{i+1}$,
the vertex set is updated as
$V_i = V_{i-1} \setminus \{m_{i-1}, v_{i+1} \} \cup m_i$,
and $H_i = \conv(V_i)$.
Note that $|V_i| = |V_{i-1}| - 1$, because two vertices are removed 
(flattened) and one added.
But $\partial H_i$ 
does not bear a similar relation to $\partial H_{i-1}$ 
because the new hull may wrap around vertices that were strictly interior to $H_{i-1}$
(green in the figures).
\item[Completion.] 
The process continues until all the original $v_i$ vertices are merged,
leaving just one vertex $V_{n-1} = \{ m_{n-1} \}$.
A full trace is illustrated in Fig.~\figref{SpiralSnaps_n20_s4}.
(The reason there are $n-1$ merges rather than $n$ is that the first merge
flattens two vertices of $V$, while all subsequent merges flatten two vertices but only
one is from $V$.)
\end{description}

%
\begin{figure}[htbp]
\centering
\includegraphics[height=0.9\textheight]{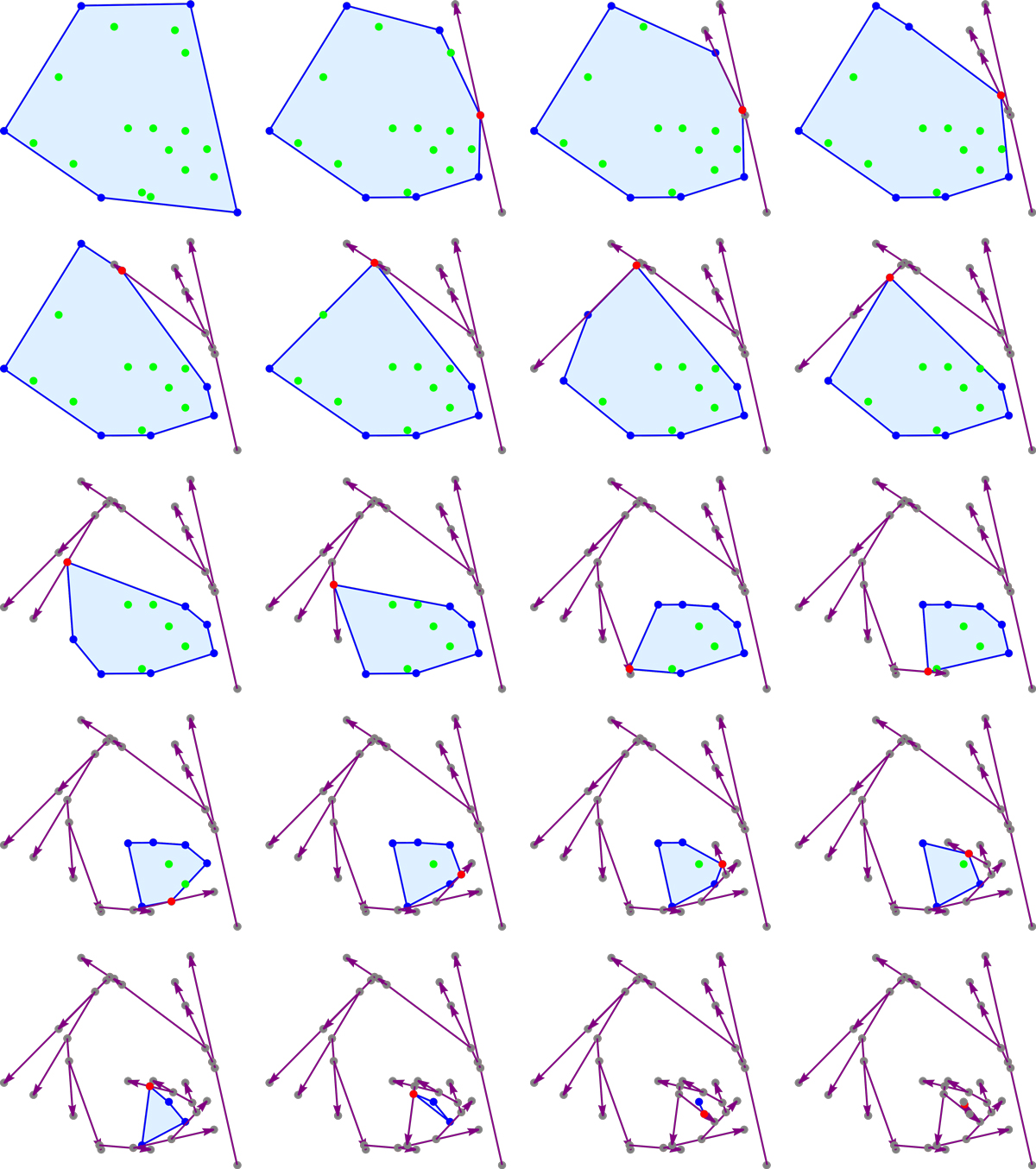}
\caption{Trace of example with $|V|=20$,
from which 
Figs.~\protect\figref{SpSnap_12} and~\protect\figref{SpSnap_mi} are 
details.}
\figlab{SpiralSnaps_n20_s4}
\end{figure}

\begin{figure}[htbp]
\centering
\includegraphics[width=0.6\linewidth]{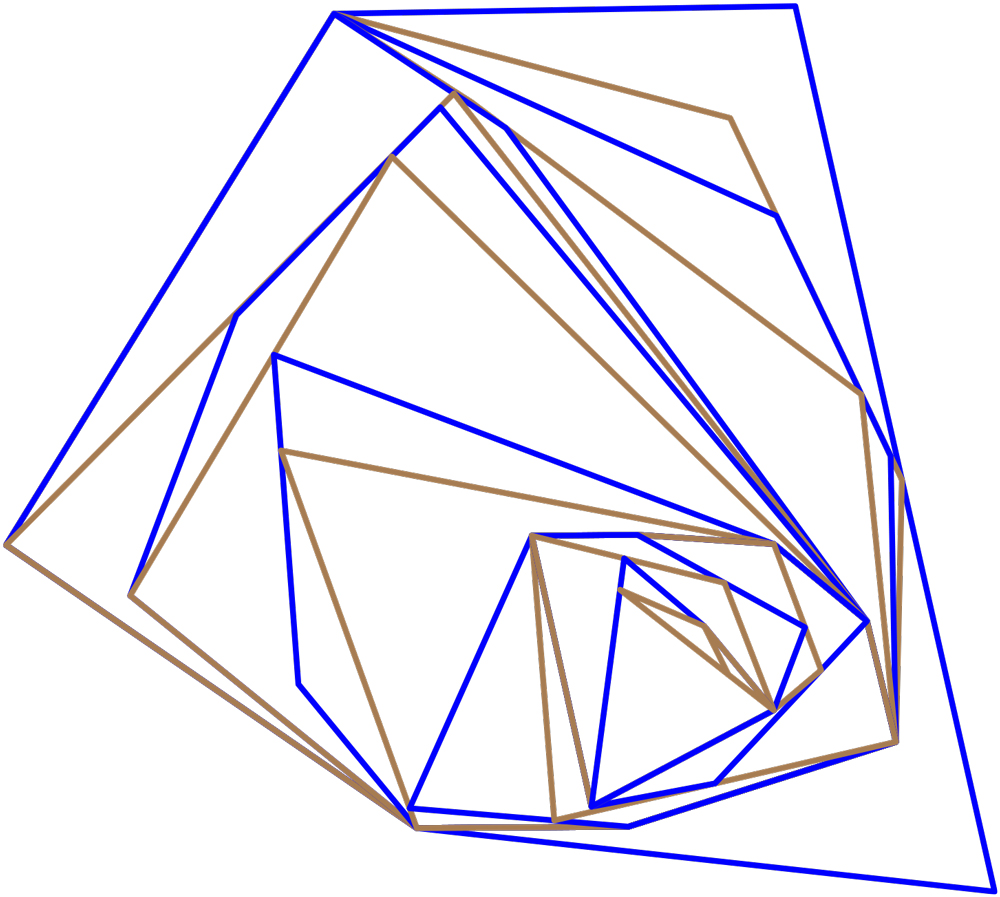}
\caption{Nested $H_i$ for all $i$, from Fig.~\protect\figref{SpiralSnaps_n20_s4},
alternately colored blue and brown.}
\figlab{SpiralNested_n20_s4}
\end{figure}


\section{Planar Proof}

We first assume that all
points---both vertices in $V$ and merge points $m_i$---are in 
\emph{general position} in the sense that no three are collinear.
Later we will see that allowing collinearities does not change the main claim
that $\L = \cup_i s_i$ is a tree.

The induction hypothesis consists of several claims:
\begin{enumerate}[label={(\arabic*)}]
\item $H_i$ is nested inside $H_{i-1}$: $H_{i-1} \supset H_i$.
See Fig.~\figref{SpiralNested_n20_s4}.

\item $s_i$ only intersects $H_i$ in one point, the merge vertex:
$s_i \cap H_i =  \{m_i\}$.
See Fig.~\figref{SpSnap_mi}(b).

\item Consecutive slit segments share just one point, an endpoint of the later segment:
$s_{i-1} \cap s_i = \{m_{i-1}\}$.
See Fig.~\figref{SpSnap_mi}.

\item The set of vertices reduces by one each iteration:
$|V_i| = |V_{i-1}|-1$, for $i > 1$.
\item $\L_i$ is a tree.
%
\end{enumerate}

\begin{figure}[htbp]
\centering
\includegraphics[width=0.6\linewidth]{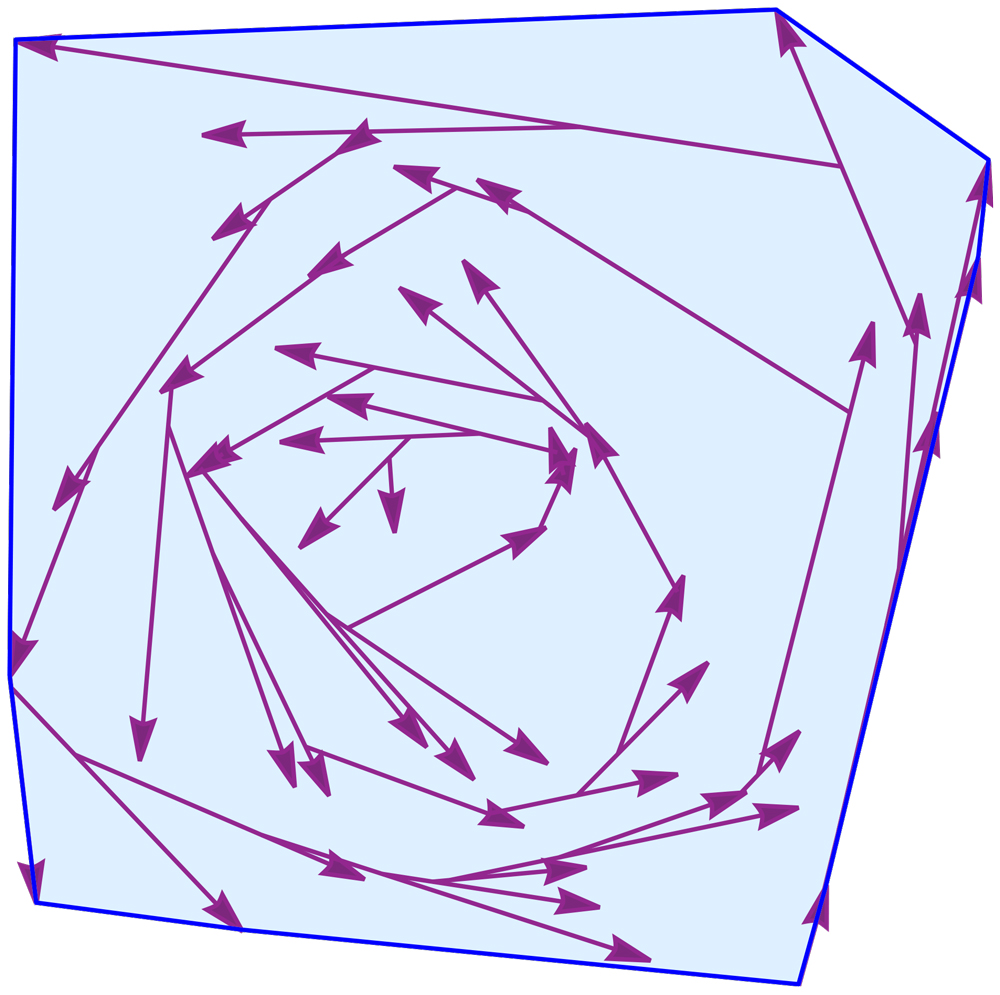}
\caption{$\L$ for a set of $|V|=50$ vertices. Only slit segments are shown.}
\figlab{SpiralTree_n50_s1}
\end{figure}

\paragraph{Basis.}
These claims are easy to see for $i=1,2$,
just by construction according to the algorithm.
Neverthless, just for completeness, we run through the basis.
Refer to Fig.~\figref{SpSnap_12}.
\begin{enumerate}[label={(\arabic*)}]
\item $H_0 \supset H_1$ 
because the edge $v_1 v_2$ is replaced by a point $m_1$ on that edge.
\item 
$s_1 \cap H_1 = \{ m_1 \}$ by construction.
\item For $i=1$, there is just one segment $s_1$.
$s_2$ starts at $m_1 \in s_1$, so indeed $s_1 \cap s_2 = \{ m_1 \}$.
\item $V_1$ has one fewer vertex than does $V=V_0$ because
two ($v_1$ and $v_2$) are removed and one ($m_1$) is added.
\item $\L_1 = s_1$ is a single-edge tree.
\end{enumerate}

\paragraph{General Step.}
Assume the induction hypotheses are satisfied up to $i-1$, and
consider the $i$-th merge.
It will help to consult Fig.~\figref{SpSnap_mi},
where (a)~is established by the induction hypothesis, and (b)~reflects
the situation to be proved.
The algorithm chooses $v_{i+1}$ as the next vertex on $H_{i-1}$,
so $s_i = m_{i-1} v_{i+1}$ is an edge $e$ of $\partial H_{i-1}$.

\begin{enumerate}[label={(\arabic*)}]
\item $H_{i-1} \supset H_i$ because the edge $e$ is replaced by the merge
vertex $m_i$.
\item $s_i \cap H_i = \{ m_i \}$ by construction.
\item 
$s_{i-1} \cap s_i = \{ m_{i-1} \}$ by construction.
\item 
$|V_i| = |V_{i-1}|-1$, again because
$V_i = V_{i-1} \setminus \{ m_{i-1},v_{i+1} \} \cup \{m_i\}$.
\item $\L_i$ is a tree.
By (3)~above, $s_i$ has an endpoint $m_{i-1}$ on $s_{i-1}$,
and otherwise does not intersect $s_{i-1}$.
It remains to prove that the interaction of $s_i$ with the earlier
segments in $\L_i$ maintains the tree property.
In fact, we show that $s_i \cap s_j = \varnothing$ for $j < i-1$.

As we've seen, $s_i$ is an edge $e$ of $\partial H_{i-1}$,
and from (1)~we know that $H_j \supset H_{i-1}$
for $j < i-1$.
From (2), each $s_j$ intersects $H_j$ in just one point $m_j$,
which is removed from $V_j$ in the next step.
So $s_j$ is disjoint from $H_{j+1} = \conv(V_j)$.
And so $s_j$ cannot intersect $s_i$, which is an edge of $\partial H_{i-1}$.

Therefore, we have shown that $s_i$ just intersects $\L_{i-1}$
at the one point $m_{i-1}$, and so maintains the tree structure for $\L_i$.
\end{enumerate}

The tree structure is illustrated in Fig.~\figref{SpiralTree_n50_s1}.
It is in a sense a geometric directed binary tree,
spiraling from the boundary of $H_0$ into the center.

\paragraph{Collinearities.}
If there are collinearities, then several consecutive segments can
collinearly overlap.
So it could be, for example, that $s_{i-1} = m_{i-2} v_i$ overlaps with $s_i=m_{i-1} v_{i+1}$
in the portion from $m_{i-1}$ to $v_i$.
Thus several slit segments could collinearly overlap with an edge of $H_i$.
If we union each group of collinearly overlapping segments to one segment,
then all the properties claimed in the induction hypothesis continue to hold,
although different notation would be required to capture the unioning of several 
collinearly overlapping segments into one.

\begin{thm}
\thmlab{SlitTree2D}
For any set $V$ of vertices in the plane, the sequential merging algorithm
detailed in Section~\secref{SpiralAlgorithm}
results in a slit graph $\L$ that is a tree
(see Fig.~\figref{SpiralTree_n50_s1}).
\end{thm}

As the figures suggest, we have implemented this algorithm.

\bs 

Convexity and, in particular, the convex hull, plays an essential role in this
algorithm. Before we can apply the same overall idea to a set $V$ on
a three-dimensional $P^+$ in Chapter~\chapref{SpiralTree3D}, 
we need to explore convexity 
and convex hulls on $P$. This is the topic of the next chapter.






\chapter{Convexity on Convex Polyhedra}
\chaplab{Convexity}

We've set as our goal proving that there is a vm-reduction ordering
of the vertices $V$ inside a quasigeodesic $Q$
that results in a slit tree $\L$, a goal achieved in 
Chapter~\chapref{SpiralTree3D}.
Before we reach that point, we need to develop a clear notion of what constitutes the convex hull
of $V$.
This in turn requires a clear notion of convexity on polyhedral convex surfaces,
which is our focus in this chapter.

Our investigation into convexity on convex polyhedra is, to our knowledge, the first  
in this direction, and seems to have a rich potential.
We barely touch on some classical convexity results
(such as Helly's and Radon's theorems: Examples~\exref{NoHelly} and~\exref{NoRadon}),
and do not attempt to provide their adaptations in this new theory of convex sets.
Instead we focus on developing basic facts
that ultimately lead to a characterization of the ``relative convex hull" of a set of
vertices inside a simple closed quasigeodesic $Q$ (Section~\secref{RelativeHullVerts}).

We first argue that, for our purposes, the proper notion 
of a convex set $S$ is one that includes every geodesic segment between points
in $S$, in contrast to, say, including at least one geodesic segment.
Defining the convex hull of a set of vertices $V$ as the smallest
convex set including $V$ (Section~\secref{ConvexHull}) leads to some surprising
and perhaps undesirable properties.
For example, the convex hull of a set inside $Q$
does not always remain inside $Q$ (Example~\exref{Pyramid}).
These properties lead us to develop a notion of ``relative" convexity,
and the relative convex hull in Section~\secref{RelativeHull}.
Among our main results is Theorem~\thmref{rconvV}, which characterizes
the relative convex hull of vertices inside a quasigeodesic $Q$.
This characterization is then employed in Chapter~\chapref{SpiralTree3D}.

Because of the many lemmas and details in this long chapter, we
offer a concise summary of $38$ individual results in
Section~\secref{Summary}.

Throughout, for a set $S$, we use $\partial S$ for its boundary,
and $\mathring{S}$ and $\bar{S}$ for the interior and closure of $S$ respectively.


\section{Convex Curves}
\label{CvxCurvex}
Let $C$ be a simple, closed, curve on the surface of a convex polyhedron $P$.
We will assume $C$ is polygonal, turning at \emph{corners} which in general may or
may not be vertices of $P$ (although we'll continue to use labels $v_i$).
View $C$ as directed counterclockwise from above.
At each corner $v_i$ of $C$, let $\a_i$ be the surface angle to the left, and
$\b_i$ the angle to the right.
$C$ is called a \emph{convex curve} if $\a_i \le \pi$; all such angles are
\emph{convex angles}.
Angles $\b_i$ strictly greater than $\pi$ are called \emph{reflex}.

If $C$ also satisfies $\a_i \le \b_i$,
we call it an \emph{$\a\b$-convex curve}.
If $\a_i < \b_i$ for all (non-zero curvature) corners, it is a \emph{strictly $\a\b$-convex curve}.

A simple, closed \emph{quasigeodesic} $Q$ is a curve $C$
that is convex to both sides: $\a_i \le \pi$ and $\b_i \le \pi$.
By a theorem of Pogorelov~\cite{p-qglcs-49},
every convex polyhedron has
at least three such quasigeodesics.\footnote{We point out here that a different notion of 
``quasi-geodesic" also exists, see e.g., \cite{bendokat2021efficient}.}

In general, quasigeodesics are not $\a\b$-convex to either side.
The total curvature of the vertices
to either side of $Q$ is $\leq 2\pi$, and only equal to $2\pi$ if $Q$ is a geodesic, passing through no vertices,
a fact we use later.

A \emph{geodesic polygon} is a polygon whose edges are geodesic arcs.
A \emph{geodesic-segment polygon} is a polygon whose edges are geodesic segments,
i.e., shortest paths.
The distinction between these two types of polygons plays a significant
role in this chapter. For brevity,
we will frequently abbreviate ``geodesic arc" with ``geoarc,"
and ``geodesic segment" with ``geoseg."
Although we are mainly interested in convex sets contained within polygons,
a convex curve could include smoothly curved arcs.


\section{Notions of Convexity}
\seclab{Notions}

Several notions of convexity have been employed so far on surfaces, see for example 
\cite{alexander1978local}, 
\cite{bangert1981totally}, 
\cite{gm-cgs-01}, 
\cite{mitrea2016geodesic},
\cite{zamfirescu1991baire},
and the references therein. 
We have not made a comprehensive accounting of all the 
different definitions found in the literature,
and instead mention next only three.

A subset $S$ of $P$ is said to be 
\begin{itemize}
\item \emph{geodesically convex}
if, given any two points in $S$, there is a unique geodesic segment joining them in $S$;
\cite{udriste2013convex}.
\item \emph{totally convex} 
if any geodesic which joins two points of $S$ is contained in $S$;
\cite{bangert1981totally}.
\item \emph{metrically convex}
if, given any two points in $S$, there is at least one geodesic segment joining them in $S$;
\cite{gm-cgs-01}.
\end{itemize}
We should mention that there is variation in the literature.
For example, \cite{vishnoi2018geodesic}
uses the term ``geodesically convex" to mean what we list above as ``totally convex."

Geodesic convexity is not suitable for our framework, because an open geodesically convex set could contain no vertex $v$---two geodesic segments would wrap around $v$.
Since our main concern is with geodesic segments forming a slit graph, 
total convexity, which focuses on geodesics (not necessarily shortest paths), is also inappropriate.

We will need a notion of a convex hull of a set of points $V$ on $P$.
In analogy with the Euclidean case, we would like to define the
convex hull of $V$ to be the intersection of all convex sets containing $V$.
We next indicate why the natural metric convexity, and in particular,
its ``at least one geodesic segment'' criterion, is not the correct
version in our context.

\begin{ex}
Let $\Delta$ be a doubly-covered triangle $v_1 v_2 v_3$, with $V= \{v_1, v_2, v_3 \}$.
The front triangle is a metric-convex set containing $V$, as is the back triangle.
So the metric-convex hull of $V$ consists of the three edges $v_1 v_2$, $v_2 v_3$, $v_3 v_1$.
However, one would expect to obtain $\Delta$ as the convex hull of its vertices,
as we do in Theorem~\thmref{ExtPts}.
Moreover, the situation illustrated by the doubly-covered triangle could be troublesome in some 
specific situations as well.
We address these situations in Section~\secref{ConvexHull} below.
\end{ex}



\section{Ag-convexity}
\seclab{Ag-convexity}
Define $S$ to be \emph{ag-convex}---all-geodesics convex---if
every geodesic segment between two points of $S$ is in $S$.
In the following we will focus on ag-convexity, which we henceforth abbreviate 
as, simply \emph{convexity}.

This same notion of convexity has been used in different contexts, see for 
example~\cite{alexander2019alexandrov} and~\cite{lytchak2021every},
although we have not seen it applied to convex polyhedra.
Other types of convexity will be specifically identified to differentiate
them from ag-convexity.

\bs

Clearly, every convex set is path-connected, but one can say more.

\begin{lm}[Classification]
\lemlab{int_convex}
Let $S$ be a closed convex set on $P$.
Then $S$ is either a point, or a geodesic arc, or a simple closed geodesic, or it has interior points.
\end{lm}

\begin{proof}
Assume $S$ contains two distinct points, hence it contains a geodesic segment $\g$.
Assume $\G$ is the maximal (with respect to inclusion) geodesic in $S$ including $\g$; so it could be an arc, or a simple closed curve.
If the image set of $\G$ is $S$, we are done.
(Here the \emph{image set} of $\G$ is the set of points on $P$ comprising the curve.)

Assume there exists a point $x$ in $S$ but not on $\G$.
Join it with an interior point $y$ of $\G$, say with the geodesic segment $\g_y$.
Locally around $\g_y$ the set $S$ is flat, so moving $y$ continuously on $\G$ provides a continuous family of geodesic segments $\g_y$, all included in $S$.
Therefore $S$ has interior points.
\end{proof}

\begin{ex}
The whole surface $P$ is obviously convex. We also consider 
the empty set and the single-point sets to be convex.
\end{ex}

\begin{ex}
\exlab{Closure_not_cvx}
Notice that the closure of a convex set is not necessarily convex.
The interior of a face of a doubly-covered triangle is convex, while the whole face is not,
because two points on 
different edges of the triangle are connected by
geodesic segments on both sides of the triangle.
\end{ex}

We'll repeatedly use the following simple fact.

\begin{lm}[Local behaviour]
\lemlab{local-conv}
Let $S\subset P$ be a convex set and $x \in S$.
Then there exists a small ball $B_x$ around $x$ such that 
$S\cap B_x$ is isometric either to a planar convex set (if $x$ is not a vertex of $P$), or to a convex set on a cone of apex $x$.
\end{lm}
\begin{proof}
Just choose $B_x$ around $x$ to include no vertices of $P$ (other than possibly $x$), and notice that 
$S \cap B_x$ is still convex on $B_x$.
\end{proof}


Recall from Section~\ref{CvxCurvex} that an $\a\b$-convex curve is a \emph{strictly $\a\b$-convex curve} if $\a_i < \b_i$ for all (non-zero curvature) corners.
The following characterization of $\a\b$-convexity will be frequently invoked subsequently.

\begin{lm}[$\a\b$-convexity] 
\lemlab{ag-is-ab}
\hfill \break \vspace{-3ex} 
\begin{enumerate}[label={(\roman*)}]
\item
Let $S\neq P$ be a closed convex subset of $P$, $S \neq P$, having 
at least two distinct points.
Then $\partial S$ is a convex curve
(including the case when $S$ is a geodesic).
If, moreover, $\partial S$ is a geodesic polygon, then it is strictly $\a\b$-convex.
\item
Let $S$ be an open convex subset of $P$, $S \neq P$, and $C$ a 
connected component of $\partial S$.
Then $C$ is either a vertex, or a convex curve.
If, moreover, $C$ is a geodesic polygon, then it is (not necessarily strictly) $\a\b$-convex.
\end{enumerate}
\end{lm}

The conclusion in Lemma~\lemref{ag-is-ab}(ii) holds with a similar proof for convex sets which are neither closed nor open, but have interior points.
A surface containing such a set is given later in Example~\exref{Illuminating}.

\begin{proof}
(i)~The first claim follows from Lemma~\lemref{local-conv}.
Indeed, each component of $\partial S$ is a locally-convex curve.
If $S$ has interior points then each component of $\partial S$ is a simple closed curve, hence convex.

For the second claim, let $v_i$ be a (non-zero curvature) corner of $\partial S$ at which $\a_i \ge \b_i$.
If $\a_i = \pi$, then $\b_i = \pi$, and $v_i$ has zero curvature. So we may
assume that $\a_i < \pi$.

Let $x,y$ be two points on $\partial S$ close to and on either side of $v_i$.
Then, as shown in Fig.~\figref{alphabeta}, the shortest path between $x$ and $y$
inside $S$ is either longer than, or equal in length to, the shortest path outside $S$.
In either case, a geodesic segment between points of $S$ is not inside $S$,
contradicting the fact that $S$ is convex.

(ii)~Assume $\partial S$ is not $\a\b$-convex at $v$, hence $\a > \b$.
Then there exist points $x,y \in \partial S$ with $v$ between them, and any geoseg $\g$ joining them is external to $S$, 
as in the above argument.
There also exist sequences of points $x_n,y_n \in S$ with $x_n \to x$ and $y_n \to y$.
Choose geosegs $\g_n$ from $x_n$ to $y_n$, hence $\g_n$ lies in $S$.
It is well known that the limit of geosegs is a geoseg.
Possibly passing to a subsequence, assume $\g_n \to \g'$, hence $\g'$ is a geoseg joining $x$ to $y$ in $\bar S$, 
which is only possible if $\a=\b$ at $v$, contradicting our assumption.
\end{proof}

\begin{ex}
\exlab{no-strict-ab-conv}
To complete Example~\exref{Closure_not_cvx}, 
notice that the interior of a face is an open convex set on a doubly covered triangle, but its boundary is not strictly $\a\b$-convex.
It remains, nevertheless,  $\a\b$-convex.
\end{ex}

\begin{figure}[htbp]
\centering
\includegraphics[width=0.75\linewidth]{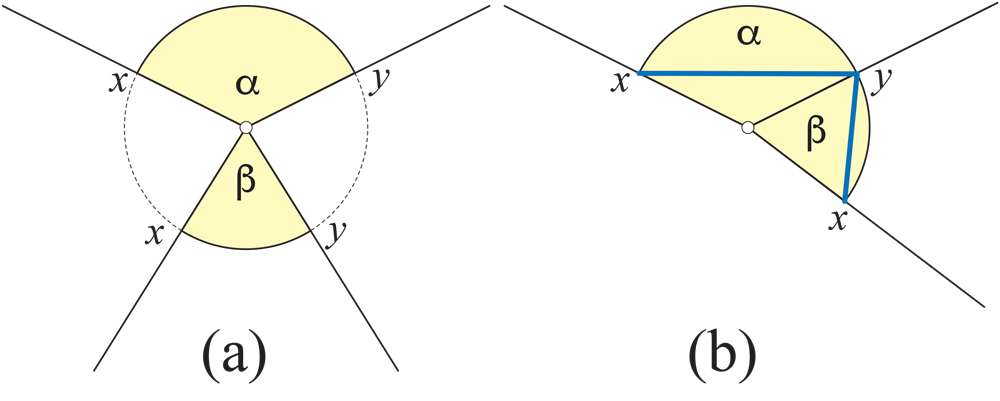}
\caption{(a)~$\a > \b$. (b)~$|xy|$ is shorter through angle $\b$ than through $\a$.}
\figlab{alphabeta}
\end{figure}

An open convex set may have several vertices as boundary components, 
as we will see in Example~\exref{P-V} below, and at most two convex curves as boundary components, as will be evident in the next result and its proof.

\begin{lm}[Simply-connected or cylinder]
\lemlab{si-con_convex sets}
Let $S\neq P$ be a closed convex subset of $P$.
Then either $S$ is simply-connected, or it is isometric to a cylinder without lids.
\end{lm}
Recall from Chapter~\chapref{Punfoldings} (Section~\secref{PunfReshaping}) 
that in our usage, simply-connectedness does not 
assume path-connectedness. However, all convex sets are path-connected.

\begin{proof}
If $S$ has empty interior, the conclusion follows from Lemma~\lemref{int_convex};
so we may assume in the following that $S$ has interior points.

The convexity of $S$ implies that each component of $\partial S$ is a simple closed curve, and since it is locally convex by Lemma~\lemref{ag-is-ab}, it is 
a convex curve.

Let $C_1, C_2$ be two components of $\partial S$.
There are two possibilities: without loss of generality, $C_2$ is nested inside $C_1$,
or they are not nested. This latter case violates convexity, because a geodesic
segment connecting a point inside $C_1$ to a point inside $C_2$ would necessarily
include points exterior to both. So let $C_2$ be nested inside $C_1$.

The Gauss-Bonnet Theorem implies that the total curvature inside $C_1$ is $\leq 2 \pi$, and similarly for $C_2$.
So the total curvature outside $C_1$, and also outside $C_2$, is $\geq 2 \pi$.
But the exterior of $C_2$ is strictly inside $C_1$, and the exterior of $C_1$ is strictly inside $C_2$.
Therefore, either the region between $C_1$ and $C_2$ has zero curvature, and therefore it is isometric to a cylinder without lids,
or at least one the above inequalities is strict and we get a contradiction, 
establishing the conclusion.
\end{proof}

\begin{lm}[Vertex point-hole]
\lemlab{conv-pt}
Let $S\subset P$ be a convex set and $v$ a vertex in 
$S$.
Then $S \setminus \{v\}$ is convex.
\end{lm}
\begin{proof}
The case $S =\{v\}$ is clear, because we consider the empty set to be a convex set.

Consider distinct points $x, y \in S$, and a geodesic segment $\g$ between them, hence $\g \subset S$.
If $v \not\in \{x,y\}$, $\g \subset S \setminus \{v\}$, because no geodesic segment will pass through a vertex.

Assume now that $x=v$. Then each point of $\g \setminus \{v\}$ is inside a geodesic subsegment of $\g$, hence included in $S \setminus \{v\}$.
\end{proof}


\begin{ex}
\exlab{ParaGeods}
An example illustrating the cylinder claim of the Lemma~\lemref{si-con_convex sets} is provided
by $P$ a tall rectangular block, and $S$ the region between two parallel simple closed geodesics close to the middle
of the block. Then $S$ is a closed convex set.
\end{ex}

\begin{ex}
\exlab{P-V}
Lemma~\lemref{si-con_convex sets} does not hold for arbitrary convex sets with interior points.
Indeed, the set $P \setminus V$ is convex (see Lemma~\lemref{conv-pt}), open, but not simply-connected.
\end{ex}

\begin{ex}
\exlab{S-v}
More generally,
let $S$ be a simply-connected convex set, and $v$ a vertex interior to $S$. Then $S \setminus \{v\}$ is convex, but not simply-connected.
However, the closure of $S\setminus \{v\}$ is convex and simply-connected.
\end{ex}

\begin{ex}
\exlab{NoHelly}
A particular instance of Helly's Theorem states the following:
let $S_1, ..., S_n$ be a finite collection of convex subsets of $\Rp$, with $n > 3$. 
If the intersection of every $h=3$ such sets is nonempty, then the whole collection has a nonempty intersection.

The simple example of the four faces of a tetrahedron,
each face a closed, convex set, shows that no analogous result holds in our framework.

A more elaborated example places the tetrahedron over a tall right triangular prism.
There, the union of the tetrahedron's lateral faces, minus the top vertex, is a convex set.
That set, together with the lateral faces, provides another counterexample.
\end{ex}

It remains for future work to consider the above 
Helly-like problem for $h \geq 4$: Open Problem~\openref{HellyEtc}.

\begin{lm}[$S \supseteq Q$]
\lemlab{S>Q}
If the interior of the closed convex set $S$ contains a simple closed quasigeodesic
$Q$, then either $\partial S$ determines with $Q$ a topologically closed cylinder, or $S=P$.
\end{lm}
\begin{proof}
Assume the simple closed quasigeodesic $Q$ is included in $S$.

If $Q$ parallels $\partial S$ (in the sense that they bound a cylinder) then $\partial S$ is itself a simple closed quasigeodesic.
Assume this is not the case.

Because $Q$ is interior to $S$, the sum of curvatures of vertices inside or on $Q$ is at least $2 \pi$.
If $\partial S \neq \varnothing$ then, since it is convex towards $S$, the total curvature of the interior of $S$ is $< 2 \pi$, 
by Lemmas~\lemref{si-con_convex sets} and~\lemref{ag-is-ab}, 
which violates the Gauss-Bonnet Theorem.
\end{proof}

We don't know if every closed convex set is either included in a half-surface bounded by a simple closed quasigeodesic, or is the whole surface.
This is not settled by the previous lemma because there $Q$ is interior to $S$. 
This question is Open Problem~\openref{ConvHalfSurf}.
However, the proof of the previous lemma does show that no closed convex set can strictly enclose more than $2 \pi$ curvature.
This picture will be completed by Proposition~\lemref{dense_conv} and Theorem~\thmref{ExtPts} below.


\section{Geodesic Segments and Convex Sets}
\label{segm-pts-cvx}

In this section we clarify the behaviour of geosegs between points in, or on the boundary of, or outside of, a convex set.
The first part of the next result will be particularly useful later.

\begin{lm}[Segment vs. convex set]
\lemlab{out-segm-conv}
Let $S\subset P$ be a convex, $S \neq P$.

(i) If $x,y$ are interior points of $S$ and $\g$ a geodesic segment joining them,
then $\g$ is interior to $S$ (i.e., it does not intersect $\partial S$).

In particular, adding a boundary point to an open convex set yields a convex set.

(ii)~If $x$ is an interior point of $S$ and $y \in \partial S$, then all geodesic segments from $x$ to $y$ are interior to $S$, possibly excepting at their extremity $y$.

(iii)~If $x,y \in \partial S$ then there is a geodesic segment between them included in the 
closure $\bar S = S \cup \partial S$ of $S$,
possibly excepting its extremities $x,y$. 
Other geodesic segments from $x$ to $y$ might not be included in $\bar S$.

In particular, adding two boundary points to an open convex set does not necessarily yield a convex set.

(iv) If $x \in \partial S$ is not by itself a connected component of $\partial S$, then there exists a geodesic segment starting at $x$ and 
exterior to $S$.
\end{lm}

\begin{proof}
(i)
Let $B_1,B_2$ be open intrinsic balls
around $v_1,v_2$ respectively, with $B_1 \cup B_2 \subset S$.

Let $N_\g \subset P$ be 
a tubular neighborhood of the image set of $\g$.
View $N_\g$ as created by translations of $\g$ sufficiently small so that 
$N_\g$ is included in $B_i$ around $v_i$, $i=1,2$, and so that
$N_\g$ is an open set included in $S$.
$N_\g$ is flat everywhere, except at $v_1,v_2$.
Because $\g$ is strictly inside $N_\g$, $\g$ does not intersect $\partial S$.

This claim can also be established by assuming that
some geodesic segment between interior points touches $\partial S$, 
and applying Lemma~\lemref{local-conv} to obtain a contradiction.
 
\medskip

(ii)
Assume there exists a geodesic segment $\g$ from $x$ to $y$ that goes outside $S$,
so there is a boundary point $z \in \partial S$ on $\g$ between $x$ and $y$.
Assume, moreover, that $z$ is closest to $y$ with these properties.

Notice that such a $z$ exists, i.e., is at positive distance to $y$.
To see this, consider a small ball $B_y$ around $y$ which contains no vertex of $P$.
Then $S \cap B_y$
is a convex set in the plane (if $y$ is not a vertex of $P$) or on a cone (if $y$ is a vertex),
by Lemma~\lemref{local-conv}.
Denote by $\g_z$ the sub-arc of $\g$ from $z$ to $y$.
The boundary of this small convex set $\partial(S \cap B_y)$ is a convex curve, and the image of $\g_z$ is a segment, with the two intersecting in at most two points
(because a segment can intersect a convex curve at most twice).
This shows $z$ exists as described.

Denote by $A_z$ the component of $\partial S$ between $z$ and $y$ such that no point of $S$ is inside the lune $L$ bounded by $A_z$ and $\g_z$.
$L$ is then bound by a 
curve $A_z$ concave towards $L$ and a straight segment $\g_z$
exterior to $A_z$. 
The Gauss-Bonnet Theorem implies this is only possible if $L$ contained negative curvature,
a contradiction.

\medskip

(iii) 
Because $x,y \in \partial S$, there exist 
a sequence of points $x_n, y_n \in S$ such that $x_n \to x$, $y_n \to y$.
Let $\g_n$ be a  geodesic segment joining $x_n$ to $y_n$, $n \in \N$.
Then there exists a subsequence of $\{\g_n\}_n$,
which converges to a geodesic segment $\g$ between $x$ and $y$.
Since all $\g_n$ are in $S$, their limit is in $\bar S$.

\medskip

(iv) 
This claim can be obtained from Lemma~\lemref{local-conv}.
An alternative proof follows.

The space of unit tangent directions $T_x$ at $x$ is a circle,
of course of arc-length $\leq 2 \pi$.

Assume, in contradiction to the claim 
that, for any $\nu \in T_x$, the maximal geodesic segment $\g^{\nu}$ starting at $x$ in direction $\nu$ intersects $S \cap B_x$,
i.e., it has an initial portion of $\g^{\nu}$ is inside $S$.
Choose three directions $\nu_i \in T_x$, $i=1,2,3$,
with the angle between adjacent pairs $\nu_i, \nu_j$ less than $\pi$.
With points $x_i \in S$ on $\g^{\nu_i}$,
the three flat triangles $x x_i x_j$ are in $S$, and their 
union forms an open set around $x$, contradicting $x \in \partial S$.
\end{proof}

\bs
\noindent
Example~\exref{Closure_not_cvx} can be easily elaborated to illustrate all the claims of the above lemma.

\begin{co}
\lemlab{int-cvx}
It follows directly from (i) in  Lemma~\lemref{out-segm-conv} that the interior of a convex set is still convex.
\end{co}

\begin{lm}[Supporting angle] 
\lemlab{support_angle}
Let $S\neq P$ be a closed convex subset of $P$ with interior points, and $x$ a boundary point of $S$.
There exists tangent directions $\mu, \nu \in T_x$ at $x$ of angle $\theta=\theta(\mu, \nu) \leq \pi$ toward $S$,
such that:

(i)~the geodesic segments $\g^{\mu}$ and $\g^{\nu}$ in the directions $\mu, \nu$ do not intersect the interior of $S$, locally; and

(ii)~for each $\tau \in T_x$ inside $\theta$, the geodesic segment $\g^{\tau}$ in the direction $\tau$ does intersect the interior of $S$, locally.
Here, $\theta$ is regarded as a subarc of $T_x$.
\end{lm}

\begin{proof}
This can established from Lemma~\lemref{local-conv}.
An alternative way to prove the claims follows.

Let $B_x$ be a small ball around $x$ which contains no vertex of $P$.

Lemma~\lemref{out-segm-conv} shows that there exists 
a geodesic segment $\g_z$ connecting $x$ to a point $z \in P\setminus S$.
Because $P\setminus S \neq \varnothing$ is an open set, there are other
points in the neighborhood of $z$ similarly connected to $x$ by
geodesic segments that do not intersect the interior of $S$.
Denote by $A_x$ the maximal open subarc of $T_x$ determined by directions of such geodesic segments $\g_z$.
Let $\mu, \nu \in T_x$ be the extremities of $A_x$.

The convexity of $S$ implies that for all $\tau \in A_x$, the geodesic segment $\g^\tau$ does not intersect $S \cap B_x$.
Moreover, $A_x \geq \pi$. This establishes claim~(i).

The convexity also implies that $\theta=T_x \setminus A_x \leq \pi$ has the opposite property: 
for each $\tau \in T_x$ inside $\theta$, the geodesic segment $\g^{\tau}$ intersects the interior of $S \cap B_x$.
This establishes Claim~(ii). 
\end{proof}


\section{Relative Convexity}
\seclab{RelConv}
For the next 
result (Lemma~\lemref{conv-vm}),
we need to modify the notion of convexity to \emph{relative convexity}, a
variation on a notion of relative convexity we employed in~\cite{ov-ceccc-14}.

Let $S\subset P$ be a contractible 
closed convex set with interior points.
We proved in the previous section that $S$ is bounded by a closed convex curve 
$C = \partial S$.
We glue to $C$ a tall cylinder $L$ with a base but without a top.
This satisfies AGT, because a point $p \in C$ is convex to the $S$-side and
has angle $\pi$ on the cylinder rim.

Denote by $P^\#$ the resulting convex surface.
AGT implies that
$P^\#$ is a polyhedron if and only if $\partial S$ is a geodesic polygon on $P$.
Call $S$ \emph{relatively convex} if its image (also denoted by $S$) 
is convex on $P^ \#$.

Note that $C$ is a quasigeodesic on $P^\#$.

There are two considerations that lead us to introduce relative convexity in this section,
and the relative convex hull $\rconv(S)$
in Section~\secref{ConvexHull}.
First, the failure of vertex-merging to preserve convexity,
Example~\exref{Quad} below.
Second, the failure of the convex hull of points inside a simple closed quasigeodesic $Q$ to remain inside $Q$, Example~\exref{Pyramid} below.

In Section~\secref{RelativeHull} we shall use a 
construction similar to the above, but
with respect to simple closed quasigeodesics $Q$ instead of convex curves $C$.
The purpose there is to construct the relative convex hull, starting from a non-convex set inside $Q$.
So the term ``relative'' will refer to either construction, depending on the context, and without confusion.

\begin{lm}[Vertex merging]
\lemlab{conv-vm}
Let $v_1,v_2$ be two vertices interior to the convex set $S \subset P$.
Merging $v_1$ and $v_2$ produces a new polyhedron $P'$ and a set $S' \subset P'$ obtained as the union of $S$ with the two merge triangles $T$.
Then $S'$ is relatively convex on $P'$.
\end{lm}

\noindent Before we prove this lemma, we show that it is false for
convexity, without the ``relatively" modifier.

\begin{ex}[doubly-covered quadrilateral]
\exlab{Quad}
Figure~\figref{Merge_v1v2} shows an example $S \subset P$ where
the insertion of merge triangles causes some geodesic segments
in the new set $S'$ to leave $S'$, violating convexity.
In~(a), geodesic segments from any point $x$ on the front face
to its image $x'$ on the back face stay within $S$,
whereas after the $v_1,v_2$ merge~(b), there are points $x$ whose shortest
path to its back-face image $x'$ crosses $ab$ outside of $S'$.
\end{ex}
\begin{figure}[htbp]
\centering
\includegraphics[width=0.8\textwidth]{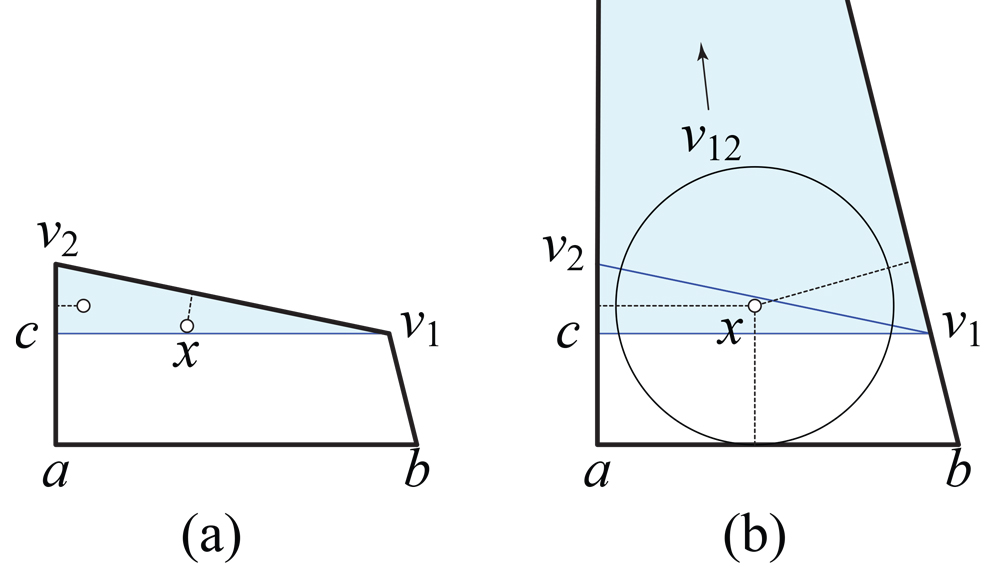}
\caption{(a)~$P$ is the doubly-covered quadrilateral $a b v_1 v_2$.
Convex set $S\subset P$ is
the two-sided triangle $c v_1v_2$, shaded.  
(b)~Merging $v_1$ and $v_2$ yields $v_{12}$ far above $P$.
The resulting set $S'$ is $c v_1 v_{12}$. 
}
\figlab{Merge_v1v2}
\end{figure}

\begin{proof}
We know by Lemma~\lemref{out-segm-conv} 
that the insertion of the triangles $T$ has not affected the boundary of $S'$:
$\partial S' = \partial S = C$, where $C$ is a convex curve.
In $P'^\#$, denote by $L$ the sides of the cylinder glued to $C$.
When $L$ is unfolded flat, it forms a rectangle $R$.

Let $x,y \in S'$ be two points in $S'$, and $\g$ a geodesic segment between them.
We argue that $\g$ cannot cross $C$.
Suppose $\g$ properly enters $R$ at $x'$ and exits at $y'$. Then the segment
$x' y'$ along the top of $R$ is shorter, a contradiction to 
the assumption that $\g$ is a geodesic segment.

Therefore, $S'$ is convex on $P'^\#$, i.e.,
$S'$ is relatively convex.
\end{proof}


\section{Convex Hull}
\seclab{ConvexHull}

We mentioned earlier (Section~\secref{Notions}) that one reason we are
not using metrical convexity is that it leads to an unsatisfactory notion of a convex hull.
Here we explore the convex hull, and conclude that we need a variation for our
purposes in the next chapter.

In analogy with the Euclidean case, define the \emph{convex hull} $\conv(S)$ of an arbitrary set $S \subseteq P$ as the intersection of
all the convex sets that enclose $S$ on $P$. So, in a familiar sense, 
it is the smallest convex set with this property.

In the plane, the convex hull of $S$ can equivalently be defined as the set enclosed by a minimal length curve enclosing $S$.
However, as we will show in Section~\secref{Lmin_rconv}, 
the two notions do not coincide in our context.
Here we focus on the intersection definition.

\begin{ex}
\exlab{Delta}
\exlab{convV}
The convex hull $\conv(V)$ of the vertices $V$ of a doubly-covered triangle 
$\Delta$ is the whole surface $\Delta$.
%
Similarly, the convex hull of all the vertices of any convex polyhedron $P$ is $P$ itself: $\conv(V)=P$.
\end{ex}

\begin{ex}
\exlab{convS}
In contrast to the Euclidean situation, for $S=P \setminus V$ 
we have $\conv(S)=S$, by Lemma~\lemref{conv-pt}, because $P$ is itself convex,
i.e., the vertex holes are not filled-in by the convex hull operation.
\end{ex}

\begin{ex}
\exlab{NoRadon} 
A particular instance of Radon's Theorem on (extrinsically) convex sets states that any set of $r=4$ points in $\Rp$ can be partitioned into two sets whose convex hulls intersect.

The simple example of a tetrahedron shows that no analoguos result holds in our framework.
Indeed, the convex hull of any two vertices is the corresponding edge, so any $2:2$ partion of the vertices provides disjoint convex hulls.
And the convex hull of any three vertices is the corresponding face, so any $3:1$ partion of the vertices also provides disjoint convex hulls.

If, as in Example~\exref{NoHelly},
we again place the tetrahedron over a tall right triangular prism,
then the convex hull of the tetrahedron base vertices is the ``roof'' minus the fourth vertex.
\end{ex}

It remains for future work to study the Radon problem for $r \geq 5$, possibly considering closures of convex hulls:

Also worth studying seems to be the existence of a  Carath\'eodory type theorem in our framework.
These questions form Open Problem~\openref{HellyEtc}.

\bs

The following two properties follow immediately from the definition of $\conv(S)$.

\begin{lm}
\lemlab{conv-elem-prop}
For every $S \subset P$, $\conv(\conv(S))=\conv(S)$.

For every two sets $S, S' \subset P$ with $S'\subset S$, $\conv(S)' \subseteq \conv(S)$.
\end{lm}

The next result is particularly useful for computing the convex hull of finite sets.

\begin{lm} 
\lemlab{conv-split}
For every pair of subsets $S,T$ of $P$ with $T \subset S$,
the following holds:
$$ \conv (S) = \conv \left(  \left(  \conv (S \setminus T ) \right) \cup T \right).$$
\end{lm}

\noindent In words: If $S$ is partitioned into two parts, $T$ and $S \setminus T$,
then the hull of $S$ is the hull of the union of $T$ with the hull of $S \setminus T$.

\begin{proof} 
The set $\conv \left(  \left(  \conv (S \setminus T ) \right) \cup T \right)$ is convex and includes $S$, hence it also includes $\conv (S)$.

On the other hand, the convex set $\conv (S)$ includes $\conv (S \setminus T)$, by Lemma~\lemref{conv-elem-prop}, and also includes $T$.
Hence it includes
$\conv \left(  \left(  \conv (S \setminus T ) \right) \cup T \right)$ as well.
\end{proof}

\bs

The next simple result follows from Lemma~\lemref{out-segm-conv}. 

\begin{lm}
\lemlab{conv_open}
The convex hull of every open set is open.
\end{lm}

\begin{ex}
\exlab{Box}
Let $\Box$ be a doubly-covered square with vertices $a,b,c,d$, and 
let $S=\{a,b,c\} \subset \Box$.
Then $S$ is closed and, as we prove below, $\conv(S)= \Box \setminus \{d\}$, hence $\conv(S)$ is not closed.
Refer to Fig.~\figref{SquareHole}.
\begin{figure}[htbp]
\centering
\includegraphics[width=0.4\linewidth]{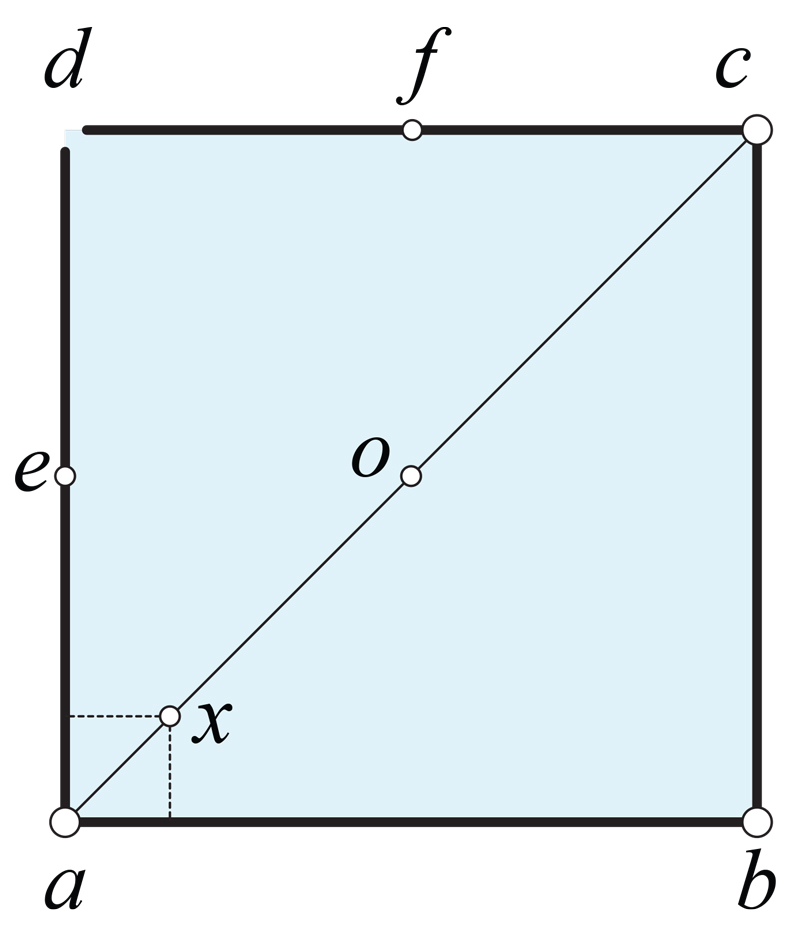}
\caption{$\conv( \{a,b,c\} )$ is $\Box \setminus \{d\}$.}
\figlab{SquareHole}
\end{figure}

Clearly, $\conv(S)$ contains the two triangles $abc$, hence also the simple closed quasigeodesic $Q$ determined by the diagonals $ac$.
Let $x \in Q$ be a point close to $a$ on the front side, and $x'$ its corresponding point on the back side.
Then there are two geodesic segments between $x$ and $x'$, one ``horizontal'' and another one ``vertical''---parallel to $ab$ and $bc$ respectively.
Therefore, there are points of $\conv(S)$ on the edge $ad$, close to $a$.
Moving $x$ continuously towards the center $o$ of $\Box$, shows that the half-edge $ae$ is included in $\conv(S)$, where $e$ is the midpoint of $ad$.
Similarly, $\conv(S)$ contains the half-edge $cf$, where $f$ is the midpoint of $cd$.
Iterating this process establishes that $\conv(S)$ contains halves of the segments $ed$ and $fd$, and so on.
Therefore $\conv(S)$ contains the whole edges $cd$ and $ad$, except
missing the corner point $d$.
\end{ex}

The next result 
will be useful for proving the last statement in Theorem~\thmref{ExtPts}.

\begin{prop}[Dense hull]
\lemlab{dense_conv}
Let $S\subset P$ be a closed convex set with interior points, enclosing strictly less than $2 \pi$ curvature.
Then the convex hull of its complement 
$S'=P \setminus S$ is dense in $P$, and the convex hull of $\partial S$ is dense in $S$.
\end{prop}

\begin{proof} 
By Lemma~\lemref{conv_open}, $\conv(S')$ is open.
Then, by Lemma~\lemref{ag-is-ab}, the boundary of $\conv(S')$ may consist of vertices and convex curves, convex towards $\conv(S')$.
If such a boundary curve existed,
it would enclose more than $2 \pi$ curvature,
contradicting the the Gauss-Bonnet Theorem.


For the second claim, notice that $\conv (\partial S) \subset S$.
Moreover, a continuity argument proves that $\conv (\partial S)$ has interior points.
The remainder is analogous
to Lemma~\lemref{ag-is-ab}, and to the above argument.
\end{proof}

\begin{ex}
\exlab{conv-bd}
Example~\exref{Box} (Fig.~\figref{SquareHole})
can be adapted to illustrate Proposition~\lemref{dense_conv}.
See Fig.~\figref{Penta2x}.
\begin{figure}[htbp]
\centering
\includegraphics[width=0.4\linewidth]{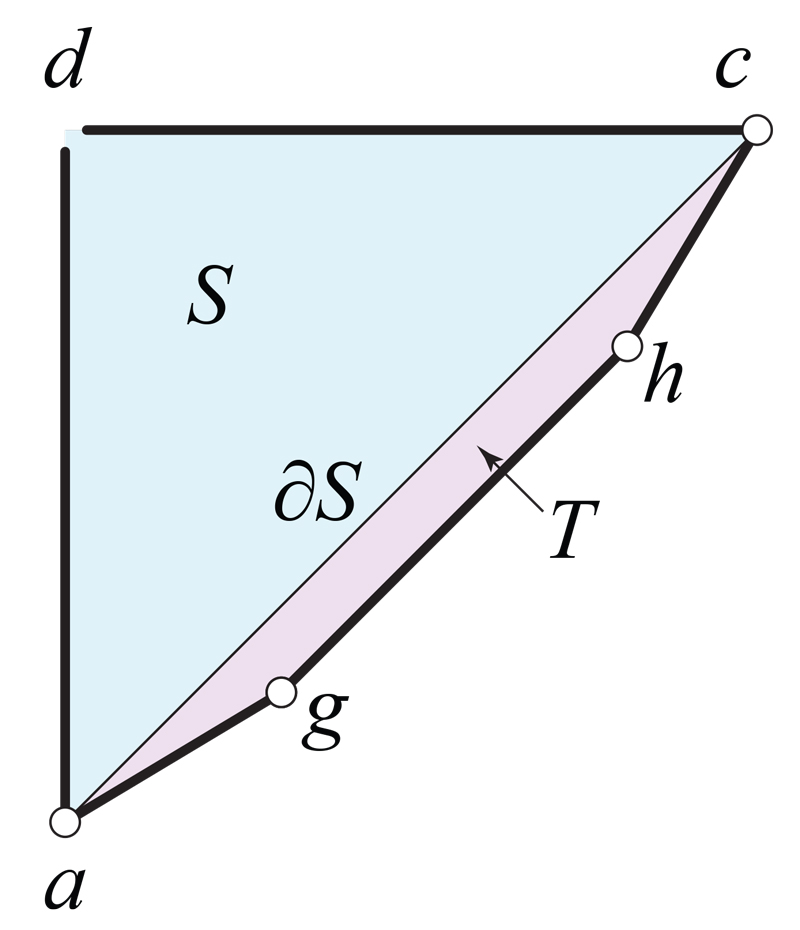}
\caption{$\conv(\partial S)$ is dense in $S$
and $\conv(P \setminus S)$ is dense in $P=S \cup T$.}
\figlab{Penta2x}
\end{figure}

Let $T=ghca$ be an isosceles trapezoid with $gh || ac$, 
$|g-h|<|a-c|$, and the base angle at $g$ almost $\pi$.
Construct, outside $T$, an isosceles right triangle $S=acd$, with $|a-d|=|c-d|$.
Let $P$ denote the doubly-covered pentagon $T \cup S$.

The (two-sided) subset $S$ of $P$
is convex, and $\partial S$ consists of the two geodesic segments $ac$.
As in the proof of Example~\exref{Box}, $\conv(\partial S)= S \setminus \{d\}$ is dense in $S$, and therefore $\conv(P \setminus S)=P \setminus \{d\}$ is dense in $P$.
\end{ex}

These examples show that the convex hull has some desirable properties,
but, for our purposes, some undesirable properties.
The doubly-covered square (Example~\exref{Box})
shows that $\conv(V)$ for $V$ 
a subset of vertices of $P$ could be a set with point holes---so the
convex hull of a closed set can be an open set.
And in the next section we will show that $\conv(S)$ for a set $S$ is not
always the convex hull of the extreme points of $S$.
These considerations lead us to use the
notion of relative convexity
as first introduced in Section~\secref{RelConv} above,
and the relative convex hull $\rconv(S)$,
which will be developed in Section~\secref{RelativeHull} below.

\begin{ex}
\exlab{B_h}
Let $B_h$ be a rectangular box of height $h$, and $Q$ a simple closed geodesic on $B$ parallel to its top and bottom faces, at height $h/2$.
If $h$ is large enough then $Q$ is itself a convex set, and $\conv(Q)=Q$; see Lemma~\lemref{reduction-hull-ag-conv} and its proof for details. 
However, for small $h$, $Q$ is not a convex set, and $\conv(Q) =B$.
This shows that the convex hull of a set $S$ is sensitive to the
surface of $P$ outside of $S$, and justifies the next section.
\end{ex}


\section{Relative Convex Hull}
\seclab{RelConvHull}
The phenomena  illustrated in the previous examples lead us to a new notion.

By ``relative convex hull" $\rconv(S)$ we mean
the intersection of all relative convex sets containing $S$,
with relative convexity as previously defined in Section~\secref{RelConv}.
We will further explore this notion in Section~\secref{RelativeHull} below.

\begin{ex}[Pyramid]
\exlab{Pyramid}
Let $P$ be a pyramid, the top of a regular octahedron.
In Fig.~\figref{PyrRelative_12}(a), both $C_1$ and $C_2$ are
$\a\b$-convex curves.
The convex hull of the corners of $C_1$, 
is the surface above $C_1$, including the apex $v_5$.
However, the convex hull of the corners of $C_2=(v_1,v_2,v_3,v_4)$ is not the surface above $C_2$,
but instead all of $P$. This is because there are geodesic segments that cross
the square base. For example, the shortest path between the midpoints of consecutive
edges of $C_2$ traverse the square base.
Fig.~\figref{PyrRelative_12}(b) shows the construction of $P^\#$ for $C_2$.
Then $\rconv(C_2)$ is the surface above $C_2$, because no shortest
paths will enter the cylinder inserted below $C_2$. 

\begin{figure}[htbp]
\centering
\includegraphics[width=0.75\linewidth]{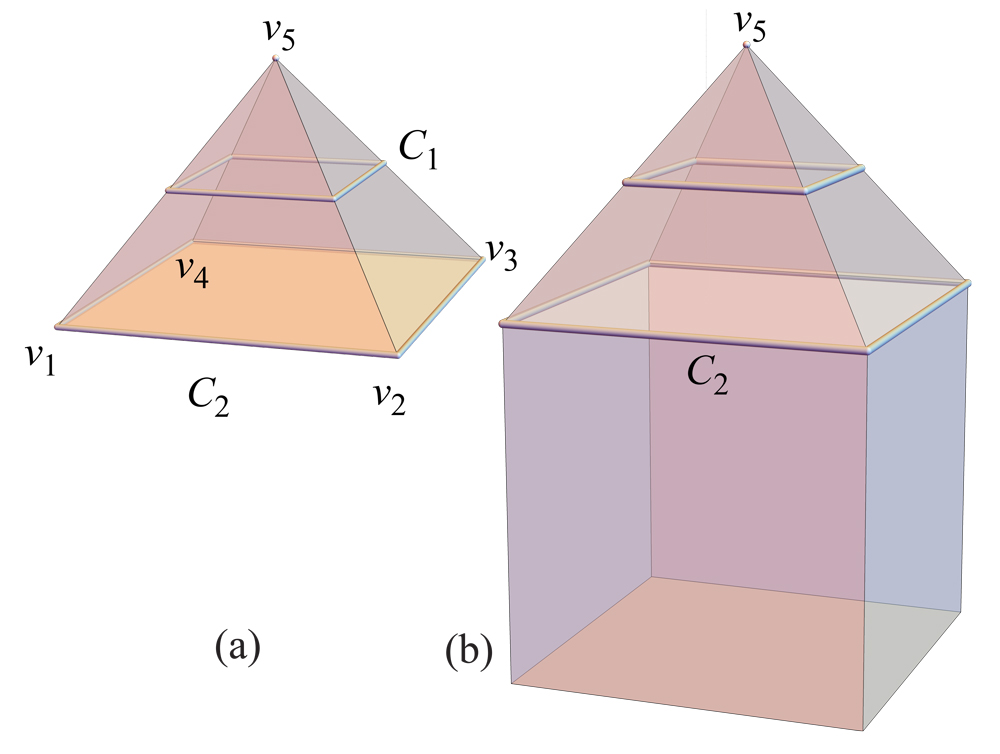}
\caption{%
(a)~$\conv(C_1)$ is the pyramid top, but $\conv(C_2) = P$.
(b)~$P^\#$ for $C_2$.}
\figlab{PyrRelative_12}
\end{figure}
\end{ex}

\begin{ex}
\exlab{Illuminating}
With the following example, we show that:
\begin{itemize}
\squeezelist
\item neither a digon nor a triangle is necessarily convex;
\item Neither $\conv(V)$ nor $\rconv(V)$  is necessarily closed,
for $V$ the set of all vertices inside a simple closed quasigeodesic;
\item the exceptional case in Theorem~\thmref{ExtPts} below may well appear for the closure of $\rconv(V)$, 
which in our case is convex but not the convex hull of its extreme points.
\end{itemize}

Start with the double of the quadrilateral $D=abcd$, where $Q=aba$ is a simple closed quasigeodesic.
See Figure~\figref{Exceptional_rconv}.
We could insert rectangles below $Q$, so that everything happens above $Q$.

\begin{figure}[htbp]
\centering
\includegraphics[width=0.40\linewidth]{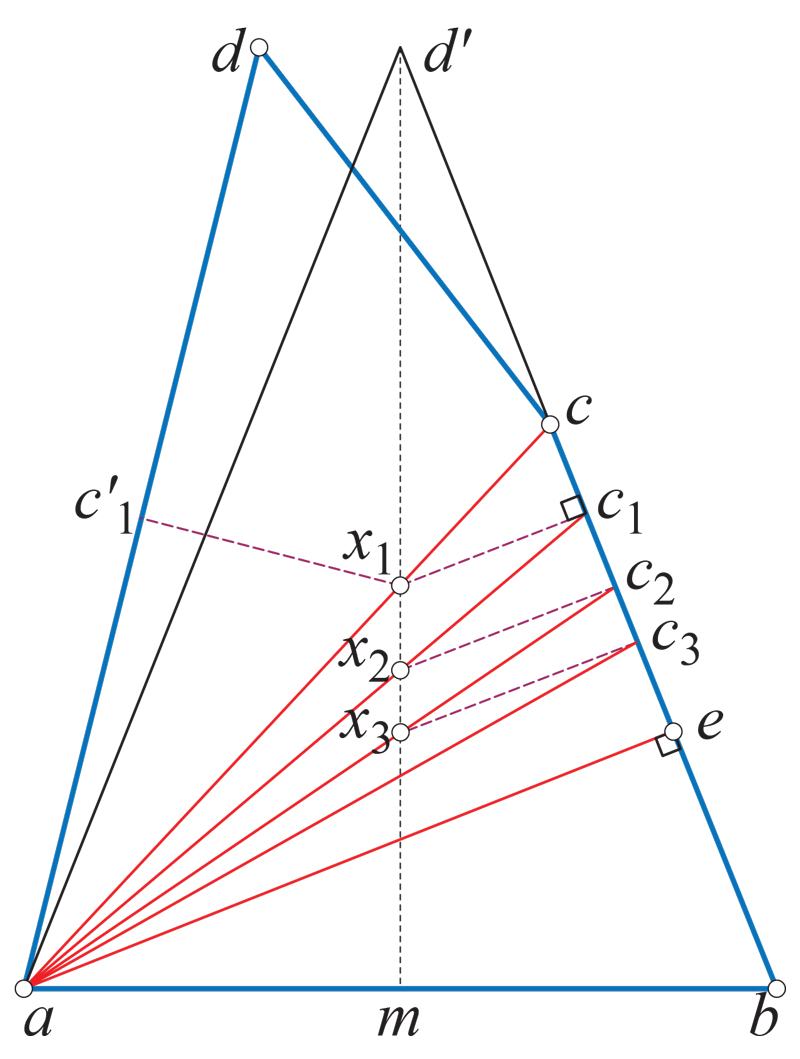}
\caption{$V=\{a,c,d\}$, $\partial \, \rconv(V) = aea$, 
but $ae \setminus \{a\} \not\in \rconv(V)$.}
\figlab{Exceptional_rconv}
\end{figure}

The construction of $D$ follows:
start with an arbitrary line-segment $ab$,
and denote by $m$ the mid-point of $ab$.
Consider an isosceles triangle $d'ab$, and a point $d$ slightly to the left of $d'$.
Take $c$ on $d'b$.

Consider, moreover, $ae$ perpendicular to $d'b$, with $e$ on $d'b$.

Put $V=\{a,c,d\}$.
Then $\rconv(V)$ is the half-surface bounded by the two segments $ae$, 
without those segments (but including $a$).
To see this, let $x_1$ be the intersection point between $ac$ and $d'm$,
and $c_1$ the foot of $x_1$ onto $d'b$.
Then the geodesic segment from $x_1$ to its ``opposite'' point goes through $c_1$, by construction.
So $c_1$ belongs to $\rconv(V)$, and therefore $ac_1$ is included in $\rconv(V)$.
(Notice that $|x_1-c_1|< |x_1 - c_1'|$.)

Now let $x_2$ be the intersection point between $ac_1$ and $d'e$,
and $c_2$ the foot of $x_2$ onto $d'b$.
Then the geodesic segment from $x_2$ to its ``opposite'' point goes through $c_2$, by construction.
So $c_2$ belongs to  $\rconv(V)$, and therefore $ac_2$ is included in $\rconv(V)$.

Iterate, and pass to the limit.

It follows (as one can prove by running the above procedure backward)
that all points above $ae$ belong to $\rconv(V)$, but not $ae$ (see Lemma~\lemref{conv-exc-case}).

Clearly, the digon $aca$ is not convex.
Also, consider that if we ``split'' the vertex $a$ into vertices $a',a''$,
then neither is the triangle $a'a''c$ convex.
\end{ex}


\section{Extreme Points}
\seclab{Extreme}

An \emph{extreme point} of a convex set 
$S \subseteq P$ is a point in $S$ that is not interior to any geodesic segment joining two points of $S$. 
For example, every vertex is an extreme point for every convex set containing it.

Roughly, the Krein–Milman Theorem states that every compact convex set is the extrinsic convex hull of its extreme points.
The closest analogy for ag-convexity we could prove is the following theorem, which will be invoked later.

\begin{thm}
\thmlab{ExtPts}
\lemlab{conv-vert}
Let $S$ be a closed convex subset of $P$ 
with $\partial S$ a closed curve, enclosing strictly less than $2\pi$ total curvature. 
Then either $S$ is the relative convex hull of its extreme points, or $\partial S$ contains a geodesic arc which is not a geodesic segment.
In the former case, the interior extreme points of $S$ are all vertices.
\end{thm}

\noindent
We do not know if the above theorem statement is still true without 
the ``relative" modifier;
Open Problem~\openref{ExtPts}.

\begin{proof}
{\bf Case 1.} If $S \subset P$ has empty interior then
the conclusion follows from Lemma~\lemref{int_convex}, 
which shows that $S$ is either a point, a geodesic arc, or a simple closed geodesic.

\bs

{\bf Case 2.} Assume $S \subset P$ has non-empty interior. 
The proof proceeds by induction over the number of vertices of $P$ interior to $S$.

\medskip

For the base case, assume $S$ has exactly one interior vertex $v$.
By Lemma~\lemref{ag-is-ab}, $\partial S$ is locally isometric to a planar convex curve.
If $\partial S$ contains a geodesic arc which is not a geodesic segment, we 
have established the lemma claim.

So assume now that $\partial S$ contains a geodesic segment, maximal with respect to inclusion, between its corners $v_i$, $v_{i+1}$.
Then we get a geodesic-segment triangle $v v_i v_{i+1}$, which is  the convex hull of its extreme points.
This is valid for all such geodesic-segment arcs in $\partial S$.

If $\partial S$ contains an arc $A$ locally isometric to a planar strictly convex curve, then each point of $A$ is an extreme point, and the conclusion follows.

\medskip

For the general case, again if $\partial S$ contains a geodesic arc that is not a geodesic segment, we are finished. So now we prove that $S$ is
the relative convex hull of its extreme points.

Assume that $S$ has at least two interior vertices, say $v_1$ and $v_2$.
Because they are inside $Q$, $\o_1+\o_2 < 2\pi$.
Let $\g$ be a geodesic segment joining them in $S$.

Merge $v_1$ and $v_2$ to $v_{12}$ along $\g$ to obtain a new surface $P'$, hence 
where $P'$ is $P$ cut open along $\g$ and the union $T$ of two twin triangles is inserted.
The set $S$ with $T$ inserted yields $S'$, a relatively 
convex set of $P'$ (Lemma~\lemref{conv-vm}).

Because $\partial S' = \partial S$ by Lemma~\lemref{out-segm-conv},
$S'$ is closed, because $S$ is closed.

By the induction hypotheses, $S'$ is the relative convex hull of its extreme points $E'$ on $P'^ \#$.
Clearly, the only extreme point of $S'$ in $T$ is $v_{12}$. 
Put $E''= E' \setminus \{v_{12} \}$ and, on $P$, $E = E'' \cup \{v_1,v_2\}$.
It suffices to show that (i) $E$ is the set of extreme points of $S$, and (ii) $S = \conv (E)$.

(i) Since $E''$ is included in the part of $S'$ isometric to $S$, each point in $E''$ in also extreme for $S$, 
as are the vertices $v_1$, $v_2$.

(ii) One can easily see that each convex set on $P'$ containing $E'$ corresponds, via digon-tailoring $T$, to a convex set on $P$ containing $E$, and conversely.
So the intersection of all such sets on $P'$ corresponds to the intersection of their correspondents on $P$, and conversely.
Consequently, $S= \rconv(E)$.

\medskip

The last claim of the theorem---that the interior extreme points of $S$
are all vertices---follows from Proposition~\lemref{dense_conv} and its proof.
\end{proof}

\noindent Note that throughout the induction, the boundary of all the
sets remains fixed at $\partial S$.

\medskip

The next example illustrates the exceptional situation of Theorem~\thmref{ExtPts}.


\begin{ex}[Extreme Points]
\exlab{ExPts}
Consider a ``tall'' triangular prism with top $abc$ and base $a'b'c'$.
Now add a point $v$ on the line containing $b b'$, above $b$, and
let $P$ be the boundary of the extrinsic convex hull of $\{a',b',c',a,c,v\}$.
See Fig.~\figref{TriPrism}.
Then $Q=abc$ is a simple closed quasigeodesic and the ``roof'' above 
$Q$---all the surface of $P$ above $Q$---is our convex set.
(Here we need the prism to be sufficiently 
tall, but for clarity the figure is more squat.)
Take points $x,y \in P$ with $x$ on $ab$ close to $a$ and $y$ on $bc$ close to $c$.
Then the arc of $Q$ from $x$ to $y$ through $b$ is not a geodesic segment. 

Note that $b$ is not an extreme point of $S$, because $b$ is interior to 
geodesic segments connecting points on $ab$ and $bc$ that are close to $b$.
The only extreme points of $S$ are $a$, $c$ and $v$, 
and so $S$ is not the convex hull of its extreme points.
This is the exception in Theorem \thmref{ExtPts}.
\end{ex}

\begin{figure}[htbp]
\centering
\includegraphics[width=0.75\linewidth]{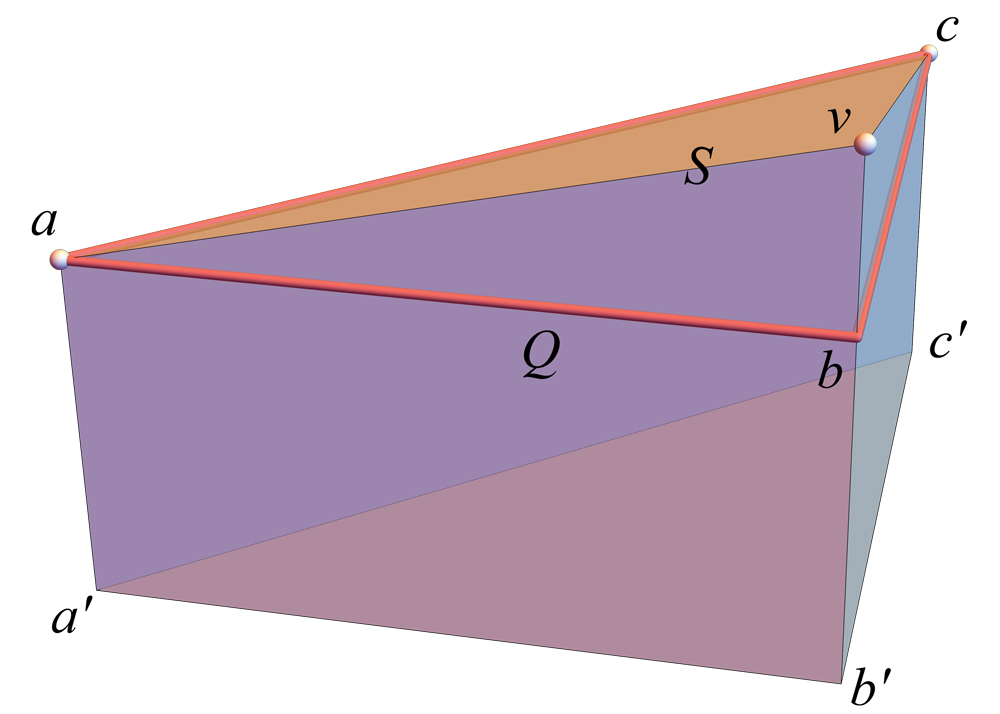}
\caption{The region above the simple closed quasigeodesic $Q$ is convex, but is not the convex hull of its extreme points $a,c,v$.}
\figlab{TriPrism}
\end{figure}

\begin{ex}[Simple closed geodesic]
\exlab{ExtPts}
For the simple closed geodesic $Q$ on $B_h$, the rectangular box of height $h$ considered in Example~\exref{B_h},
$\conv(Q)=Q$ holds, and there are no extreme points on $Q$.
\end{ex}

Denote by $\ext (S)$ the extreme points of the convex set $S$.

\begin{lm}
\lemlab{conv-ext} 
For any set $S \subset P$, $\ext(\conv(S)) \subset S$.
\end{lm}

Notice that not all extreme points of the closure of $\conv(S)$ are in $S$, as Example~\exref{Box} shows.

\begin{proof}
Assume there exists a point $x$ in $\ext(\conv(S)) \setminus S$.
Then, since $x$ is 
not interior to any geodesic segment joining two points of $\conv(S)$
(because it is extreme),
the set $\conv(S) \setminus \{x\}$ is still convex (see the proof of Lemma~\lemref{conv-pt}).
So we have $S \subset \conv(S) \setminus \{x\}  \subset \conv(S)$ and $\conv(S) \setminus \{x\} \neq \conv(S)$, contradicting the minimality of $\conv(S)$.
\end{proof}

\begin{lm}
\lemlab{conv-exc-case}
For any $S \subset P$, if $\conv(S)$ contains a boundary
geoarc which is not a geoseg, then $S$ also contains that arc. 
\end{lm}

\begin{proof}
Assume the contrary be true, and let $A \subset \conv(S)$ be a maximal (with respect to inclusion) such geoarc not in $S$. So
$A \subset \partial (\conv(S)) \setminus S$.
Remove $A$ from $\conv(S)$, but leave its extremities, and denote by $K$ the resulting set.
It follows, just as in the proof of Lemma~\lemref{conv-pt}, that $K$ is a convex set including $S$ and strictly included in $\conv (S)$, 
a contradiction to the inclusion-minimality of $\conv(S)$. 
\end{proof}

Example~\exref{Illuminating} also illustrates this result. Indeed, in Fig.~\figref{Exceptional_rconv}, the arc $aea$ is a boundary geoarc of $\rconv(\{a,c,d\}$), but not a geoseg. And that arc is not included in $\{a,c,d\}$, hence neither is it included in $\rconv(\{a,c,d\})$.


\section{Relative Convex Hull of Vertices}
\seclab{RelativeHull}

We return to
the notion of \emph{relative convexity} and the \emph{relative convex hull},
used in the proof of Theorem~\thmref{ExtPts}.
For our purposes 
we only consider the relative convex hull of vertices $V$
that fall to one side of (i.e., \emph{inside}, or above) a simple closed quasigeodesic $Q$.

\paragraph{Notation.} 
\noindent
\begin{enumerate}[label={(\arabic*)}]
\item $Q$ quasigeodesic. 
Orient $Q$ counterclockwise (ccw), and let
$P^+$ be the half-surface to the left of (above) $Q$.
$P^\#$ is $P^+$ union the unbounded cylinder to the right of (below) $Q$.
\item $V$: Set of (positive curvature) vertices inside or on $Q$.
\item For $C$ any simple closed curve in $P^+$,
let $R(C)$ be the region of $P^+$ to the left of $C$. So $\partial R(C) = C$.
This is well-defined because $C$ is a simple curve.
And because it is oriented ccw, $R(C) \subset P^+$.
\end{enumerate}

Given $Q$ oriented so that $V$ is to $Q$'s left, define 
a polyhedron $P^\#(Q) = P^\#$ as 
in Section~\secref{RelConv}:
Cut $P$ along $Q$ and insert below $Q$ a sufficiently 
tall cylinder; for example, a cylinder of height equal to $\textrm{diam}(P)$.
Because $Q$ is convex to both sides, AGT implies that $P^\#$
is a convex polyhedron.

\medskip

Note that $Q$ is strictly $\a\b$-convex on $P^\#$, because $\b_i=\pi$ at all 
corners\footnote{Of course, the vertices of $Q$ with $\a =\pi$ are flattened when passing to $P^\#$.}
of $Q$, and corners have positive curvature.

\medskip

Now we define the \emph{relative convex hull $\rconv(V)=H$} of $V$ to be the intersection of
all the convex sets that enclose $V$ on $P^\#$.
This is the same notion earlier explored in Section~\secref{RelConvHull},
but here we focus on vertices inside $Q$.

\begin{lm}
\lemlab{reduction-hull-ag-conv}
Let $V$ be a set of points on $P$, inside a simple closed quasigeodesic $Q$.
The relative convex hull $\rconv(V)=H$ of $V$ can be obtained 
from only employing geodesic segments, and consequently convex sets, inside $Q$ 
(instead of constructing $P^\#$).
\end{lm}

\begin{proof}
Let $x,y \in H$ be two points in $H$, and $\g$ a geodesic segment between them.
We argue that $\g$ cannot cross $Q$.

When the cylinder attached below $P^\#$ is unfolded flat, it forms a rectangle $R$.
Suppose $\g$ properly enters $R$ at $x'$ and exits at $y'$. Then the segment
$x' y'$ along the top of $R$ is shorter, a contradiction to our assumption.
(Note $x' y'$ is a portion of $Q$, and it is possible that $H$ shares that
portion of $Q$.)
\end{proof}

A \emph{node} $n$ of a geodesic polygon $N$ 
is a point of $N$ interior to no geodesic subarc of $N$.
Call it a \emph{g-node} if
at least one side of $N$ incident to $n$ is not a geoseg
(instead a geoarc), and a \emph{gs-node} otherwise.

Recall that Lemma~\lemref{ag-is-ab} showed that, if $S$ is closed and convex,
then $\partial S$ is strictly $\a\b$-convex;
and if $S$ has interior points
then $\partial S$ is $\a\b$-convex but not necessarily strictly.
We now explore to what extent there is a converse to this lemma,
a result we need to compute $\rconv(S)$.

\begin{lm}[$\a\b$ converse]
\lemlab{Converse_ab}
Let $S$ be a set satisfying these conditions:
(1)~$S \subset R(Q)$ contains $V$, (2)~$\partial S$ is a geodesic polygon,
and (3)~any geoarcs of $\partial S$ which are not geosegs, are not included in $S$.

Then $S$ is relatively convex if and only if $\partial S$ is 
strictly $\a\b$-convex at each gs-node, and $\a\b$-convex at each g-node.
\end{lm}

\begin{proof} 
If $S$ is relatively convex, the statement is covered by 
the $\a\b$-convexity Lemma~\lemref{ag-is-ab}.
The remainder of the proof establishes the converse:
That if $\partial S$ is $\a\b$-convex 
at the gs- and g-nodes as above, then $S$ is relatively convex.

\medskip

Consider first two boundary points $x,y \in \partial S$ 
(the roles $x$ and $y$ will play, will be decided later)
and let $\g'$ be a geoseg joining them,
with $\g'$ not included in $\bar S = S \cup \partial S$.
We now show that there is a subarc $\g$ of $\g'$ (possibly $\g=\g')$
that lies completely
external to the interior $\mathring S$ of $S$.

Notice that $\g'$ intersects $\partial S$ finitely many times.
Otherwise, since $\partial S$ is formed by geoarcs, hence locally by geosegs,
we could find arcs $A$ of $\partial S$ arbitrarily small, having common extremities with subarcs of $\g'$.
But those arcs would be ramifying geosegs, impossible.

Therefore, possibly replacing $\g'$ with a subarc 
$\g$, we may assume that $\g$ lies completely
external to the interior $\mathring S$ of $S$.

Let $F$ be the region of $P$ bounded by $\g$ and the arc $\partial^*S$ of $\partial S$ from $x$ to $y$, so that $F$ doesn't include $S$.
Then $\mathring F$ contains no vertex, by hypothesis.

Now we are going to identify a geoarc $g$ from $x$ to $y$ inside $\bar S$.
The boundary of $S$ is, in particular, a convex curve, hence we may construct with AGT the double $S^{\#}$ of $\bar S$.
Let $g$ be a geoseg on $S^{\#}$ from $x$ to $y$. It is included in a half-surface of $S^{\#}$, because geosegs do not branch.
So we may assume that $g$ is in $\bar S$.

\smallskip

We prove next that $g$ is shorter than or equal to $\g$: $\ell(g) \le \ell(\g)$.
Notice that the cut locus $C(x)$ of $x$ on $S^{\#}$ doesn't intersect $g$, which is a geoseg on $S^{\#}$.
On $S^{\#}$, put $\{c_1, \ldots,c_k \} = C(x) \cap \partial \bar S$.
Each point $c_j$ is the extremity of two digons of $S^{\#}$ at $x$ (one on each copy of $\bar S$).
Remove all those digons from $\bar S \subset S^{\#}$ and zip close the result by AGT.
We get a doubly covered flat surface, by Lemma~\lemref{Path}. 
Denote by $S^*$ the part of it bounded by $\partial^*S$ and $g$ which doesn't include $\partial S$.

We compare now $F$ and $S^*$, which are both flat (hence planar) polygons.
Clearly, the angles $\a^*$ of $\partial^* S$ towards $S^*$ are at most $\pi$, and they are smaller or equal to those towards $F$, denoted by $\b^*$.
(For convenience, we suppress the indices of $\a^*$ and $\b^*$.)
By the extension of Cauchy's Arm Lemma 
(Theorem~\thmref{CAL}). 
we obtain that $g$ is shorter than or equal to $\g$.


Next we will consider separately the cases of $x$ or $y$ (or both) belongs to $S$.
That case-distinction and the non-strict inequality will imply convexity, via strictness or non-strictness of $\a\b$-convexity.

\emph{Case~1}. $x$ and $y$ in $S$.
The equality case appears if $\a^*=\b^*$ everywhere, meaning that $\partial^* S$ is outside $S$.
So if $x,y \in \partial S \cap S$ 
the inequality 
$\ell( g ) < \ell (\g )$ 
is strict because of the strictness of $\a\b$-convexity, hence all geosegs between them are in $S$.

\medskip

\emph{Case~2}.
Consider now two points $x',y' \in \mathring S$.
From these two points we will obtain $x$ and $y$.
Assume $\g'$ exits $S$.
Because $\partial S$ is a convex curve towards $S$, $\g'$ must intersect it at least two times.
Choosing two such points $x,y$, the above argument implies the ramification of $\g'$ assuming 
$| g | = | \g |$,
hence not strict $\a\b$-convexity, impossible.
So $\g'$ lies inside $\mathring S$.

\medskip

\emph{Case~3}.
The case $x \in \partial S$ and $y' \in \mathring S$ can be treated analogously,
completing the proof.
\end{proof}

\begin{ex} 
\exlab{No_converse_ab}
The statement of Lemma~\lemref{Converse_ab} is not necessarily true if 
$S$ does not include all of $V$: $V \setminus S \neq \varnothing$.

\begin{enumerate}[label={(\alph*)}]
\item 
To see this, take the 
quadrilateral $H= abcd$ formed by three equilateral triangles 
$obc, ocd, oda$, and let $P'$ be the double of $H$.
Let $P$ be obtained from $P'$ by extending it below by long rectangles.
See Figure~\figref{relconv_NotConverse}(a).
$Q=aobo'a$ is a simple closed quasigeodesic on $P$  (with $o'$ ``opposite'' to $o$).
Let $S=R(oco'do)$ be the closed set constituted by $\triangle ocd$ and $\triangle o'cd$.
Observe that the boundary of $S$ is not strictly $\a\b$-convex:
both the angles $\a$ and $\b$ at $c$ are $120^\circ$ (and similarly at $d$).
Then by Lemma~\lemref{ag-is-ab}(i), which established that
closed $S$ are strictly $\a\b$-convex, $S$ is not convex.
Indeed, there is a geoseg outside $S$ from $x$ to $x'$ on the back, for $x$ on $oc$ 
(and similarly for points on $od$).
And for the same reason, 
there are geosegs from $o$ to $o'$ outside $S$, namely $oeo'$ and $ofo'$.
%
\begin{figure}[htbp]
\centering
\includegraphics[width=1.0\linewidth]{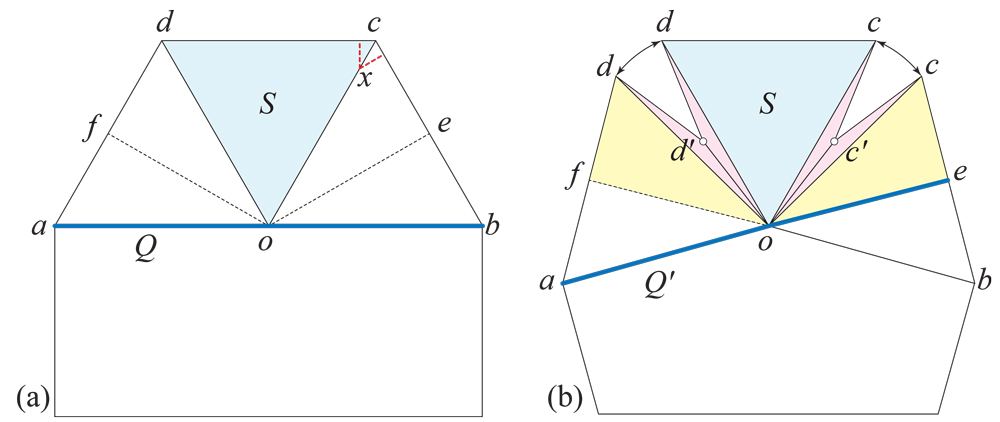}
\caption{$S$ is not convex.
(a)~Lemma~\lemref{Converse_ab} doesn't apply because $\partial S$ is not 
strictly $\a\b$-convex at $c$ and $d$. 
(b)~After insertion of digons, 
$\partial S$ is strictly $\a\b$-convex, but Lemma~\lemref{Converse_ab} 
does not hold in this case because  $S$ doesn't include $c'$ and $d'$.
}
\figlab{relconv_NotConverse}
\end{figure}
%
\item
Now we show that even altering the example so that every gs-node 
on $\partial S$ is strictly $\a\b$-convex, does not always permit
the converse conclusion that $S$ is convex.
Modify $P$ by inserting
two digon doubled triangles  along $oc$ and $od$; see Figure~\figref{relconv_NotConverse}(b).
Each triangle has base angles $7.5^\circ$, and so adds $15^\circ$ to 
$c$, and $d$, 
and adds a total of $30^\circ$ of angle at $o$. 
Cut out $30^\circ$ at $o$ below $o$, so that $o$ is again flat.
Now the $Q$ used above is no longer a quasigeodesic, but
$Q'=aoeo'a$ is a simple closed quasigeodesic.
The effect of the digons is two-fold. 
First, $c$ and $d$ now are both strictly $\a\b$-convex, and
$\partial S$ remains strictly $\a\b$-convex at $o$ and at $o'$.
So $S$ is closed, $\partial S$ is a geoseg polygon,
strictly $\a\b$-convex at every gs-node (and it has no g-nodes).
So $S$ satisfies all but one condition of Lemma~\lemref{Converse_ab},
which would allow concluding that $S$ is convex.
The missing condition is that $V \setminus S \neq \varnothing$:
$S$ does not include the two digon vertices---call them
$c'$ and $d'$---which are now part of $V$.
And again the geosegs $oeo'$ and $ofo'$ lie outside of $S$,
showing that indeed $S$ is not convex.
\item
Next, alter $S$ to include the digon vertices $c'$ and $d'$:
$S' = R( o c' c o' d d' o )$. So now $V \setminus S = \varnothing$.
But $\partial S$ is not $\a\b$-convex at $c'$ and $d'$: they
both have $\a=\b=165^\circ$. So again the preconditions for 
Lemma~\lemref{Converse_ab} are not satisfied, and $S'$ is not convex.
\item 
If $S = R( o e o' f o )$, then $o$ is a g-node of $\partial S$,
which is strictly $\a\b$-convex
at $o$: $\a=150^\circ, \b=210^\circ$.
However, the geoarcs $oeo'$ and $o' f o$ must not be included in $S$
for Lemma~\lemref{Converse_ab} to apply.
However, $S$ is nevertheless relatively convex.
\item
If $S=Q=aoeo'a$ in  Figure~\figref{relconv_NotConverse}(b), then $\partial S$ is not strictly $\a\b$-convex at $o$
($\a=\b=180^\circ$),
but again the lemma does not apply because $\partial S$ contains the geoarc $oeo'$.
Again, $S$ is relatively convex.
\end{enumerate}
\end{ex}

The next result gathers the proven facts about $\rconv(V)$.

\begin{thm}[$\rconv(V)$]
\thmlab{rconvV}
Let $V$ be the set of all (at least three) vertices of $P$, inside a simple closed quasigeodesic $Q$.
$V$ might also include some, but not all, vertices on $Q$, such that $\sum_{v \in V} \o_v < 2 \pi$.

Then $H=\rconv(V)$ is simply-connected and its boundary $\partial H$ is a geodesic 
polygon, strictly $\a\b$-convex at gs-nodes and $\a\b$-convex at g-nodes, on $P^\#$.
Moreover, $\ext(H) = V$.
\end{thm}
It will be established in
Theorem~\thmref{Walgorithm} 
that $H=\rconv(V)$ can be constructed in time $O(n^5 \log n)$ where $n = | V |$.

\begin{proof}
By Lemma~\lemref{conv-ext}, $\ext(\rconv V) \subset V$, and clearly $V \subset \rconv V$, hence $\ext(H) = V$.
Therefore, each point in $\partial \, \rconv(V) \setminus V$ is interior to a geoarc, so $\partial rconv(V)$ is a geodesic polygon.
The $\a\b$-convexity follows now from Lemma~\lemref{Converse_ab}.

The simply-connectedness follows from the algorithmic construction in Section~\secref{WAlgorithm}, justified by  Lemma~\lemref{Converse_ab}.
\end{proof}

\begin{ex}
\exlab{ext(H)neqV}
Notice that the claim that $\ext(H) = V$
may not hold when $V$ does not include
all vertices of $P$ inside $Q$.
This can be seen by adapting
Example~\exref{Box}, Fig.~\figref{SquareHole},
as follows.

Imagine there is another vertex $a'$ close to $a$, above $a$, and almost flat.
(This vertex $a'$ is added so that $| V | \ge 3$.)
Then, in the upper-left half-surface $S$ bounded by the simple closed quasigeodesic $Q=aca$, take 
$V=\{a,a',c\}$ (i.e., skip one vertex $d$ inside $Q$).
Then $H=\conv(V)=S \setminus \{d\}$, by Proposition~\lemref{dense_conv}.
Therefore, $H$ is not simply-connected and $\partial H$ is not a geodesic-segment polygon.
\end{ex}

%



\section{Summary of Properties}
\seclab{Summary}
We attempt to succinctly summarize below the facts about convexity
established in this chapter,  
either in lemmas and theorems, or via examples.
The goal is to help the reader interested in applications of this theory.
Please note that, unless otherwise specified, ``convexity" means ``ag-convexity."

\subsubsection{Ag-Convexity}  
\begin{enumerate}[label={(\arabic*)}]
\item $S$ closed convex $\implies$ Either point, arc, geodesic, or has interior points.
Lemma~\lemref{int_convex}.
\item Closure $\bar{S}$ of convex $S$ $\centernot\implies$ $\bar{S}$ convex. Example~\exref{Closure_not_cvx}.

\item $S$ closed convex $\implies$ $\partial S$ a convex curve, and if a geodesic polygon, then $\a\b$-convex.
Lemma~\lemref{ag-is-ab} 

\item $S$ closed convex $\implies$ $S$ encloses $\le 2 \pi$ curvature.
\item $S$ closed convex $\implies$ simply-connected, or cylinder.
Lemma~\lemref{si-con_convex sets}.

\item $S$ convex and $v \in S$ a vertex $\implies$ $S \setminus \{v\}$ convex.
Lemma~\lemref{conv-pt}.

\item $S$ convex and 
not closed $\centernot\implies$ simply-connected.
Example~\exref{S-v}.

\item No Helly type theorem for $h=3$. Example~\exref{NoHelly}.
\item $S$ closed convex and includes $Q$ $\implies$ cylinder, or $S=P$.
Lemma~\lemref{S>Q}.
\end{enumerate}

\subsubsection{Geodesic Segments and Convex Sets}

\begin{enumerate}[resume,label={(\arabic*)}]
\item $S$ convex $\implies$ $\g$ connecting interior points $x_1,x_2$
does not meet $\partial S$. Lemma~\lemref{out-segm-conv}.
\item $S$ convex,  $x$ interior to $S$ and $y \in \partial S$ $\implies$ all geodesic segments from $x$ to $y$ are interior to $S$, possibly excepting their extremity $y$.
Lemma~\lemref{out-segm-conv}.
\item $S$ convex, $x,y \in \partial S$ then there is a geodesic segment between them included in the closure $\bar S$ of $S$. Lemma~\lemref{out-segm-conv}.
\item $S$ convex; $x \in \partial S$ is not itself a component of $\partial S$ then there exists a geodesic segment starting at $x$ and 
exterior to $S$. Lemma~\lemref{out-segm-conv}.
\item  $S$ convex $\implies$ the interior of $S$ is convex. Corollary~\lemref{int-cvx}.
\item  $S \neq P$ convex with interior points, $x \in \partial S$ $\implies$ supporting angle at $x$. Lemma~\lemref{support_angle}.
\end{enumerate}

\subsubsection{Relative Convexity}
\begin{enumerate}[resume,label={(\arabic*)}]
\item $S$ convex and $v_1,v_2$ merge  $\implies$ $S'$ relatively convex.
Lemma~\lemref{conv-vm}.
\item $S$ convex and $v_1,v_2$ merge  $\centernot\implies$ $S'$ convex.
Example~\exref{Quad}.
\end{enumerate}


\subsubsection{Convex Hull Properties: $\conv(S)$}

\begin{enumerate}[resume,label={(\arabic*)}]
\item $V$ vertices of $P$
$\implies$ $\conv(V)=P$. Example~\exref{convV}.
\item  $\conv(P \setminus V)=P \setminus V$.  Example~\exref{convS}.
\item No Radon type theorem for $r=4$. Example~\exref{NoRadon}.
\item For $S'\subset S$, $\conv(S') \subseteq \conv(S)$. Lemma~\lemref{conv-elem-prop}.
\item For $T \subset S$, $ \conv (S) = \conv \left(  \left(  \conv (S \setminus T ) \right) \cup T \right)$. 
 Lemma~\lemref{conv-split}.
\item $S$ open $\implies$ $\conv(S)$ open.
Lemma~\lemref{conv_open}.
\item $S$ closed $\centernot\implies$ $\conv(S)$ closed.
Example~\exref{Box}.
\item $S$ closed, convex, with interior points, enclosing $< 2 \pi$ curvature
$\implies$ $\conv (P \setminus S)$ dense in $P$, $\conv(\partial S)$ dense in $S$.
Proposition~\lemref{dense_conv}.
\item For box $B_h$ and $Q$: $\conv(Q) = Q$ or $\conv(Q) = B$, depending on $h$.
Example~\exref{B_h}.
\end{enumerate}


\subsubsection{Relative Convex Hull: $\rconv(S)$}
\begin{enumerate}[resume,label={(\arabic*)}]
\item Convexity and relative convexity can differ. Example~\exref{Pyramid}.
\item 
Neither a digon nor a triangle is necessarily convex; $\conv(V)$ is not necessarily closed;
the closure of $\rconv(V)$ is not the convex hull of its extreme points. Example~\exref{Illuminating}.
\end{enumerate}


\subsubsection{Extreme Points: $\ext(S)$}
\begin{enumerate}[resume,label={(\arabic*)}]

\item $S$ closed convex with $\partial S$ enclosing $< 2 \pi$ curvature.
Either $S = \rconv(\ext(S))$, or $S$ contains a non-segment geoarc.
Theorem~\thmref{ExtPts}. Exception: Examples~\exref{Illuminating},~\exref{ExPts}.
\item $S$ closed, convex and $\conv(S) \neq \conv(\ext(S))$.
Example~\exref{ExPts}.

\item $S$ closed convex with $\ext(S) = \varnothing$. Example~\exref{ExtPts}.

\item $\ext(\conv(S)) \subset S$. Lemma~\lemref{conv-ext}.

Not all  extreme points of the closure of $\conv(S)$ are in $S$. Example~\exref{Box}.
\item $\partial \conv(S)$ contains a 
non-segment geoarc only if $S$ does so. Lemma~\lemref{conv-exc-case}.
\end{enumerate}

\subsubsection{Relative Convex Hull of Vertices}
\seclab{RelativeHullVerts}

\begin{enumerate}[resume,label={(\arabic*)}]
%
%
\item $S \subset R(Q)$, points $V$ in $S$: $\rconv(S)$ can be constructed without $P^\#$.
Lemma~\lemref{reduction-hull-ag-conv}.

\item  $S \subset R(Q)$, $V\subset S$,  $\partial S$  $\a\b$-convex  $\implies$ $S$ relatively convex.
Lemma~\lemref{Converse_ab}.

\item Lemma~\lemref{Converse_ab} is not true if $V \setminus S \neq \varnothing$. Example~\exref{No_converse_ab}.

%
\item 
$V$ all vertices inside $Q$, + some on $Q$  $\implies$
$H=\rconv(V)$ simply-connected, $\partial H$ a geodesic polygon, $\a\b$-convex on $P^\#$.
$\ext(H)=V$. Theorem~\thmref{rconvV}.
\item
$\ext(H) \neq V$ if $V$ does not include
all vertices inside $Q$.  Example~\exref{ext(H)neqV}.
\end{enumerate}

\chapter{Minimal-length Enclosing Polygon}
\chaplab{MinEnclPoly}
In this chapter we explore the minimial-perimeter polygon $Z$ enclosing
a set of vertices $V$ in or on a quasigeodesic $Q$.
We derive its key properties, and provide a polynomial-time algorithm
for constructing it (Section~\secref{ShortAlg}).
We then show that this polygon is not always the same
as $\partial \, \rconv(V)$ 
(Example~\exref{minell_neq_rconv}).
The algorithm for $Z$ also works to construct $\partial \, \rconv(V)$ with minor modifications
(Section~\secref{WAlgorithm}).

In the next chapter we will show that either of these notions
can support an algorithm that ensures that sequential vertex-merging cuts
do not disconnect $P^+$. 


\section{Properties of the Minimal Enclosing Polygon}
\seclab{ZProperties}

\paragraph{Notation.}
\noindent
We repeat some previous notation and introduce new notation to be used subsequently.
We again use the abbreviation \emph{geoseg} to mean 
``geodesic segment" and \emph{geoarc} to mean a simple geodesic
that may or may not be a geoseg.

\begin{enumerate}[label={(\arabic*)}]
\item $Q$ quasigeodesic. 
Orient $Q$ counterclockwise (ccw), and let
$P^+$ be the closed half-surface to the left of (above) $Q$.
$P^\#$ is $P^+$ union the unbounded cylinder to the right of (below) $Q$.
\item $V$: Set of (positive curvature) vertices inside or on $Q$.
\item For $C$ any simple closed curve in $P^+$,
let $R(C)$ be the closed region of $P^+$ to the left of $C$
(``R" for region).
So $\partial R(C) = C$.
This is well-defined because $C$ is a simple curve.
And because it is oriented ccw, $R(C) \subset P^+$.
\item For $C$ any curve on $P$, $\ell (P)$ denotes its length. 
\end{enumerate}

\paragraph{Definition of $\min \ell$.} 
For a finite set $V$ of points in the plane, $\conv (V)$ is precisely the set bounded by the minimal length curve enclosing $V$.
We adapt this property for vertices $V \subset P^+$.
Toward that end, we define $Z=Z(V)=\min \ell[V]$ to be the 
minimal-length geodesic polygon enclosing $V$.
Here \emph{enclosing} is to the left when $Z$ is oriented ccw.
We must also stipulate that 
``enclosing" means that the interior of the region between $Z$ and $Q$ is
empty of vertices.
Otherwise a clockwise traversal of the boundary of any triangle $\triangle \subset P^+$
would enclose $V$ to its left, and in fact would enclose all of $P \setminus \triangle$.
Also, for a simple closed geodesic enclosing $V$, its position is not
fixed by the above requirements, in the sense that it could be ``slid" parallel to itself while remaining 
a geodesic of the same length. In such a case, we define $Z$ to be
``lifted" to a position that touches a vertex.

Each edge of the geodesic polygon $Z$ is a geoarc but not necessarily a geoseg.
See, e.g., Fig.~\figref{ZZnested}(b).
We view polygons, and so $Z$, as the boundary, not the region enclosed,
which is denoted by $R(Z)$.

Here we concentrate on properties of $Z$ and
algorithms to construct $Z$.
We stress that all the following considerations take place on $P^\#$ and not on $P$.
For simplicity, we will omit in the rest of this section the modifier ``relative,'' which would refer precisely to $P^\#$.

We first establish three basic properties of $Z=Z(V)=\min \ell[V]$.

\begin{lm}[$\a\b$-convex]
\lemlab{Zstrictly}
$Z$ is strictly $\a\b$-convex.
\end{lm}
\begin{proof}
At vertices  $v \in Z \cap Q$, $\b = \pi$ on $P^\#$.
On the other hand, at vertices $v \in Z \setminus Q$ with $\b < \pi$ one could further shorten $Z$, 
a contradiction to minimality.

Since $\b \geq \pi$ and $\a+\b < 2 \pi$ at vertices, $\a < \pi \leq \b$.
\end{proof}

\begin{lm}[$Z$ simple]
\lemlab{Zsimple}
$Z$ is simple, i.e., non-self-crossing.
\end{lm}
\begin{proof}
Suppose $Z$ is not simple. At each self-crossing, clip off the portions of
the geodesics toward the inside. The resulting outermost geodesic polygon
still encloses all of $V$, and is shorter.
See Fig.~\figref{SimpleGeoPoly}.
\begin{figure}[htbp]
\centering
\includegraphics[width=0.75\linewidth]{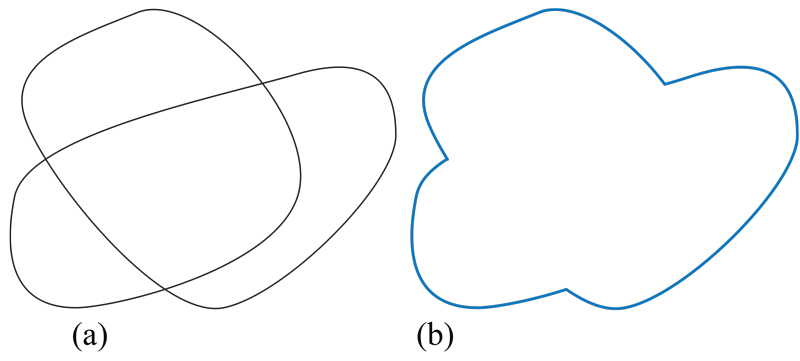}
\caption{(a) Non-simple polygon. (b)~Outermost arcs.}
\figlab{SimpleGeoPoly}
\end{figure}
\end{proof}

\begin{lm}[$Z$ Unique]
\lemlab{Zunique}
Let $Z$ be a shortest geodesic polygon enclosing $V$.
$Z$ is unique, up to isometry.
\end{lm}
\begin{proof}
Assume there is another geodesic polygon $Z'$ also enclosing all of $V$ and
of the same length: $\ell ( Z ) = \ell ( Z' )$.
Then either one is enclosed in the other, or the two polygons cross at two or
more points.

\paragraph{Case~1.} Assume $Z'$ is nested inside $Z$.

\emph{Case~1a}: First assume $Z$ and $Z'$ are disjoint.
Because all vertices are on or interior to $Z'$, $Z$ must be a simple closed
geodesic. Let $A$ be the annulus between $Z$ and $Z'$; $A$ contains no vertices.
Intrinsically, $A$ is a subset of a cylinder with lower circular rim $Z$,
and upper boundary $Z'$.
Thus we can project any point $p \in Z'$ orthogonally onto $Z$.
See Fig.~\figref{ZZnested}(a).
Thus $\ell ( Z )  \le \ell ( Z' )$, with equality when $Z$ and $Z'$ are parallel.
In either case, we can ``lift" $Z$ parallel to itself until it touches a vertex in $V$.
As mentioned earlier,
we choose this isometric version of $Z$ as $\min \ell [V]$.
\begin{figure}[htbp]
\centering
\includegraphics[width=1.0\linewidth]{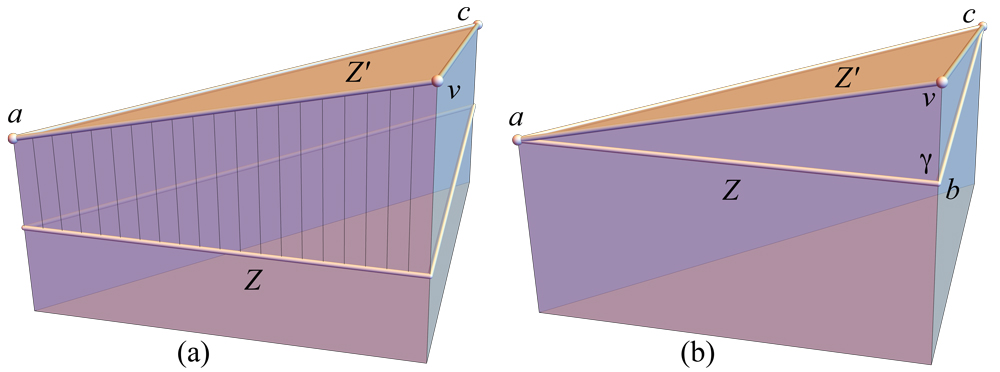}
\caption{$Z'=avc$. 
(a)~Edge $av$ projected onto $Z$.
(b)~$\g$ with $ab,bc$ forms a planar polygon.}
\figlab{ZZnested}
\end{figure}

\emph{Case~1b}: Assume $Z$ and $Z'$ share some portion of their boundaries.
Again because all vertices are on or inside $Z'$, this means that there must
be a geodesic arc $\g$ of $Z$, outside of $Z'$,
between two shared vertices $v_1$ and $v_2$.
See Fig.~\figref{ZZnested}(b).
Let $B$ be the region of $P$ bounded by $\g$ and the edges $E$ of $Z'$ between $v_1$ and $v_2$.
Then (a direct induction shows that) $B$ is isometric to a plane polygon, with one straight edge $\g$,
and $E$ a polygonal chain.
Therefore $E$ is strictly longer than $\g$, and so $\ell ( Z ) < \ell ( Z' )$.

\paragraph{Case~2.} Neither of $Z$ or $Z'$ is inside the other.
So they must cross at least twice, with always the outer arc a geodesic $\g$,
and the inner sequence of edges possibly containing vertices. 
Call such a region a lune $L$.
Now the argument in Case~1b applies: $L$ is isometric to a planar polygon
with straight edge $\g$, and so the outer arc of $L$ is strictly shorter than the
inner arc.
Now form $Z''$ by following all outer arcs,
and otherwise the common portions of the two polygons. $Z''$ still encloses all of $V$,
but it is strictly shorter than $Z$ and $Z'$.
This contradiction to the assumption that both $Z$ and $Z'$ are shortest
establishes the lemma.
\end{proof}

This next lemma is a counterpart to
the $\a\b$-converse lemma,
Lemma~\lemref{Converse_ab} in Chapter~\chapref{Convexity}.
In the following, a \emph{pinned node} $v$ of a geoarc polygon $G$ is a flat point which behaves like a vertex, in the sense that
$G$ is not allowed to be shorten at $v$ towards its interior.

\begin{lm} 
\lemlab{Converse_Lmin}
Assume that all nodes of a geodesic polygon $G \subset R(Q)$ are either vertices or pinned, and that $V =V(P^+) \subset R(G)$.
Then $G=Z(V)$ if and only if $\b \geq \pi$ at each node of $G$.
\end{lm}

\begin{proof}
That $\b \geq \pi$ for $Z(V)$ was established 
in Lemma~\lemref{Zstrictly}.

In the following we show that $G=Z(V)$,
using reasoning similar to that used to prove uniqueness in Lemma~\lemref{Zunique}.

In Case~1 of that lemma's proof, one of $G$ or $Z$ is nested inside the other.
Case~1a: Assume the two polygons are disjoint and $Z$ is inside $G$.
Then the lemma argument leads to $\ell ( G ) \le \ell ( Z )$, but because $Z$ is
shortest, we must have $\ell ( G ) = \ell ( Z )$; and so $G = Z$.
If instead $G$ is inside $Z$, then recall that $Z$ must be a simple closed quasigeodesic.
If $Z$ and $G$ are not identical, then the arc of $Z$ between touching points of $G$
is a short-cut of $G$, contradicting the assumption that $G$ has no further short-cuts.

In Case~1b, $G$ and $Z$ share some portion of their boundaries.
If $Z$ is inside $G$, the lemma argument shows that $\ell ( G ) < \ell ( Z )$,
contradicting the minimality of $Z$.
If $G$ is inside $Z$, then consider the region $B$ in the proof
of Lemma~\lemref{Zunique},
bound by $\g$ and the edges $E$ of $G$ between $v_1$ and $v_2$.
$B$ is isometric to a plane polygon, and so some vertex along $E$ must
be convex. We have thus identified a vertex on $G$ with a right angle
$\b < \pi$. So $G$ can be short-cut there, contracting the assumption
that no more $G$-short-cuts were possible.

In Case~2, neither is nested in the other, so they must cross.
This leads to a polygon shorter than either $G$ or $Z$, a contradiction
to the minimality of $Z$.
\end{proof}

\begin{ex}
\exlab{ConverseOptimal}
The assumption of pinned nodes in Lemma~\lemref{Converse_Lmin} is clearly necessary.
Here we show that the second assumption, that $V \subset R(G)$, is also necessary.

Consider an equilateral triangle $abd'$ inscribed in the rectangle $abce'$, and choose $d \in ce'$ between $d'$ and $e'$, and $e \in ae'$ such that $ed \perp bd$.
See Figure~\figref{V_RG}.
Extend the sides $e'a$ and $cb$ sufficiently beyond $a$ and $b$ respectively, and let $L$ be the double of the extended $abcde$.
Let $V=\{b,c,d\}$. Note that $e \not\in V$.

On $L$, $Q=aba$ is a simple closed quasigeodesic, and the geodesic polygon $G=bdb$ has 
each of its angles $\b$ towards $Q$ at least $\pi$.
However, $e \notin R(G)$ and $Z( V ) = Q$, 
because $|a-b|=|b-d'|<|b-d|$.

\begin{figure}[htbp]
\centering
\includegraphics[width=0.5\linewidth]{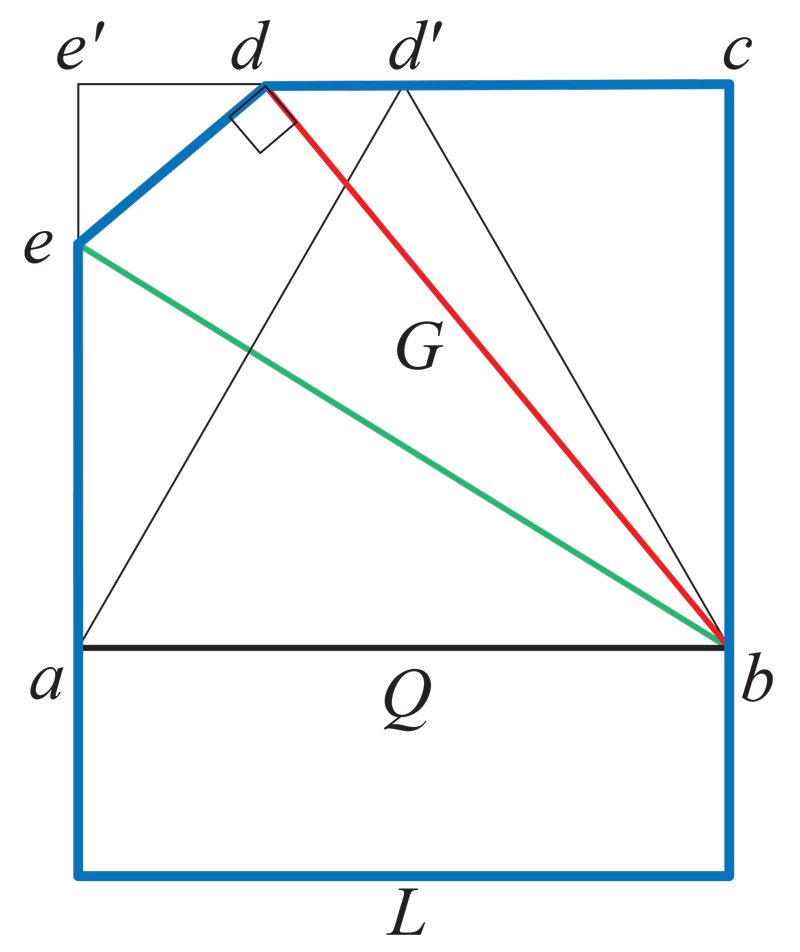}
\caption{Lemma~\lemref{Converse_Lmin} is optimal. 
$V=\{b,c,d\}$, $G=bdb \neq Z =Q$.}
\figlab{V_RG}
\end{figure}
\end{ex}

We prove here the following result,
which will be employed in the next chapter.

\begin{lm}
\lemlab{V+1_subset_V}
For any point $p \notin V$, $R (Z(V)) \, \subseteq \, R (Z(V \cup \{p\}))$.%
\footnote{We believe 
that also $\ell (Z(V)) \leq \ell (Z(V \cup \{p\}))$, but this inequality is not important in the following.}
\end{lm}

\begin{proof}
The argument is similar to the one proving the uniqueness of $Z$.

Set $V'=V \cup \{p\}$.
Clearly, $R (Z(V)) = R (Z(V'))$ if $p \in R(Z(V))$.
So assume that $p$ is strictly exterior to $R(Z(V))$.

Put $I=R (Z(V)) \cap R (Z(V'))$ and assume there exists $q \in R (Z(V)) \setminus R (Z(V'))$.
If the lemma were to hold, then $I=R(Z(V))$ and $R (Z(V)) \setminus R (Z(V'))=\varnothing$,
so there could be no such $q$. We now prove that the existence of $q$ leads
to a contradiction.

Since $V \subset I$ and $q$ is outside $I$,
$q$ lies on an geoarc $A \subset Z(V)$  with extremities on $Z(V')$. 
Those extremities determine an arc $A'  \subset Z(V')$ 
such that $A \cup A'$ bounds a flat polygonal domain $L$.

Because of the length minimality of both $Z(V')$ and $Z(V)$, $\ell (A')= \ell(A)$. 

On the other hand, if $A'$ is not a geoarc, the flat polygon bounded by $A \cup A'$ has the side $A$ necessarily shorter than $A'$, a contradiction.
And if $A'$ is a geoarc, $L$ is a flat digon, contradicting the Gauss-Bonnet Theorem.
\end{proof}


\section{Shortening Algorithm}
\seclab{ShortAlg}

Our goal in this section is to provide a polynomial-time algorithm to construct
$\min\ell[V]$.

\subsection{Curve-Shortening Flow}
\seclab{Flow}
Here we digress to highlight the intuition behind curve-shortening.
The famous Gage-Hamilton-Grayson (GHG) theorem says that,
\begin{quotation}
\noindent
If a smooth simple closed curve undergoes the curve-shortening flow, it remains smoothly embedded without self-intersections. 
It will eventually become convex, and once it does so it will remain convex. ... The curve will then converge ... to a ``round point,"
[a vanishingly small circle].
\cite{Wikipedia} 
\end{quotation}
First assume that all of $V$ is strictly inside $Q$: $Q$ and $V$ are disjoint.
At every vertex $u$ of $Q$ strictly convex to its left,
replace $u$ with a small arc of a circle, smoothly joined to the incident
edges of $Q$. This is now a smooth simple closed convex curve $C$,
a smooth approximation of $Q$,
to which the GHG theorem applies.
Shortening $C$ will either result in a simple closed geodesic, which can no longer be
shortened, or it will hit one or more vertices in $V$.

Consider these vertices \emph{pinned}, partitioning $C$ into sections between pinned vertices.
A recent result~\cite{allen2012dirichlet} established that
``open curves with fixed endpoints evolving ...
do not develop singularities, and evolve to geodesics.''
Thus the arcs of $C$ between pinned vertices shorten to geodesics.
The result is our shortest enclosing geodesic polygon $Z$.

Although there is recent work on discrete curve-shortening,
e.g., \cite{avvakumov2019homotopic} and \cite{eppstein2020grid},
it seems it could be difficult to turn this smooth curve-shortening viewpoint into 
a formal proof or a finite algorithm.
So we take a different approach.


\subsection{Algorithm Overview}

We now turn to a discrete, finite algorithm,
which we will see is in some sense the reverse of the curve-shortening flow.
Imagining $Q$ as roughly equitorial, 
the flow as described above shortens $Q$ ``upward" to $Z$,
whereas the algorithm to be described
starts with $G$ ``above" $Z$ and shortens it ``downward" to $Z=\min\ell[V]$.

The algorithm consists of two steps. 
The first step constructs a geodesic polygon $G$ enclosing $V$.
The second step shortens $G$ to $Z$.


\subsection{Finding an Enclosing Geodesic Polygon}
\seclab{GeosegPolygon}

Our goal here is to construct a geodesic polygon $G$ that encloses all vertices in $V$. 

It is tempting to build $G$ from the subgraph of the $1$-skeleton
of $P$ induced by the vertices in $V$.
However, this subgraph is not always connected, 
as illustrated in Fig.~\figref{InducedGraph}.
As it seems difficult
to base an algorithm on this subgraph, we pursue a different approach.
\begin{figure}[htbp]
\centering
\includegraphics[width=0.7\linewidth]{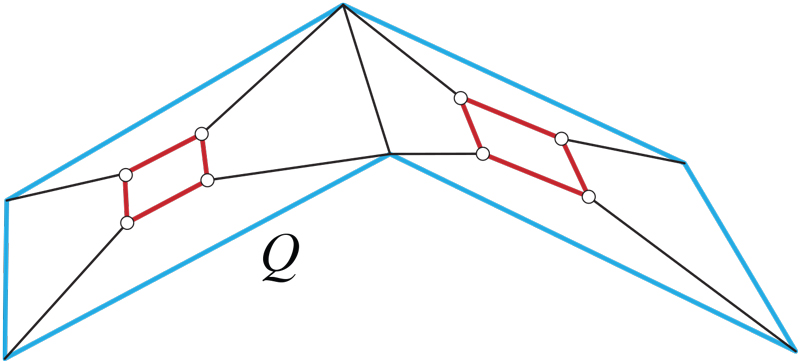}
\caption{$Q$ (blue) encloses $V$, the $8$ marked vertices.
The subgraph $V$ induces is the two disconnected red quadrilaterals.
}
\figlab{InducedGraph}
\end{figure}

We start by computing all geosegs between pairs of vertices in $V$,
and then identify those geoseg subsets (``edgelets") on the outer boundary
of the resulting planar subdivision.

A \emph{planar subdivision} is a data structure $\mathcal{R}$ that maintains
a subset of the plane partitioned into regions.
Often called a Planar Straight-Line Graph (PSLG) when its edges are straight-line segments,
$\mathcal{R}$ represents an embedding of a planar graph,
and maintains the incidence relations between its regions and edges.
Such data structures support updates (for us, the addition of new geosegs)
in $O(\log n)$ time per intersection.

The algorithm described below has time complexity at worst $O(n^5 \log n)$,
although a careful analysis might lead to $O(n^4 \log n)$.
Our focus is more to show it is polynomial-time, as opposed to
optimizing efficiency.

We first describe the three stages of the algorithm at a high-level.
\begin{enumerate}[label={(\arabic*)}]
\item Calculate all geosegs between pairs of vertices in $V$.
Let $\G$ be the collection of these geosegs.
\item Intersect all geosegs in $\G$, incrementally building a planar subdivision
data structure $\mathcal{R}$ on $P^+ =R(Q)$.
As each $g \in \G$ crosses the others already processed,
the surface is partitioned into convex regions,
call them \emph{g-faces} bounded by \emph{edgelets}.
Each g-face is 
flat and isometric to a planar convex polygon;
an edgelet is a (generally short) subsegment of a geoseg $g$,
between two points where other geosegs cross $g$.
G-faces do not necessarily lie in one face of $P$,
but they are always empty of strictly interior vertices, or internal edgelets.
Fig.~\figref{gfaces} illustrates a simple such a partition;
Fig.~\figref{EdgeletsCex} below is a bit more complicated.
As the intersections are calculated, the
data structure $\mathcal{R}$ is updated
to record ccw orientation of each newly created g-face.
The process parallels the incremental algorithm for constructing
an arrangement of lines; see Fig.~\figref{IncrementalAlg}.
\item Each interior edgelet is shared by exactly two g-faces,
recorded in $\mathcal{R}$.
A edgelet $e$ shared with just one g-face must be on the boundary of the
outer face of the planar subdivision.
Connecting all the outer-boundary edgelets
yields a geodesic polygon $G$ that encloses all vertices in $V$.
\end{enumerate}
\begin{figure}[htbp]
\centering
\includegraphics[width=0.5\linewidth]{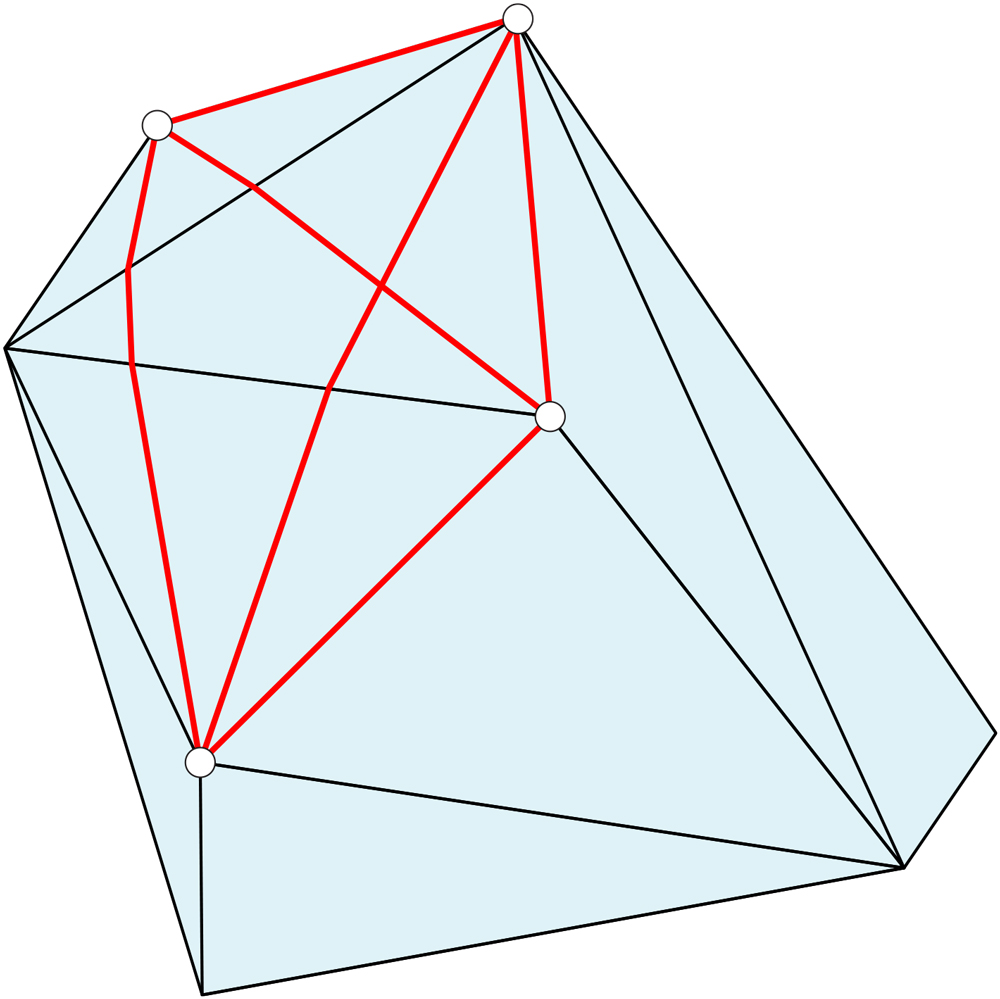}
\caption{$V$ is the four marked vertices. Geosegs (red) form four g-faces.}
\figlab{gfaces}
\end{figure}
\begin{figure}[htbp]
\centering
\includegraphics[width=0.5\linewidth]{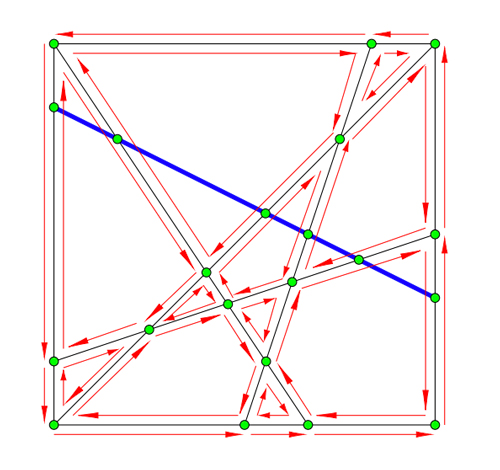}
\caption{Inserting a line into a line-arrangement data structure.
(Image from Jan Verschelde Lecture Notes:
\protect\url{http://homepages.math.uic.edu/~jan/mcs481/arrangements.pdf})}
\figlab{IncrementalAlg}
\end{figure}

Next we discuss some computational details for each of the three stages.
\begin{enumerate}[label={(\arabic*)}]
\item The total number of geosegs between vertices of $P$ is $O(n^2)$,
each of which can be found in $O(n \log n)$ by the
algorithm in~\cite{SchreiberSharir}.
Although two vertices may have more than one
geoseg connecting them, 
the number of geosegs from one vertex $v$ and the $n-1$ other vertices
is still $O(n)$.
This claim can be proved by induction using
the cut locus partition (Lemma~\lemref{Partition} in Chapter~\chapref{Preliminaries}. 
For another proof idea, notice that each extra pair of geosegs between two
vertices must have a vertex separating them. And there are only $n$ vertices.
\item
The total combinatorial size of the data structure $\mathcal{R}$ is
$O(n^4)$, the same size as an arrangement of $\binom{n}{2}$ lines.
That each of the $O(n^4)$ edgelets may cross several edges of $P$ is
accounted for in the implicit data structure of~\cite{SchreiberSharir}.
\item We note that
the sides of $G$ are not necessarily full geosegs between vertices.
$G$ in the example shown in Fig.~\figref{EdgeletsCex}
includes two edgelets $v_1 x$ and $v_3 x$ whose containing geosegs
are not wholly on the outer boundary.
This example also shows that $G$ is not necessarily convex:
it is reflex at $x$.\\
The complexity of any face, including the outer face, of an arrangement of segments
is just slightly more than linear in the number of segments.
Therefore, since
$G$ is built from $O(n^2)$ geosegs,
the total number of edgelets on $G$ is
$O( n^2 \a(n) )$,
where $\a$ is the inverse Ackermann function~\cite{chazelle1993computing}.

\end{enumerate}

\begin{figure}[htbp]
\centering
\includegraphics[width=0.8\linewidth]{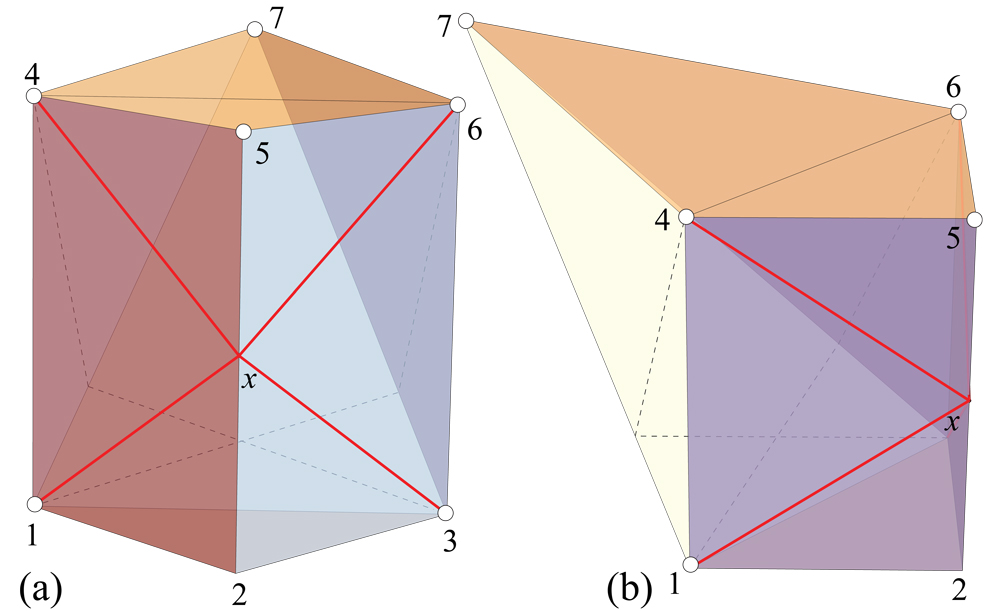}
\caption{$V$ is all the vertices except excluding $v_2$.
$Q=v_1 v_2 v_3$.
$G= v_1 x v_3 v_1$.  $g_{16}$ and $g_{35}$ (red) on the front are geosegs.
The dashed paths around the back are longer and so not geosegs.}
\figlab{EdgeletsCex}
\end{figure}

\medskip

A more formal presentation is displayed below as Algorithm~\ref{G0Alg}.

There are two special cases. When $V$ is a single vertex $v$, we
allow $G=\{v\}$.
And when $|V|=2$, say, $V=\{ a, b \}$, $G$ is the ccw loop $aba$.

\begin{algorithm}[htbp]
\caption{Algorithm to find an enclosing geodesic polygon}
\label{G0Alg}
\DontPrintSemicolon 
    \SetKwInOut{Input}{Input}
    \SetKwInOut{Output}{Output}

    \Input{$n$ vertices $V$ inside or on quasigeodesic $Q$, on $P^\#$.}
    \Output{Geodesic polygon $G$ enclosing $V$.}
    
    \BlankLine
    \tcp{Calculate all geosegs between pairs of vertices: $O(n^3 \log n)$.}
    
     \ForEach{$\binom{n}{2}$ pairs of vertices in $V$}{
     
       Compute geosegs $v_i$ to $v_j$, each in $O(n \log n)$ time~\cite{SchreiberSharir}.

     }
    
      Let $\G$ be the set of these $O(n^2)$ geosegs.
       \BlankLine
       \tcp{Compute g-faces and edgelets: $O(n^5 \log n)$.}
       
       \ForEach{geoseg $g \in \G$}{
       Intersect $g$ with previously built data structure $\mathcal{R}$.
       
       Create g-face regions, orient each ccw, bounded by edgelets.
       
       $O(n \log n)$ per geoseg insertion.
       
        }
        
        \BlankLine
        \tcp{Identify outer boundary: $O(n^4 \log n)$.}
        
        \ForEach{edgelet $e$}{
        
        Check in data structure $\mathcal{R}$ if $e$ is shared between two g-faces.
        
        If not, $e$ is on the boundary of the outer face of the subdivision:
        
        Then add $e$ to $G$.
              
        }
     
      \KwRet $G$.
      
\end{algorithm}

We have established this lemma:

\begin{lm}
\lemlab{G0}
Let $P$ have $n$ vertices, and $V$ a set of vertices in or on 
quasigeodesic $Q$ (possibly including all vertices of $Q$).
Then a geodesic polygon $G$ enclosing all of $V$ can
be constructed in $O(n^5 \log n)$ time.
\end{lm}


\subsection{Algorithm for Curve Shortening}
\seclab{ShorteningAlgorithm}

We now describe an algorithm for shortening $G=G_0$ to $Z=\min\ell [V]$.
We will repeatedly shorten $G_0$ to $G_1,G_2,\ldots,G_k = G$, until $G_k=Z$
for some $k \le n = | V |$.
We next describe one shortening step.

\bs

\emph{Case~1}. $|V \cap G_j| \ge 3$.
Let $v_{i-1}, v_i, v_{i+1}$ be three consecutive nodes 
of $G_j$ (for some $j$),
and let $\a_i$ and $\b_i$ be the angles at $v_i$ left and right of $G_j$.
If $v_i \in Q$, or if
$\b_i \ge \pi$, we take no action---$G_j$ should not or can not be shortened at $v_i$.
If however $v_i \not\in Q$ and $\b_i < \pi$, then we can locally shorten $G_j$ in a neighborhood of $v_i$.
But rather than locally shorten, we shorten by replacing the sequence
$(v_{i-1}, v_i, v_{i+1})$ by $(v_{i-1}, v_{i+1})$.
This is a shortening because the triangle inequality holds 
for geodesic triangles on a convex polyhedron.
Let $\g= v_{i-1} v_{i+1}$.

Before proceeding with the description, we re-examine the example
in Fig.~\figref{short-cut_v}.
\begin{figure}[htbp]
\centering
\includegraphics[width=0.75\linewidth]{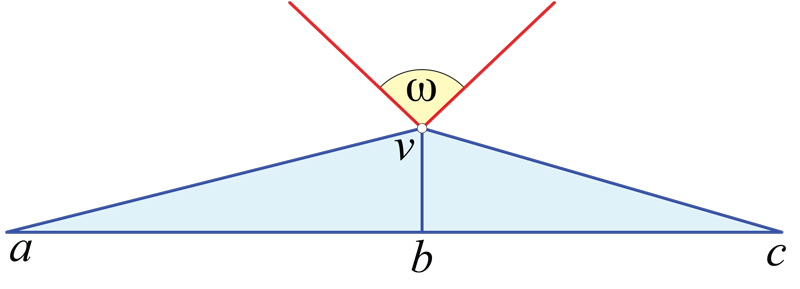}
\caption{
Flattening the region between $G$ and $Q$ in Fig.~\protect\figref{ZZnested}(b).
$\a \approx 94^\circ$, $\b \approx 150^\circ$, $\o \approx 116^\circ$.}
\figlab{short-cut_v}
\end{figure}
In the example previously shown in Fig.~\figref{ZZnested}, we have
$G_0 = (a,v,c)$, and $Q=Z$ as illustrated.
Here $\b < \pi$, and flattening the region below $v$ leads to the
planar triangle, short-cut by geoarc $\g = abc$.

Returning to the shortening step, the region of $P^+$ between $G_j$ and $Q$ is
empty of vertices. So the path $\g$ does not cross any portion of $G_j$,
nor can it cross $Q$, because the $P^\#$ construction would require
$\g$ to cross below and return above $Q$, violating the fact
that geodesics do not branch.
So the triangle $v_{i-1} v_i v_{i+1}$ is between $G_j$ and $Q$, and empty of vertices.
By the triangle inequality, $\g$ is shorter than 
$|v_{i-1}- v_i| + |v_i - v_{i+1}|$.

Another way to view this construction
is to imagine a point $p$ initially at $v_i$, sliding down
the edge $v_{i-1} v_i$ until it reaches $v_{i-1}$. Then the line segment $p v_{i+1}$
never encounters a vertex during this motion,
and leads to $p v_{i+1} = v_{i-1} v_{i+1} = \g$.


Applying one of these shortening steps moves a node
that was originally on $G$---either a vertex or an intersection of edgelets---to
the interior of $G_j$.
Because the combinatorial size of $G$ is bounded by
$O( n^2 \a(n) )$,
there can be no more than that nearly quadratic number of steps.

When no further shortening steps are possible, $G_k$ is a geodesic polygon
with $\b_i \ge \pi$ at every vertex $v_i$.
By Lemma~\lemref{Converse_Lmin} then $G_k = Z = \min\ell[V]$.

\bs
Before passing to the other cases, we notice that the total curvature $\o$ of $V \subset R(Q)$ is at most $2 \pi$.

\bs

\emph{Case~2}. $|V \cap G_j|=2$ and $G_j$ is a digon
with endpoints $v_1$ and $v_2$.
Let $\b_i$ be the exterior angles at $v_i$, and let $\tau_i=\b_i-\pi$ be the turn angles at $v_i$.
Gauss-Bonnet requires that $\o + \tau = 2 \pi$. 
If both $\b_i < \pi$, then $\tau_1+\tau_2 = \tau < 0$, and $\o + \tau = 2 \pi$ cannot be satisfied.
If both $\b_i \ge \pi$, the algorithm halts.
So let $\b_2 < \pi$.

Because $Q$ is a convex curve, 
we can merge inside $R(Q)$ all vertices it encloses to obtain one apex $a$ of a cone $\Upsilon$.
Clearly, the vertex-free region between $Q$ and $G_j$ is 
isometric to a subset of $\Upsilon$.
Unfold $\Upsilon$ by cutting it open along the cone generator $a v_1$.
Then there exists a geodesic loop at $v_1$ enclosing $R(G_j)$ and strictly shorter than $G_j$.

A special case is when $|V|=|V \cap G_j|= 2$ and $\a_i=0$.
For example, in Fig.~\figref{GeosegsNonconvex},
with $V=\{c,d\}$, $G=cdc$, and short-cutting at $d$ leads to $Z= cfc$.
\begin{figure}[htbp]
\centering
\includegraphics[width=0.4\linewidth]{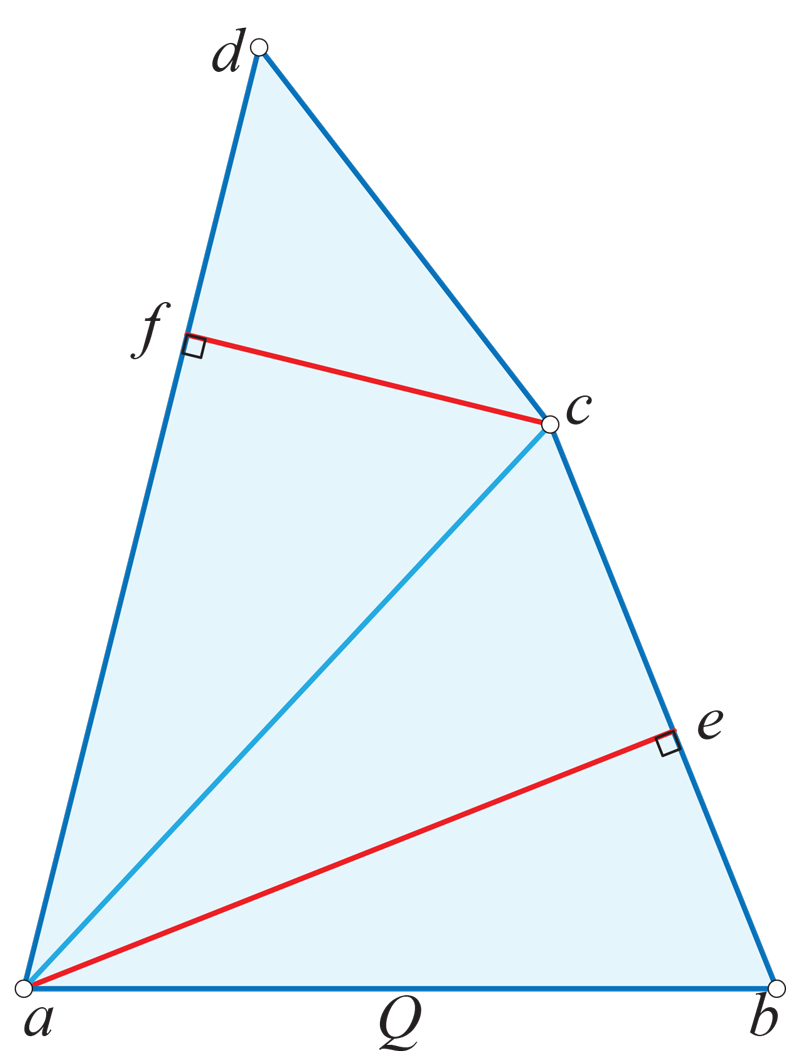}
\caption{Doubly-covered quadrilateral. 
If $V=\{a,c,d\}$, $G= aca$, $Z=aea$. 
If $V=\{c,d\}$, $G=cdc$, $Z= cfc$.
Vertex $c$ is nonconvex in the doubled $\triangle acd$.
}
\figlab{GeosegsNonconvex}
\end{figure}

\bs

\emph{Case~3}.
$|V \cap G_j|=1$. Then
$G_j$ is a geodesic loop at $v$, where 
Gauss-Bonnet implies that
necessarily $\b \geq \pi$.
Therefore the algorithm halts.

This completes the description of the shortening algorithm.
\bs

A few remarks on the shortening algorithm:
\begin{itemize}
\item The short-cut edge $\g= v_{i-1} v_{i+1}$ is a geoarc, not a geoseg, because if $\g$
were a geoseg, it would have been incorporated into the construction of $G=G_0$.
So every shortening step inserts a geoarc.
Later (in Chapter~\chapref{SpiralTree3D}),
algorithms will vertex-merge along such a geoarc (beyond a geoseg).
\item The algorithm relies on the property that the region of $P^+$ between $G$ and $Q$ is
empty of vertices. 
All the shortenings insert arcs in this vertex-free region.
For that reason, $V$ must include all vertices inside $Q$,
i.e., the algorithm may not work if $V$ is a proper subset of the vertices inside $Q$.
(However, the construction of $G$ still works, but the shortening may not.)
So in this sense, it is not a general ``convex hull" algorithm,
but rather a construction tailored to our needs.
\item The algorithm also works
for all vertices inside a convex curve $C$, because doubling $R(C)$ via AGT transforms $C$ into a quasigeodesic.
\item Even when $V$ includes all vertices on $Q$, it may be that $G \neq Q$,
if $Q$ contains geoarcs. But then the shortening algorithm will result in $Z=Q$.
\item Viewing $P^+$ as the upper half-surface, one can view $Q$ as 
the rim of
the base of $P^+$, $Z$ a closed convex curve ``above"/inside $Q$, 
and $G$ a closed curve ``above"/inside $Z$. The shortening
algorithm works by, in some sense ``area growing" $G$ down to $Z$, even though each growing step is in fact length-shortening.
In contrast, the curve-shortening flow described in Section~\secref{Flow}
``raises" $Q$ until it matches $Z$.
\end{itemize}


We have now established this theorem:

\begin{thm}[$Z$-Algorithm]
\thmlab{Zalgorithm}
The minimum-length geodesic polygon enclosing $V$,
$Z=Z(V) = \min\ell [V]$, can be constructed in polynomial-time,
specifically in time $O(n^5 \log n)$ where $n = | V |$.
\end{thm}


\section{AG-convexity and $Z$}
\seclab{Lmin_rconv}

To parallel the notation $Z=Z(V)=\min\ell [V]$,
define $W=W(V) = \partial \, \rconv(V)$.

\begin{lm}
\lemlab{Z_ag_convex}
The set $R(Z)$ enclosed by $Z=Z(V)$ is a relatively ag-convex set containing $V$, but it is not necessarily equal to $\rconv (V)$.
\end{lm}

\begin{proof}
The relative convexity of $R(Z)$ follows directly from the strict $\a\b$-convexity of $Z$ and Lemma~\lemref{Converse_ab}.
The second claim is established by the next example.
\end{proof}

\begin{ex}
\exlab{minell_neq_rconv}
It could be that $Z(V) \neq W(V)$.

Let $P$ be the octahedron illustrated in Fig.~\figref{minl_rconv}.
The shape is symmetric front-to-back and top-to-bottom.
Triangles $acd$ and $ac'd'$ are congruent $45^\circ-45^\circ-90^\circ$ triangles,
and triangles $acc'$ and $add'$ are congruent equilateral triangles.
$Q$ (red) is
$a b a' e a$, and $V=\{a, c, d, a' \}$.
Let $G=acda'$ as illustrated. Then 
$\a_i < \b_i$ at each of the four vertices / gs-nodes.
By Lemma~\lemref{Converse_ab}, 
$\rconv(V)$ is the union of the two triangles $acd$ and $a'cd$.
However, because $\b$ at both $c$ and $d$ is strictly less than $\pi$,
the algorithm would short-cut there, leading
to $Z=a b a' e a = Q$.
So $Z(V) \neq W(V)$.

%
%
%
%
%
%

\begin{figure}[htbp]
\centering
\includegraphics[width=0.6\linewidth]{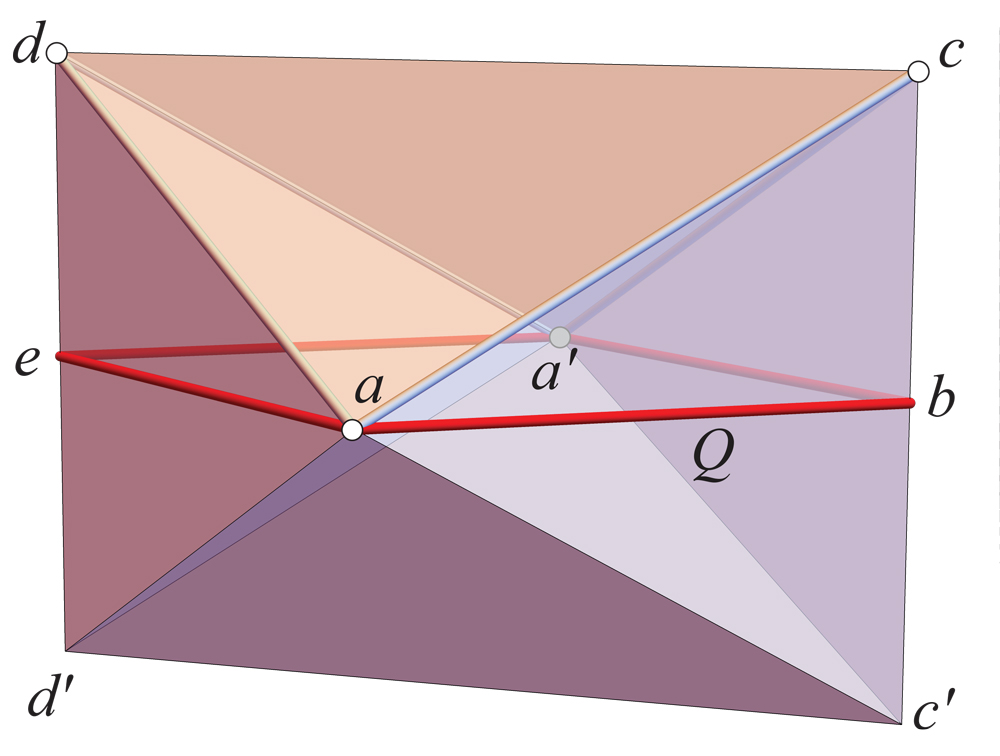}
\caption{$Q=a b a' e a$, $V=\{a, c, d, a' \}$.
$Z=a b a' e a = Q$.
$W(V) = a c a' d a$.
$Z(V) \neq W(V)$.}
\figlab{minl_rconv}
\end{figure}
\end{ex}


It follows, directly from Lemma~\lemref{Z_ag_convex} and the definition of $\rconv (V)$, that
$R(W)  \subseteq R(Z)$, with $Z=Z(V)$ and $W=W(Z)$.
So of the two generalizations of the planar convex hull,
$\rconv(V)$ is ``tighter" than $\min\ell[V]$.


\section{Algorithm for $\rconv(V)= R(W)$}
\seclab{WAlgorithm}

An algorithm to construct $W=W(V)$ can be obtained by minor modification
of the algorithm for $Z=Z(V)$.
The first part, calculation of a geodesic polygon $G$ enclosing $V$, is
identical.
The second part, short-cutting $G$, can be followed with just different criteria
of when to short-cut.
We will continue to call it a ``short-cut" even though it is no longer aimed at
length-shortening.

Just as the construction of $Z$ relies on Lemma~\lemref{Converse_Lmin},
the construction of $W$ relies on the $\a\b$-converse Lemma~\lemref{Converse_ab}.
Assuming the preconditions of these lemmas are satisfied, then the short-cutting
decisions of the two algorithm are as follows.
\begin{itemize} 
\item[$Z$:] If $G_j$ is such that, for every $v$ on $G_j$, 
$\b \ge \pi$, then $G_j=Z$.
So we short-cut whenever $\b < \pi$.
\item[$W$:] If $G_j$ is such that, for every $v$ on $G_j$, 
\begin{enumerate}[label={(\alph*)}]
\item if $v$ is a gs-node, then $v$ is strictly $\a\b$-convex.
So we short-cut whenever $\a \ge \b$.
For positive curvature $v$, this implies $\b < \pi$.
\item if $v$ is a g-node, then $v$ is $\a\b$-convex.
So we short-cut whenever $\a > \b$.
And this implies $\b < \pi$.
\end{enumerate}
then $G_j=W$.
\end{itemize}

\noindent
We next explain in some detail how the short-cutting conditions
for $W$ imply $\b < \pi$, as claimed in (a) and (b) above.
Because $\a+\b=2\pi$, we have:
$\a \ge \b$ implies that $\b \le \pi$, and
$\a > \b$ implies that $\b < \pi$.
Below we abbreviate ``the $Z$-algorithm" and ``the $W$-algorithm"
with just $Z$ and $W$, for readability.

\paragraph{Positive Curvature $v$.}
If $v$ has positive curvature, then $\a \ge \b$ implies that $\b < \pi$
(because if $\b=\pi$, then $\a=\b=\pi$ and $v$ is flat).
So, for positive curvature $v$, whenever $W$ short-cuts, $\b < \pi$ and also $Z$
short-cuts.
And for a gs-node, if $\a < \b$, then $W$ does not short-cut independent of $\b$, but
$Z$ will short-cut if $\b < \pi$.
And for a g-node, if $\a \le \b$, then $W$ does not short-cut independent of $\b$, but
$Z$ will short-cut if $\b < \pi$.

\paragraph{Flat $v$.}
For flat $v$, if $\a=\b=\pi$, then short-cutting effectively leaves a geodesic through $v$,
and so neither algorithm takes action.
For flat $v$ and $\a < \b$, neither algorithm short-cuts:
$Z$ because $\b > \pi$; $W$ because strictly $\a\b$-convex.
For flat $v$ and $\a > \b$, both algorithms short-cut:
$Z$ because $\b < \pi$; $W$ because not $\a\b$-convex.
\bs
We have now established this result:
\begin{lm}
\lemlab{short-cuts}
The $W$-algorithm's short-cuts are a subset (or equal) to 
the $Z$-algorithm's short-cuts. 
\end{lm}
\noindent
This accords with our conclusion from Lemma~\lemref{Z_ag_convex}
that $R(W)  \subseteq R(Z)$.

The arguments that the $W$-algorithm's short-cuts are possible,
and that the algorithm halts after at most $n$ steps, are identical
to those for the $Z$-algorithm,
because every $W$-algorithm short-cut is a $Z$-algorithm short-cut,
by Lemma~\lemref{short-cuts}.
So we achieve a result parallel to Theorem~\thmref{Zalgorithm}:

\begin{thm}[$W$-Algorithm]
\thmlab{Walgorithm}
The boundary $W(V)$ of $\rconv( V )$
can be constructed in polynomial-time,
specifically in time $O(n^5 \log n)$ where $n = | V |$.
\end{thm}

%
%
%
%

\chapter{Spiral Tree on Polyhedron}
\chaplab{SpiralTree3D}
In the previous chapters we have established the planar model for our spiral tree (based on convex hulls), 
and the extensions to convex polyhedra of the planar notions of convex hull/minimal length enclosing polygon.
In this chapter we continue our program and 
prove that the spiraling idea works as well for vertex-merging in polyhedral half-surfaces bounded by simple closed quasigeodesics,
with respect to either convex hull, or enclosing geodesic polygon.

Although in some sense the 3D algorithm follows the 2D algorithm
in Chapter~\chapref{SpiralTree2D} closely,
there are several significant differences.
One is that, in 2D, no triangles are inserted along each slit, whereas in 3D
these insertions change $P$ to a different (larger) polyhedron $P_i$ at each step.
Another difference is that the 2D slit segments were all geosegs, whereas
in 3D some of the slits are geoarcs. 
These and other differences make the proofs in this chapter somewhat intricate.


\section{Notation}
\seclab{Notation}
The notation is a bit complex, so we list the central symbols here for 
later reference.
It may help to refer ahead to Fig.~\figref{IcosaCuts3D} to illustrate the definitions below.

\begin{itemize}
\item $V = V_0 = \{v_1,v_2,v_3, \ldots, v_n \}$: Vertices on the surface $P$.

\item $P_i$: Polyhedron after the $i$-th vertex-merge.

\item $m_i$: merge vertex, the vertex created by merging $m_{i-1}$ with $v_{i+1}$,
with $m_0=v_1$. 

\item $g_i = m_{i-1} v_{i+1}$. The $i$-th slit/merge geoarc. 

\item $v_1$ is also given the label $m_0$, so $g_1 = v_1 v_2 = m_0 v_2$.

\item $v_i$ is called \emph{flattened} if it has already been merged;
the merge will reduce $v_i$'s curvature to zero.

\item $V_i$ is the set of not-flattened vertices on $P_i$ remaining after the $i$-th merge 
along $g_i = m_{i-1} v_{i+1}$. $m_i \in V_i$.
$|V_i|$ is the number of vertices in $V_i$.

\item $T^2_i$ represents the pair of triangles inserted along $g_i$
in a vertex-merge, and $T_i$ refers to the one of the pair crossed by $g_{i+1}$.

\item $H_i = \rconv(V_i)$ is the relative ag-convex hull of $V_i$,
defined in Section~\secref{RelConvHull}.
We view $H_i$ as a region of $P_i$ and let $\partial H_i$ denote its boundary.
$H_0$ is the convex hull of $V=V_0$.

\item With some abuse of notation, we identify objects on $P_i$ with their image on $P$ and vice-versa. 

For an object $X_i$,
we put $\tilde{X}_i \subset P$ for the image on $P$ of $X_i \subset P_i$.
So, $\tilde{H}_i=H_i  \cap P$, and $\tilde{g}_i=g_i \cap P$, for all $i$.

\item Occasionally, we may also use notation from Chapter~\chapref{VMSlitGraph}.
For example, as in that chapter,
we also identify objects on $P_{i-1}$ with their image on $P_i$ and vice-versa.

\item $\tilde{\L} _i= \cup_i \, \tilde{g}_i$ is the slit graph after the $i$-th merge. $\tilde{\L}$ is the full slit graph.
\end{itemize}


\section{Icosahedron Example}
\seclab{IcosahedronExample2}
Before we describe and prove the spiral algorithm in Section~\secref{rconvAlg},
we illustrate its application to an icosahedron $P$.
This repeats our discussion in
Section~\secref{IcosahedronExample1}
but following a spiral sequence of vm-reductions.
In Chapter~\chapref{VMSlitGraph} the endpoint of the reductions was
the doubly-covered triangle shown in
Fig.~\figref{IcosaDoubleTriangle}.
Here the endpoint of the reductions will be a half-cylinder,
described in the next chapter (Fig.~\figref{IcosaCylPhotos}).

We label the $12$ vertices of $P$ as shown in Fig.~\figref{IcosahedronLabels_v2}.
$Q$ is a simple closed geodesic
around the ``equator" of the icosahedron,
and $V$ the six vertices above $Q$,
$v_1,v_2,v_3,v_4,v_5,v_6$.
The vertices are merged in that order, as indicated in Fig.~\figref{IcosahedronLabels_v2}.
(This differs from the order in Section~\secref{IcosahedronExample1},
which merged $v_1,v_2,v_3,v_4,v_6$, leaving $v_5$ unmerged.)

\begin{figure}[htbp]
\centering
\includegraphics[width=0.4\linewidth]{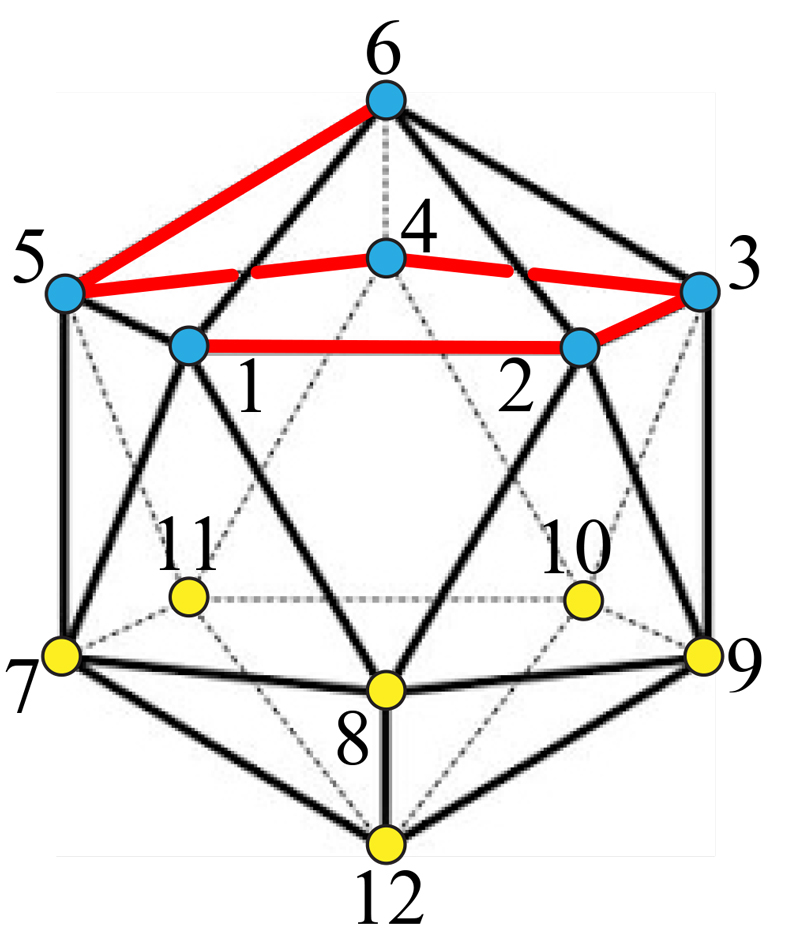}
\caption{Labels $i$ for vertices $v_i$. Sequential merge path in red.
Cf.~Fig.~\protect\figref{IcosahedronLabels_v1}.}
\figlab{IcosahedronLabels_v2}
\end{figure}

Each of the five merges $i$ is accomplished by inserting two
copies of a triangle $T_i$, whose apex is the merge vertex $m_i$.
By convention $m_0=v_1$.
For $i=1,\ldots,5$,
$$
m_{i-1} \; + \; v_{i+1} \; \to \; m_i  \;,\; T_i \; \textrm{along} \;
g_i \;.
$$
Here $m_5$ is special for the icosahedron
in that $T_5$ is actually an infinite parallelogram
sending $m_5$ off to $z=+\infty$,
because the sum of the curvatures of $v_1,\ldots,v_6$ is exactly $2\pi$.

Now we describe the steps of the algorithm, 
referring to Fig.~\figref{IcosaCuts3D} throughout.
\begin{figure}[htbp]
\centering
\includegraphics[height=0.8\textheight]{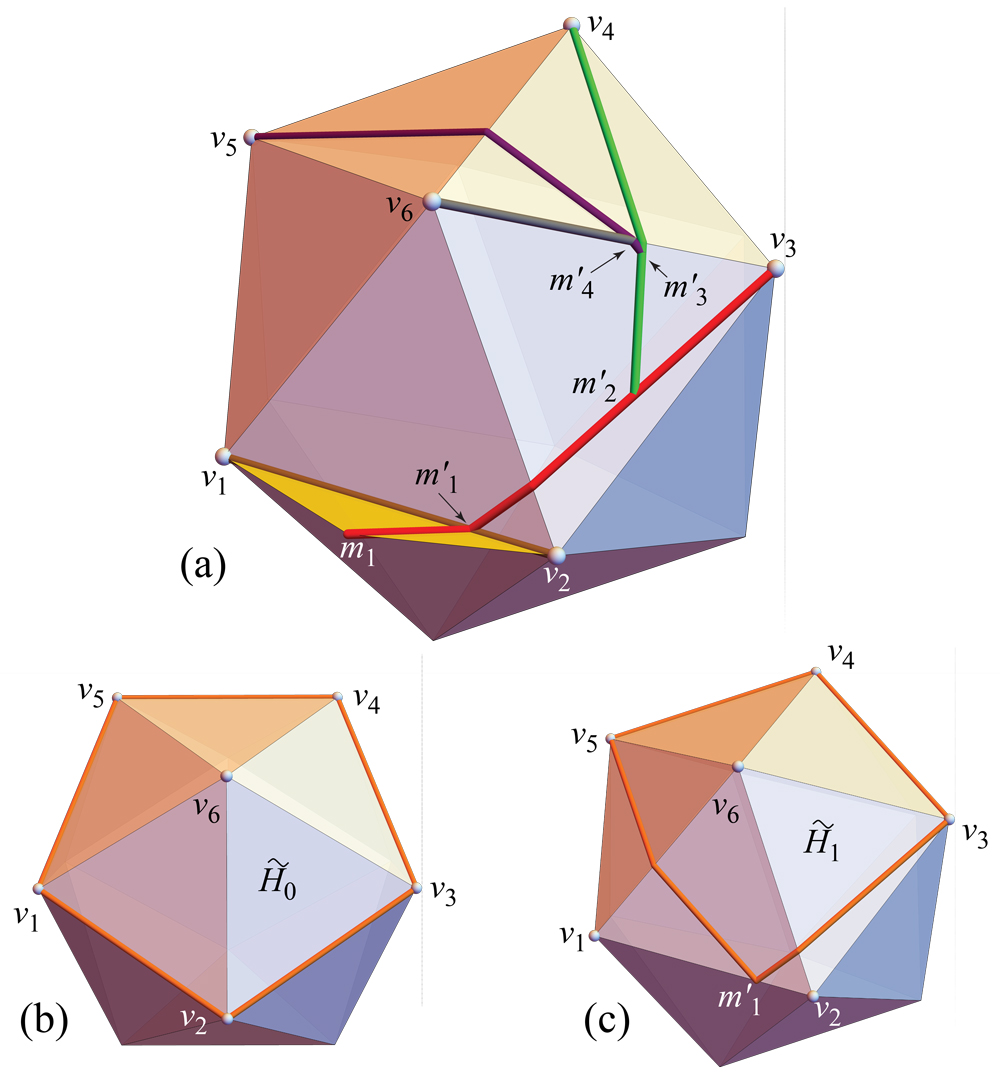}
\caption{(a)~Five geodesic slits $\g_i$ on $P$,
each entering $P$ from $T_i$ at $m'_i$.
(b,c)~$\tilde H_0$ and $\tilde H_1$.
Cf.~Fig.~\protect\figref{Icosahedron_v1}.}
\figlab{IcosaCuts3D}
\end{figure}

With $V=V_0$ on $P=P_0$,
the initial convex hull containing $V_0$ is
$\rconv (V_0) = H_0 = \tilde{H}_0$,
with $\partial H_0 = v_1 v_2 v_3 v_4 v_5$.
See Fig.~\figref{IcosaCuts3D}(b).

The first step, $i{=}1$, merges two consecutive vertices of $\partial H_0$,
say $v_1$ and $v_2$ (as illustrated).
So the slit segment is $g_1 = v_1 v_2 = m_0 v_2$,
where $v_1=m_0$.
Doubled triangles $T^2_1$ are inserted along $g_1$, which flatten
$v_1$ and $v_2$, and introduces a new vertex $m_1$ at the apex of the
triangles. The result is a new polyhedron $P_1$.
We cannot easily illustrate $P_1$
(because of the nonconstructive nature of AGT),
but it can be imagined from 
the (yellow) triangle $T_1$ shown attached to $v_1 v_2$ in Fig.~\figref{IcosaCuts3D}(a).

On $P_1$, $V_1=\{m_1,v_3,v_4,v_5,v_6\}$.
$H_1 = \rconv(V_1)$,  with $\partial H_1 = m_1 v_3 v_4 v_5$,
with $m_1 v_3$ a geodesic from the apex $m_1$ of $T_1$,
to $m'_1$ where it enters $P$, and then on $P$ to $v_3$.
$\tilde{H}_1$ has boundary $\partial \tilde{H}_1 =  m'_1 v_3 v_4 v_6$,
as shown in Fig.~\figref{IcosaCuts3D}(c).
Notice that $\tilde{H}_1 \subset \tilde{H}_0$ on $P$,
but $H_1 \not\subset H_0$ on $P_1$: $m_1 \not\in H_0$.

$P_2$ is obtained by inserting $T^2_2$ triangles
along the slit $g_2$ connecting $m_1$ to $v_3$,
red in Fig.~\figref{IcosaCuts3D}(a).
These triangles flatten $m_1$ and $v_3$,
leaving $V_2=\{ m_2, v_4, v_5, v_6 \}$,
and $H_2 = \rconv(V_2)$,  with $\partial H_2 = m_2 v_4 v_5 $
and $\partial \tilde{H}_2 =  m'_2 v_4 v_5$.

The process continues, with the final set of slits as 
depicted in Fig.~\figref{IcosaCuts3D}(a).
Note that, on $P$, each slit $\tilde{g}_i$
starts from a point $m'_{i-1}$ on $\tilde{g}_{i-1}$ and ends at $v_{i+1}$.


\section{Spiraling Algorithm for $\rconv$}
\seclab{rconvAlg}

\paragraph{First Step.}
As described above in the icosahedron example, 
$H_0 = \rconv(V)$ is computed via the algorithm
described in Section~\secref{WAlgorithm}.
Two vertex endpoints of an edge of $\partial H_0$, $v_1 = m_0$ and $v_2$,
are merged along $g_1=m_0 v_2$.
Triangles $T^2_1$ are inserted along $g_1$, flattening $v_1$ and $v_2$
and introducing a merge vertex $m_1$.
Now $V_1 = ( V \setminus \{v_1, v_2\} ) \cup \{m_1\}$
and $H_1 = \rconv(V_1)$.

\paragraph{General Step.}
We now describe the general step, referring to Fig.~\figref{mi_on_hull}.
Suppose step $i-1$ has been completed.
$H_{i-1}=\rconv( V_{i-1} )$ is on $P_{i-1}$, with
$V_{i-1} = \{ m_{i-1}, v_{i+1}, \ldots \}$.
The merge vertex $m_{i-1}$ is a vertex of $\partial H_{i-1}$,
and $v_{i+1}$ is the next vertex of that hull boundary,
counterclockwise.

\begin{figure}[htbp]
\centering
\includegraphics[width=0.6\linewidth]{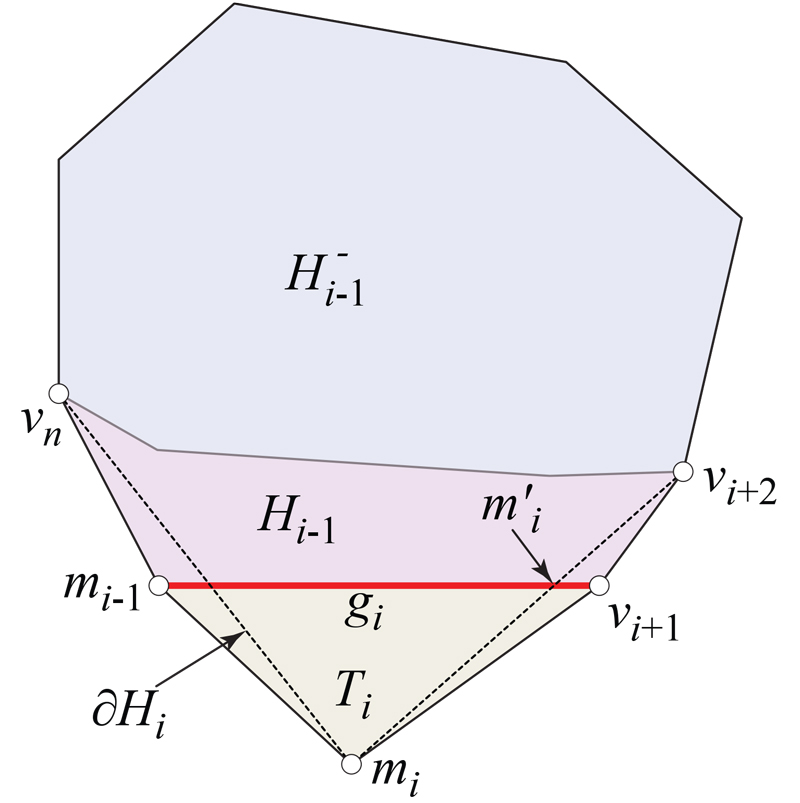}
\caption{$H_{i-1}$ and $H_i$ on $P_i$.}
\figlab{mi_on_hull}
\end{figure}

For the $i$-th step, we consider two cases:

\emph{Case~1:} $m_{i-1} \neq v_{i+1}$.
Then $g_i = m_{i-1} v_{i+1}$ is an edge of $\partial H_{i-1}$.
Triangles $T^2_i$ are inserted along $g_i$, 
flattening $m_{i-1}$ and $v_{i+1}$, and forming $P_i$.
$V_i$ loses these two vertices and gains the new merge vertex $m_i$.
$H_i = \rconv( V_i )$.

\begin{figure}[htbp]
\centering
\includegraphics[width=0.35\linewidth]{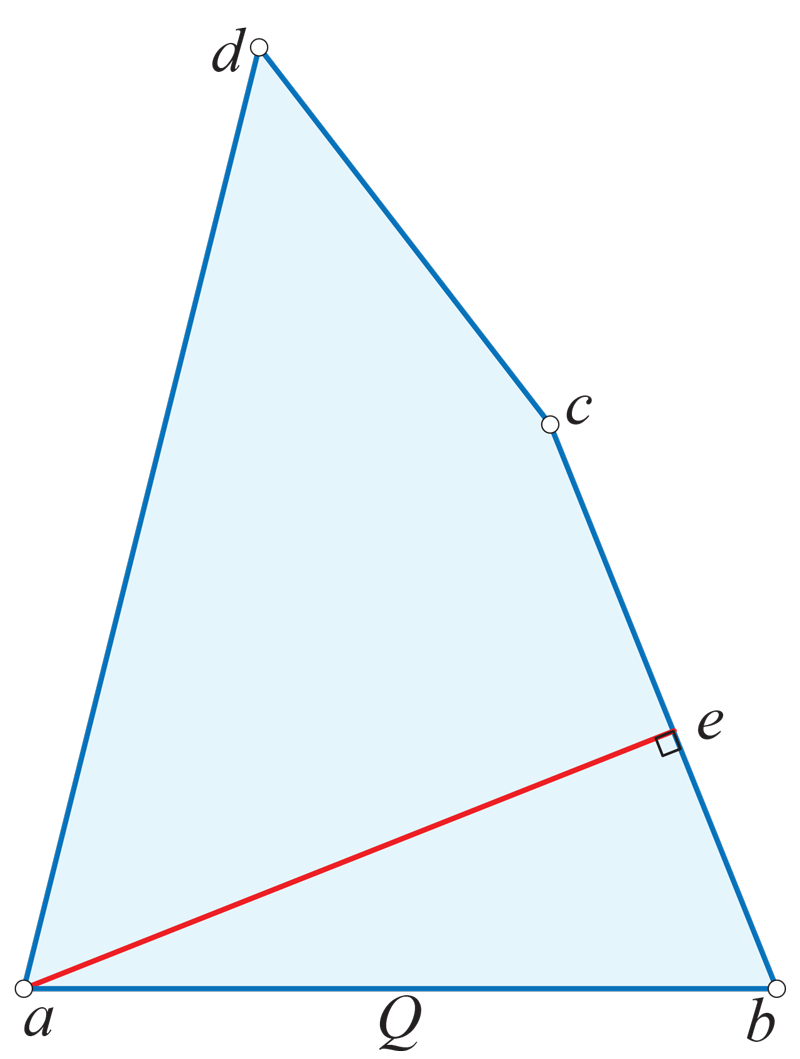}
\caption{Double-sided quadrilateral, $Q=(a,b,a)$. The path $(a,e,a)$ is a geodesic loop.
For $V=\{a,c,d\}$, $\partial \, \rconv(V) = (a,e,a)$.}
\figlab{GeosegsNonconvex_v2}
\end{figure}

\emph{Case~2:} $m_{i-1} = v_{i+1}$, 
hence $\partial H_{i-1}$ is a geodesic loop.
We have seen that this case can occur. 
For example, in Fig.~\figref{GeosegsNonconvex_v2}, $(a,e,a)$ is a geodesic loop. (See Example~\exref{Illuminating} for the proof.)
In such a case, $\partial H_{i-1}$ consists only of the vertex $m_{i-1}$, 
and a geodesic loop at $m_{i-1}$.
Clearly we cannot execute vertex-merging along such a loop,
so we need another strategy for this exceptional case.

Let $V'_{i-1}=V_{i-1} \setminus \{ m_{i-1} \}$. 
Construct $H'_{i-1}=\rconv V'_{i-1}$, 
and let $p$ be the point on $\partial H'_{i-1}$ closest to $m_{i-1}$,
say $p \in \g$, where $\g$ is a geoarc of  $\partial H'_{i-1}$.
Notice that either $p$ is a vertex of $V'_{i-1}$, or there is a unique geoseg from $m_{i-1}$ to $p$, which in particular is orthogonal to $\g$ at $p$.

If $p$ is a vertex of $V'_{i-1}$, take $g_i = m_{i-1} p$ and notice that $g_i$ intersects neither $\partial H_{i-1}$ nor $\partial H'_{i-1}$, other than at $ m_{i-1}$ and at $p$, respectively, by construction.

Assume now that $p$ is a flat point of $H'_{i-1}$, hence interior to $\g$.
Then take $v_{i+1}$ to be either endpoint-vertex of $\g$
(or the only vertex on $\g$ if $\g$ is a geodesic loop).
Next we show there exists a geoarc $g_i = m_{i-1} v_{i+1}$
intersecting neither $\partial H_{i-1}$, nor $\partial H'_{i-1}$, other than at $ m_{i-1}$ and at $v_{i+1}$, respectively.

\begin{figure}[htbp]
\centering
\includegraphics[width=0.4\linewidth]{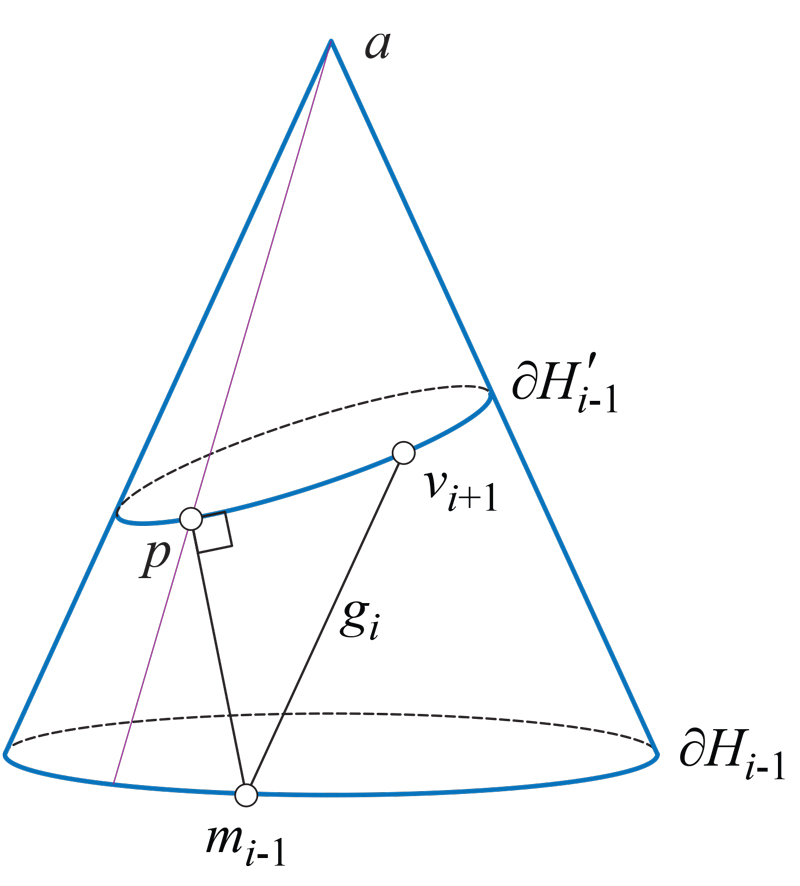}
\caption{On the cone $\Upsilon$, $g_i = m_{i-1} v_{i+1}$
crosses neither $\partial H_{i-1}$ nor $\partial H'_{i-1}$.
}
\figlab{ConeGeoLoop}
\end{figure}

Notice first that $V'_{i-1}$ is strictly interior to $H_{i-1}$
(because $\partial H_{i-1}$ only contains the one vertex $m_{i-1}$),
hence so is $H'_{i-1}=\rconv (V'_{i-1})$ 
(by Lemma~\lemref{out-segm-conv},
the interior of a convex set is convex).
Therefore, $\partial H_{i-1} \cap \partial H'_{i-1} = \varnothing$.

Because both $\partial H_{i-1}$ and $\partial H'_{i-1}$ are convex curves, we can merge all vertices of  $V'_{i-1}$ to obtain one apex $a$ of a cone $\Upsilon$.
It follows that both $\partial H_{i-1}$ and $\partial H'_{i-1}$ ``live on $\Upsilon$,''
in the sense that the closed region $A$ of $P_{i-i}$ they bound is isometric to a subset of 
$\Upsilon$ containing $a$. 
(For details see~\cite{ov-ceccc-14}.)
Refer to Fig.~\figref{ConeGeoLoop}.

Unfold $A$ by cutting it open along the cone generator 
through $ap$. Also denote by $A$ the result.
Then (the images of) $\partial H_{i-1}$ and $\partial H'_{i-1}$ are convex planar curves in $A$, and $m_{i-1} p \perp \g \subset \partial H'_{i-1}$. 
Then we see that $g_i = m_{i-1} v_{i+1}$
again intersects neither $\partial H_{i-1}$ nor $\partial H'_{i-1}$ except at its endpoints.

\begin{lm}
\lemlab{mi_on_hull}
The new merge vertex $m_i$ created at step $i$ is
on $\partial H_i$ (as opposed to strictly interior to $\partial H_i$).
\end{lm}

\begin{proof}
The algorithm constructs $H_i$ by removing
the two flattened vertices $m_{i-1}$ and $v_{i+1}$ and adding the new merge vertex $m_i$.
Formally,
$$
H_i = \rconv\left( \, (V_{i-1}\setminus \{m_{i-1},v_{i+1}\}) \cup \{m_i\}  \right) \;.
$$
Let 
$H^-_{i-1} = \rconv( V_{i-1} \setminus \{m_{i-1},v_{i+1}\} )$,
i.e., 
the relative convex hull of $V_{i-1}$ with the flattened vertices removed, but the new merge
vertex $m_i$ not yet added.
See Fig.~\figref{mi_on_hull}.
The geoarc $g_i = m_{i-1} v_{i+1}$ is strictly exterior to $H^-_{i-1}$,
because it is included in $\partial H_{i-1}$.
Now attach the doubled $T^2_i$ triangle to $g_i$, with triangle apex $m_i$.
It is clear that $m_i$ is also strictly exterior to  $H^-_{i-1}$:
some points on a geoseg
from $m_i$ to a point in $H^-_{i-1}$ must
be exterior to $H^-_{i-1}$. So $m_i \not\in H^-_{i-1}$.
Therefore, with $H_i = \rconv( H^-_{i-1} \cup m_i )$, it must be that
$m_i$ is a vertex of $\partial H_i$.

For the exceptional geodesic-loop case of the algorithm, 
the conclusion follows by construction.
\end{proof}

Notice that $H_{i-1}$ is not necessarily closed, which in some circumstances leads to $\partial H_{i-1}$ being not strictly $\a\b$-convex,
but it remains nevertheless $\a\b$-convex. 
See Lemma~\lemref{ag-is-ab} and Example~\exref{no-strict-ab-conv}
concerning $\a\b$-convexity.

\begin{lm}[Visibility]
\lemlab{mi_visib}
Let $\g$ be the geoarc
connecting the new merge vertex $m_i$ created at step $i$
to the next vertex $v_{i+2}$ to be merged.
Then $\g$ crosses the geoarc $g_i$ at a point $\{m'_i\} = g_i \cap \g$.
In a sense, $v_{i+2}$ is ``visible" to $m_i$ through $g_i$.
\end{lm}

\begin{proof}
We use the $\a\b$-convexity property of $\partial H_{i-1}$:
the angles $\a_{i+1}, \b_{i+1}$  at the $v_{i+1}$ endpoint of the 
geoarc $g_i = m_{i-1} v_{i+1}$ satisfy $\a_{i+1} \le \b_{i+1}$.

We aim to show that, after insertion of the double triangle along $g_i$,
the new angle at $v_{i+1}$, $\angle m_i v_{i+1} v_{i+2}$, is convex, which proves
the claim of the lemma.

So we aim to prove that $\a_{i+1} + \frac{1}{2} \o_{i+1} \le \pi$;
the factor $\frac{1}{2}$ appears
because that is the angle of $T_i$'s corner at $v_{i+1}$.
To simplify the derivation, we suppress the index $i+1$;
so the goal is $\a + \frac{1}{2} \o  \le \pi$.
Because the triangle insertion flattens $v_{i+1}$, we know that
$\o = 2 \pi - (\a + \b)$.
So we get equivalent inequalities:
\begin{align*}
\a + \tfrac{1}{2} \o & \le \pi \\
\a + \pi - \tfrac{1}{2} \a  - \tfrac{1}{2} \b & \le \pi \\
\tfrac{1}{2} \a  & \le \tfrac{1}{2} \b
\end{align*}
which holds at $v_{i+1}$ by $\a\b$-convexity.

The same argument applies to the other endpoint $m_{i-1}$ of $g_i$.
\end{proof}

\noindent
This lemma verifies that the drawing in Fig.~\figref{mi_on_hull} is a correct depiction.
If this lemma did not hold, then $\g = m_i v_{i+2}$ would not necessarily cross $g_i$.


\section{Proof: Slit Graph is a Tree}
\seclab{ProofSlitTree}

The next lemma is the counterpart to the nesting
property of the 2D spiral algorithm, previously illustrated in Fig.~\figref{SpiralNested_n20_s4}.

\begin{lm}[Nesting]
\lemlab{Hi_Nesting}
$\tilde{H}_i \subset \tilde{H}_{i-1}$.
\end{lm}

It may be useful to keep in mind the situation on $P_i$ in Fig.~\figref{mi_on_hull},
and the $\tilde{H}_i$ examples on $P$ in Fig.~\figref{IcosaCuts3D}(bc),
which shows $\tilde{H}_1 \subset \tilde{H}_0$.
With some abuse of notation in the proof, 
we will identify objects on $P_{i-1}$ with their image on $P_i$ and vice-versa.

\begin{figure}[htbp]
\centering
\includegraphics[width=0.6\linewidth]{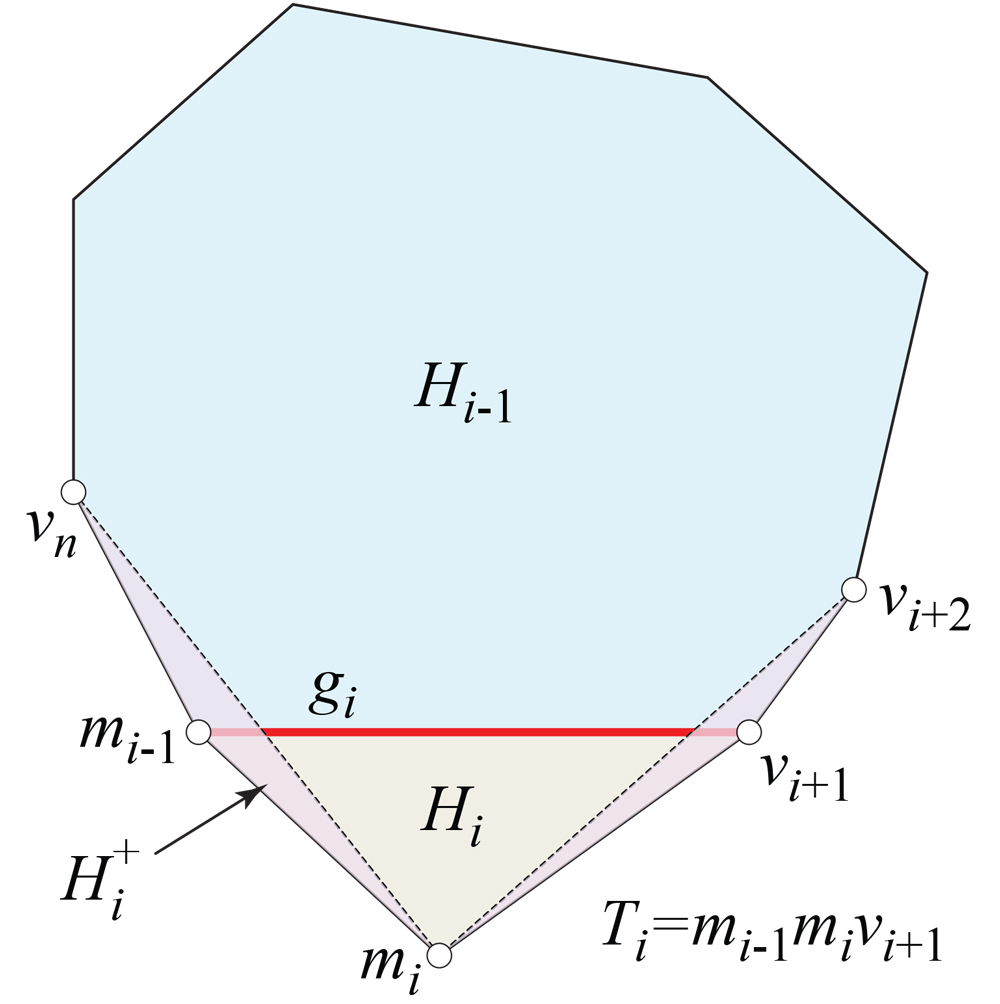}
\caption{$H_i \subset H^+_i$.}
\figlab{HHH}
\end{figure}

\begin{proof}
The proof is in two parts.
We show first that $\tilde{H}_i  \subset \tilde{H}_{i-1}$.
Set $V^-_{i-1}= V_{i-1} \setminus \{m_{i-1},v_{i+1}\}$, i.e.,
$V_{i-1}$ without the two vertices that will form the next vertex-merge insertion.
To keep track of the notation, we
refer to Fig.~\figref{HHH} (a variation of Fig.~\figref{mi_on_hull})
and this display, where we use ``$-$" to 
indicate ``missing" vertices 
among the three $m_{i-1}, m_i, v_{i+1}$:
\begin{align*}
\mbox \;&\;\; V^-_{i-1} = \{ v_{i+2}, \ldots, v_n \} \\
H_{i-1} = \rconv( V_{i-1} ) \;,&\;\; V_{i-1} = \{ m_{i-1}, -, v_{i+1} \} \cup V^-_{i-1} \\
H_{i} = \rconv( V_{i} ) \;,&\;\;  V_{i} = \{ -, m_i, -, \} \cup V^-_{i-1} \\
H^+_i = \rconv( V^+_i ) \;,&\;\;  V_{i} = \{ m_{i-1}, m_i, v_{i+1} \} \cup V^-_{i-1}
\end{align*}
Set $H^+_i$ as above.
Then as Fig.~\figref{HHH} shows, $H_i \subset H^+_i$ and thus $\tilde H_i \subset \tilde{H}^+_i$.

\medskip
Denote by $T_i= \triangle m_{i-1} m_i v_{i+1}$ one of the two inserted triangles along $g_i$, 
the one sharing $g_i$ in common with $H_{i-1}$.

For the second part of the proof, we next show
that, on $P_i$, $H^+_i = H_{i-1} \cup T_i$.
First, notice that $H_{i-1} \cup T_i \subset H^+_i$ 
(by Lemma~\lemref{conv-split} on the hull of a set partition),
so $\rconv (H_{i-1} \cup T_i) \subset H^+_i$.
Second, the $\a\b$-convexity of $\partial H_{i-1}$ at $m_{i-1}$ and $v_{i+1}$ on $P_{i-1}$ shows that, on $P_i$, 
the two sides of $T_i$ incident to $m_i$ either extend two boundary arcs of  $H_{i-1}$ 
(if the $\a\b$-convexity is not strict),
or they make with the respective boundary arcs angles $<\pi$ towards $H_{i-1} \cup T_i$ 
(if the $\a\b$-convexity is strict);
see Lemma~\lemref{mi_visib} (Visiblity).
Therefore, $H_{i-1} \cup T_i$ is a convex set
(by Lemma~\lemref{Converse_ab})
containing $V^-_{i-1} \cup \{m_{i-1}, m_i, v_{i+1}\}$,
hence  $H_{i-1} \cup T_i \supset  \rconv ( V^-_{i-1} \cup \{m_{i-1}, m_i, v_{i+1}\} )= H^+_i$.

Finally, since $H^+_i = H_{i-1} \cup T_i$ and the 
image on $P$ of $T_i$ is a subarc of $\tilde g_i$, 
we have $\tilde{H}^+_i = \tilde {H}_{i-1}$.

In conclusion, $\tilde H_i \subset \tilde{H}^+_i=\tilde {H}_{i-1}$.
\end{proof}



\noindent
Notice that, unlike the planar case where $g_i$ only intersects $H_i$ in one point, $g_i \cap H_i = \{m_i\}$ (Fig.~\figref{SpSnap_mi}(b)),
on polyhedra $g_i$ intersects $H_i$ on a sub-segment (Fig.~\figref{mi_on_hull}).

\bs

Recall that, in our notation, $g_{i-1}$ is a curve (a geoarc) 
on $P_{i-2}$ and $g_i$ is a curve on $P_{i-1}$.

\begin{lm}[$g_i \cap g_{i-1}$]
\lemlab{consec_cuts}
On $P_{i-1}$, $g_i$ intersects the two images
(``banks") of $g_{i-1}$ at precisely one point.
Consequently, $\tilde{g}_i$ has an endpoint on $\tilde{g}_{i-1}$ and no other common point.
\end{lm}

\begin{proof}
On $P_{i-1}$, $g_i$ joins the point $m_{i-1}$ to a vertex $v_{i+1}$.
But $m_{i-1}$ is inside the inserted region (a pair of triangles)
$T_{i-1}$, while $v_{i+1}$ is outside $T_{i-1}$.
Remember that the boundary of $T_{i-1}$ consists of two copies of  $g_{i-1}$, say $g'_{i-1}$ and $g''_{i-1}$.
So $g_i$ intersects those two copies, $g'_{i-1} \cup g''_{i-1}$, at least once.\footnote{
With yet another abuse of notation, we use the same symbol $g$ for an arc $g:I \to P$ 
(with $I \subset \mathbb{R}$ an interval) and its geometrical image on $P$.}

Assume, for the sake of contradiction, that there are two distinct points 
$$
x',x'' \in g_i \cap (g'_{i-1} \cup g''_{i-1}) \;.
$$

Assume first that $\{x' \} = g_i \cap g'_{i-1}$ and $\{x'' \}= g_i \cap g''_{i-1}$.
It follows that $g_i$ exits $T_{i-1}$ and enters it again, and so to reach $v_{i+1}$ it has to exit  $T_{i-1}$ again, crossing $g'_{i-1} \cup g''_{i-1}$ once more.
Therefore, we may assume next that both $x'$ and $x''$ 
arise from $g_i$ intersecting one $g_{i-1}$ image, say $g'_{i-1}$.
Hence $g_i$ and $g'_{i-1}$ determine a geodesic digon, which necessarily contains a vertex $v \in V$ (by the Gauss-Bonnet Theorem).
Since there is no vertex both outside $H_{i-1}$ and inside $H_{i-2}$, 
and there is no vertex both outside $H_{i-2}$ and inside $H_{i-1}$ other than $m_{i-1}$, a contradiction is obtained.

Therefore, $g_i$ intersects the two images of $g_{i-1}$ at precisely one point,
and thus we get that $| \tilde{g}_{i-1} \cap \tilde{g}_i |= 1$.
\end{proof}


\begin{lm}
The set of vertices reduces by one each iteration:
$|V_i| = |V_{i-1}|-1$, for $i > 1$.
\end{lm}

\begin{proof}
Clearly $|V_i| = |V_{i-1}|-1$, because
$V_i = V_{i-1} \setminus \{ m_{i-1},v_{i+1} \} \cup \{m_i\}$.
\end{proof}

\begin{lm}
$\tilde{\L}_i$ is a tree.
\end{lm}

\begin{proof}
By Lemma \lemref{consec_cuts} above, $\tilde{g}_i$ has an endpoint on $\tilde{g}_{i-1}$,
and otherwise does not intersect $\tilde{g}_{i-1}$.
It remains to prove that the interaction of $\tilde{g}_i$ with the earlier
segments in $\tilde{\L}_i$ maintains the tree property.
In fact, we show that $\tilde{g}_i \cap \tilde{g}_j = \varnothing$ for $j < i-1$.

By our sequential-merge choice, $g_i$ is an edge $e=m_{i-1} v_{i+1}$ of $\partial H_{i-1}$,
and so $\tilde{g}_i$ is an edge $\tilde{e}$ of $\partial \tilde{H}_{i-1}$.
For example, in Fig.~\figref{IcosaCuts3D}(bc), 
$\tilde{g}_1=v_1 v_2$ is an edge of $\tilde{H}_0$,
and $\tilde{g}_2=m'_1 v_3$ is an edge of $\tilde{H}_1$.

From Lemma~\lemref{Hi_Nesting} (Nesting) we know that 
$\tilde{H}_j \supset \tilde{H}_{i-1}$ 
for $j < i-1$.

Therefore, $\tilde{g}_j \cap \tilde{H}_{i-1}$ is either the empty set, or an edge of $\tilde{H}_{i-1}$,
so it cannot intersect other edges of $\tilde{H}_{i-1}$ excepting those adjacent to it 
(by Theorem~\thmref{rconvV},
which established the simplicity of $\rconv$ boundaries).

Since the edge $\tilde{g}_{i-1} \cap \tilde{H}_{i-1}$ is between 
$\tilde{g}_j \cap \tilde{H}_{i-1}$ (not empty by Lemma~\lemref{LmSlitSegment})
and $\tilde{g}_i$
(``between'' with respect to the circular order of sides in $\partial \tilde{H}_{i-1}$),
$\tilde{g}_j$ cannot intersect $\tilde{g}_i$ for $j < i-1$.

For example, returning to Fig.~\figref{IcosaCuts3D},
when $i=3$ and $j=1$,
we have $\tilde{H}_{i-1}=\tilde{H}_2$,
$\tilde{g}_i=\tilde{g}_3=m'_2 v_4$,
$\tilde{g}_j=\tilde{g}_1=v_1 v_2$,
and indeed $\tilde{g}_1$ does not intersect $\tilde{g}_3$.

Therefore, we have shown that $\tilde{g}_i$ just intersects $\tilde{\L}_{i-1}$
at the one point, and so maintains the tree structure for $\tilde{\L}_i$.
\end{proof}

\bs

Putting together the above lemmas 
yields the following theorem, one of our primary goals.

\begin{thm}
\thmlab{SpiralAlg}
Let $Q$ be a simple closed quasigeodesic on the convex polyhedron $P$, and $V$ the set of vertices of $P$ enclosed by $Q$, to either side of $Q$.
Then the sequential 
vertex-merging algorithm detailed in Section~\secref{SpiralAlgorithm} results in a slit graph $\tilde{\L}$ that is a tree.
\end{thm}

The consequence of this theorem is that the slits do not disconnect 
a half -surface of $P$, i.e., it remains simply connected. This will be expanded upon
in the following chapter.


\section{Spiraling Algorithm for $Z(V)=\min\ell[V]$}
\seclab{Zalg}
The spiraling approach, developed in Section~\secref{rconvAlg} into Algorithm~1
based on the sets $H_i=\rconv(V_i)$, can be also based on the sets $M_i=R(Z(V_i))$.
And, because it can be that $W(V) \neq Z(V)$
(Example~\exref{minell_neq_rconv}), in 
these cases we get a different but analogous 
Algorithm~2.

The general description of Algorithm~2 is precisely the same as for Algorithm~1, as it only uses the geodesic polygon $\partial H_i$ which can be replaced by $Z$.
So we will retain the same notation, but always remember to replace $H_i$ by $M_i$.

For the icosahedron example in Fig.~\figref{IcosaCuts3D}, Section~\secref{IcosahedronExample1},
$\tilde H_i = \tilde M_i$ for all $i$, and the resulting slit tree is the same.

%
%

Recall that Fig.~\figref{minl_rconv} is an example where $Z \neq W$.
Using the $Z$-algorithm, the first merge would be along the $a b a'$ geoarc,
whereas the $W$-algorithm would start with, say, the $ac$ geoseg.
So the two vm-reductions would not be identical, and would result in different
slit trees. 

\begin{lm}
\lemlab{mi_on_Z}
The new merge vertex $m_i$ created at step $i$ is
on $Z_i$ (as opposed to strictly interior to $Z_i$).
\end{lm}

\begin{proof} Similar to that for Lemma~\lemref{mi_on_hull}.
The algorithm constructs $M_i$ by removing
the two flattened vertices $m_{i-1}$ and $v_{i+1}$ and adding the new merge vertex $m_i$.
Formally,
$$
M_i = R\left( Z \left( \, (V_{i-1}\setminus \{m_{i-1},v_{i+1}\}) \cup \{m_i\}  \right) \right) \;.
$$
Let 
$M^-_{i-1} =R(Z ( V_{i-1} \setminus \{m_{i-1},v_{i+1}\} ))$;
see Fig.~\figref{mi_on_hull} with $N \to M$.
The geoarc $g_i = m_{i-1} v_{i+1}$ is strictly exterior to $M^-_{i-1}$,
because it is included in $Z_{i-1}$.
Now attach the doubled $T^2_i$ triangle to $g_i$, with triangle apex $m_i$.
It is clear that $m_i$ is also strictly exterior to  $M^-_{i-1}$.
Therefore, with $M_i = R(Z( M^-_{i-1} \cup m_i ))$, it must be that
$m_i$ is a vertex of $Z_i$.

For the exceptional geodesic-loop case of the algorithm, 
the conclusion follows by construction.
\end{proof}



\begin{lm}[Visibility]
\lemlab{mi_visib2}
Let $\g$ be the geoarc
connecting the new merge vertex $m_i$ created at step $i$
to the next vertex $v_{i+2}$ to be merged.
Then $\g$ crosses the geoarc $g_i$ at a point $\{m'_i\} = g_i \cap \g$.
In a sense, $v_{i+2}$ is ``visible" to $m_i$ through $g_i$.
\end{lm}

\begin{proof} As in the proof of Lemma~\lemref{mi_visib},
we only make use of the $\a\b$-convexity property of $Z_{i-1}$, which in this case is strict.
\end{proof}

\begin{lm}[Nesting]
\lemlab{Ni_Nesting}
$\tilde{M}_i  \subset \tilde{M}_{i-1}$.
\end{lm}

With some abuse of notation in the proof, 
we will identify objects on $P_{i-1}$ with their image on $P_i$ and vice-versa.


\begin{proof}Similar to that for Lemma~\lemref{Hi_Nesting}.

The proof is in two parts.
We show first that $\tilde{M}_i  \subset \tilde{M}_{i-1}$.
Set $V^-_{i-1}= V_{i-1} \setminus \{m_{i-1},v_{i+1}\}$, i.e.,
$V_{i-1}$ without the two vertices that will form the next vertex-merge insertion.
To keep track of the notation, we
refer to Fig.~\figref{HHH} (a variation of Fig.~\figref{mi_on_hull}) with $N \to M$,
and this display, where again we use ``$-$" to 
indicate ``missing" vertices 
among the three $m_{i-1}, m_i, v_{i+1}$:
\begin{align*}
\mbox \;&\;\; V^-_{i-1} = \{ v_{i+2}, \ldots, v_n \} \\
M_{i-1} = R( Z_{i-1} ) \;,&\;\; V_{i-1} = \{ m_{i-1}, -, v_{i+1} \} \cup V^-_{i-1} \\
M_{i} = R(Z_{i} ) \;,&\;\;  V_{i} = \{ -, m_i, -, \} \cup V^-_{i-1} \\
M^+_i = R(Z( V^+_i) ) \;,&\;\;  V^+_{i} = \{ m_{i-1}, m_i, v_{i+1} \} \cup V^-_{i-1}
\end{align*}
Set $M^+_i$ as above.
Then as Fig.~\figref{HHH} shows, $M_i \subset M^+_i$ 
(by Lemma~\lemref{V+1_subset_V} applied twice) 
and thus $\tilde M_i \subset \tilde{M}^+_i$.

\medskip
Denote by $T_i= \triangle m_{i-1} m_i v_{i+1}$ one of the two inserted triangles along $g_i$, 
the one sharing $g_i$ in common with $M_{i-1}$.

For the second part of the proof, we next show that, on $P_i$, $M^+_i = M_{i-1} \cup T_i$.

First, notice that $M_{i-1} \cup T_i \subset M^+_i$.
Second, the strict $\a\b$-convexity of $Z( V_{i-1})$ at $m_{i-1}$ and $v_{i+1}$ on $P_{i-1}$ shows that, on $P_i$, 
the two sides of $T_i$ incident to $m_i$ 
make with the respective boundary arcs angles $<\pi$ towards $M_{i-1} \cup T_i$, hence $> \pi$ on the other side;
see Lemma~\lemref{mi_visib2} (Visiblity).

Therefore, the boundary of $M_{i-1} \cup T_i$ is precisely the minimal length enclosing polygon of $V^+_i$ (by Lemma~\lemref{Converse_Lmin}).

Finally, since $M^+_i = M_{i-1} \cup T_i$ and the image on $P$ of $T_i$ is a subarc of $\tilde g_i$, 
we have $\tilde{M}^+_i = \tilde{M}_{i-1}$.
In conclusion, $\tilde M_i \subset \tilde{M}^+_i=\tilde{M}_{i-1}$.
\end{proof}

The proofs for the next four results are identical to the proofs or their counterparts in Section~\secref{ProofSlitTree} 
(Lemma~\lemref{consec_cuts} and the following results).

\begin{lm}[$g_i \cap g_{i-1}$]
\lemlab{consec_cuts_minl}
On $P_{i-1}$, $g_i$ intersects the two images
(``banks") of $g_{i-1}$ at precisely one point.
Consequently, $\tilde{g}_i$ has an endpoint on $\tilde{g}_{i-1}$ and no other common point.
\end{lm}

\begin{lm}
The set of vertices reduces by one each iteration:
$|V_i| = |V_{i-1}|-1$, for $i > 1$.
\end{lm}

\begin{lm}
\lemlab{Ztree}
$\tilde{\L} _i$ is a tree.
\end{lm}

Putting together the above lemmas, we obtain a counterpart to Theorem~\thmref{SpiralAlg}.

\begin{thm}
\thmlab{Ztree}
Let $Q$ be a simple closed quasigeodesic on the convex polyhedron $P$, and $V$ the set of vertices of $P$ enclosed by $Q$, to either side of $Q$.
Then the sequential vertex-merging algorithm,
based on $Z=\min\ell[V]$ and
detailed in Section~\secref{SpiralAlgorithm},
results in a slit graph $\tilde{\L}$ that is a tree.
\end{thm}

Application of these slit trees to unfoldings of $P$ is discussed in the next chapter.

%

\chapter{Unfoldings via Slit Trees}
\chaplab{Appl3DAlg}
In this chapter we gather together results from Chapter~\chapref{VMSlitGraph}
concerning the connectivity structure of slit graphs,
and the spiral slit trees just obtained in Chapter~\chapref{SpiralTree3D},
to draw conclusions about unfolding convex polyhedra $P$ 
via vertex-merging.

Depending on circumstances we will detail, vertex-mergings result in $P$
embedded in a doubly-covered triangle (or isosceles tetrahedron), 
in a pair of joined cones, or in a cylinder.
The slits may leave $P$ in one piece, or cut into several pieces.
In some situations, we can develop $P$ to the plane as a net.

\section{Notation}
We first recall previous notation to be used in this
chapter, and introduce some new notation.

\begin{itemize}
\item $Q$ is a simple closed quasigeodesic on $P$.
\item $P^+$ and $P^-$ are the closed half-surfaces bounded by $Q$.
Because closed, both of these half-surfaces include $Q$.
\item $V(Q)=\{q_1, \ldots, q_k\}$ is set of vertices on $Q$.
\item $\e$ is either of $+, -$.
\item We refer to the $Z$-algorithm (minimum enclosing polygon)
and the $W$-algorithm ($\partial \, \rconv$),
summarized in
Theorems~\thmref{Zalgorithm}
and~\thmref{Walgorithm} respectively.
%
\item $\tilde{\L}^\e$ is the full slit tree on $P^\e$, obtained from a set $V$ of vertices containing all vertices interior to $P^\e$, 
and possibly some vertices in $V(Q)$, via either the $Z$- or $W$-spiral algorithm.
\item 
$\a_i$ and $\b_i$ are the angles incident to $q_i$
in $P^+$ and in $P^-$ respectively, i.e., above and below $Q$.
\item The curvature of $P$ at $q_i$ is
$\o_i = 2 \pi - (\a_i + \b_i)$.
\item 
The \emph{partial curvatures} at $q_i$ toward $P^+$ and $P^-$
are $\o_i^+=\pi - \a_i$ and $\o_i^-=\pi-\b_i$, respectively.
Hence $\o_i=\o_i^+ + \o_i^- < 2\pi$.
\item $\O^\e$ is the total curvature of all interior vertices of $P^\e$.
\end{itemize}

The Gauss-Bonnet Theorem gives, for $\e \in \{+,-\}$, 
\begin{align}
\O^\e + \sum_{i=1}^k \o_i^\e &= 2 \pi \;, \eqnlab{GBepsilon}\\
\O^+ + \O^- + \sum_{i=1}^k \o_i &= 4 \pi \;.
\end{align}


\section{Unfoldings via spiraling algorithms}

Assume in the following that $V$ contains \emph{all} vertices of $P$ enclosed by $Q$, to either side, including those on $Q$: $V \supset V(Q)$.

Next we see that the spiraling algorithm works fine in this case, too.

The total curvature of $V$ may be $> 2 \pi$ on $P$, but on $P^\#$ it is precisely $2 \pi$, because all angles $\b_i$ are $\pi$ on $P^\#$.

\bs

Choose two consecutive vertices on $Q$, $q_1, q_2$. 
If there are no such vertices, then 
$Q$ is either a closed geodesic or a geodesic loop, handled later.

Merge $q_1$ and $q_2$ along the geoarc $g_{12}$ of $Q$ joining them (either arc, if there are two), to produce $m_1 \in V_1 \subset P_1$.
This merging inserts a double triangle of base $g_{12}$ and base angles $\o^+_i /2$ at $v_i$.
Call this a \emph{partial merging}.

By Lemmas~\lemref{mi_on_hull} and~\lemref{mi_on_Z},
$m_1$ is interior to $P_1$. 
Further merging $m_1$ to another vertex $q_3 \in Q$ would force $\tilde{\L}^+$ to have a leaf on $Q$.
Therefore, if $k = |V(Q)| > 2$ then $\tilde{\L}^+$ intersects $Q$ at $g_{12}$ and at each $q_j$ with $\a_j< \pi$, for $j=3, \ldots, k$.

For example, return to Fig.~\figref{IcosaCuts3D}(a),
repeated below as Fig.~\figref{IcosaCutsOnly3D}.
Suppose that $Q=v_1 v_2 v_3 v_4 v_5$ is the pentagon of edges surrounding $v_6$.
Then $\tilde{\L}^+$ has leaves at $v_3,v_4,v_5$.

Similar considerations lead to the the conclusion that  $\tilde{\L}^-$ also intersects $Q$ at $g_{12}$ 
(if it starts by merging $q_1$ and $q_2$)
and at each $q_j$ with $\b_j< \pi$, for $j=3, \ldots, k$.


Continuing with the example in Fig.~\figref{IcosaCutsOnly3D},
if $\tilde{\L}^-$ also had leaves at, say, $v_3$ and $v_4$,
then a piece of $P$, bounded on $P^+$ by $v_3 m'_2 v_4$,
and similarly bounded on $P^-$,
would be disconnected from the remainder of $P$.

The above discussion of 
the way the trees $\tilde{\L}^\e$ connect to 
one another along $Q$ leads to several main results.


We see two main options: merge the vertices strictly inside $P^\e$,
or, in addition, include all vertices on $Q$ via partial merging.


\subsection{Two Cones}
\seclab{TwoCones}
We start with merging all vertices inside $Q$.

\begin{lm}
\lemlab{Unf2Cones}
Let $P$ be a convex polyhedron and $Q$ a simple closed quasigeodesic on $P$.
Merging all vertices strictly inside $Q$ to either side, along the slit trees $\tilde{\L}^\e$ ($\e = +,-$), unfolds $P$ onto the union $\cal{U}$ of two cones, each of base $Q$, glued 
together along $Q$.
The unfolding of $P$ onto $\cal{U}$ 
can be decomposed into two simple geodesic polygons, one to each side of $Q$, sharing $Q$.
\end{lm}
\begin{proof}
Merging all vertices inside $P^\e$ reduces all those vertices to one merge vertex,
the apex of a cone. Because $Q$ is unaffected by the merges,
it remains shared by the cones.
Theorem~\thmref{Connectivity} 
(that $\tilde{\L}^\e$ a tree implies that $P^\e$ is a simple polygonal domain)
guarantees the $P$ unfolding is a simple geodesic polygon to each side of $Q$.
\end{proof}

Notice that the image of $Q$ on $\mathcal{U}$ above is not necessarily planar.

A special case of this lemma is that, if there is just one vertex $q_1$ on $Q$,
so $Q$ is a geodesic loop,
then $\cal{U}$ has three vertices---two cone apexes and the unmerged 
$q_1 \in Q$---and so is a doubly-covered triangle.

Another special case of this lemma is if $Q$ is a simple closed geodesic;
i.e., there is no vertex on $Q$.
We defer this case to Theorem~\thmref{Geodesic_VM_net} below.


\paragraph{Example: Spiral $Q$ on Box.}
\seclab{SpiralQBox}
We illustrate Lemma~\lemref{Unf2Cones}
with an example based on Fig.~2 in~\cite{demaine2020finding},
reproduced in Fig.~\figref{ErikSpiralGeo}.
Its interesting feature is that the illustrated 
simple closed quasigeodesic spirals around
the box, arbitrarily many times as $L$ grows large.
\begin{figure}[htbp]
\centering
\includegraphics[width=0.3\textheight]{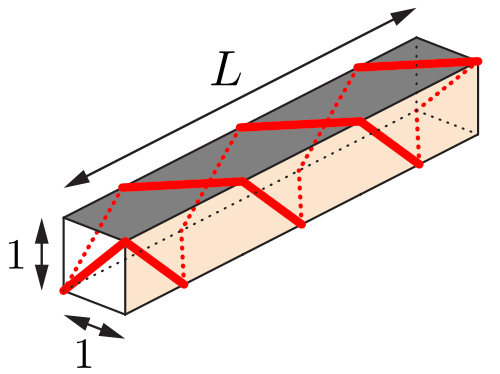}
\caption{Detail from Fig.~2 in~\protect\cite{demaine2020finding}.}
\figlab{ErikSpiralGeo}
\end{figure}
For simplicity, in our version,
Fig.~\figref{QgeoSpiral_abc}(a),
the quasigeodesic $Q$ spirals just one turn, but
the example could be extended to many turns.
\begin{figure}[htbp]
\centering
\includegraphics[width=0.8\textheight]{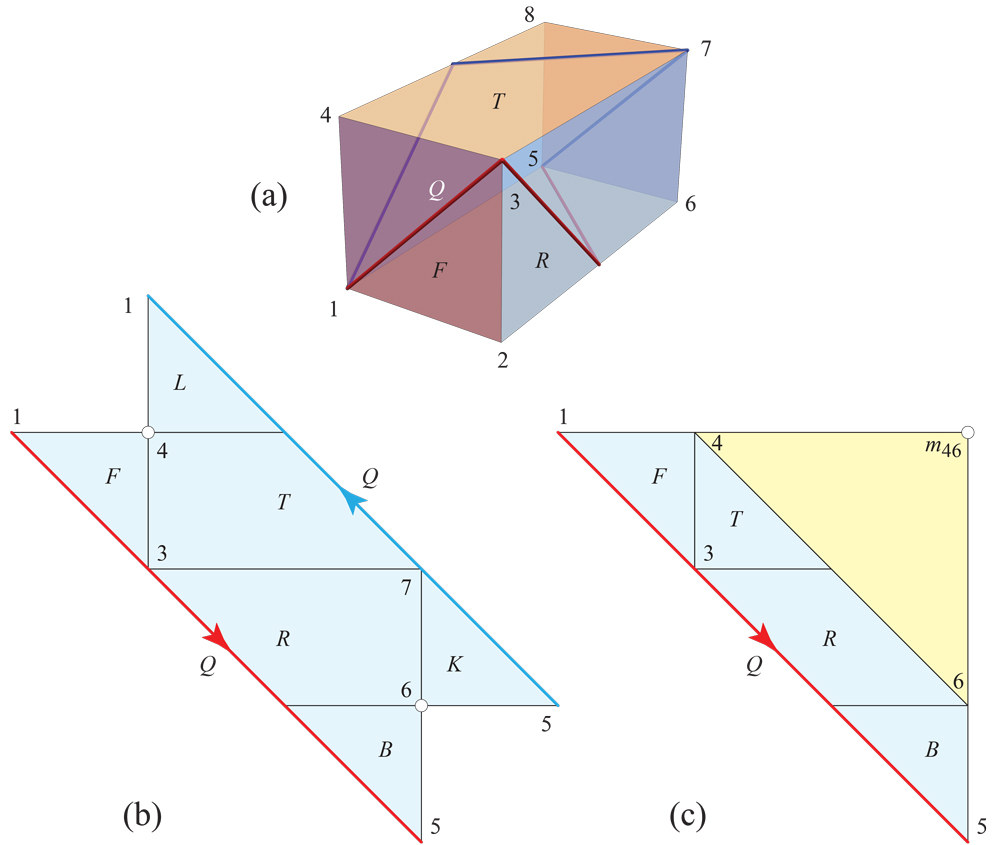}
\caption{(a)~A $1 \times 1 \times 2$ box.
(b) The surface $P^+$ to the left of $Q$. Note the two images of $v_1$
are identified, as are the images of $v_5$.
$\o(v_4) = \o(v_6) = \pi/2$.
(c)~After merging $v_4$ and $v_6$, a cone, here flattened to a doubly-covered triangle.}
\figlab{QgeoSpiral_abc}
\end{figure}

Our quasigeodesic $Q= v_1 v_3 v_5 v_7$ passes through four vertices,
and encloses two vertices, $v_4$ and $v_6$, strictly to its left.
Fig.~\figref{QgeoSpiral_abc}(b) shows an
unfolding of $P^+$, the half-surface left of $Q$.
$P^-$ (not shown) is symmetrical to the right of $Q$.

Merging the two vertices $v_4$ and $v_6$ produces a cone
apexed at the merge vertex $m_{46}$.
For display purposes, it is flattened to the plane as a doubly-covered triangle
in (c)~of the figure.
One can see that the portion of $P$ to the left of $Q$ is indeed a simple geodesic
polygon, as claimed by Lemma~\lemref{Unf2Cones}.

Merging $v_2$ and $v_8$ on $P^-$ produces a second cone apexed at $m_{28}$.
The two cones are glued together along $Q$, which still contains four vertices.
So ${\cal U}$ is a polyhedron of $6$ vertices, an octahedron.

\bs

It is tempting to develop each cone to the plane to achieve nets of 
$P^\e$.
This works in the example just presented, but not always, as we now detail.

Let ${\cal U}^+$ be the upper cone containing $P^+$.
$P^+$ is topologically an annulus, with $\partial P^+ = Q \cup Q'$.
Parametrize points $p(t)$ on the upper boundary $Q'$, $t \in [0,1]$,
with $p(0)=p(1)$.
Define $\phi(t)$ to be the angular turn of the segment  from $a p(0)$
to $a p(t)$ about the apex $a$ on the surface of ${\cal U}^+$ . 
Then with $\phi(0)=0$, we have that $\phi(1) = \a$, for the segment
makes one full turn to return to $p(0)$.
However, it could be that if $Q'$ spirals around the cone, intermediate positions
represent turns greater than $\a$.
See Fig.~\figref{ConeSpiral}.

\begin{figure}[htbp]
\centering
\includegraphics[width=0.5\textheight]{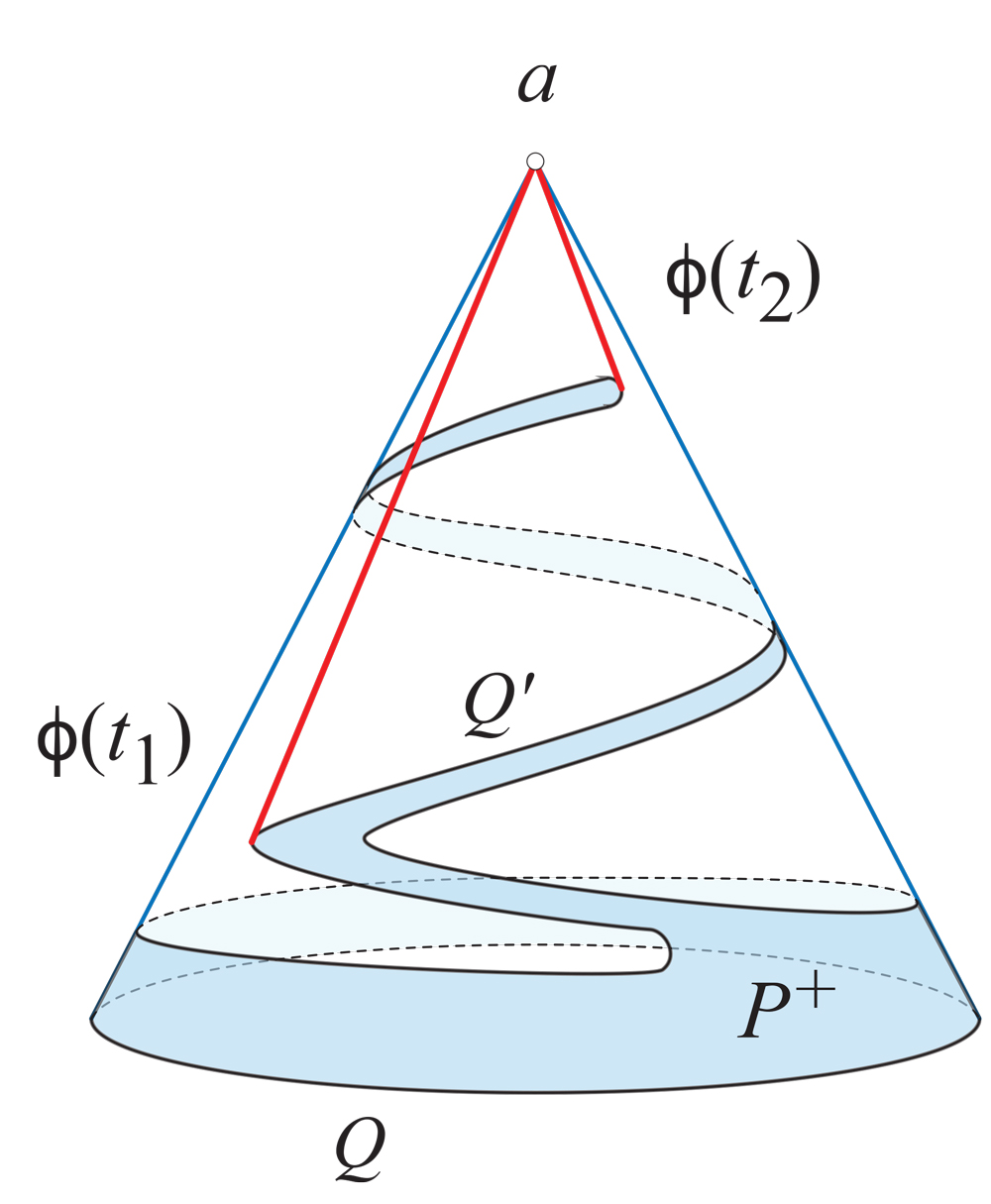}
\caption{$\phi(t_2)-\phi(t_1) > \a$.}
\figlab{ConeSpiral}
\end{figure}

Let $\phi_{\max} = \max_{t_1 < t_2} | \phi(t_2)-\phi(t_1) |$.
Then $k = \lceil \phi_{\max} / \a \rceil$ represents the number of full turns
around the cone that might be needed to fully develop $P^+$ to the plane.
No overlap in the development can occur unless $k \a > 2 \pi$,
so that the imprint of the cut-open cone cycles around $a$ more than once.
Therefore we have this result.

\begin{prop}
\lemlab{ConeSpiral}
If $k \a \le 2 \pi$, where $k = \lceil \phi_{\max} / \a \rceil$,
then $P^+$ develops to the plane without overlap.
\end{prop}

In the absence of the limit detailed in this 
proposition, it is indeed possible for
the development of $P^+$ to overlap.
This was established with a nontrivial example in~\cite[Fig.~14]{ov-ceccc-14}
by a curve that spirals around $a$ four times.
Building this curve into $P^+$ would establish an overlapping development of $P^+$.

Another very special case occurs when $\a$ evenly divides $2 \pi$. For then, even if the development
spirals around $a$ more than $2 \pi$, the imprint of the cut-open cone lays exactly
on top of its earlier-rolled image. So any overlap in the development would have been mirrored
by overlap on the cone. 
But $P^+$ does not overlap on ${\cal U}^+$.

\begin{prop}
\lemlab{Divides2pi}
If $\a$ evenly divides $2 \pi$, then $P^+$ develops without overlap.
\end{prop}


\subsection{Reduction to Cylinder}
\seclab{Cylinder}

We return now to the second option:
including all vertices on $Q$ via partial merging.

\begin{lm}
\lemlab{Unf2Cylinder}
With $P$ and $Q$ as before, let there be $k=|V(Q)|$ vertices on $Q$.
Merging all vertices of $P$,
along the trees $\tilde{\L}^\e$,
including the vertices $V$ via partial merges as described above,
unfolds $P$ onto a cylinder $\cal{C}$.
The unfolding of $P$ onto $\cal{C}$ is the union of at most $k-2$ simple geodesic polygons, joined circularly at $k-2$ vertices of $Q$.
\end{lm}

\begin{proof}
The reason the final surface is a cylinder is that 
(a)~all vertices are merged; none remain,
and (b)~the curvature to each side is exactly $2\pi$: Eq.~\eqnref{GBepsilon}.
Although Theorem~\thmref{Connectivity} 
again guarantees the $P$ unfolding to each side of $Q$ is
a simple geodesic polygon, as we've seen, the coinciding leaves of
$\tilde{\L}^\e$ partitions $P$ into at most $k-2$ pieces.
\end{proof}

In contrast to the two-cones case, rolling the cylinder $\cal{C}$ on the plane cannot
cause overlap. 

%


\bs

We've already mentioned that $k=|V(Q)|= 1$ is a special case
of Lemma~\lemref{Unf2Cones},
and $k \le 2$ plays a special role in Lemma~\lemref{Unf2Cylinder}.
We highlight the latter in the following theorem.

\begin{thm}
\thmlab{RollingNet}
Let $P$ be a convex polyhedron and $Q$ be a simple closed quasigeodesic on $P$ containing $k=|V(Q)|$ vertices.
If either
\begin{enumerate}[label={(\arabic*)}]
\item $k \le 2$, or
\item
if the total number of non-$\pi$ angles $\a_i$ and $\b_i$ (i.e.,  $< \pi$ angle to both sides) 
is at most two,
\end{enumerate}
then the unfolding of $P$ onto $\cal{C}$ 
(Lemma~\lemref{Unf2Cylinder})
is a simple geodesic polygon.
In this case, rolling $\cal{C}$ onto a plane develops $P$ to a net.
\end{thm}

\begin{proof}
Lemma~\lemref{Unf2Cylinder} established the specialness of $k \le 2$: then
the unfolding of $P$ onto $\cal{C}$ is a single, simple geodesic polygon.

When either of the angles $\a_i$ or $\b_i$ incident to $q_i$ is equal to $\pi$,
there is no need to partially merge to $q_i$ from one side or the other.
The two slit trees can have a common leaf only at a vertex $q_i$ 
having to both sides angles  strictly less than $\pi$.
This reduces the number of pieces from $k-2$.
\end{proof}

Returning to Fig.~\figref{IcosaCutsOnly3D}
with $Q=v_1 v_2 v_3 v_4 v_5$, $\a_i=\frac{2}{3}\pi$ and $\b_i=\pi$,
so there are no angles $< \pi$ to both sides, and Theorem~\thmref{RollingNet} applies. 
We will indeed see below (Fig.~\figref{IcosaCylPhotos}) this leads to a net
for the icosahedron.

We believe that every convex poyhedron has a a simple closed quasigeodesic containing at most one vertex, see Open Problem~\openref{QVert2}.
If this would be true, the above result would provide vertex-merging nets for all convex polyhedra.
We discuss this question of 
the number of vertices on a simple closed quasigeodesic in Chapter~\chapref{VertQuasigeo}.

\bs

The case $k=0$  (i.e., $Q$ is a simple closed geodesic) is considered next.
In this case, $P$ contains a region $R$ isometric to a cylinder; assume $R$ is maximal with respect to inclusion.
Then the two boundary components of $R$ are simple closed quasigeodesics, denoted by $Q^\e$.
We then denote by $P^\e$ the caps of $P$ bounded by $Q^\e$ outside $R$, and apply the previous considerations.
Precisely, we unfold both $P^\e$ onto the same cylinder, which in this case contains $R$, and further roll the cylinder to 
obtain a net of $P$.

The case $k=0$ is also special in that we can obtain an unfolding of $P$ onto an isosceles tetrahedron 
(as opposed to the doubly covered triangles obtained in Lemma~\lemref{Unf2Cones} for $k=1$).
In this case, $Q$ is a simple closed geodesic.
Then the merging processes end with two vertices of total curvature $2 \pi$, to each side of $Q$. 
If both of those vertices (to the same side of $Q$) have curvature $\pi$, we 
have reached half of an isosceles tetrahedron.
Otherwise, suitably choosing a partial merge between the two vertices results in one of them having
curvature $\pi$, and a new merge vertex also of curvature $\pi$. Globally, we obtain $\cal{U}$ an isosceles tetrahedron,
the special case of vm-irreducible surfaces.
This discussion is consistent with the above argument: once we have an isosceles tetrahedron, 
we can obtain the cylinder by cutting open a pair of opposite edges.

\begin{thm}
\thmlab{Geodesic_VM_net}
If $P$ contains a simple closed geodesic then, unfolding it as described above provides a net.
\end{thm}

We next provide two examples illustrating 
this theorem.


\subsection{Cube Example}

To illustrate Theorem~\thmref{Geodesic_VM_net},
we revisit the cube example in Section~\secref{CubeExamples}
(Fig.~\figref{Cube_VM_triangle_3D}).
There we sequentially-merged three vertices on the top face,
$v_7,v_8$ and then $v_5$ with $v_{78}$, and symmetrically three on
the bottom face. 

Instead now we 
merge all four vertices on the top face, and then
on the bottom face. So, for a quasigeodesic $Q$ around the middle of the cube, we
are merging all vertices inside $P^+$, and all vertices inside $P^-$,
leading to a cylinder.

We perform the same two initial merges on the top vertices, resulting again in
$v_{578}$. 
Let $a = m'_1$ and $b=m'_2$, the two points on the top face where the
geodesic segments from the merge vertices $m_1=v_{12}$
and $m_2=v_{578}$ enter the top face of the cube.

Note that the angle incident to the merge vertex $v_{578}$
is $90^\circ$, 
and the angle incident to the as-yet unmerged top vertex $v_6$
is $270^\circ$. 
So a merge of $v_{578}$ with $v_6$ would not produce a triangle pair,
because the sum of their curvatures is $2 \pi$.
However, we can imagine a merge resulting in a pair of unbounded parallelograms.
If we cut the surface along the geodesic segment $v_6 v_{578}$
(of length $2 \sqrt{2}$ for a unit cube)
and insert twin parallelograms, the result is a cylinder unbounded above.
See Fig.~\figref{CubeCyl_3D}.
Note the two $45^\circ$ angles inserted at $v_6$ flattens that vertex,
and the insertion of the two $135^\circ$ angles flattens $v_{578}$.
Symmetric merges on the bottom-face vertices leads to an unbounded
cylinder in both directions.

\begin{figure}[htbp]
\centering
\includegraphics[width=0.85\textheight]{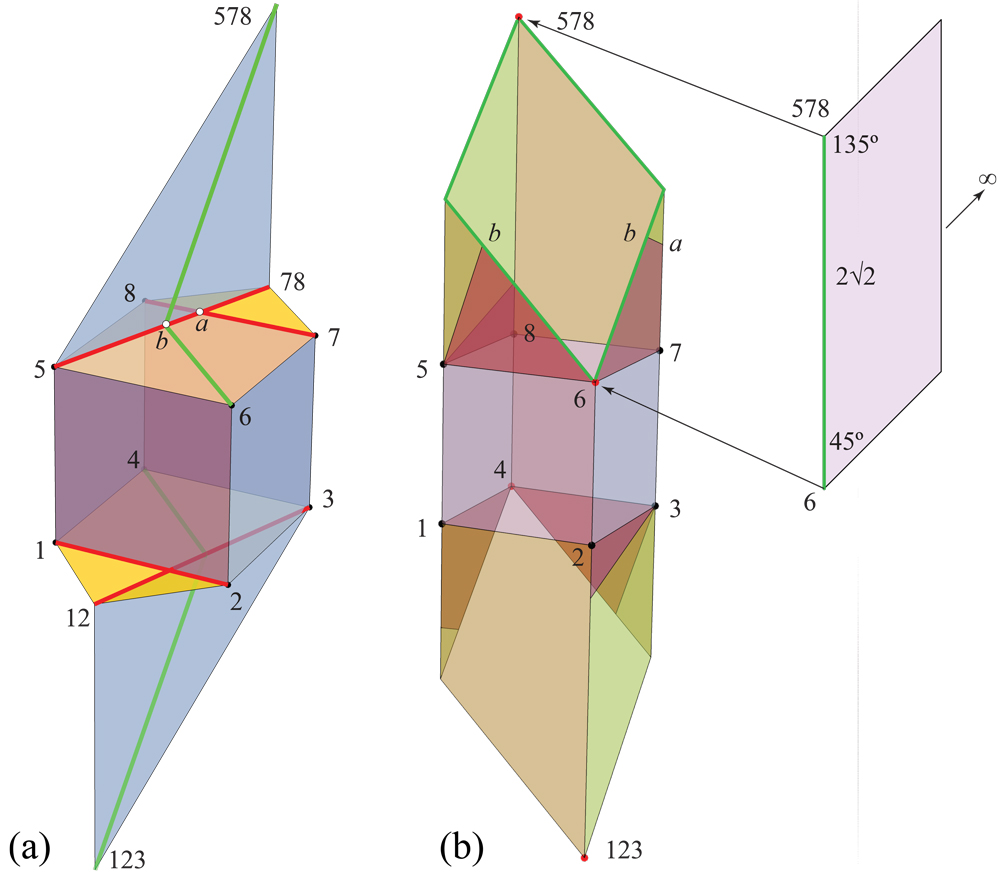}
\caption{(a)~Cutting along the $v_6 v_{578}$ geodesic segment (green),
and
inserting double parallelograms~(b) 
leads to an unbounded cylinder above,
and similarly below.
In~(b) the yellow regions are inserted merge triangles; pink regions
pieces of the top and bottom cube faces.}
\figlab{CubeCyl_3D}
\end{figure}

Unrolling the cylinder unfolds the cube to a non-overlapping net:
Fig.~\figref{CubeCylUnf}.
\begin{figure}[htbp]
\centering
\includegraphics[width=0.75\linewidth]{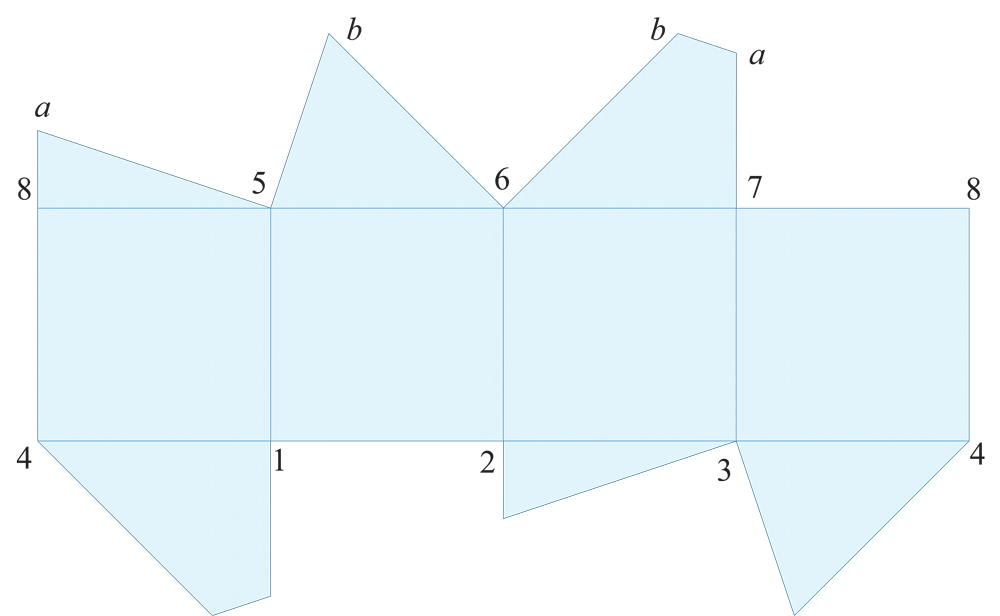}
\caption{Unfolding of cube to a net by rolling the cylinder on the plane.}
\figlab{CubeCylUnf}
\end{figure}

\newpage
\subsection{Icosahedron Example}
For a second example of Theorem~\thmref{Geodesic_VM_net},
we revisit and continue the example of the icosahedron from 
Sections~\secref{IcosahedronExample1} and~\secref{IcosahedronExample2}.
Fig.~\figref{IcosaCutsOnly3D} below repeats
for easy reference
Fig.~\figref{IcosaCuts3D}(a) 
showing the cuts $\tilde{g}_i$.
\begin{figure}[htbp]
\centering
\includegraphics[width=0.5\linewidth]{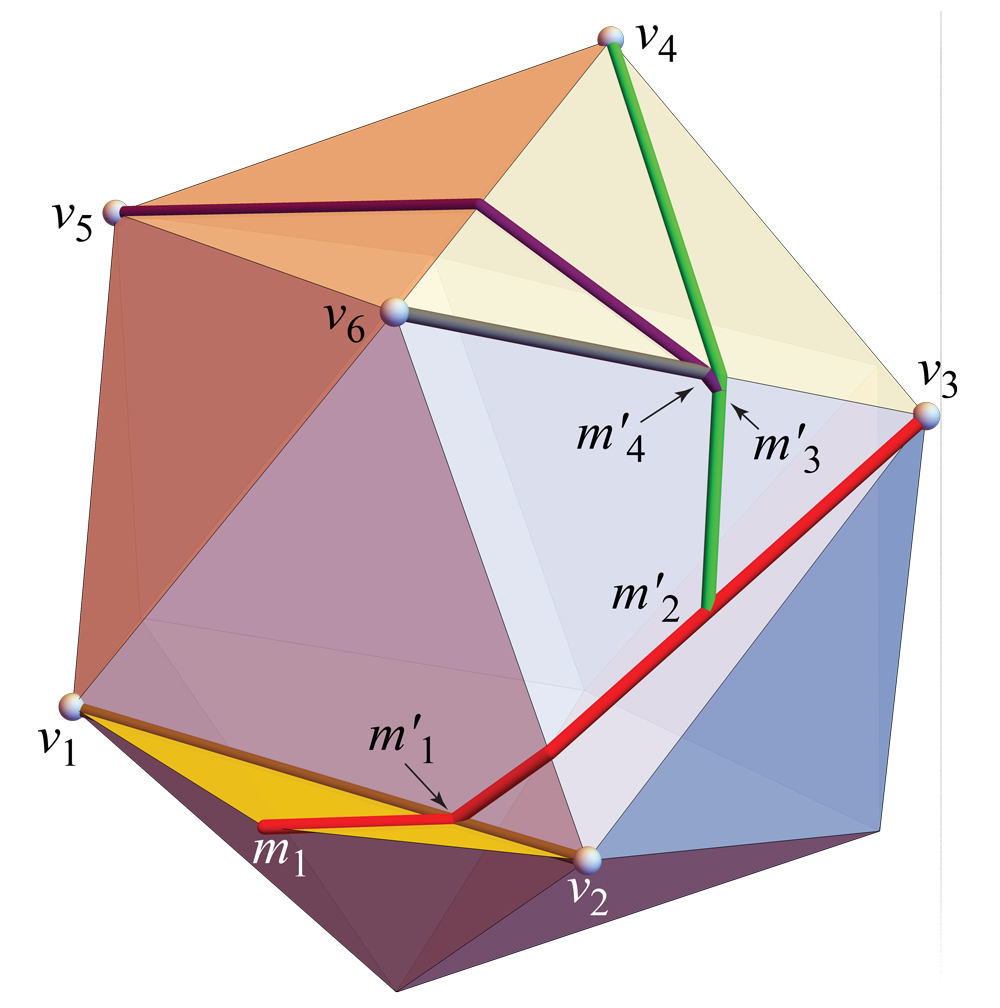}
\caption{The vertex-merge cuts on $P$.
Repeat of Fig.~\protect\figref{IcosaCuts3D}(a).}
\figlab{IcosaCutsOnly3D}
\end{figure}

The triangles inserts $T_1,T_2,T_3,T_4$ are shown in
Fig.~\figref{IcosahedronTrianglesOnly}.
Note the half angles at the merge vertex $m_i$ apexes of the triangles 
are $120^\circ,90^\circ,60^\circ,30^\circ$ respectively.
The fifth merge vertex $m_5$ would have zero angle, and represents the
infinite parallelogram that sends $m_5$ to infinity.
\begin{figure}[htbp]
\centering
\includegraphics[width=0.75\linewidth]{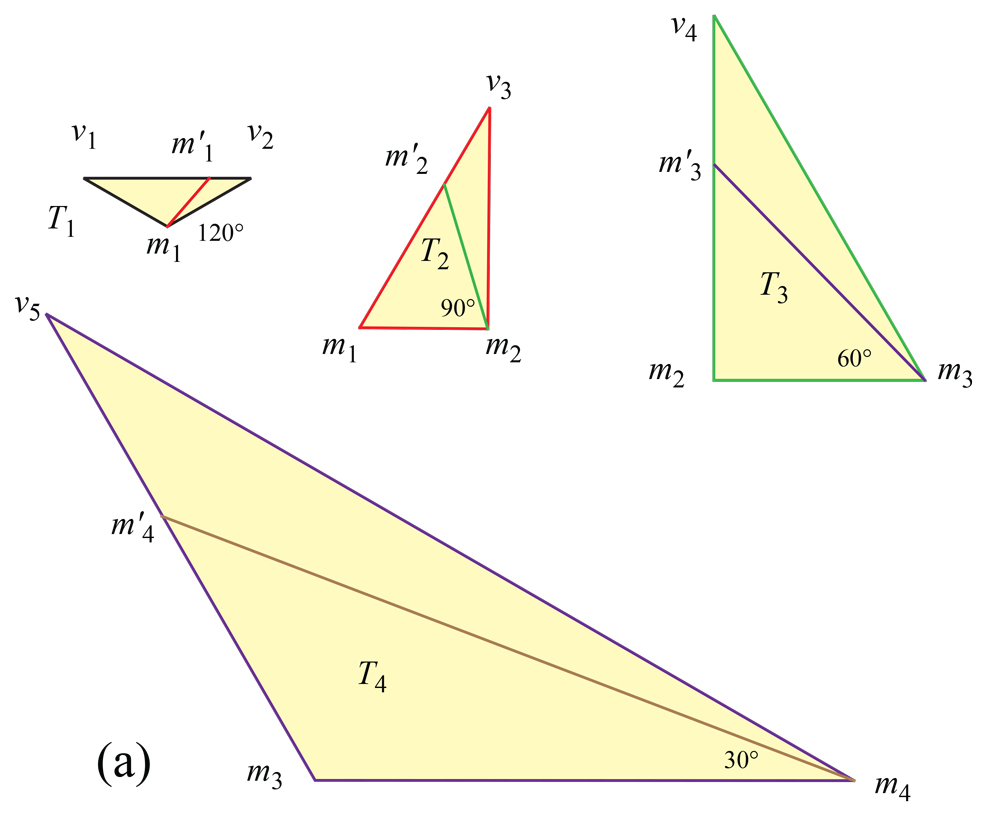}
\caption{The four triangles $T_i$. Each is doubled to $T^2_i$ and inserted
along $g_i$.}
\figlab{IcosahedronTrianglesOnly}
\end{figure}

Inserting the doubled triangles $T^2_i$ along the cuts $g_i$ 
produces the layout shown in Fig.~\figref{IcosaAssembly}.
Here we only show the top half of the icosahedron.
\begin{figure}[htbp]
\centering
\includegraphics[width=1.0\linewidth]{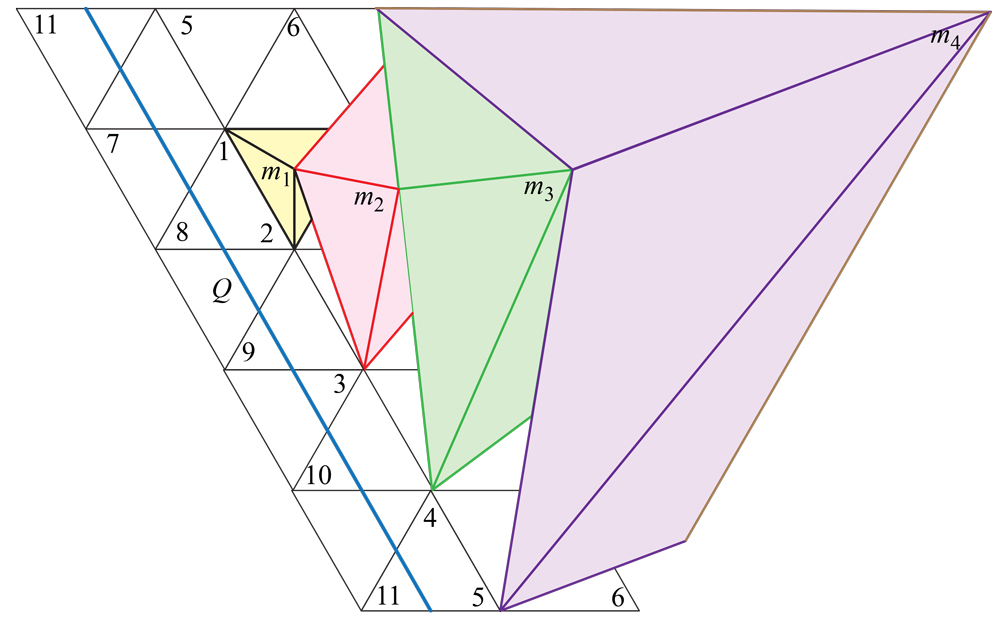}
\caption{Unfolding of the top half of the icoshedron with $T^2_i$ 
inserted (shaded polygonal domains), $i=1,2,3,4$.
Icosahedron faces are white; $Q$ is blue.}
\figlab{IcosaAssembly}
\end{figure}

The assembly closes to a half-cylinder shown in Fig.~\figref{IcosaCylPhotos}.
The infinite parallelogram $T_5$ with angles $60^\circ$ and $120^\circ$
(not shown)
would attach at $v_6$ and $m_4$ respectively.
\begin{figure}[htbp]
\centering
\includegraphics[width=0.75\linewidth]{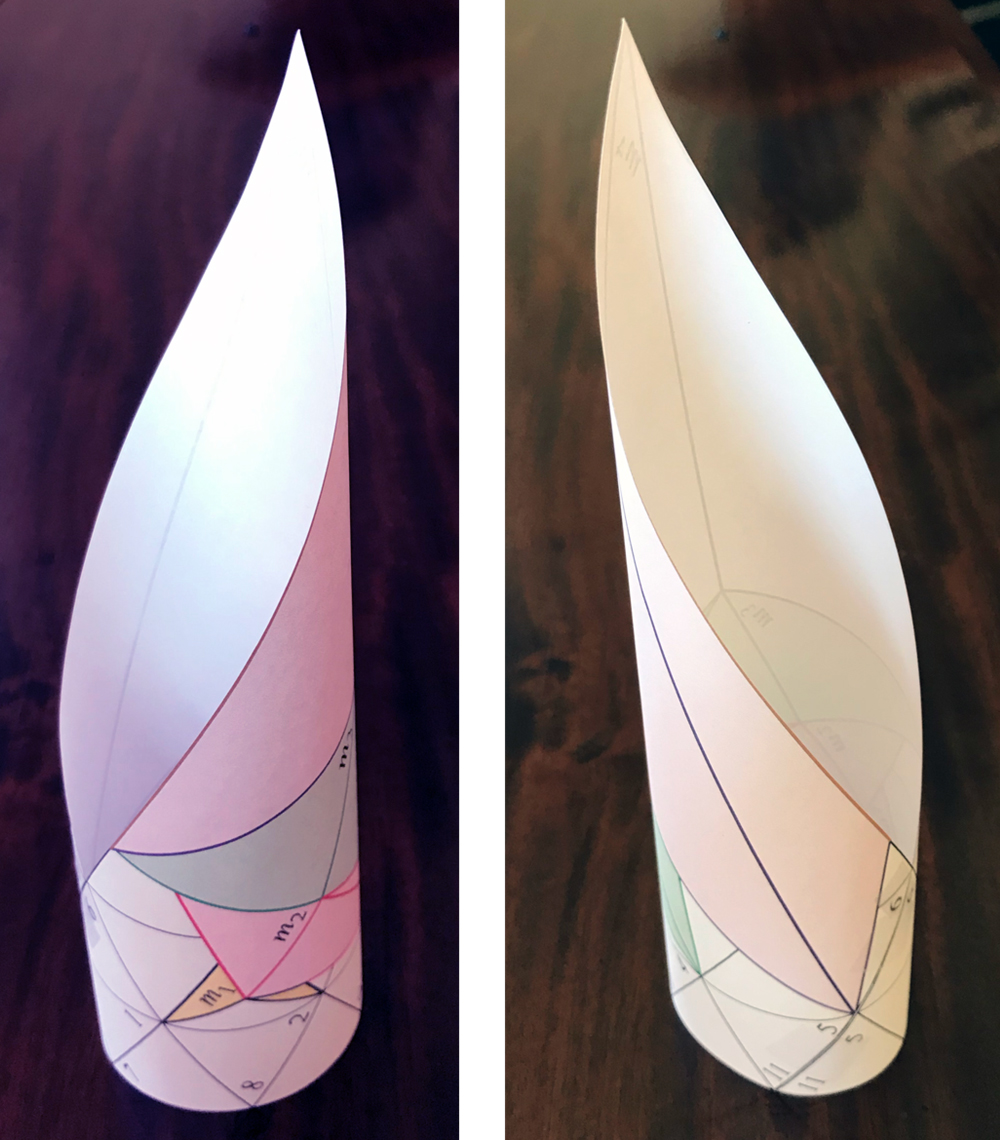}
\caption{The half-cylinder obtained from Fig.~\protect\figref{IcosaAssembly}
by joining the two images of $v_{11} v_5 v_6$.}
\figlab{IcosaCylPhotos}
\end{figure}
Joining the two symmetric half-cylinders together and rolling on the plane produces
a net for the icosahedron, the white faces in Fig.~\figref{IcosaAssembly}.

Earlier we discussed the quasigeodesic $v_1 v_2 v_3 v_4 v_5$, call it $Q'$.
Note that $Q'$ can be viewed as the $Q$ in Fig.~\figref{IcosaAssembly}
slid upward parallel to itself
until it touches those vertices.
So both lead to the same embedding on the cylinder $\cal{C}$.


%
%
%


\chapter{Vertices on Quasigeodesics}
\chaplab{VertQuasigeo}
Theorem~\thmref{RollingNet} demonstrated the importance 
in our context of the number of vertices on a quasigeodesic.
If, as we conjecture in Open Problem~\openref{QVert2},
every convex polyhedron $P$ has a quasigeodesic $Q$ containing at most
one vertex,
then the vertex-merging described in that theorem leads to an unfolding
of $P$ to a cylinder $\cal{C}$ and then to a net for $P$.
In this chapter, we prove that the space ${\cal P}_n$ of convex polyhedra of $n$ vertices
contains infinitely many polyhedra with quasigeodesics through one or two vertices,
or in fact through any number $k$ of vertices, $1 \le k \le n$.
``Infinitely many" is expressed as a set 
in ${\cal P}_n$ with non-empty interior.
Precise definitions are given below.


\section{Notation}
We continue to use $P$ for a convex polyhedron and
$Q$ for a quasigeodesic.
As in the previous chapter,
$V(Q)=\{q_1, \ldots, q_k\}$ is the set of vertices on $Q$; hence $k=|V(Q)|$.
The spaces of polyhedra are defined as follows:

\begin{itemize}
\item For any $n \in \N$, let  ${\cal P}_n$ be the space of all convex polyhedra in $\Rs$ with $n$ vertices,
with the topology induced by the usual Pompeiu-Hausdorff metric.
Two polyhedra in ${\cal P}$ are then close to each other if and only if they have close respective vertices.
(See Chapter~\chapref{Punfoldings}.) 
\item For all $0 \leq k \leq n$, ${\cal Q}_k ={\cal Q}_k(n)$ denotes  the subset of ${\cal P}_n$ such that
every $P \in {\cal Q}_k$ has a simple closed quasigeodesic $Q$ through exactly $k$ vertices.
\end{itemize}

\section{${\cal Q}_k$ Theorem}
The goal of this chapter is to prove the following result.

\begin{thm}
\thmlab{Q_012}
$\mbox{}$ \\ 
\vspace{-12pt}
\begin{enumerate}[label={(\arabic*)}]
\item ${\cal Q}_0$ has empty interior in ${\cal P}_n$.
\item ${\cal Q}_k$ has non-empty interior in ${\cal P}_n$, for all $1 \leq k \leq n$.
\end{enumerate}
\end{thm}

\noindent
If our conjecture in Open Problem~\openref{QVert2}
is true, then ${\cal Q}_0 \cup {\cal Q}_1 = {\cal P}_n$.

It is not surprising that ${\cal Q}_0$ is rare, as it requires a partition of
vertex curvatures into two halves of exactly $2\pi$ each.
Indeed, (1)~is a known result~\cite{g-tcscc-91}, \cite{gal2003convex}.
As far as we are aware, (2)~is new.

The proof of the second part of Theorem~\thmref{Q_012} will follow directly from 
several lemmas.
Our approach is to identify a polyhedron $P$
in ${\cal Q}_k$ for each $k$, and then ``fatten" $P$ via the next lemma.

\begin{lm}
\lemlab{Q_on-neighborhood}
Assume the convex polyhedron $P$ has a simple closed quasigeodesic $Q$ with 
$|V(Q)|=k \geq 1$ and 
\begin{equation}
\label{ab-strict}
\a_i < \pi, \: \: \b_i < \pi, \: \: \forall \, 1\leq i \leq k. \tag{$\ast$}
\end{equation}
Then, all polyhedra $P'$ sufficiently close to $P$ in ${\cal P}_n$ have such a quasigeodesic.
\end{lm}
\begin{proof}
The complete angles at the vertices of $P$ depend continuously on the vertex positions in $\Rs$.
Also, the geoarcs between two vertices remain separated from other vertices by positive distances,
so they also depend continuously on the vertex positions,
for small perturbations of the vertices of $P$.
Therefore, the strict inequalities $\a_i < \pi$, $\b_i < \pi$ everywhere
also remain strict for small perturbations of the vertices of $P$.

In other words, there exists a neighborhood of $P$, 
such that each polyhedron $P'$ in this neighborhood
still has a quasigeodesic $Q'$ with the same number of vertices as $Q$:
$|V(Q)|=|V(Q')|$.
And this is true for all $Q$ on $P$.
\end{proof}

In view of Lemma~\lemref{Q_on-neighborhood}, in order to prove the second part of Theorem~\thmref{Q_012}, it suffices to find examples of polyhedra
of $n$ vertices with $Q$ satisfying $|V(Q)|=k$ and Eq.~\eqref{ab-strict}, for 
each $1\leq k \leq n$.


\subsection{$|V(Q)|=1$}
Following the plan just articulated, we will identify doubly-covered convex 
polygons in ${\cal Q}_1$, and then apply Lemma~\lemref{Q_on-neighborhood}.

Recall that the \emph{width} of a convex polygon $P$ is the shortest distance between parallel supporting lines.
The characterization~(a) in the lemma below has long been known.
We also need~(b), which we could not find in the literature.

\begin{lm}
\lemlab{Width}
The width of $P$ is always achieved by (a)~one of the supporting lines flush with an edge $e$
of $P$, and (b)~the other line including a vertex that projects onto $e$.
\end{lm}

\begin{proof}
Claim~(a) is Theorem~2.1 in~\cite{houle1988computing}, 
and known earlier. Their proof 
uses a lemma (their Lemma~2.1) that says the following.
Let $u$ and $v$ be two points in the plane,
and parallel lines $L_u$ and $L_v$ through each respectively.
Then there exists a ``preferred direction of rotation" of the lines,
maintaining parallelism and maintaining contact with $u$ and $v$, 
that diminishes the separation between the lines. 
They call this the PDR lemma.
In the special case when the lines are
orthogonal to $vu$, then both directions of rotation reduce separation.
%
%

Claim~(b) relies on this PDR lemma.
But as they skip proving their lemma, we include a proof here.
Let $a b = e$ be the edge of $P$ with supporting line
$L_{ab}$ including $e$. 
Let the parallel line be $L_c$, including only vertex $c$,
with the separation between the lines width $w$.
Assume, in contradiction to claim~(b), that $c$ does not project onto $ab$.
The situation is then as depicted in Fig.~\figref{WidthPoly},
with $ac$ playing the role of $uv$ in the PDR lemma.

Rotate $L_{ab}$ and $L_c$ about $a$, clockwise in the figure---this is
the preferred direction.
Then $L_c$ pivots around point $p$, and is no longer supporting $P$.
Now move it (down in the figure) until it again contacts $c$. Call the new lines $L'_{ab}$ (red)
and $L'_c$ (green), and let $w'$ be their separation.
Note that $h^2 = w^2 + x^2$,
and $h^2 = (w')^2 + (x')^2$.
The rotation ensures that $x' > x$, and therefore $w' < w$, contradicting
the assumption that the width is $w$.

If $c$ instead does project into $ab$, as per claim~(b), 
say to point $q \in ab$,
then the line $L_{ab}$
is blocked from the preferred direction of rotation because it would penetrate $P$ at $q$.

If $c$ is the endpoint of an edge $cd$ parallel to $ab$, then
at least one of the four vertices $\{a,b,c,d\}$ must project onto the opposite edge,
so claim~(b) holds.
\end{proof}

\begin{figure}[htbp]
\centering
\includegraphics[width=1.0\linewidth]{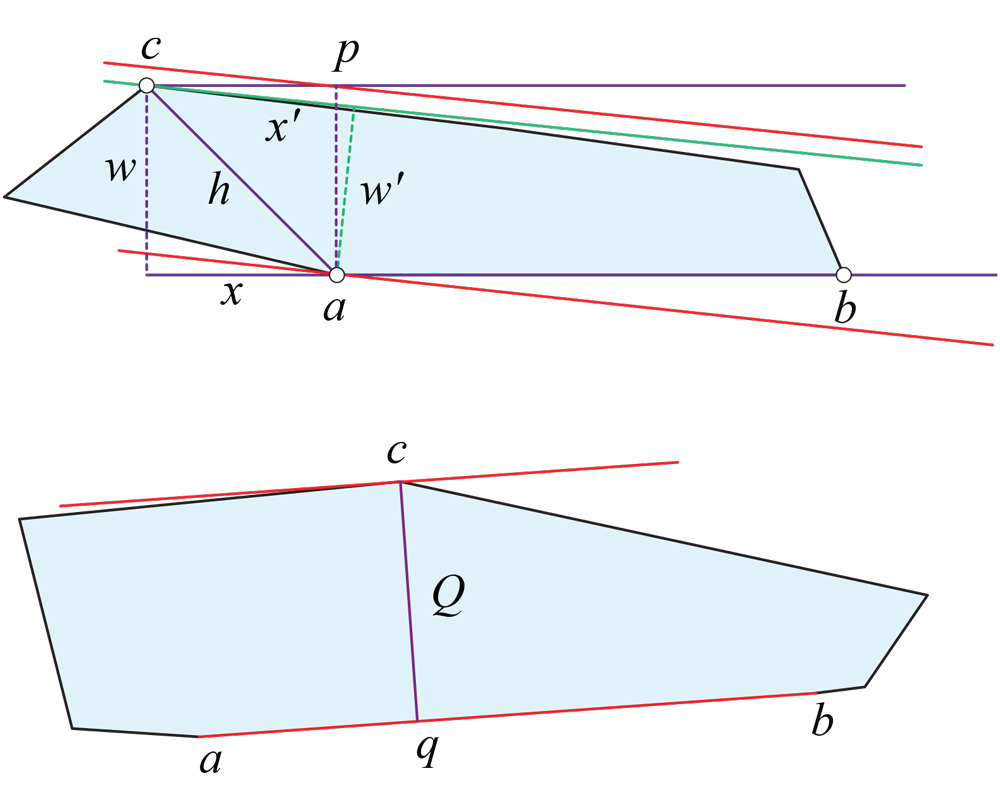}
\caption{(a)~Vertex $c$ projects outside $ab$.
(b)~$cq$ realizes the width.}
\figlab{WidthPoly}
\end{figure}

\begin{lm}
\lemlab{QuasiWidthNoParallel}
Every doubly-covered convex polygon $P$ with no parallel edges
has a quasigeodesic through one vertex.
\end{lm}

\begin{proof}
Let $Q$ be the segment from a vertex $c$ orthogonal to an edge $ab$ at a point $q$,
as guaranteed by Lemma~\lemref{Width} and illustrated in Fig.~\figref{WidthPoly}(b).
The angles on the doubled polygon
at $q$ are both $\pi$, and the left and right angles at $c$
are both $< \pi$
(strictly less than because of the no-parallel-edges assumption).
Therefore $Q=cq$ is a quasigeodesic.

If $q=a$, i.e., if $c$ projects to an endpoint of $ab$, then
the PDR lemma shows that $c$ and $ab$ could not have realized the width;
see Fig.~\figref{Parallelogram}(a).
Both directions of rotation diminish the separation, with one direction blocked
by the flush edge.
Therefore, $q$ must project into the interior of $ab$, and $Q$ includes just the one
vertex $c$.
\end{proof}

Although not needed for Theorem~\thmref{Q_012}, for completeness
we also handle parallel edges.

\begin{lm}
\lemlab{QuasiWidthParallel}
If the width of a convex polygon $P$ is achieved on parallel edges,
then the doubly covered $P$
has a simple closed geodesic.
\end{lm}
\begin{proof}
Let the width of $P$ be realized by edges $ab$ and $cd$.
If the projection of $cd$ onto $ab$ is a positive-width interval,
then a closed geodesic can be achieved by a pair of points,
one on each edge, strictly interior to this interval.

If instead the projection interval is a point, then that point corresponds
to a vertex-to-vertex projection, say $c$ to $b$ in
Fig.~\figref{Parallelogram}(b).
But in this circumstance, again the PDR lemma
shows that $|cb|$ cannot have been the width,
as rotation of the supporting lines toward the $< \pi/2$ side of $b$ and $c$
(counterclockwise in the figure) reduces the separation between the lines.
\end{proof}
\begin{figure}[htbp]
\centering
\includegraphics[width=\linewidth]{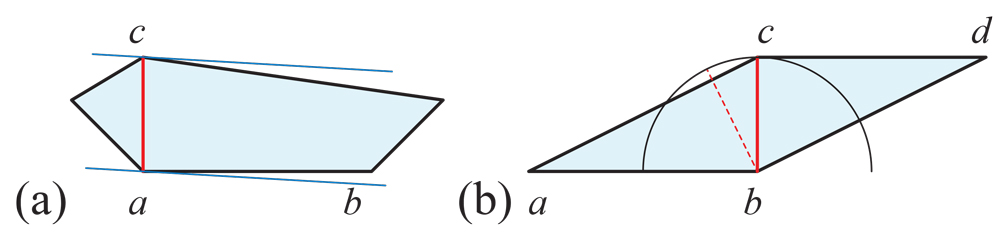}
\caption{Neither $ca$ (a) nor $bc$ (b) realizes the width.}
\figlab{Parallelogram}
\end{figure}

Note that the previous two lemmas together ensure that every doubly-covered
convex polygon has a quasigeodesic through at most one vertex.

The following lemma is not needed for our main goal, but 
might possibly help in resolving Open Problem~\openref{QVert2}.

\begin{lm}
\lemlab{ShortestQ}
The width-quasigeodesic identified in 
Lemmas~\lemref{QuasiWidthNoParallel}--\lemref{QuasiWidthParallel} 
is the shortest quasigeodesic for $P$.
\end{lm}

\begin{proof}
Call the width-quasigeodesic $Q_w$, and the width $w$.
So $\ell( Q_w ) = 2 w$.
First we show there is no shorter $1$-vertex quasigeodesic.
Suppose $Q$ were such a quasigeodesic through $v$. It must be orthogonal
to an edge $e$ of $P$, and then the line parallel to $e$ through $v$ is
a supporting line for $P$. Therefore, if $Q$ were shorter than $Q_w$, the
width of $P$ would not be $w$.

Suppose now there is a shorter diagonal-quasigeodesic, connecting
vertices $a$ and $b$. Then lines orthogonal to $ab$ through $a$ and $b$
must be supporting lines, to satisfy the angle constraints needed for
the quasigeodesic to be convex to both sides at $a$ and $b$.
This again would contradict $Q_w$ realizing the width.

A $0$-vertex quasigeodesic $Q$ must cross parallel edges of $P$.
It could be slid left or right until it touches a vertex, in which case it is now
a $1$-vertex quasigeodesic, and the argument above applies.
\end{proof}


\subsection{$|V(Q)|=2$}

\begin{lm}
Every doubly covered polygon has a simple closed quasigeodesic $Q$ with $|V(Q)|=2$
\end{lm}

\begin{proof}
Each extrinsic diameter of a convex polyhedron $P$, 
$$\mathrm{diam}\left( P \right): = \max_{x,y \in P}|x-y|,$$
is realized between two vertices of $P$,
see for example Proposition~1 in~\cite{IRV-tetrahedra}.
Because of its length maximality, such a diameter provides, for doubly covered polygons, the desired quasigeodesics through $2$ vertices.
\end{proof}


\subsection{$|V(Q)|=k$, with $3\leq k \leq n$}

For these cases, we construct particular examples
of polyhedra $P_n$ of $n$ vertices each of which has at least one
quasigeodesic through $k$ vertices.

\begin{lm}
For every $3\leq k \leq n$ there exist doubly covered $n$-gons having simple closed quasigeodesics $Q_k$ with $|V(Q_k)|=k$ and satisfying Eq.~\eqref{ab-strict}.
\end{lm}

\begin{proof}
Consider a regular $k$-gon $R_k=v_1 \ldots v_k$, where $3\leq k \leq n$.
Figure~\figref{HexagonQk} illustrates the case $k=6$.
\begin{figure}[htbp]
\centering
\includegraphics[width=0.7\linewidth]{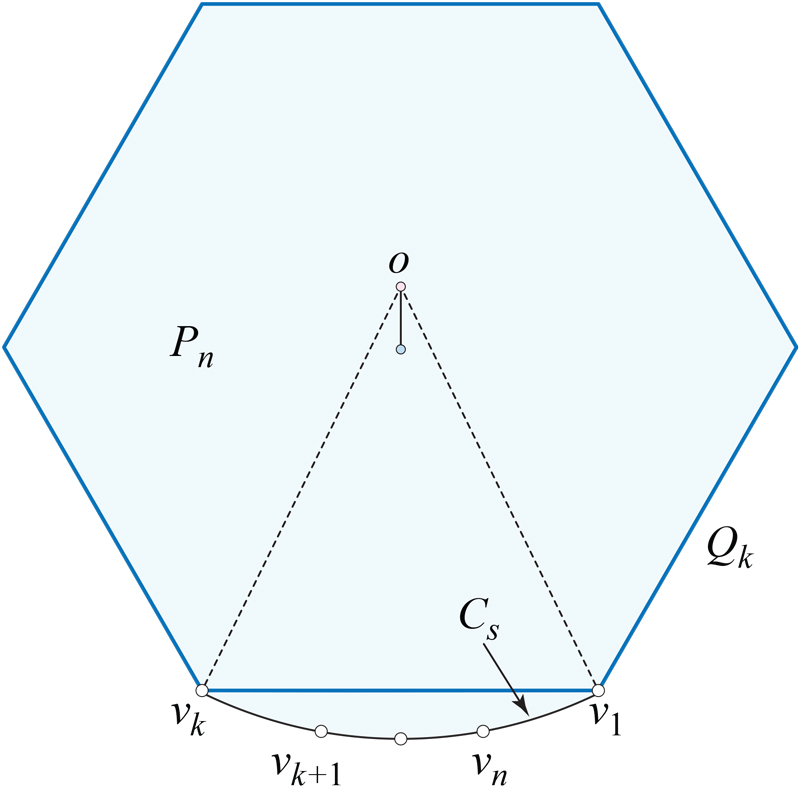}
\caption{$Q_6$ is the blue hexagon.
$C_s$ is centered at $o$ above the hexagon center.}
\figlab{HexagonQk}
\end{figure}

Also consider the circle $C$ through $v_1$ and $v_k$, of centre $o$ on the same side of $v_1 v_k$ as $R_k$.
On the short arc $C_s$ of $C$ determined by $v_1$ and $v_k$, choose points $v_{k+1}, \ldots, v_n$, to increase the polygon to $n$ vertices, 
and denote by $P_n$ the doubly covered polygon $v_1 \ldots v_n, v_1$.

When the center $o$ is beyond the center of $R_k$ as illustrated,
the measure (length) of $m(C_s)$ satisfies $m(C_s)<2\pi/k$.
It follows that $P_n$ has two simple closed quasigeodesics $Q_k$ with $k$ vertices, both corresponding to the polygon $v_1 \ldots v_k, v_1$,
one including edge $v_1 v_k$ on the front, and one on the back.

Indeed, for $Q_k$
\begin{itemize}
\item the angles $\a_i$ are each the vertex angles of $R_k$, 
equal to $\frac{(k-2)}k \pi < \pi$, for $1 \le i \le k$.
%
\item the angles $\b_i=\a_i$ for $1<i<k$.
\item For $i=1$, $\b_1$ is $\a_1$ plus twice $\angle v_k v_1 v_n$, and $\angle v_k v_1 v_n < m(C_s)/2$. Similarly for $i=k$.
Therefore, for $i \in \{1,k\}$,
$$\b_i < \a_i + m(C_s)<\frac{k-2}k \pi + \frac 2k \pi < \pi \;.$$
\end{itemize}
Thus $Q_k$ is a quasigeodesic satisfying Eq.~\eqref{ab-strict} 
of Lemma~\lemref{Q_on-neighborhood}.
\end{proof}

Returning to Theorem~\thmref{Q_012},
we have identified polygons with 
$|V(Q)|=k$, for $1 \leq k \leq n$,
and applying Lemma~\lemref{Q_on-neighborhood}
creates ${\cal Q}_k$ with nonempty interiors.
This establishes Theorem~\thmref{Q_012}.



%
%
%
%
%


\chapter{Conclusions\\ and Open Problems}
\chaplab{OpenII}
%

Our work leaves open several questions of various natures, most of which have been mentioned in the text in some form.
Here we list open problems and indicate possible directions for future research.
The exposition roughly follows the order we treated the subjects, with the comments on both parts joined in this 
final chapter.

\section{Part~I}

In the first part of this work, we mainly studied properties of the tailoring operation on convex polyhedra.

\begin{itemize}[leftmargin=*] 
\item 
We have presented three methods for tailoring, of very different flavors.
On one hand, the methods of tailoring with sculpting given in 
Chapters~\chapref{TailoringSculpting} 
and~\chapref{Crests}---digon-tailoring or crest-tailoring---seem 
appropriate for local tailoring, and can produce any $Q$
inside $P$.

On the other hand, the method of tailoring via flattening
presented in Chapter~\chapref{TailoringFlattening} is purely intrinsic, in that it doesn't need the spatial structure of $P$ and $Q$ to work. 
The surfaces can be given, in this case, as a collection of polygons glued together as in AGT.
But it has the disadvantage of being ``non-economical,'' in the sense that it discards a lot of 
$P$'s surface area. 

Even with ``surface-removal optimal" tailoring, we could be forced to lose almost all the
surface area of $P$ to reshape to $Q$, for example 
if approximating by tailoring a sphere inscribed in a very long convex surface.

All three methods can be reversed to enlarge surfaces, where the results are the same, but not-requiring the spatial structure might be a clear advantage.

\item 
We made little attempt to optimize our algorithm complexities,
resting content with polynomial-time upper bounds.
Likely several algorithms could be improved, or lower bounds established.
Especially notable is this problem, which dominates the complexity of 
the tailoring-via-flattening algorithm, Theorem~\thmref{FlatteningAlg}.

\begin{open}
Given a polyhedron $P$ of $n$ vertices, find a
cut-locus generic point $x$ in less than $O(n^4)$ time.
\openlab{CutLocusGeneric}
\end{open}


Such a generic point has a unique geodesic segment (shortest path) to each vertex of $P$.

%


\item Recall that Lemma~\lemref{Slice2gDomes}
established that, with a slice along the plane of a face $X$ of $Q \subset P$,
the portion of $P$ sliced-off can be partitioned into a fan of g-domes.
Then Theorem~\thmref{DomePyr} showed how to reduce a g-dome to
its base by removing pyramids.

Rather than repeating this for each face $X$ of $Q$, it is conceivable
that a single base $X$ and a single g-dome suffice.

\begin{open}
Let $X$ be a ``base" face of $P$, and $D$ a fixed g-dome over $X$, interior to $P$.
Is it possible to partition $P$ into pyramids and $D$, with planes through the edges of $X$?
After each sectioning we remove the sliced pyramid.
\openlab{OneXgdome}
\end{open}

If g-dome $D$ is replaced with an arbitrary interior polyhedron, then the answer
is easily seen to be {\sc no}.
Without the restriction to planes through base edges,
the answer is {\sc yes} as shown by 
Lemma~\lemref{Slice2gDomes}.

\item 
As discussed in
Section~\secref{OtherDigonOrderings},
the proof of Theorem~\thmref{SealGraphTree}---that the seal graph $\S$
for a pyramid is a tree---depends on the ordering of digon removal.
\begin{open}
Is the seal graph for a pyramid a tree for other orderings of digon removal?
\openlab{OtherDigonOrderings}
\end{open}

\item 
We have shown in Chapter~\chapref{Punfoldings} 
that different unfoldings of $Q \subset P$ onto $P$ exist. 
Clearly, the $P$-unfolding obtained via tailoring
depends on the order of tailoring operations.
The $P$-unfolding of $Q$ does not necessarily have connected interior,
which suggests this problem:
\begin{open}
Is there some method and/or orderings of tailoring operations that would render 
interior-connected a $P$-unfolding of arbitrary $Q \subset P$? 
A less ambitious goal would be to minimize the number of interior-connected pieces.
\openlab{PunfConnected}
\end{open}

Notice that Theorem~\thmref{QPflat} established simply connectedness in general.
But in our usage, simply connected does not imply connected, e.g., the union of several disjoint disks is a disconnected set but simply connected.
In Part~II we pursue this question for particular $P$, for example, a doubly-covered triangle.

\item
Theorem~\thmref{ContinuousFolding} established the \emph{existence} of a continuous
folding of $P$ onto $Q$ when $Q \subset P$, but we do not know of an algorithm:
\begin{open}
Is there a finite algorithm that determines a continuous unfolding of $P$ onto $Q$?
\openlab{Continuous}
\end{open}

\item 
It seems that at least a part of the present work could apply to \emph{$1$-polyhedra}.
These are polyhedra whose faces are (congruent to) geodesic polygons on the unit sphere. 
They can approximate convex surfaces with curvature bounded below by $1$
(in the sense of A.~D.~Alexandrov), 
just as convex polyhedra can approximate ordinary 
convex surfaces~\cite{az-igs-67},~\cite{itoh2015moderate}.

\begin{open}
How much of the presented results can be carried over 
and applied to $1$-polyhedra?
\openlab{1Polyhedra}
\end{open}

\item 
One could also define tailoring for general (i.e., not necessarily convex) polyhedra, of arbitrary topology.
Of course, the methods we developed here apply locally to convex caps. Globally, 
a necessary condition for $Q$ to be tailored from a homothetic copy of $P$ is to have the same topology as $P$. 
Our Theorem~\thmref{MainTailoring} might suggest this is also sufficient, but that is not true.
By Alexandrov's Gluing Theorem (AGT), tailoring a convex polyhedron always produces a convex polyhedron, 
never a nonconvex polyhedron homeomorphic to the sphere but having negative curvature at some vertex.
There is as yet no counterpart to AGT for nonconvex polyhedra.
Therefore, in the general framework, tailoring could be a much subtler topic.

\begin{open}
Is there a type of tailoring operation that permits global reshaping of
non-convex polyhedra?
\openlab{NonConvex}
\end{open}
\end{itemize}



\section{Part~II}
In the second part of this work, we mainly studied the vertex-merging (as opposed to tailoring) processes on convex polyhedra,
with the aim of producing ``nice'' polyhedral and planar unfoldings. 
Toward that aim, we needed to develop a theory of convex sets on convex polyhedra.

\begin{itemize}[leftmargin=*] 
\item 
For any convex polyhedron $P$, there clearly exist vertex-merging reductions of $P$
that increase the surface area by the least, or the most amount.
\begin{open}
\openlab{SurfaceArea}
Find upper and lower bounds on the added surface area for vertex-merging reductions of $P$.
Find the reduction yielding the minimal, and respectively maximal area. 
\end{open}
\item
General relationships between the properties of the slit graph and those of the corresponding unfolding of $P$ are presented in Theorem~\thmref{Connectivity}.
Finding particular vertex-mergings with ``good'' behaviour seems to be a difficult task, only partly fulfilled in this study.

\begin{open}
\openlab{NoSlitCycle}
Does there exist for every convex polyhedron $P$ a vertex-merging reduction 
onto a vm-irreducible surface
whose slit graph has no cycle? If not, for which $P$ are there such vm-reductions?
\end{open}
\item
A positive answer to the above question raises another one, related to our discussion in Section~\secref{Unf2xTri}.
\begin{open} 
\openlab{NotDisconnect}
For any simple unfolding $P_S$ of $P$ onto a vm-irreducible surface $S$, 
does there exists an unfolding of $S$ in the plane which results in a net for $P$?
\end{open}
\item
Our attempt to answer Open Problem~\openref{NoSlitCycle} is based on a spiral-merging idea.
In order to formalize our spiraling algorithms, we needed to develop a theory of convex sets on convex polyhedra.
We only proved some basic results in this theory, and much remains for future study.
In particular, the next question is related to Lemma~\lemref{S>Q} for the case when $S \supseteq Q$.\footnote{
Recall that in Part~II, we use $Q$ to denote a simple closed quasigeodesic,
in contrast to Part~I's use of $Q$ to represent the target polyhedron.}
\begin{open}
\openlab{ConvHalfSurf}
Is it the case that every closed convex subset $S$ of a convex polyhedron
is either included in a half-surface bounded by a simple closed quasigeodesic $Q$, or is the whole surface?
\end{open}
\item
Of the classical results in the combinatorial theory of convex sets bearing a name, we only adapted to our framework
the Krein–Milman Theorem
(that a convex set is the convex hull of its extreme points);
see Theorem~\thmref{ExtPts}.
Perhaps the statement of that result, even though suitable to our purpose, could be improved.
\begin{open}
\openlab{ExtPts}
Is Theorem~\thmref{ExtPts} concerning extreme points and the relative convex hull,
still true without the ``relative" modifier?
\end{open}
\item

Examples~\exref{NoHelly} and~\exref{NoRadon} show that the precise planar versions
of Helly's and Radon's theorems are false on the surface of a convex polyhedron.
It seems worth studying if there are versions of these theorems, and of
Carath{\'e}odory's theorem, in our context.

\begin{open}
\openlab{HellyEtc}
\begin{sloppypar}
Are there versions of Helly's,
Radon's, and Carath\'eodory's theorems
that hold on the surface of a convex polyhedron,
using ag-convexity?
\end{sloppypar}
\end{open}
The Radon point of any four points in the plane is their geometric median, the point that minimizes the sum of distances to the other points~\cite{clarkson1996approximating}.
Therefore, it could also be of some interest to study such points in our framework.
Note, for example, on a doubly covered square there are two such points, the centers of the two faces.

\item
Theorem~\thmref{RollingNet} established that if a polyhedron $P$ has
a quasigeodesic including at most two vertices, vm-reduction can lead
to a net for $P$.
One might view the reduction process as moving $P$ closer and closer to
planarity, as $P_i$ has fewer and fewer vertices. Finally $P$ is on a cylinder
which rolls to a net.

Little seems known about the structure of quasigeodesics.
Pogorelov's original proof is nonconstructive~\cite{p-qglcs-49},~\cite{p-egcs-73},
and it has long been an open
problem to design an algorithm to find a simple closed quasigeodesic
(Open Problem~24.2~\cite{do-gfalop-07}).
It is not difficult to find a plane $\Pi$ through one vertex whose slice
curve $C=\Pi \cap P$ halves the curvature as needed.
But $C$ is not necessarily convex~\cite{o-dipp-03}, and so
in general is not a quasigeodesic.

\begin{sloppypar}
However, there has been recent progress on quasigeodesics.
First, a pseuopolynomial-time algorithm was designed that
finds a possibly self-crossing quasigeodesic~\cite{kane2009pseudopolynomial}.
Second, an exponential algorithm for finding all the
simple closed quasigeodesics was described in~\cite{ChartierArnaud}.
Despite this progress, there remains no practical algorithm for finding
simple closed quasigeodesics.
\end{sloppypar}

\begin{open}
\openlab{QVert2}
We conjecture that every convex polyhedron either has a simple closed geodesic,
or a simple closed quasigeodesic through just one vertex.
\end{open}

Note that this conjecture is verified for doubly-covered convex polygons,
by Lemmas~\lemref{QuasiWidthNoParallel}
and~\lemref{QuasiWidthParallel}.
We recently proved the conjecture for tetrahedra~\cite{QonT}.
\end{itemize}

\bibliographystyle{alpha}
\addcontentsline{toc}{chapter}{Bibliography}
\label{Bibliography}
\bibliography{refs}

\end{document}